\documentclass[11pt]{article}

\usepackage[utf8]{inputenc}
\usepackage[english]{babel}
\usepackage[font=small]{caption}
\usepackage{amsmath,amssymb,amsfonts,amsthm}
\usepackage{mathrsfs} 
\usepackage{mathabx} 
\usepackage{bbm}
\usepackage{graphicx}
\usepackage{hyperref}
\usepackage[hmargin=2.5cm,vmargin=2cm]{geometry}
\usepackage{calc}

\usepackage{enumitem}

\usepackage{color} 
\usepackage{longtable}

\theoremstyle{plain}
\newtheorem{theorem}{Theorem}[section]
\newtheorem{lemma}[theorem]{Lemma}
\newtheorem{corollary}[theorem]{Corollary}
\newtheorem{proposition}[theorem]{Proposition}

\theoremstyle{definition}

\theoremstyle{remark}
\newtheorem{remark}[theorem]{Remark}

\newcommand{\tconst}[1]{\ensuremath{t_{\textnormal{\ref{#1}}}}}
\newcommand{\const}[1]{\ensuremath{c_{\textnormal{\ref{#1}}}}}

\def\N{\ensuremath{\mathbf{N}}}
\def\Q{\ensuremath{\mathbf{Q}}}
\def\R{\ensuremath{\mathbf{R}}}
\def\Z{\ensuremath{\mathbf{Z}}}

\def\ep{\varepsilon}

\def\E{\ensuremath{\mathbf{E}}}
\def\P{\ensuremath{\mathbf{P}}}
\DeclareMathOperator{\Var}{Var}
\def\F{\ensuremath{\mathscr{F}}}

\def\Ind{\ensuremath{\mathbf{1}}}

\renewcommand\Re{\operatorname{Re}}

\DeclareMathOperator{\supp}{supp}
\DeclareMathOperator{\med}{qu}

\def\to{\rightarrow}

\def\tand{\ensuremath{\text{ and }}}
\def\tif{\ensuremath{\text{ if }}}
\def\tas{\ensuremath{\text{ as }}}
\def\ton{\ensuremath{\text{ on }}}
\def\tor{\ensuremath{\text{ or }}}

\newcommand{\dd}{\mathrm{d}} 

\newcommand{\executeiffilenewer}[3]{%
\ifnum\pdfstrcmp{\pdffilemoddate{#1}}%
{\pdffilemoddate{#2}}>0%
{\immediate\write18{#3}}\fi%
}

\def\A{\ensuremath{\mathscr{A}}}
\def\G{\ensuremath{\mathscr{G}}}
\def\H{\ensuremath{\mathscr{H}}}
\def\p{\ensuremath{\mathbf{p}}}
\def\Ptilde{\widetilde{\P}}
\def\Etilde{\widetilde{\E}}

\def\Phat{\widehat{\P}}
\def\Ehat{\widehat{\E}}

\def\Varhat{\widehat{\Var}}

\def\thbar{\overline{\theta}}
\def\Bfl{B$^\flat$}
\def\Bsh{B$^\sharp$}

\def\numcst{b}
\def\PB{\mathbf{P}^{\mathrm{B}}}
\def\EB{\mathbf{E}^{\mathrm{B}}}
\def\VarB{\Var^{\mathrm{B}}}
\def\Pfl{\mathbf{P}^{\flat}}
\def\Pfltilde{\widetilde{\mathbf{P}}^{\flat}}
\def\Efl{\mathbf{E}^{\flat}}
\def\Efltilde{\widetilde{\mathbf{E}}^{\flat}}
\def\Varfl{\Var^\flat}
\def\Psh{\mathbf{P}^{\sharp}}
\def\Esh{\mathbf{E}^{\sharp}}
\def\Varsh{\Var^\sharp}

\numberwithin{equation}{section}

\author{Pascal Maillard\thanks{Département de Mathématiques, Université Paris-Sud, 91405 Orsay Cedex, France. \newline \hspace*{\parindent} e-mail: pascal DOT maillard AT u-psud DOT fr}}
\title{Speed and fluctuations of $N$-particle branching Brownian motion with spatial selection}

\begin{document}

\maketitle

\begin{abstract}
We consider branching Brownian motion on the real line with the following selection mechanism: Every time the number of particles exceeds a (large) given number $N$, only the $N$ right-most particles are kept and the others killed. After rescaling time by $\log^3N$, we show that the properly recentred position of the $\lceil \alpha N\rceil$-th particle from the right, $\alpha\in(0,1)$, converges in law to an explicitly given spectrally positive L\'evy process. This behaviour has been predicted to hold for a large class of models falling into the universality class of the FKPP equation with weak multiplicative noise [Brunet et al., Phys. Rev. E \textbf{73}(5), 056126 (2006)] and is proven here for the first time for such a model.

\bigskip


\bigskip

\end{abstract}

\tableofcontents

\section{Introduction}
\label{sec:NBBM_intro}

\subsection{Definition of the model and statement of the main result}

In this article, we consider an instance of branching Brownian motion (BBM) with \emph{selection}, dubbed the $N$-BBM and defined as follows: Given a probability measure $(q(k))_{k\ge 0}$ on $\N = \{0,1,2,\ldots\}$, called the \emph{reproduction law}, with $m=\sum (k-1)q(k) > 0$ and finite second moment, particles diffuse according to standard Brownian motion and branch at rate $\beta q(k)$ into $k$ particles, for every $k\ge0$, where $\beta = 1/(2m)$ is called the \emph{branching rate}. We fix a (large) parameter $N$ and select particles according to the following simple mechanism: Each time the number of particles exceeds $N$, we keep only the $N$ right-most and instantaneously kill the others. In Section~\ref{sec:motivation}, we explain the motivation behind this system. We first state our result.

For a finite counting measure $\nu$ on $\R$, $\alpha\in(0,1)$ and $N\in\N$, we define
\begin{equation}
 \label{eq:def_med}
\med^N_\alpha(\nu) = \inf\{x\in\R: \nu([x,\infty)) < \alpha N\}.
\end{equation}
For $N\in\N$ large enough, we then define $a_N = \log N + 3 \log\log N$ and
\[
\mu_N = \sqrt{1-\frac{\pi^2}{a_N^2}} = 1 - \frac{\pi^2}{2\log^2N} +  \frac{3\pi^2\log\log N}{\log^3 N} + o\Big(\frac 1 {\log^3 N}\Big).
\]
Let $\nu^N_t$ be the counting measure formed by the positions of the particles of $N$-BBM at time~$t$. Define $$M^N_\alpha(t) = \med_\alpha^N(\nu^N_t) - \mu_N t.$$  Our main theorem is the following:

\begin{theorem}
  \label{th:1}
Suppose that at time $0$ there are $N$ particles distributed independently according to the density proportional to $\sin(\pi x/a_N)e^{- x}\Ind_{(0,a_N)}(x)$. Let $\alpha\in(0,1)$ and define $x_\alpha>0$ by $\int_{x_\alpha}^\infty  y e^{- y}\,\dd y = \alpha.$ Then the finite-dimensional distributions of the process 
\(
\left(M^N_\alpha\big(t\log^3 N\big)\right)_{t\ge0}
\)
converge as $N\to\infty$ to those of the L\'evy process $(L_t+x_\alpha)_{t\ge0}$ with $L_0 = 0$ and 
\begin{equation}
\label{eq:laplace_levy}
 \log E[e^{i\lambda L_1}] = i\lambda c + \pi^2\int_0^\infty (e^{i\lambda x} - 1 - i\lambda x\Ind_{(x\le 1)})\,\Lambda(\dd x).
\end{equation}
Here, $\Lambda$ is the image of the measure $x^{-2}\Ind_{(x> 0)}\dd x$ by the map $x\mapsto \log(1+x)$ and $c\in\R$ is a constant depending only on the reproduction law $q(k)$ (in fact, $c = \pi^2(\const{eq:W_expec} + \const{eq:c_log} - \log(2\pi^2))$, with the constants from \eqref{eq:W_expec} and \eqref{eq:c_log}, respectively).
\end{theorem}

\subsection{Motivation and related work}
\label{sec:motivation}

The $N$-BBM and the results of this article are related to several other mathematical models, a few of which we wish to outline in this section. We are not going to elaborate on the applications to other sciences such as physics or biology, for which we refer to \cite{Brunet2006,Brunet2008a} and the review articles \cite{VanSaarloos2003,Panja2004}.

Most importantly, the $N$-BBM is a prototype for \emph{noisy travelling waves of FKPP type}, i.e.\ travelling waves of the stochastic partial differential equation $\partial_t u = \partial_x^2 u + u(1-u) + \sqrt {u/N}\, \dot W$, where $\dot W$ is space-time white noise. This equation is known as the FKPP equation with weak multiplicative noise and is believed to share the same   phenomenology as the $N$-BBM. The connection between BBM \emph{without} selection and the FKPP equation \emph{without} noise is indeed well known, at least since McKean's work \cite{McKean1975}, and has had many fruitful applications, one of the most important being Bramson's study of the law of the right-most particle of BBM \cite{Bramson1983}. The relation between the $N$-BBM and the noisy FKPP equation is less explicit, but there is an exact duality relation between this equation and a system of branching and coalescing Brownian motions \cite{Shiga1988}. More generally, particle systems governed by the FKPP equation in the large 
population limit can often be modelled by the noisy FKPP equation, the parameter $N$ having the meaning of an ``effective population size'' (\cite{Brunet2001}, \cite[Chapter~7]{VanSaarloos2003}, \cite{Panja2004}).

Another example of BBM with selection is BBM with absorption at a linear space-time barrier. This process is well-studied \cite{Kesten1978,Neveu1988,HHK2006,Derrida2007,Gantert2011,Berestycki2010,Berestycki2010a} and is much more tractable than $N$-BBM due to the greater independence between the particles and its connection with some differential equations \cite{Neveu1988,HHK2006,Maillard2010}. It has been known since Kesten \cite{Kesten1978} that this process dies out almost surely if and only if the slope of the barrier is at least 1 (with our choice of parameters). Of particular importance to this article is the near-critical case, where the slope of the barrier is slightly smaller than one. This case has been studied in detail in the literature \cite{Derrida2007,Gantert2011,Berestycki2010,Berestycki2010a}. Our article draws a lot on these results, especially on \cite{Berestycki2010}.

We finally mention the relation between the $N$-BBM or $N$-BRW and an instance of a Fleming--Viot-type process: Let $N$ particles perform independent continuous-time nearest-neigh\-bour random walks on the integers with drift towards the origin. Furthermore, as soon as a particle reaches the origin, let it jump at random onto one of the other particles. This type of model was introduced in order to yield a particle representation of quasi-stationary distributions \cite{Burdzy1996}, and it is indeed conjectured that the empirical distribution of this system under the stationary distribution converges, as $N$ goes to infinity, to the quasi-stationary distribution of the random walk (see e.g.\ \cite{Asselah2012,Asselah2012a,Villemonais2014} for recent progress on this subject). The relation with $N$-BRW (with binary branching) is obvious.

\subsection{Comparison of Theorem~\ref{th:1} with previous results and discussion}

The quantitative study of $N$-BBM and similar models has started with the seminal article by Brunet and Derrida \cite{Brunet1997}. There, the authors define a variant of the FKPP equation with ``cutoff'', which (heuristically) models the effects of an $N$-particle discretisation on the solutions to the FKPP equation. In a later article with co-authors  \cite{Brunet2006}, they extended this model with random shifts, supposed to model the fluctuations of the $N$-BBM. In terms of the $N$-BBM, this yielded the following heuristic semi-deterministic description of the system:
\begin{enumerate}
 \item Most of the time, the particles are in a meta-stable state. In this state, the cloud of particles (also called the \emph{front}) has a diameter of $\log N+O(1)$, its empirical density of particles seen from the left-most is proportional to $e^{- x} \sin(\pi  x/\log N)$, and the system moves at a linear speed $v_{\mathrm{cutoff}} = 1 - \pi^2/(2\log^2N)$. In particular, most of the particles are at $O(1)$ distance from the left-most particle. This is the description provided by the \emph{cutoff} approximation from \cite{Brunet1997} mentioned above.
 \item This meta-stable state is perturbed from time to time by particles moving far to the right. Fix a point in the bulk, for example the median or the barycentre, and say that a particle reaches a point $y$ if it moves to distance $y$ of that fixed point. Playing with the initial conditions of the FKPP equation with cutoff, the authors of \cite{Brunet2006} find that a particle moving up to the point $\log N + x$ causes, after a relaxation time of order $\log^2 N$, a shift of the front by
\(
 \Delta = \log \Big(1+\frac{Ce^{x}}{\log^3 N}\Big),
\)
for some constant $C>0$. In particular, in order to have an effect on the position of the front, a particle has to reach a point near $\log N + 3 \log \log N$.
\item Assuming that such an event where a particle ``escapes'' to the point $\log N + x$ happens at rate proportional to $e^{-x}$, one sees that the time it takes for a particle to come close to $\log N + 3 \log \log N$ (and thus causing shifts of the front) is of the order of $\log^3 N$, which is greater than the relaxation time.
\item With this information, the speed of the front is found to satisfy \(v_N - v_{\mathrm{cutoff}} \approx \pi^2 \frac{3\log\log N}{\log^3 N}\) and the fluctuations are found to be given by the L\'evy process from Theorem~\ref{th:1}.
\end{enumerate}

As mentioned above, this description relies on the validity of the approximation of the $N$-BBM by the FKPP equation with cutoff and certain additional random shifts. There have been several attempts to render these results rigorous. B\'erard and Gou\'er\'e \cite{Berard2010} prove that the speed\footnote{By \emph{speed} we mean here the limit $v_N = \lim_{t\to\infty} X_t/t$, where $X_t$ is the barycentre, say, at time $t$.} of a general $N$-BRW satisfies $(1-v_N)/\log^2 N\to c$ for some explicit constant $c$. An important ingredient of their proof is a certain coupling with BRW with absorption at a linear space-time barrier, studied in \cite{Gantert2011}. The empirical distribution of $N$-BRW has been studied by Durrett and Remenik \cite{Durrett2009}, who show that if the empirical distribution of the initial particles converges to a deterministic limit, then the evolution of the empirical distribution of the $N$-BRW converges to the solution of a free-boundary partial integro-differential equation. 
Finally, Mueller, Mytnik and Quastel \cite{Mueller2010} study the propagation speed of solutions to the noisy FKPP equation, which is heuristically related to the $N$-BBM as mentioned in the last section. They show that this propagation speed satisfies $v_N = 1 - \pi^2/(2\log^2N) + O (\log\log N/\log^3 N)$ (again with our choice of parameters), which partly confirms the physicists' predictions.

The present article is therefore to the knowledge of the author the first example of a rigorous proof of the full statistics of $N$-BBM or of a related model. Naturally, Theorem~\ref{th:1} has some important limitations. 
First, it only describes the model on time scales of order $\log^3N$. It is still possible (but in our opinion unlikely), that on larger time scales other behaviour might be observed, leading to a different speed or larger fluctuations. 
Second, the requirement on the initial condition in Theorem~\ref{th:1} is very strong. The results should certainly hold for every initial condition (modulo an additional random shift determined by the beginning of the process). However, we were not able to show this; notably the study of the \Bfl- and \Bsh-BBM (see below) highly depends on the assumption of independence of the initial particles. 
Third, the convergence in finite-dimensional distributions in Theorem~\ref{th:1} is certainly not optimal; one could expect convergence in a suitable function space, possibly w.r.t.\ Skorokhod's $M_1$-topology\footnote{One cannot expect convergence in ``the'' Skorokhod topology, also called the $J_1$ topology, since the jumps of the L\'evy process from Theorem~\ref{th:1} are no real jumps in the $N$-BBM, but rather gradual shifts in the position over the time scale $\log^2N$.
}. 
Fourth, Theorem~\ref{th:1} does not say anything about the stationary distribution of the $N$-BBM, seen from the left-most particle, say. It has been shown by Durrett and Remenik \cite{Durrett2009} that this stationary distribution exists (even for general $N$-BRW), but it is still unknown what it looks like and in particular, how it behaves in the large $N$ limit. By analogy with the $N$-particle Fleming--Viot process mentioned in the last section, we expect the empirical measure to be deterministic under this limit, namely, the quasi-stationary distribution of the underlying random walk with a certain drift towards the origin. 
Finally, we do not consider the genealogy of the $N$-BBM in this article. By analogy with the results in \cite{Berestycki2010}, we expect it to be governed in the large $N$ limit by the Bolthausen--Sznitman coalescent (at the timescale $\log^3 N$). Unfortunately, our methods do not seem to be immediately applicable (see the end of Section~\ref{sec:overview} for details).

\subsection{Overview of the proof}
\label{sec:overview}

Our approach to the proof of Theorem~\ref{th:1} is inspired by the ideas of Berestycki, Berestycki and Schweinsberg \cite{Berestycki2010}, who consider BBM with absorption at the origin and with near-critical drift $-\mu_N$. They show that the evolution of the number of particles in this process converges on the $\log^3 N$ time scale to Neveu's continuous-state branching process\footnote{A continuous-state branching process (CSBP) $(Z_t)_{t\ge0}$ is a time-changed L\'evy process without negative jumps: at time $t$, time is sped up by the factor $Z_{t-}$. CSBPs are scaling limits of Galton--Watson processes and thus have an inherent notion of genealogy. Neveu's CSBP is the CSBP with L\'evy measure $x^{-2}\Ind_{(x>0)}\,\dd x$, whose genealogy is given by the Bolthausen--Sznitman coalescent \cite{Bertoin2000}.} and the genealogy of the system to the Bolthausen--Sznitman coalescent. We will briefly recall the basic ideas of their proof.

Their starting point is to introduce a second barrier at the point\footnote{\label{fn:aNA}They actually consider $a_{N,A} = a_N -  A$ for some large positive constant $A$, which slowly goes to infinity  with $N$. This is an important detail, but we ignore it for the moment.} $a_N = \log N + 3 \log\log N$ and divide the particles at time $t$ into two parts; on the one hand those that have stayed inside the interval $(0,a_N)$, on the other hand those that have hit the point $a_N$ before hitting $0$. This corresponds roughly to the division of the process into a deterministic and a stochastic part. Indeed, killing the particles at $a_N$ prevents the number of particles from fluctuating too much and thus permits to estimate the particle density \emph{via} first and second moment calculations. For example, if at time $0$ we have $N$ 
particles distributed according to the meta-stable density, then the variance of the number of particles at time $\log^3 N$ is of order $N^2$.
The particles inside the interval $(0,a_N)$ therefore behave almost deterministically at the time scale $\log^3 N$. Moreover, the leading term in the Fourier expansion of the transition density of Brownian motion [with drift $-\mu_N$ and killed at the border of the interval $(0,a_N)$] is proportional to $e^{-\mu_N x} \sin(\pi x/a_N)$, which explains the meta-stable density predicted by the physicists. As for the particles that hit $a_N$, the authors of \cite{Berestycki2010} find that 1) the number of descendants at a later time of such a particle is of the order of $WN$, where $W$ is a random variable with tail $P(W > x) \sim 1/x$, as $x\to\infty$ and 2) the rate at which particles hit the right barrier is of the order of $\log^{-3}N$. Putting the pieces together yields their result.

In the present article, we define an auxiliary process, the B-BBM (``B'' stands for ``barrier''), which is a better approximation of the $N$-BBM than BBM with a linear barrier. The B-BBM is a BBM with drift $-\mu_N$ and absorption at a random barrier, the role of which is to keep the number of particles roughly constant over a time scale of order $\log^3 N$. This random barrier should therefore be seen as the approximate position of the left-most particle in the $N$-BBM. Its shape is remarkably simple: most of the time it is constant as in \cite{Berestycki2010} and only when a particle reaches a point at distance $a_N$ from the barrier \emph{and} spawns a large number of particles, does the barrier receive an additional random shift to the right. We call this event a \emph{breakout}. The magnitude of the random shift is approximately $\log(1+\pi W)$, where $W$ is the variable from the last paragraph. This turns out to be the shift that is necessary to keep the number of particles roughly constant. The form of this shift explains\footnote{As mentioned in 
Footnote~\ref{fn:aNA}, we actually use $a_{N,A}$ instead of $a_N$. The random shift 
is then approximately $\log(1+\pi e^{-A}W)$, so that in the limit only the 
large values of $W$, and therefore only the asymptotic $1/x$ of its tail, contributes.} the L\'evy measure which appears in Theorem~\ref{th:1}.

Summing up the above arguments: 1) According to \cite{Berestycki2010}, the population size in BBM with absorption at a linear barrier evolves at the time scale $\log^3N$ like a CSBP with L\'evy measure $x^{-2}\,\dd x$. 2) In order to counteract an increase  in the population by $wN$, we have to shift the barrier by $\log(1+w)$. 3) As a consequence, the new barrier evolves like a L\'evy process whose L\'evy measure  is the image of the measure $x^{-2}\,\dd x$ by the map $x\mapsto \log(1+x)$.

We now describe how we use the results on the B-BBM in order to prove Theorem~\ref{th:1}. Initially, our plan was to \emph{couple} the $N$-BBM and the B-BBM, i.e.\ construct them on the same probability space. We would then assign a colour to each particle: \emph{blue} to the particles which appear in the $N$-BBM but not in the B-BBM, \emph{red} to those that appear in the B-BBM but not in the $N$-BBM, and \emph{white} to the particles that appear in both processes. Our aim was then to show that the number of blue and red particles was negligible after a time of order $\log^3 N$. This, unfortunately, did not work out, because we were not able to handle the intricate dependence between the red and blue particles.

Instead, we couple the $N$-BBM with two different processes, the \Bfl- and the \Bsh-BBM, which are variants of the B-BBM and which bound the position of the $N$-BBM in a certain sense from below and above, respectively. The \Bfl-BBM is defined as follows: Initially, all particles are coloured white and evolve as in the B-BBM. A white particle is coloured red as soon as it has $N$ or more white particles to its right. Children inherit the colour of their parent. After a breakout and the subsequent relaxation, all the red particles are killed immediately and the process restarts with the remaining particles. It is intuitive that the collection of white particles then bounds the $N$-BBM from below (in some sense), because we kill ``more'' particles than in the $N$-BBM. Indeed, in Section~\ref{sec:coupling}, we show by a coupling method that the empirical measure of the white particles in \Bfl-BBM is stochastically dominated by the one of the $N$-BBM with respect to the usual stochastic ordering of measures. The 
study of the \Bfl-BBM therefore yields a lower bound in Theorem~\ref{th:1}.

For an upper bound, we couple the $N$-BBM with another system, the \Bsh-BBM. Again, we colour all initial particles white and particles evolve as in B-BBM with the following change: Whenever a white particle hits $0$ \emph{and} has less than $N$ particles to its right, instead of killing it immediately, we colour it blue and let it survive for a time of order $\log^2 N$. Since we kill ``less'' particles than in the $N$-BBM, we can again show that empirical measure of the white and blue particles in \Bsh-BBM stochastically dominates the one of the $N$-BBM. 

In the limit as $N\to\infty$, the position of the $\lceil\alpha N\rceil$-th particle from the right in both the \Bfl- and \Bsh-BBM will then be shown to evolve according to the L\'evy process defined in the statement of Theorem~\ref{th:1}, which implies the theorem.

We note that although our technique of bounding the $N$-BBM from below and from above works well for the \emph{position} of the particles, it does not give us information about the genealogy; the reason being that the coupling deforms the genealogical tree of the process. Thus, although it should not be difficult to show that the genealogy of the B-BBM (and of the \Bfl- and \Bsh-BBM) converges to the Bolthausen--Sznitman coalescent we do not know at present how one could transfer this information to the $N$-BBM.

\subsection{Organisation of the article}

Here is a brief description of the organisation of this article: Section~\ref{sec:approx} introduces the processes used to approximate the $N$-BBM: the B-, \Bfl- and \Bsh-BBM. The main results about these processes are stated and a proof of Theorem~\ref{th:1} assuming these results is given. Sections~\ref{sec:BBBM}, \ref{sec:Bflat} and \ref{sec:Bsharp} then prove the results about the B-, \Bfl- and \Bsh-BBM, respectively. An appendix consisting of three sections contains a mixture of known and new results on general branching Markov processes and (branching) Brownian motion in an interval.

As a guide to the reader: Start with Section~\ref{sec:approx}, maybe refreshing your knowledge about branching processes (notably about stopping lines) by skimming through the relevant pages in Section~\ref{sec:preliminaries}. Then read Sections~\ref{sec:BBBM}, \ref{sec:Bflat} and \ref{sec:Bsharp}, jumping to the appendix whenever it seems necessary. Note that in general, results are presented as a series of lemmas, with the proofs usually immediately following the statement. This is merely a means to ensure that one may quickly and easily locate a particular result with its proof and should not encourage strictly linear reading. Sections~\ref{sec:Bflat} and~\ref{sec:Bsharp} may be read independently of each other, but they highly depend on Section~\ref{sec:BBBM}.

\subsection{Notation guide}
\label{sec:notation}

This article is quite long and uses a wealth of different notation. Below is a list of recurrent symbols, together with their meaning, often informal, and the section where they are defined. Note that in the appendix, these symbols sometimes have a similar, but different meaning (e.g.\ $\P^x$ in Section~\ref{sec:preliminaries} or $Z_t$ in Section~\ref{sec:interval}). 
Sections \ref{sec:Bflat} and \ref{sec:Bsharp} each use special notation, not listed below, which only appears in those sections.
Following the list are some further remarks about notational conventions.

\begin{longtable}{lp{10.3cm}r}
\textbf{Symbol} & \textbf{Meaning} & \textbf{Sect.}\\[4pt]
$(q(k))_{k\ge\N}$ & Reproduction law & \ref{sec:NBBM_intro}\\
$m$ & $m = \sum_k (k-1) q(k)$ & \ref{sec:NBBM_intro}\\
$\beta$ & Branching rate, $\beta = 1/(2m)$ & \ref{sec:NBBM_intro}\\
$a$ & Large parameter of B-BBM, particles hitting $a$ have a chance to \emph{break out}, i.e.\ spawn a large number of descendants & \ref{sec:BBBM_parameters}\\
$A$ & Large parameter of B-BBM, the number of particles is approx.\ proportional to $e^{A+a}/a^3$. We first let $a$, then $A$ go to infinity & \ref{sec:BBBM_parameters}\\
$\mu$ & $\mu = \sqrt{1-\pi^2/a^2}$. In B-BBM: drift towards the origin  & \ref{sec:BBBM_parameters}\\
$w_Z(x)$ & $w_Z(x) = a\sin(\pi x/a)e^{\mu(x-a)}\Ind_{(x\in[0,a])}$ & \ref{sec:BBBM_parameters}\\
$w_Y(x)$ & $w_Y(x) = e^{\mu(x-a)}$ & \ref{sec:BBBM_parameters}\\
$\P^x,\E^x,\P^x_f,\E^x_f$ & Law of and expectation w.r.t.\ BBM with constant drift $-\mu$ or drift $-\mu - a^{-2} f'(t/a^2)$, started from $x$ & \ref{sec:BBBM_parameters}\\
$\P^{(x,t)},\cdots$ & Same as above, but started at the space-time point $(x,t)$ & \ref{sec:BBBM_parameters}\\ 
$\P^\nu,\cdots$ & Same as above, but started from a collection of particles distributed according to a finite counting measure $\nu$ & \ref{sec:BBBM_parameters}\\
$X_u(t)$ & Position of the individual $u$ at time $t$ & \ref{sec:BBBM_parameters}\\
$o(1)$ & Non-random term that vanishes as $A$ and $a$ go to infinity & \ref{sec:BBBM_parameters}\\
$o_a(1)$ & Non-random term that vanishes as $a$ goes to infinity (and $A$ is fixed) & \ref{sec:BBBM_parameters}\\
$\ep$ & Small parameter of B-BBM, depends on $A$: $e^{-A/6} \le \ep \le A^{-17}$ & \ref{sec:BBBM_parameters}\\
$\eta$ & Small parameter of B-BBM, depends on $A$: $\eta \le e^{-2A}$ & \ref{sec:BBBM_parameters}\\
$y,\zeta$ & Large parameters  of B-BBM depending on $\eta$ and $(q(k))_{k\in\N}$ & \ref{sec:BBBM_parameters}\\
$B$ & Event of a breakout  & \ref{sec:BBBM_breakout}\\
$p_B$ & Probability of a breakout: $p_B = \P^a(B) = (\ep e^A)^{-1}(\pi+o(1))$  & \ref{sec:BBBM_breakout}\\
$\mathscr A_0(t)$ & Particles at time $t$ which have not yet hit the origin & \ref{sec:BBBM_tiers}\\
``critical line'' & (w.r.t.\ $t$): the space-time line $a-y+(1-\mu)(s-t)$ & \ref{sec:BBBM_tiers}\\
$\tau_l(u)$ & Time at which the individual $u$ hits $a$ for the $l$-th time (and returns to the critical line in between) & \ref{sec:BBBM_tiers}\\
$\sigma_l(u)$ & First time after $\tau_{l-1}(u)$ the individual $u$ hits the critical line & \ref{sec:BBBM_tiers}\\
$\mathscr R_t^{(l)}$ & Stopping line of tier $l$ particles hitting $a$ before time $t$, $\mathscr R_t = \bigcup_{l\ge 0} \mathscr R_t^{(l)}$ & \ref{sec:BBBM_tiers}\\
$\mathscr S_t^{(l)}$ & Stopping line of tier $l$ particles descending from $\mathscr R_t^{(l-1)}$ at the moment they hit the critical line & \ref{sec:BBBM_tiers}\\
$\mathscr S^{(u,t)}$ & Stopping line of descendants of $(u,t)$ hitting the critical line & \ref{sec:BBBM_tiers}\\
$Z^{(u,t)}$, $Y^{(u,t)}$ & Sum of $w_Z(x),w_Y(x)$ over particles in $\mathscr S^{(u,t)}$ & \ref{sec:BBBM_tiers}\\
$\sigma_{\mathrm{max}}^{(u,t)}$ & Maximum of $s-t$ for all $(v,s)\in\mathscr S^{(u,t)}$ & \ref{sec:BBBM_tiers}\\
$B^{(u,t)}$ & Event of a breakout of the particle $(u,t)\in\mathscr R_\infty$ & \ref{sec:BBBM_tiers}\\
$T^{(l)}$ & Time of first breakout of a tier $l$ particle & \ref{sec:BBBM_tiers}\\
$\mathscr U$ (the \emph{fugitive}) & The particle that breaks out first & \ref{sec:BBBM_tiers}\\
$T,T_n$ & Time of the first/$n$-th breakout & \ref{sec:BBBM_definition_definition}\\
$T^+$ & $T+\sigma_{\mathrm{max}}^{(\mathscr U,T)}$ & \ref{sec:BBBM_definition_definition}\\
$\Theta_n$ & Beginning of the $(n+1)$-th piece of B-BBM (i.e.\ time of the $n$-th breakout plus relaxation time). $\Theta_1 = (T + e^Aa^2)\vee T^+$ & \ref{sec:BBBM_definition_definition}\\
$X^{[n]}_t$ & The position of the first $n$ pieces of the barrier at time $t$ (barrier process) & \ref{sec:BBBM_definition_definition}\\
$G_n$ & ``Good'' event related to the first $n$ pieces of B-BBM & \ref{sec:BBBM}\\
$\widetilde{\ep}$ & $\widetilde{\ep} = (\pi p_Be^A)^{-1} = (\ep/\pi^2)(1+o(1))$.& \ref{sec:BBBM}\\
$\mathscr N_t^{(l)}$ & Stopping line of tier $l$ particles at time $t$ or at time $\sigma_l(u)$ if $\tau_{l-1}(u)\le t < \sigma_l(u)$, $\mathscr N_t=\bigcup_{l\ge0} \mathscr N_t^{(l)}$  & \ref{sec:BBBM_more_definitions}\\
$Z^{(l)}_t,Y^{(l)}_t$ & Sum of $w_Z(x),w_Y(x)$ over particles from $\mathscr N_t^{(l)}$ & \ref{sec:BBBM_more_definitions}\\
$R^{(l)}_t$ & Number of tier $l$ particles hitting $a$ up to time $t$, $R^{(l)}_t = \#\mathscr R^{(l)}_t$ & \ref{sec:BBBM_more_definitions}\\
$Z_t,Y_t,R_t$ & $Z_t = Z^{(0+)}_t$, $Y_t = Y^{(0+)}_t$, $R_t = R^{(0+)}_t$ & \ref{sec:BBBM_more_definitions}\\
$\widehat{\mathscr N}_t$ & Restriction of $\mathscr N_t$ to particles unrelated to fugitive & \ref{sec:BBBM_more_definitions}\\
$\widebar{\mathscr N}_t$ & Restriction of $\mathscr N_t$ to particles spawned by fugitive between times $\sigma_l(\mathscr U)$ and $\tau_l(\mathscr U)$ for some $l\ge 0$ & \ref{sec:BBBM_more_definitions}\\
$\widecheck{\mathscr N}_t$ & Restriction of $\mathscr N_t$ to particles spawned by fugitive between times $\tau_l(\mathscr U)$ and $\sigma_{l+1}(\mathscr U)$ for some $l\ge 0$ & \ref{sec:BBBM_more_definitions}\\
$\mathscr N^{\mathrm{fug}}_t$ & Restriction of $\mathscr N_t$ to descendants of fugitive & \ref{sec:BBBM_more_definitions}\\
$\widehat Z_t$, $\widebar Z_t$, etc. & Value of $Z_t$ restricted to particles from $\widehat{\mathscr N}_t$, $\widebar{\mathscr N}_t$, etc. & \ref{sec:BBBM_more_definitions}\\
$T^-$ & $(T-e^Aa^2)\vee 0$  & \ref{sec:BBBM_more_definitions}\\
$\Delta$ & The total shift of the barrier after a breakout & \ref{sec:BBBM_more_definitions}\\
$\widecheck Z_{\Delta}$ & Contribution of \emph{check}-particles to $\Delta$ & \ref{sec:BBBM_more_definitions}\\
$\Q^a$ & $\P^a$ conditioned not to break out & \ref{sec:time_breakout}\\
$W$ & A random variable with tail $P(W>x)\sim 1/x$ & \ref{sec:time_breakout}\\
$\Phat,\Ehat$ & Law/expectation of BBM conditioned not to break out before some fixed time & \ref{sec:before_breakout_particles}\\
$\Phat_{(l)},\Ehat_{(l)}$ & Law/expectation of BBM conditioned not to break out form a tier $k\le l$ before some fixed time & \ref{sec:before_breakout_particles}\\
$Z^{(l)}_{\emptyset,t}$ & Sum of $w_Z(x)$ over particles from $\mathscr S_t^{(l)}$ & \ref{sec:before_breakout_particles}\\
$\mathscr E_{\mathscr U}$ & Functional bounding the contribution of the bar-particles & \ref{sec:fugitive}\\
$\mathscr L_{\mathscr U},\mathscr F_{\mathscr U}$ & Random line related to the bar-particles and its sigma-field & \ref{sec:fugitive}\\
$\mathscr L_{\Delta}$ & Random line consisting of the particles contributing to $\Delta$ & \ref{sec:piece_proof}\\
$Z_\Delta$, $Y_\Delta$ & Sum of $w_Z$ and $w_Y$ over particles from $\mathscr L_\Delta$ & \ref{sec:piece_proof}\\
$\mathscr F_\Delta$ & A sigma-field that contains $\F_{\mathscr L_{\Delta}}$ & \ref{sec:piece_proof}\\
\parbox{2cm}{$G_{\mathscr U},G_{\mathrm{fug}},\widehat G$,\\$\widecheck G,G_{\Delta},G_{\mathrm{nbab}}$} & Several good sets regarding B-BBM & \ref{sec:piece_proof}\\
$N^{(l)}_t(r)$ & $N^{(l)}_t(r) = \sum_{(u,s)\in\mathscr N^{(l)}_t}\Ind_{(X_u(s)\ge r)}$ & \ref{sec:Bflat_num_particles}\\
$N_t(r)$ & Number of particles to the right of $r$ at time $t$ ($\ne N^{(0+)}_t(r)$ !) & \ref{sec:Bflat_num_particles}\\
$U$ & The space of individuals & \ref{sec:preliminaries_definition}\\
$\mathscr A(t)$ & Set of individuals alive at time $t$ in a branching Markov process & \ref{sec:preliminaries_definition}\\
$\mathscr F_t$ & $\sigma$-algebra with information up to time $t$ & \ref{sec:stopping_lines}\\
$\mathscr F_{\mathscr L}$ & $\sigma$-algebra with information up to the stopping line $\mathscr L$ & \ref{sec:stopping_lines}\\
$\mathscr L_T$ & Stopping line generated by a stopping time $T$ & \ref{sec:stopping_lines}\\
$\theta,\thbar$ & Functions related to Jacobi theta functions & \ref{sec:killed_bm}\\
$W^x$ & Law of Brownian motion started at $x$ & \ref{sec:killed_bm}\\
$p_t^a(x,y)$ & Transition density of Brownian motion killed outside $[0,a]$ & \ref{sec:killed_bm}\\
$E_t$ & Error term decaying like $e^{-(3/2)\pi^2 t}$ as $t\to\infty$ & \ref{sec:killed_bm}\\
$I^a(\cdot,\cdot),J^a(\cdot,\cdot)$ & Integrals related to Brownian motion killed outside $[0,a]$ & \ref{sec:killed_bm}\\
$W^x_{\mathrm{taboo}},W^{x,t,y}_{\mathrm{taboo}}$ & Law of Brownian taboo process and its bridge & \ref{sec:taboo}
\end{longtable}

The complement of an event $G$ is denoted by $G^c$. $\N$ and $\N^*$ are the sets of natural numbers including and excluding 0, respectively. $\R$ is the set of real numbers, $\R_+ = [0,\infty)$. $a\wedge b$ denotes the minimum of $a$ and $b$. This relation is defined between real numbers, but also between individuals of the branching Brownian motion and between (stopping) lines, where it has an analogous meaning (see Section~\ref{sec:preliminaries}). The maximum of two numbers $a$ and $b$ is denoted by $a\vee b$.


The symbols $C$ and $C'$ stand for a \emph{positive} constant, which may only depend on the reproduction law $q$ and the value of which may change from line to line. Furthermore, if $X$ is any mathematical expression, then the symbol $O(X)$ stands for a possibly random term whose absolute value is bounded by $C|X|$.

\section{The approximating processes: \texorpdfstring{B-, \Bfl- and \Bsh-BBM}{B-, Bb- and B\#-BBM}}
\label{sec:approx}

In this section, we define the three processes which will be used to approximate the $N$-BBM. They are called the B-, \Bfl- and \Bsh-BBM. We also state the main results about these processes (deferring the proofs) and prove Theorem~\ref{th:1} assuming these results.

\subsection{B-BBM: definition and statement of results}
\label{sec:BBBM_definition}

The definition builds upon the ideas of \cite{Berestycki2010} presented in the introduction. The basic idea is to approximate the $N$-BBM by a BBM with absorption at a \emph{random} barrier, which is chosen in such a way that \emph{it keeps the number of particles almost constant}. We call the resulting system the ``B-BBM'' (B stands for ``barrier''). It will serve as a blueprint for the \Bfl-BBM and the \Bsh-BBM defined in \ref{sec:Bflat_Bsharp_definition}, both of which are used in the proof of Theorem~\ref{th:1}.

In defining the random barrier of the B-BBM, we tried to depart as little as possible from the linear barrier. That this has been possible is due to the fact that the $N$-BBM behaves most of the time almost deterministically, advancing with linear speed. In this phase, the approximation with the BBM with absorbing linear barrier can be established even on a rigorous level, as shown later in this article. This approximation breaks down however, as soon as a particle goes far to the right and spawns a large number of descendants (of the order of $N$), which causes a jump to the right of the cloud of particles. Consequently, a first step in the definition of a barrier is to determine exactly when such events occur. This will be done through the definition of a \emph{breakout}, an event where one particle -- the \emph{fugitive} -- reaches a point sufficiently far to the right and then spawns a large number of descendants. Only when such a breakout occurs do we shift the barrier to the right (see Figure~\ref{fig:bbbm}). After relaxation, the process starts anew. Defining this process precisely requires a certain number of definitions which we introduce now.

\begin{figure}[ht]
 \centering
\def\svgwidth{9cm}
\executeiffilenewer{bbbm.svg}{bbbm.pdf}%
{inkscape -z -D --file=bbbm.svg %
--export-pdf=bbbm.pdf --export-latex --export-area-drawing}%
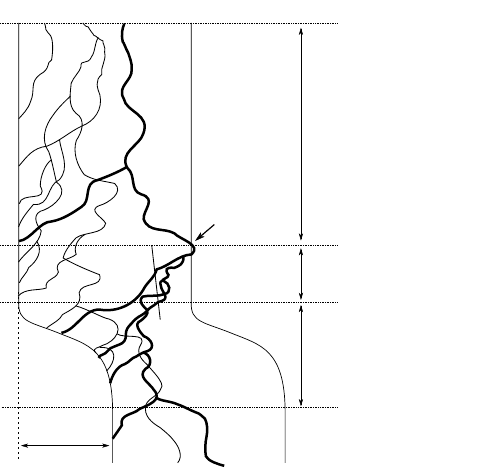%

\caption{A caricatural graphical description of the B-BBM. The fugitive and its descendants are drawn with thick lines, the other particles with thin lines. A breakout happens at time $T$ and the barrier is moved from the time $T^+$ on. The process starts afresh at time $\Theta_1$. Note that technically we increase the drift to the left instead of moving the barrier. The three important time scales ($1$, $a^2$ and $a^3$) are shown as well. The quantities $T,T^+,\Theta_1$ and $a$ are defined below. Note that $a\asymp\log N$.}
\label{fig:bbbm}
\end{figure}

\subsubsection{The parameters and basic definitions}
\label{sec:BBBM_parameters}

As in the introduction, we fix a reproduction law $(q(k))_{k\ge0}$ with $m_2 = \sum k(k-1)q(k) < \infty$ and set $m=\sum (k-1)q(k)$ and $\beta = 1/(2m)$. We further introduce several parameters which will be used during the rest of the paper. The two most important parameters are $a$ and $A$, which are both large positive constants. As indicated in the introduction, the parameter $a$ is the width of an interval the particles stay in most of the time. The parameter $A$ has a more subtle meaning: it controls the number of particles of the system and with it the intensity at which particles hit the right border of the interval. Below, we will indeed choose the initial conditions such that the number of particles is approximately $2\pi e^{A+a}/a^3$. We set
\begin{equation}
 \label{eq:mu}
\mu = \sqrt{1 - \frac{\pi^2}{a^2}},
\end{equation}
so that the following estimate holds:
\begin{equation}
 \label{eq:c0_mu}
 0\le 1 - \mu \le \frac{\pi^2}{2\mu a^2}.
\end{equation}
The following functions are used throughout the article:
\begin{equation}
\label{eq:def_wZY}
w_Z(x) = ae^{\mu (x-a)}\sin(\pi x/a)\Ind_{(x\in[0,a])},\qquad w_Y(x) = e^{\mu(x-a)}.
\end{equation}
Note that $w_Z(x)\le C$ and $w_Y(x)\le 1$ for all $x\in[0,a]$.

Let  $f:\R_+\to\R$ be continuous with $f(0) = 0$. During the whole article, except in the appendix, we denote by $\P^x_f$ the law of the branching Markov process where particles move according to the law of $(X_t-\mu t-f(t/a^2))_{t\ge0}$ for a standard Brownian motion $(X_t)_{t\ge0}$ and reproduce at rate $\beta$ according to the law $(q(k))_{k\ge0}$, starting from a single particle at $x\in\R$. A detailed definition of branching Markov processes and the space of marked trees they live in is given in Section~\ref{sec:preliminaries_definition}. We also denote by $\P_f^{(x,t)}$ and $\P_f^\nu$ the corresponding laws starting from a single particle at the space-time point $(x,t)$ or from a collection of particles given by a counting measure $\nu$, respectively. Expectation with respect to these laws will be denoted by $\E_f^x$, $\E_f^{(x,t)}$ and  $\E_f^\nu$, respectively. Most of the time, we will use these notations with $f\equiv 0$, in which case we omit the subscript.

When we study the system for large $A$ and $a$, we first let $a$ go to infinity, then $A$. Thus, the statement ``For large $A$ and $a$ we have\ldots'' means: ``There exist $A_0$ and a function $a_0(A)$, both depending on the reproduction law $q$ only, such that for $A \ge A_0$ and $a \ge a_0(A)$ we have\ldots''. Likewise, the statement ``As $A$ and $a$ go to infinity\ldots'' means ``For all $A$ there exists $a_0(A)$, depending on the reproduction law $q$ only, such that as $A$ goes to infinity and $a\ge a_0(A)$\ldots''. These phrases will become so common that in Sections~\ref{sec:BBBM} to~\ref{sec:Bsharp} they will often be used implicitly, although they will always be explicitly stated in the theorems, propositions, lemmas etc. We further introduce the symbols $o(1)$ and $o_a(1)$, which stand for a (non-random) term that only depends on the reproduction law $q$ and the parameters $A$, $a$, $\ep$, $\eta$, $y$ and $\zeta$ (defined below) and which goes to $0$ as $A$ and $a$ go to infinity ($o(1)$), or as $a$ 
goes to infinity and $A$ is fixed ($o_a(1)$).

The remaining parameters we introduce all depend on $A$, but not on $a$. First of all, there is the small parameter $\ep$, which controls the intensity of the breakouts. The mean time between two breakouts will in fact be approximately proportional to $\ep a^3$. We expect that one could choose $\ep$ such that $e^{-A/2} \ll \ep \ll A^{-1}$, but for technical reasons we will require that
\begin{align}
 \label{eq:ep_upper}
 \ep &\le A^{-17},\quad \tand\\
 \label{eq:ep_lower}
 \ep &\ge e^{-A/6}.
\end{align}
Another protagonist is $\eta$, which we choose depending on $\ep$ and $A$ and which will be used to bound the probability of very improbable events, as well as some errors. We require that
\begin{equation}
 \label{eq:eta}
 \eta \le e^{-2A},
\end{equation}
which, by \eqref{eq:ep_lower}, implies
\begin{equation}
 \label{eq:eta_ep}
\eta \le \ep^{12}.
\end{equation}
The last parameters are $y$ and $\zeta$, which only depend on $\eta$ and on the reproduction law, and are always supposed to be as large as necessary (it will be enough to assume that $y\ge \eta^{-1}$ and that $y$ and $\zeta$ are chosen in such a way that Lemma~\ref{lem:ZYW} holds). Note that the parameters $\eta$, $y$ and $\zeta$ appeared already in \cite{Berestycki2010} and had the same meaning there.

\subsubsection{The breakout}
\label{sec:BBBM_breakout}

Now suppose we start with a single particle at the point $a$. Let $\mathscr S$ be the stopping line (see Section~\ref{sec:stopping_lines}) consisting of the particles stopped at the space-time line $a-y + (1-\mu)t$, i.e., $\mathscr S = \mathscr L_H$, where $H(X) = \inf\{t: X_t = a-y+(1-\mu)t\}$ for a path $X = (X_t)_{t\ge0}$ starting at $X(0)=a$. It is well-known that $\#\mathscr S < \infty$, $\P^a$-almost surely \cite{Kesten1978}. Recalling that $X_u(t)$ denotes the position at time $t$ of the individual $u$, we define $Z = \sum_{(u,t)\in\mathscr S} w_Z(X_u(t))$, $Y = \sum_{(u,t)\in\mathscr S} w_Y(X_u(t))$ and $W_y = ye^{- y}\#\mathscr S$. As in \cite{Berestycki2010}, the quantity $Z$ is used to estimate the number of particles after a relaxation time of order $a^2$ and the quantity $Y$ is used to bound error terms. Furthermore, define $\sigma_{\mathrm{max}} = \max\{t:(u,t)\in \mathscr S\}$. 

The event of a \emph{breakout} is now defined by
\begin{equation}
 \label{eq:def_B}
B = \{Z > \ep e^A\}\cup  \{\sigma_{\mathrm{max}} > \zeta\}.
\end{equation}
Inclusion of the event $\{\sigma_{\mathrm{max}} > \zeta\}$ is for technical reasons. We also set $p_B = \P^a(B)$. 

\subsubsection{The tiers}
\label{sec:BBBM_tiers}

Having defined the event of a breakout, we would now like to define the (random) time of the first breakout. This leads to the question of how to handle the particles which hit the right barrier but do not cause a breakout. It will turn out that these particles have a non-negligible contribution, so that we cannot simply ignore them. Instead, we classify the particles into \emph{tiers}. Particles that have never hit the point $a$ form the particles of tier 0. As soon as a particle hits $a$ (at time $\tau$, say) it advances to tier 1. Its descendants then belong to tier 1 as well, but whenever a descendant hits $a$ and has an ancestor which has hit the space-time line (the ``critical line'') $a-y+(1-\mu)(t-\tau)$ after time $\tau$, it advances to tier 2 and so on. Whenever a particle advances to the next tier, it has a chance to break out.

\begin{figure}[h]
 \centering
\begin{tabular}{cc}
\executeiffilenewer{tiers.svg}{tiers.pdf}%
{inkscape -z -D --file=tiers.svg %
--export-pdf=tiers.pdf --export-latex --export-area-drawing}%
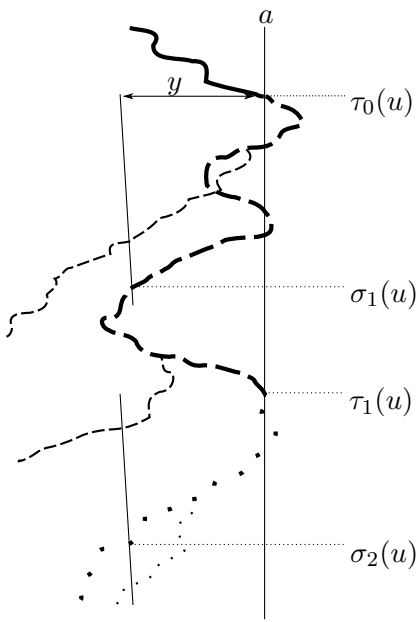%

\end{tabular}
\caption{A graphical view of the \emph{tiers}. The trajectory of one particle $u$ and its ancestors is singled out (thick lines) and the times $\tau_l(u)$ and $\sigma_l(u)$ are shown. The tier~0 trajectories are drawn with straight lines, the tier~1 trajectories with dashed and the tier~2 trajectories with dotted lines.}
\label{fig:tiers}
\end{figure}

We now define a set of stopping lines which implement this ``tier structure''. As introduced in Section \ref{sec:preliminaries_definition}, $U$ denotes the space of individuals and $\mathscr A(t)\subset U$ the set of individuals alive at time~$t$. We define the subset of particles which have not yet hit the origin:
\[
\mathscr A_0(t) = \{u\in \mathscr A(t): H_0(X_u) > t\},
\]
where $H_0$ is the hitting time functional of 0. For each $u\in U$, we now define two sequences of random times $(\tau_l(u))_{l\ge-1}$ and $(\sigma_l(u))_{l\ge0}$ by $\tau_{-1}(u) = 0$, $\sigma_0(u) = 0$ and for $l\ge 0$:
\begin{equation*}
\begin{split}
 \tau_l(u) &= \inf\{t\ge \sigma_l(u):X_u(t) = a\},\\
 \sigma_{l+1}(u) &= \inf\{t\ge \tau_l(u):X_u(t)=a-y+(1-\mu)(t-\tau_l(u))\},
\end{split}
\end{equation*}
where we set $\inf \emptyset = \infty$. See Figure~\ref{fig:tiers} for a graphical description. 
We then define for $t\in\R_+\cup\{\infty\}$ the stopping lines
\begin{align*}
\mathscr R_t^{(l)} &= \{(u,s)\in U\times\R_+: u\in\mathscr A_0(s) \tand s=\tau_l(u)\le t\},\ l \ge -1,\\
\mathscr S_t^{(l)} &= \{(u,s)\in U\times\R_+: u\in\mathscr A_0(s) \tand s=\sigma_l(u)\tand \tau_{l-1}(u) \le t\},\ l\ge 0,
\end{align*}
and set 
$\mathscr R_\infty/\mathscr R_t/\mathscr S_\infty/\mathscr S_t = \bigcup_{l\ge 0} \mathscr R_\infty^{(l)}/\mathscr R^{(l)}_t/\mathscr S_\infty^{(l)}/\mathscr S^{(l)}_t$. That means, $\mathscr R_\infty^{(l)}$ contains the particles which advance from tier $l$ to tier $l+1$ at the moment they hit the right barrier (for $l\ge0$) and $\mathscr S_\infty^{(l)}$ contains the particles of tier $l$ at the moment at which they come back to the critical line. Note that $\mathscr R^{(l-1)}_t \preceq \mathscr S^{(l)}_t \preceq \mathscr R^{(l)}_t$ for every $l\ge0$ and $t\ge 0$ (the relation ``$\preceq$'' for stopping lines is defined in Section~\ref{sec:stopping_lines}). 

Since every pair $(u,t)\in \mathscr R_\infty$ represents a particle located at $a$ and having a chance to break out, we denote by $\mathscr S^{(u,t)}$, $\sigma_{\mathrm{max}}^{(u,t)}$, $B^{(u,t)}$, $Z^{(u,t)}$, $Y^{(u,t)}$ and $W_y^{(u,t)}$ the corresponding objects from Section~\ref{sec:BBBM_breakout} (with $\sigma_{\mathrm{max}}^{(u,t)} = \max\{s-t:(u,s)\in \mathscr S^{(u,t)}\}$).

We can now define the time of the first breakout. For $l\in\N$, define the stopping times
\begin{equation*}
T^{(l)} = \inf\Big\{t\ge0: \sum_{(u,s)\in\mathscr R^{(l)}_t} \Ind_{B^{(u,s)}} = 1\Big\},\quad T = \min_{j\ge 0} T^{(j)},\quad T^{(0;l)} = \min_{0\le j\le l} T^{(j)}.
\end{equation*}
Then, $T$ is the time of the first breakout and $T^{(l)}$ is the time of the first breakout of a particle from tier $l$. We denote by $\mathscr U$ the individual that causes the first breakout, i.e.\ the almost surely unique individual s.t.\ $(\mathscr U,T)\in\mathscr R_\infty$. This individual is called the \emph{fugitive}.

\begin{remark}
\label{rem:hat_f}
  Note that all of these definitions
  also make sense for BBM with varying drift given by a continuous function $f$ as described in Section~\ref{sec:BBBM_parameters}. However, in this case we change the definition of the critical line to be the line $a-y+(1-\mu)(t-\tau)-(f(t/a^2)-f(\tau/a^2))$. The number of particles absorbed at this line, starting with a single particle at the space-time point $(a,\tau)$, then has the same law for every  $f$.
\end{remark}

\subsubsection{B-BBM: definition}
\label{sec:BBBM_definition_definition}

We call a function $f:\R_+\to\R$ a \emph{barrier function} if it is continuous, $f(0)=0$, $f(t)-f(s) \ge -1$ for all $s<t$ and the left-derivative $f'$ exists everywhere and is of bounded variation.
For such a function, we define
\begin{equation}
\label{eq:def_f_norm}
 \|f\| = \max\Big\{\|f\|_\infty, \|f'\|_\infty,\int_0^\infty [f'(s)]^2\,\dd s,\int_0^\infty |\dd f'(s)|\Big\},
\end{equation}
where $\|\cdot\|_\infty$ is the usual supremum norm. 

We are almost ready to define the B-BBM, we only need two more ingredients. The first is a family $(f_x)_{x\in\R}$ of functions $f_x \in \mathscr C^2(\R,\R)$, measurable in $x$, such that,
\begin{enumerate}[nolistsep]
 \item For all $x\in\R$, $f_x(t) = 0$ for $t\le 0$ and $\lim_{t\to\infty} f_x(t) = x$.
 \item For all $x\in\R$ and $t\ge 1$, $|f_x(t)-x| \le C|x|/t$.
 \item The function $x\mapsto \|f_x\|$ is uniformly bounded on compact sets.
 \item For all $x\ge -1$, $f_x(t) - f_x(s) \ge -1$ for all $0\le s\le t$.
\end{enumerate}
For all $x\ge -1$, the function $f_x$ is then a barrier function. The family $(f_x)_{x\in\R}$ will determine the shape of the random barrier in the B-BBM after a breakout. As an example of a family of functions with the above properties, consider the following one ($\thbar(t)$ is defined in \eqref{eq:thbar}):
 \begin{equation}
   \label{eq:wall_fn}
 f_x(t) = \log\left(1+(e^{ x}-1)\thbar(t)\right)\text{ for $t\ge 0$,}\quad f_x(t) = 0\text{ for $t<0$.}
 \end{equation}
As will be seen later, with this choice the number of particles in the B-BBM stays roughly constant over time. It is therefore used for the \Bfl- and \Bsh-BBM below.

The second ingredient is  a (possibly random) finite counting measure  $\nu_0 = \nu_0^{A,a}$ for every $A$ and $a$, which represents the initial configuration of particles. Apart from the fact that we require $\supp \nu_0 \subset [0,a]$ almost surely, this measure is arbitrary for now, but later we will choose it in such a way that $Z'_0 \approx e^A$, where
\[
 Z'_t = \sum_{u\in\mathscr A_0(t)} w_Z(X_u(t)).
\]
with $w_Z$ from \eqref{eq:def_wZY}. $Z_t'$ is an important quantity, which measures for example the number of particles at a future time $t+t'$ with $a^2\ll t' \ll a^3$ (see Section~\ref{sec:interval_number}), or the number of particles which hit the right barrier between $t$ and a future time $t+t''$ with $t'' \ll a^3$ (see Section~\ref{sec:right_border}). 

The B-BBM is now formally a branching Markov process with reproduction law $(q(k))_{k\ge0}$, branching rate $\beta$,  initial configuration of particles $\nu_0$ and where particles move according to Brownian motion with a random, time-dependent drift $-\mu_t$ (defined below) and are absorbed (i.e.\ killed) at the origin. We will denote its law (on the space of marked trees from Section~\ref{sec:preliminaries_definition}) by $\PB$ and expectation w.r.t.\ this law by $\EB$. It remains to define $\mu_t$; for this we define for each $n\in\N$ a stopping time $\Theta_n$, with $0=\Theta_0\le\Theta_1\le\cdots$, and for each $n\in\N^*$ a process $(X^{[n]}_t)_{t\in [\Theta_{n-1},\Theta_n]}$, such that $\mu_t = \mu + (\dd/\dd t) X^{[n]}_t$ for each $t\in [\Theta_{n-1},\Theta_n]$. $\Theta_n$ and $X^{[n]}$ are defined as follows:
\begin{enumerate}[nolistsep]
 \item Initially, the drift is constant, i.e.\ $\mu_t=\mu$, until the time $T$ of the first breakout. We set accordingly $X^{[1]}_t = 0$ for $t\in [0,T]$.
 \item After the breakout, we modify the drift: Define the stopping times $T^+ = T+\sigma_{\mathrm{max}}^{(\mathscr U,T)}$ and  $\Theta_1 = (T + e^Aa^2)\vee T^+$. We then set 
 \[
X^{[1]}_t = X^{[1]}_{\Theta_0} + f_{\Delta}\Big(\frac{t-T^+}{a^2}\Big) \tand \mu_t = \mu + \frac{\dd}{\dd t} X^{[1]}_t,\quad t\in[\Theta_0,\Theta_1].
\]
(In particular, $X^{[1]}_t = 0$ for $t\in[0,T^+]$). Here, $\Delta$ is an $\F_{T^+}$-measurable random variable whose exact definition will be provided later in \eqref{eq:Delta_def}. Morally, we have $\Delta \approx \log ( e^{-A} Z_{T^+}' )$, so that $Z'_{\Theta_n} \approx e^A$ w.h.p., as we will see later.
\item Having defined $T_1=T$, $T^+_1=T^+$, $\Theta_1$ and $X^{[1]}$, we define $T_2$, $T^+_2$, $\Theta_2$ and $X^{[2]}$ by  repeating the above steps, but starting at time $\Theta_1$ and with $X^{[2]}_{\Theta_1} = X^{[1]}_{\Theta_1}$. 
\end{enumerate}
We construct the \emph{barrier process} $(X^{[\infty]}_t)_{t\ge0}$ by $X^{[\infty]}_t = X^{[n]}_t$, if $t\in [\Theta_{n-1},\Theta_n]$.

\begin{remark}
 \label{rem:prob_space}
 One might remark that for the definition of the B-BBM, instead of introducing the law $\PB$, we could have stuck with the law $\P$ and instead introduced new random variables. For example, we could have defined $X^{\mathrm{B}}_u(t) = X_u(t) - X^{[\infty]}_t$ as well as variants of all other random variables (for example of the random times $\tau_n(u)$ and $\sigma_n(u)$) referring to $X^B_u(t)$ instead of $X_u(t)$. However, since below we define two more processes, the \Bfl- and the \Bsh-BBM, we would have had to do the same for these, leading to a zoo of different random variables. Introducing new laws is more economical in terms of notation.
 
 Finally, since the law of the process until time $T^+$ is the same under $\PB$ and $\P$, we can and will often use $\P$ instead of $\PB$ when considering the process up to this time.
\end{remark}

\subsubsection{B-BBM: results}
\label{sec:BBBM_results}

Theorem~\ref{th:barrier} and Theorem~\ref{th:barrier2} below are our main results on the B-BBM. Note that these results are not used directly in the proof of Theorem~\ref{th:1}, their proofs rather serve as blueprints for the proof of the corresponding result for the \Bfl-BBM and \Bsh-BBM (Theorem~\ref{th:Bflat_Bsharp}).

Define the predicate\\
\begin{tabular}{lp{\textwidth-2cm}}
 (H$_\perp$) & $\nu_0$ is obtained from $\lfloor 2\pi  e^Aa^{-3}e^{\mu a} \rfloor$ particles distributed independently according to the density proportional to $\sin(\pi x/a)e^{-\mu x}\Ind_{(0,a)}(x)$.
\end{tabular}\\
Furthermore, if for some $n>0$, $(t^{A,a}_j)_{j=1}^n\in\R_+^n$ for all $A$ and $a$, then define the predicate
\begin{tabular}{lp{\textwidth-2cm}}
 (Ht) & There exists $0\le t_1 < \dots < t_n$, such that for all $j\in\{1,\ldots,n\}$, $t^{A,a}_j\to t_j$ as $A$ and $a$ go to infinity, with $t_1^{A,a}\equiv 0$ if $t_1=0$.
\end{tabular}

\begin{theorem}
 \label{th:barrier}
Suppose (H$_\perp$). Let $n\ge 1$ and $(t^{A,a}_j)_{j=1}^n\in\R_+^n$ such that (Ht) is verified. Define the process
\[
(X_t)_{t\ge 0} = (X^{[\infty]}_{a^3t} -  \pi^2 A t)_{t\ge 0}.
\]
Then, as $A$ and $a$ go to infinity,  under $\PB$, the law of the vector $(X_{t^{A,a}_j})_{j=1}^n$ converges to the law of $(L_{t_j})_{j=1}^n$, where $(L_t)_{t\ge0}$ is the L\'evy process from Theorem~\ref{th:1} but with the constant $c = \pi^2(\const{eq:W_expec} + \const{eq:c_log} - \log \pi)$.
\end{theorem}
A stronger convergence than convergence in the sense of finite-dimensional distributions is convergence in law with respect to Skorokhod's ($J_1$-)topology (see \cite[Chapter~3]{Ethier1986}). The convergence in Theorem~\ref{th:barrier} does not hold in this stronger sense, because the barrier is continuous but the L\'evy process is not and the set of continuous functions is closed in Skorokhod's $J_1$-topology
. However, if we create artificial jumps, we can rectify this:

\begin{theorem}
 \label{th:barrier2}
Suppose (H$_\perp$). Define $J_t = X^{[\infty]}_{\Theta_n}$, if $t\in [\Theta_n,\Theta_{n+1})$, for $n\in\N$. Then as $A$ and $a$ go to infinity, under $\PB$, the process $(X'_t)_{t\ge 0} = (J_{a^3t}-\pi^2 At)_{t\ge 0}$ converges in law with respect to Skorokhod's topology to the L\'evy process defined in the statement of Theorem~\ref{th:barrier}.
\end{theorem}
It will be technically convenient to prove Theorem~\ref{th:barrier2} first, because working in Skorokhod's topology makes time changes easier to handle. Theorem~\ref{th:barrier} then follows almost immediately. We also remark that the assumption (H$_\perp$) in the above theorems can be considerably weakened, in fact, the assumptions from Proposition~\ref{prop:piece} below would be enough.

\subsection{\texorpdfstring{\Bfl}{Bb}-BBM and \texorpdfstring{\Bsh}{B\#}-BBM: definition and statement of results}
\label{sec:Bflat_Bsharp_definition}

In this section, we introduce the two auxiliary processes \Bfl- and \Bsh-BBM, which will bound the $N$-BBM from below and from above, respectively. During the whole section, we fix $\delta \in (0,1/100)$. In the definition of the phrase ``as $A$ and $a$ go to infinity'', (see Section~\ref{sec:BBBM_parameters}), $A_0$ and the function $a_0(A)$ may now also depend on $\delta$. 

\begin{figure}[ht]
 \centering
\begin{tabular}{cc}
 \def\svgwidth{7cm} %
\executeiffilenewer{bflat.svg}{bflat.pdf}%
{inkscape -z -D --file=bflat.svg %
--export-pdf=bflat.pdf --export-latex --export-area-drawing}%
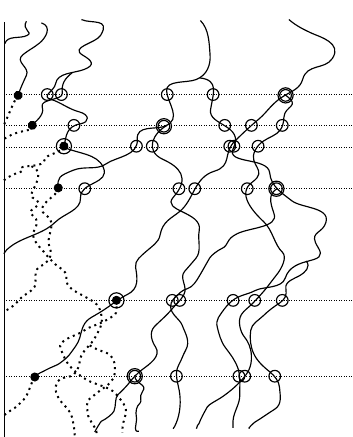%
 & 
 \def\svgwidth{7cm} %
\executeiffilenewer{bsharp.svg}{bsharp.pdf}%
{inkscape -z -D --file=bsharp.svg %
--export-pdf=bsharp.pdf --export-latex --export-area-drawing}%
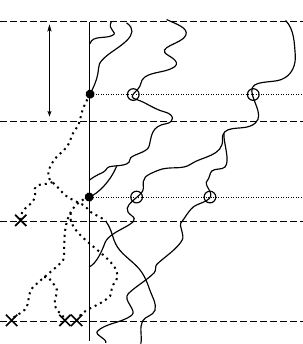%

\end{tabular}
 \caption{Left: The \Bfl-BBM with $N^\flat=6$ (no breakout is shown). White and red particles are drawn with solid and dotted lines, respectively. A blob indicates that a white particle is coloured red, the circles show the six white particles living at that time (thick circles correspond to two particles at the same position). Right: The \Bsh-BBM with $N^\sharp=3$ (no breakout is shown). White and blue particles are drawn with solid and dotted lines, respectively. A blob indicates that a white particles is coloured blue. A cross indicates that a blue particle is killed. $\widetilde K$ is a large constant.}
 \label{fig:bflat_bsharp}
\end{figure}

\subsubsection{\texorpdfstring{\Bfl}{Bb}-BBM: definition}
\label{sec:Bflat_definition}

Define $N^\flat = \lfloor 2\pi  e^{A+\delta}a^{-3}e^{\mu a}\rfloor$. The B$^\flat$-BBM is defined as follows: Given a possibly random initial configuration $\nu_0$ of particles in $(0,a)$, we let particles evolve according to B-BBM with the family of barrier functions given by \eqref{eq:wall_fn}. In addition, we colour the particles white and red as follows: Initially, all particles are coloured white. As soon as a white particle has $N^\flat$ or more white particles to its right, it is coloured red\footnote{This can be ambiguous if there are more than one of the particles at the same position, for example when the left-most particle branches. In order to eliminate this ambiguity, induce for every $t\ge0$ a total order on the particles in $\mathscr A(t)$ by $u<v$ iff $X_u(t) < X_v(t)$ or $X_u(t)=X_v(t)$ and $u$ precedes $v$ in the \emph{lexicographical order on $U$}. Whenever there are more than $N^\flat$ white particles, we then colour the particles in this order, which is well defined.}. Children inherit the 
colour 
of their parent. At each time $\Theta_n$, $n\ge1$, 
all the red particles are killed immediately and the process goes on with the remaining particles. We denote the empirical measure of the white particles at time $t$ by $\nu^\flat_t$. The law of the process is denoted by $\Pfl$, expectation w.r.t.\ this law by $\Efl$. See Figure~\ref{fig:bflat_bsharp} for a graphical description.

\subsubsection{\texorpdfstring{\Bsh}{B\#}-BBM: definition}
\label{sec:Bsharp_definition}
Define now $N^\sharp = \lfloor 2\pi  e^{A-\delta}a^{-3}e^{\mu a} \rfloor $. Further, define $K$ to be the smallest number, such that $K \ge 1$ and $E_K \le \delta/10$, where $E_t$ is defined in \eqref{eq:def_E}. The B$^\sharp$-BBM is defined as follows: Given a possibly random initial configuration $\nu_0$ of particles in $(0,a)$, we let particles evolve according to B-BBM with barrier function given by \eqref{eq:wall_fn} and with the following changes: Define $t_n = n(K+3)a^2$ and $I_n = [t_n,t_{n+1})$. Colour all initial particles white. When a white particle hits $0$ during the time interval $I_n$ and has at least $N^\sharp$ particles to its right, it is killed immediately. If less than $N^\sharp$ particles are to its right, it is coloured blue and survives until the time $t_{n+2}\wedge \Theta_1$, where all of its descendants \emph{to the left of~$0$} are killed and the remaining survive and are coloured white again. At the time $\Theta_1$, the process starts afresh. We denote the empirical measure of all particles (white and blue) at time $t$ by $\nu^\sharp_t$. The law of the process is denoted by $\Psh$, expectation w.r.t.\ this law by $\Esh$. See Figure~\ref{fig:bflat_bsharp} for a graphical description.

\subsubsection{Results}
\label{sec:Bflat_Bsharp_results}

\begin{theorem}
  \label{th:Bflat_Bsharp}
Theorems~\ref{th:barrier} and \ref{th:barrier2} still hold for the \Bfl- and \Bsh-BBM.
\end{theorem}

Recall the definitions of $\med^N_\alpha$ (Eq.\ \eqref{eq:def_med}), $x_\alpha$ (Theorem~\ref{th:1}) and (Ht) (Section~\ref{sec:BBBM_results}).
\begin{proposition}
\label{prop:Bflat_Bsharp_med}
Suppose (H$_\perp$) and let $(t^{A,a}_j)$ satisfy (Ht). Let $\alpha\in (0,e^{-2\delta})$. For large $A$ and $a$,
\begin{enumerate}[nolistsep]
 \item $\Pfl(\forall j: \med^{N^\flat}_\alpha(\nu^{\flat}_{a^3t^{A,a}_j}) \ge x_{\alpha e^{2\delta}})\to 1,$ and
 \item $\Psh(\forall j: \med^{N^\sharp}_\alpha(\nu^{\sharp}_{a^3t^{A,a}_j}) \le x_{\alpha e^{-2\delta}})\to 1.$
\end{enumerate}
\end{proposition}

\begin{lemma}
  \label{lem:Csharp}
Define a variant called C$^\sharp$-BBM of the B$^\sharp$-BBM by killing blue particles only if there are at least $N^\sharp$ particles to their right. Then the parts of Theorem~\ref{th:Bflat_Bsharp} and Proposition~\ref{prop:Bflat_Bsharp_med} concerning the \Bsh-BBM hold for the C$^\sharp$-BBM as well.
\end{lemma}

\begin{remark}
The definitions of the \Bfl- and \Bsh-BBM are quite natural, except for the choice of the length of the intervals during which the blue particles do not feel any selection in \Bsh-BBM. There is actually not much choice for the length of this interval: If we increase its length, the blue particles would spawn too many descendants.
On the other hand, if we \emph{decreased} the length of the intervals $I_n$, then the errors due to the application of Lemma~\ref{lem:N_2ndmoment} (namely, due to the appearance of the $w_Y$ term in that lemma) would blow up. 

We remark that for deriving the above results on \Bfl- and \Bsh-BBM, Lemma~\ref{lem:N_2ndmoment} and the other results from  Section~\ref{sec:interval_number} play a pivotal role in several places, in particular in order to bound the number of white particles in \Bfl- and \Bsh-BBM that get coloured red or blue. In deriving Lemma~\ref{lem:N_2ndmoment}, we were actually quite lucky: Its bound is exactly the one that is needed for all errors to be negligible (as seen for example in the case of the \Bsh-BBM explained above). Note that the corresponding result in \cite{Berestycki2010} (Proposition~14) would not yield good enough bounds for $t \asymp \log^2 N$, neither can its proof be optimised to yield these bounds.

Finally, we remark that assumption (H$_\perp$) is really necessary for the proof of the above results, weakening it considerably seems to be difficult.
\end{remark}

\subsection{A monotone coupling between \texorpdfstring{$N$}{N}-BBM and more general particle systems}
\label{sec:coupling}

In this section, we establish a monotone coupling between the $N$-BBM and a class of slightly more general BBM with selection which includes the B$^\flat$-BBM and C$^\sharp$-BBM.

A \emph{selection mechanism} for branching Brownian motion is by definition a stopping line $\mathscr L$, which has the interpretation that if $(u,t)\in\mathscr L$, we think of $u$ being \emph{killed} at time $t$. The set of particles in the system at time $t$ then consists of all the particles $u\in\mathscr A(t)$ which do not have an ancestor which has been killed at a time $s\le t$, i.e.\ all the particles $u\in\mathscr A(t)$ with $\mathscr L \not\preceq (u,t)$.

Now suppose we have two systems of BBM with selection, the $N^+$-BBM and the
$N^-$-BBM, whose selection mechanisms satisfy the following rules.
\begin{enumerate}[nolistsep]
\item Only left-most particles are killed, i.e.\ if $(u,t)\in\mathscr L$, then $X_u(t)\le X_v(t)$ for all $v\in\mathscr A(t)$ with $\mathscr L \not\preceq (v,t)$.
\item $N^+$-BBM: Whenever a particle gets killed, there are at least $N$
  particles to its right (but not necessarily all the particles which have $N$ particles
  to their right get killed). $N^-$-BBM: Whenever at least $N$ particles are to
  the right of a particle, it gets killed (but possibly more particles get
  killed).
\end{enumerate}
Note that the $N$-BBM is both an $N^+$-BBM and an $N^-$-BBM.

Let $\nu^+_t$, $\nu^-_t$ and $\nu^N_t$ be the empirical measures of the
particles at time $t$ in $N^+$-BBM, $N^-$-BBM and $N$-BBM, respectively. On the space of finite counting measures on $\R$ we denote by $\preceq$ the usual stochastic ordering: For two counting measures $\nu_1$ and $\nu_2$, we write $\nu_1\preceq\nu_2$ if and only if $\nu_1([x,\infty)) \le \nu_2([x,\infty))$ for every $x\in\R$. If $x_1,\ldots,x_n$ and $y_1,\ldots,y_m$ denote the atoms (with multiplicity) of $\nu_1$ and $\nu_2$ respectively, then this is equivalent to the existence of an injective map $\phi:\{1,\ldots,n\}\to\{1,\ldots,m\}$ with $x_i \le y_{\phi(i)}$ for all $i\in\{1,\ldots,n\}$. Furthermore, for two families of counting measures $(\nu_1(t))_{t\ge0}$ and $(\nu_2(t))_{t\ge0}$, we write $(\nu_1(t))_{t\ge0}\preceq (\nu_2(t))_{t\ge0}$ if $\nu_1(t)\preceq\nu_2(t)$ for every $t\ge 0$. If $(\nu_1(t))_{t\ge0}$ and $(\nu_2(t))_{t\ge0}$ are random, then we write $(\nu_1(t))_{t\ge0}\stackrel{\mathrm{st}}{\preceq} (\nu_2(t))_{t\ge0}$ if there exists a coupling between the two (i.e.\ a realisation of both 
on the same probability space), such that $(\nu_1(t))_{t\ge0}\preceq (\nu_2(t))_{t\ge0}$.

\begin{lemma}
\label{lem:coupling}
Suppose that $\nu^-_0 \stackrel{\mathrm{st}}{\preceq} \nu^N_0 \stackrel{\mathrm{st}}{\preceq} \nu^+_0$. Then $(\nu^-_t)_{t\ge0} \stackrel{\mathrm{st}}{\preceq} (\nu^N_t)_{t\ge0} \stackrel{\mathrm{st}}{\preceq} (\nu^+_t)_{t\ge0}$.
\end{lemma}
\begin{proof}
 We only prove the second inequality $\nu^N_t \stackrel{\mathrm{st}}{\preceq} \nu^+_t$ in detail, the proof of the first one is sketched at the end of the proof. Coupling $\nu_0^N$ and $\nu_0^+$ and conditioning on $\F_0$, it is enough to show it for deterministic $\nu^N_0$ and $\nu^+_0$. Let $n^+ = \nu^+_0(\R)$ and let $\Pi = (\Pi^{(1)},\Pi^{(2)},\ldots,\Pi^{(n^+)})$ be a forest of independent BBM trees with the atoms of $\nu^+_0$ as initial positions. We denote by $\mathscr A^\Pi(t)$ the set of individuals alive\footnote{The term ``alive'' has the same meaning here as in Section~\ref{sec:preliminaries_definition}, i.e.\ it includes the ``killed'' particles.} at time $t$ and by $X^\Pi_u(t)$ the position of an individual $u\in\mathscr A^\Pi(t)$. Denote by $\mathscr A^+(t)\subset \mathscr A^\Pi(t)$ the subset of individuals which form the $N^+$-BBM (i.e.\ those which have not been killed by the selection mechanism of the $N^+$-BBM). We set $\nu^+_t = \sum_{u\in\mathscr A^+(t)} \delta_{X^\Pi_u(t)}$.

From the forest $\Pi$ we will construct a family of forests $\left(\Xi_t = (\Xi^{(1)}_t,\ldots,\Xi^{(N)}_t)\right)_{t\ge0}$ (not necessarily comprised of independent BBM trees), such that
\begin{enumerate}[nolistsep]
 \item if $t_1\le t_2$, then the forests $\Xi_{t_1}$ and $\Xi_{t_2}$ agree on
   the time interval $[0,t_1]$,
\item the initial positions in the forest $\Xi_0$ are the atoms of $\nu^N_0$,
\item for every $t\ge 0$, the $N$-BBM, regarded as a measure-valued process, is embedded in $\Xi_t$ up to the time $t$, in the sense that for every $s\in[0,t]$, if $\mathscr A^\Xi(s)$ denotes the set of individuals\footnote{Note that this does not depend on $t$.} from $\Xi_t$ alive at time $s$ and $X^\Xi_u(s)$ the position of the individual $u\in\mathscr A^\Xi(s)$, then there is a subset $\mathscr A^N(s)\subset \mathscr A^\Xi(s)$ such that $(\nu^N_s)_{0\le s\le t} = \big(\sum_{u\in\mathscr A^N(s)} \delta_{X^\Xi_u(s)}\big)_{0\le s\le t}$ is equal in law to the empirical measure of $N$-BBM run until time $t$,
\item for every $t\ge0$, there exists a (random) injective map $\phi_t:\mathscr A^N(t) \to \mathscr A^+(t)$, such that $X^\Xi_u(t) \le X^\Pi_{\phi_t(u)}(t)$ for every $u\in\mathscr A^N(t)$.
\end{enumerate}
We will say that the individuals $u$ and $\phi_t(u)$ are \emph{connected}. If at a time $t$ an individual $v\in \mathscr A^{+}(t)$ is not connected to another individual (i.e.\ $v\notin \phi_t(\mathscr A^N(t))$), we say that $v$ is \emph{free}.

The construction of the coupling goes as follows: Since $\nu^N_0 \preceq \nu^+_0$, we can construct $\Xi_0$, $\mathscr A^N(0)$ and $\phi_0$, such that for every $u\in \mathscr A^N(0)$ we have $X^\Xi_u(0) \le X^\Pi_{\phi_0(u)}(0)$ and the trees $\Xi^{(u)}_0$ and $\Pi^{(\phi_0(u))}$ are the same up to translation. We now construct the forests $(\Xi_t)_{t\ge0}$ along with a sequence of random times $(t_n)_{n\ge0}$ recursively as follows: Set $t_0 = 0$. For a BBM forest $\Sigma$ and an individual $u$ in the forest alive at some time $t$, denote by $\Sigma^{(u,t)}$ the subtree rooted at $u$ and $t$, i.e., for an individual $uw$ alive at some time $t+s$, $s\ge0$, we have $X^{\Sigma^{(u,t)}}_w(s) = X^\Sigma_{uw}(t+s)$. Let $n\in\N$ and suppose that
\begin{itemize}[nolistsep]
 \item[a)] $\Xi_t$ and $\phi_t$ have been defined for all $t\le t_n$ and satisfy the above-mentioned points 1.-4. up to the time $t_n$,
 \item[b)] for each $u\in \mathscr A^N(t_n)$, the subtrees $\Xi_{t_n}^{(u,t_n)}$ and $\Pi^{(\phi_{t_n}(u),t_n)}$ are the same, up to translation.
\end{itemize}
Note that this is the case for $n=0$. We then define $t_{n+1}$ to be the first time after $t_n$ at which either
\begin{enumerate}[nolistsep]
 \item[i)] a particle of the $N$-BBM branches, or
 \item[ii)] the left-most particle of the $N^+$-BBM dies without a particle of the $N$-BBM branching.
\end{enumerate}
Between $t_n$ and $t_{n+1}$, we simply set $\Xi_t = \Xi_{t_n}$ and $\phi_t = \phi_{t_n}$ for all $t\in[t_n,t_{n+1})$. The above-mentioned points 1.-4.  are then still satisfied for $t < t_{n+1}$, because by hypothesis b), connected particles move and branch together. At the time $t_{n+1}$ however, we may need to ``rewire'' particles. We show that it is always possible to do this in such a way that the above conditions on the family $\Xi_t$ are satisfied. Figure~\ref{fig:coupling} is a graphical description of the key step.
\begin{figure}[ht]
 \begin{center}
\executeiffilenewer{coupling.svg}{coupling.pdf}%
{inkscape -z -D --file=coupling.svg %
--export-pdf=coupling.pdf --export-latex --export-area-drawing}%
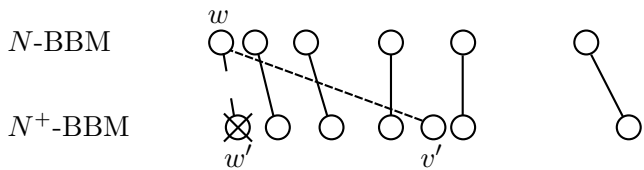%

 \caption{The connection between the particles $w$ of the $N$-BBM and $w'$ of the $N^+$-BBM breaks. By definition of the $N^+$-BBM, there exists a free particle $v'$ to the right of $w'$ and $w$ is rewired to that particle.}
 \label{fig:coupling}
 \end{center}
\end{figure}

We distinguish between the two cases above, starting with the second:

\emph{Case ii):} The left-most particle $w'$ of the $N^+$-BBM gets killed without a particle of the $N$-BBM branching. If $w'$ is free, nothing has to be done, i.e.\ we set $\Xi_{t_{n+1}} = \Xi_{t_n}$ and $\phi_{t_{n+1}} = \phi_{t_n}$. Suppose therefore that $w'$ is connected to a particle $w$ of the $N$-BBM. Then, since there are at most $N-1$ remaining particles in the $N$-BBM and there are at least $N$ particles to the right of $w'$ in the $N^+$-BBM (otherwise it would not have been killed), at least one of those particles is free. Denote this particle by $v'$. We then ``rewire'' the particle $w$ to $v'$ by setting $\phi_{t_{n+1}}(w) := v'$ and define $\Xi_{t_{n+1}}$ by replacing the subtree $\Xi_{t_n}^{(w,t_{n+1})}$ in $\Xi_{t_n}$ by $\Pi^{(v',t_{n+1})}$, properly translated. Note that we then have 
\[
X^\Pi_{\phi_{t_{n+1}}(w)}(t_{n+1}) = X^\Pi_{v'}(t_{n+1}-)\ge X^\Pi_{w'}(t_{n+1}-) \ge X^\Xi_w(t_{n+1}-),
\]
where the first inequality follows from the fact that $w'$ is the left-most individual in $N^+$-BBM at time $t_{n+1}$ and the second inequality holds by the induction hypothesis.

If more than one particle of the $N^+$-BBM gets killed at time $t_{n+1}$, we repeat the above procedure for every particle, starting from the left-most.

\emph{Case i):} A particle $u$ of the $N$-BBM branches at time $t_{n+1}$. By the hypothesis b), the particle $\phi_{t_n}(u)$ then branches as well into the same number of children. We then define $\phi_{t_n}(uk) = \phi_{t_n}(u)k$ for each $k\in\{1,\ldots,k_u\}$ (recall that $k_u$ denotes the number of children of $u$), i.e.\ we connect each child of $u$ to the corresponding child of $\phi_{t_n}(u)$. Now first define $\phi_{t_{n+1}}'$ to be the restriction of $\phi_{t_n}$ to the surviving particles, i.e.\ to $\mathscr A^N(t_{n+1})$. Then continue as in Case~i), i.e.\ for each particle $w'$ of the $N^+$-BBM which gets killed and which is connected through  $\phi_{t_{n+1}}'$ to a particle $w$ of the $N$-BBM, rewire $w$ to a free particle $v'$. In the end, we get $\phi_{t_{n+1}}$. 

This procedure gives a recursive definition of $(\Xi_t)_{t\ge0}$, $(\phi_t)_{t\ge0}$ and the sequence $(t_n)_{n\ge0}$. Note that each time we rewire a particle, we rewire it to a particle whose subtree is independent of the others by the strong branching property, whence the particles from $\mathscr A^N(t)$ and $\mathscr A^+(t)$ still follow the law of $N$-BBM and $N^+$-BBM, respectively. Furthermore, we have for every realisation (hence, almost surely): $\nu^N_t \preceq \nu^+_t$ for every $t\ge 0$. This finishes the proof of the second inequality.

The first inequality is proven very similarly. Here, one defines for every $t\ge0$ a random injective map $\phi_t:\mathscr A^-(t) \to \mathscr A^N(t)$, such that $X^\Pi_u(t) \le X^\Xi_{\phi_t(u)}(t)$ for every $u\in\mathscr A^N(t)$. This connects each particle from the $N^-$-BBM to a particle from the $N$-BBM, at each time $t$. Each time a connection breaks because of particle from the $N$-BBM being killed, we then rewire the corresponding particle from the $N^-$-BBM, similarly as above. The fact that in the $N^-$-BBM the number of particles is always at most $N$ ensures that there always exists a particle to rewire to. We omit the details. 
\end{proof}

\subsection{Proof of Theorem \ref{th:1} assuming the results from Sections~\ref{sec:BBBM_definition} and \ref{sec:Bflat_Bsharp_definition}}
Recall the definitions from the introduction, in particular, let $N\in\N$ be large and set $a_N = \log N + 3\log\log N$ and $\mu_N = \sqrt{1-\pi^2/a_N^2}$.
Let $(\nu_t^N)_{t\ge 0}$ be the empirical measure process of $N$-BBM starting from $N$ particles independently distributed according to the
density proportional to $\sin(\pi x/a_N) e^{- x}\Ind_{(x\in(0,a_N))}$. We wish to show that for every $\alpha \in (0,1)$, the finite-dimensional distributions of the process $(M^N_\alpha(t \log^3 N))_{t\ge0}$ (defined in the introduction)
converge weakly as $N\to\infty$ to those of the L\'evy process $(L_t)_{t\ge0}$ stated in Theorem~\ref{th:1} starting from $x_\alpha$. We will do this by proving separately a lower and an upper bound and show that in the limit these bounds coincide and equal the L\'evy process $(L_t)_{t\ge0}$.

\paragraph{Lower bound.} Fix $\alpha\in(0,1)$ and $\delta > 0$. We assume that $\delta$ is small enough such that $\alpha < e^{-2\delta}$. We let $N$ and in parallel $A$ and $a$ go to infinity (in the meaning of Section~\ref{sec:BBBM_parameters}) in such a way that $N = N^\flat = 2\pi  e^{A+\delta}a^{-3}e^{\mu a}$, where $\mu$ is defined in \eqref{eq:mu}. Since $\mu = 1 + O(1/a^2)$ for large $a$, this gives, 
\begin{align*}
 \log N = a - 3 \log a  + A+\delta +\log(2\pi) + o(1) = (1+o(1))a.
\end{align*}
This gives $\log \log N = \log a + o(1)$ and thus, 
\begin{equation}
 \label{eq:a_asymp}
 a = a_N - (A+\delta +\log(2\pi) + o(1)).
\end{equation}
Write $\ep = a_N - a$. Then
\[
 \frac 1 {a^2} = \frac 1 {a_N^2}\frac 1 {(1-\ep/a_N)^2} = \frac 1 {a_N^2} \left(1+\frac {2\ep}{a_N} + O((\ep/a_N)^2)\right),
\]
so that
\begin{align*}
\nonumber
\mu &= \sqrt{1-\frac{\pi^2}{a^2}} = 1 - \frac{\pi^2}{2a^2} + O(1/a^4)\\
\nonumber
&= 1 - \frac{\pi^2}{2a_N^2} + \frac{\pi^2\ep}{a_N^3} + O(1/a^4 + \ep^2/a_N^4).
\end{align*}
With $\mu_N = \sqrt{1-\pi^2/a_N^2} =  1 - \pi^2/(2a_N^2) + O(1/a_N^4)$ and $1/a_N^3 = 1/a^3 + O(\ep/a^4)$ we get,
\begin{equation}
\label{eq:127}
\mu = \mu_N + \frac{\pi^2}{a^3}(A+\delta+\log(2\pi) + o(1)).
\end{equation}
Let $(\nu^{\flat}_t)_{t\ge 0}$ be the empirical measure process of B$^\flat$-BBM starting from the initial configuration (H$_\perp$), i.e.\ $\nu^\flat_0$ is obtained from $\lfloor e^{-\delta}N\rfloor$ particles distributed independently according to the density proportional to $\sin(\pi x/a)e^{-\mu x}\Ind_{(0,a)}(x)$. We claim that one can couple $\nu^\flat_0$ and $\nu^N_0$ such that $\nu^\flat_0 \stackrel{\mathrm{st}}{\preceq} \nu^N_0$ with high probability for large $N$ (in the sense of  Section~\ref{sec:coupling}). Indeed, write
\[
 \varphi(x) = c_\varphi \sin(\pi x/a)e^{-\mu x}\Ind_{(0,a)}(x),\quad \varphi_N(x) = c_N \sin(\pi x/a_N)e^{-x}\Ind_{(0,a_N)}(x),
\]
where $c_\varphi$ and $c_N$ are such that the functions are probability densities. Expanding the sine functions at $0$ and using the fact that $a/a_N\to1$ and $\mu\to 1$, we get $c_\varphi/c_N\to 1$ as $N\to\infty$. We now apply the inequality $b\sin(\pi x/b)\le c\sin(\pi x/c)$, $b\le c$, $0\le x\le b$ with $b=a$ and $c = a_N$, to get for large $N$, for every $\delta'\in(0,\delta)$,
\begin{align*}
 e^{-\delta'} \varphi(x) &\le e^{-\delta'} c_\varphi \frac{a_N} a \sin(\pi x/a_N) e^{-x + O(x/a^2)} \Ind_{(0,a)}(x)\\
&\le c_N \sin(\pi x/a_N) e^{-x} \Ind_{(0,a_N)}(x) = \varphi_N(x).
\end{align*}
The rejection algorithm then allows us to construct from $\nu^N_0$ a random number $\widetilde N$ of iid particles distributed according to the density $\varphi(x)$ whose empirical measure $\widetilde \nu$ satisfies $\widetilde \nu \stackrel{\mathrm{st}}{\preceq} \nu^N_0$. We can then construct $\nu^\flat_0$ from $\widetilde \nu$ by choosing $\lfloor e^{-\delta} N\rfloor$ particles amongst these, completing with independent particles if necessary. Since $\widetilde N$ is binomially distributed with parameters $N$ and $e^{-\delta'}$ and therefore greater than $e^{-\delta} N$ with high probability, this gives a coupling between $\nu^\flat_0$ and $\nu^N_0$ such that $\nu^\flat_0 \stackrel{\mathrm{st}}{\preceq} \nu^N_0$ with high probability for large $N$.

Now, if $X_t^\flat$ denotes the barrier process of the B$^\flat$-BBM, then $(\nu^-_t)_{t\ge 0} = (\nu^\flat_t+\mu t + X^\flat_t)_{t\ge0}$  is by definition an instance of the $N^-$-BBM defined in Section~\ref{sec:coupling}. Here, for a measure $\nu$ and a number $x$ we denote by $\nu+x$ the measure $\nu$ translated by $x$. Lemma~\ref{lem:coupling} now gives for large $N$, $(\nu^-)_{t\ge0} \stackrel{\mathrm{st}}{\preceq} (\nu^N_t)_{t\ge0}$ w.h.p., which by definition implies
\begin{equation}
  \label{eq:128}
 (\med^N_\alpha(\nu^-_t))_{t\ge0} \stackrel{\mathrm{st}}{\le} (\med^N_\alpha(\nu^N_t))_{t\ge0},\quad\text{w.h.p.}
\end{equation}
Given $0\le t_1<\ldots<t_n$, we now define $t_i^N = a^{-3}t_i \log^3 N$, so that $t_i^N\to t_i$ for every $i$, as $N\to\infty$. We then have w.h.p.,
\begin{align*}
\Big(M^N_\alpha\big(t_i \log^3 N\big)\Big)_{i=1}^n 
&\stackrel{\mathrm{st}}{\ge} \big(\med^N_{\alpha}(\nu^-_{t_i \log^3 N})-\mu_Nt_i \log^3 N\big)_{i=1}^n && \text{by \eqref{eq:128}}\\
&= \big(\med^N_{\alpha}(\nu^\flat_{a^3t_i^N}) + X^\flat_{a^3t_i^N}-\pi^2(A+\delta+\log(2\pi)+o(1))t_i^N\big)_{i=1}^n && \text{by \eqref{eq:127}.}
\end{align*}
By the \Bfl-BBM part of Theorem~\ref{th:Bflat_Bsharp}, the vector  $(X^\flat_{a^3t_i^N}-\pi^2(A+\log(2\pi)t_i^N))_{i=1}^n$ converges in law to $(L_{t_i})_{i=1}
^n$, with $(L_t)_{t\ge 0}$ being the L\'evy process from the statement of Theorem~\ref{th:1}. Together with the previous inequality and the \Bfl-BBM part of Proposition~\ref{prop:Bflat_Bsharp_med}, this yields for large $N$, w.h.p.,
\[
 \Big(M^N_\alpha\big(t_i \log^3 N\big)\Big)_{i=1}^n \stackrel{\mathrm{st}}{\ge} \big(x_{\alpha e^{2\delta}} + L_{t_i} - O(\delta)\big)_{i=1}^n.
\]
Letting first $N\to\infty$, then $\delta\to 0$ yields the proof of the lower bound.

\paragraph{Upper bound.} The proof is analogous to the previous case, with two differences: First, the B$^\sharp$-BBM is not a realisation of the $N^+$-BBM, as defined in Section~\ref{sec:coupling}. However, the C$^\sharp$-BBM (defined in Lemma~\ref{lem:Csharp}) is such a realisation and by that lemma, the B$^\sharp$-BBM parts of Theorem~\ref{th:Bflat_Bsharp} and Proposition~\ref{prop:Bflat_Bsharp_med} hold for the C$^\sharp$-BBM as well. Second, the construction of the coupling between $\nu^\sharp_0$ and $\nu^N_0$ has to be slightly modified since it is not true that for all $\delta'\in(0,\delta)$ and $x\ge0$, we have $\varphi(x) \ge e^{-\delta'}\varphi_N(x)$. However, one can check that $\varphi(x) \ge e^{-\delta'}\varphi_N(x)\Ind_{x\in[0,a_N-A^2]}$ for large $N$. Hence, if $\widetilde \nu^N_0$ is obtained from $\nu^N_0$ by removing the particles in the interval $[a_N-A^2,a_N]$, then we can show as above that $\widetilde \nu^N_0\stackrel{\mathrm{st}}{\preceq}\nu^\sharp_0$ w.h.p. Furthermore, since $N\varphi_N([a_N-A^2,a_N]) = O(A^2 Ne^{-a}) = o(1)$, we have $\widetilde\nu^N_0 = \nu^N_0$ with high probability, whence $\nu^N_0\stackrel{\mathrm{st}}{\preceq}\nu^\sharp_0$ w.h.p. This finishes the proof of the upper bound and of Theorem~\ref{th:1}.

\section{The B-BBM}
\label{sec:BBBM}

In this section, the longest of the article, we prove Theorems~\ref{th:barrier} and~\ref{th:barrier2} about the B-BBM (refer to Section~\ref{sec:BBBM_definition} for the definition of the B-BBM and several quantities associated to it). The core of the proof is Proposition~\ref{prop:piece} below. Recall that $\nu_0$ is the empirical measure of the initial particles and let $\nu_t$ be the empirical measure of the particles of B-BBM at time $t$. Further recall the definitions of the functions $w_Z$ and $w_Y$ from \eqref{eq:def_wZY}:
\begin{equation*}
w_Z(x) = ae^{\mu (x-a)}\sin(\pi x/a)\Ind_{(x\in[0,a])},\qquad w_Y(x) = e^{\mu(x-a)}.
\end{equation*}
The variable $Z'_t$ defined in Section~\ref{sec:BBBM_definition_definition} then can be written as \(Z_t' = \int w_Z(x)\,\nu_t(\dd x)\). We also define \(Y_t' = \int w_Y(x)\,\nu_t(\dd x)\). Note that later (in Section~\ref{sec:BBBM_more_definitions}), we will define $Z_t$ and $Y_t$, whose definitions closely match those of $Z_t'$ and $Y_t'$ but are technically more convenient to work with.

We now define a sequence of events $(G_n)_{n\ge 0}$ recursively as follows. The event $G_0$ is defined as 
\[
G_0 = \{\supp \nu_0 \subset (0,a),\ |e^{-A}Z_0' - 1| \le \ep^{3/2},\ Y_0' \le \eta\},
\]
The role of this event is basically to ensure that the time before a particle hits the right barrier is of order $e^{-A} a^3$. It is easy to show that under the hypothesis (H$_\perp$), the event $G_0$ holds with high probability (tending to 1 as $A$ and $a$ go to infinity). 

The event $G_n$ for $n\ge 1$ is then defined to be the intersection of $G_{n-1}$ with the following events
\begin{itemize}[nolistsep]
 \item $\supp \nu_{\Theta_n} \subset (0,a)$, $|e^{-A}Z_{\Theta_n}' -1| \le \ep^{3/2}$ and $Y_{\Theta_n}' \le \eta$.
 \item $\forall u\in\mathscr A_0(\Theta_n): \Theta_n \not\in \bigcup_{l\ge0} [\tau_l(u),\sigma_{l+1}(u))$, i.e.\ no particles at time $\Theta_n$ are in an ``intermediate state'' (this is equivalent to $\mathscr N_{\Theta_n} \subset U\times \{\Theta_0\}$),
 \item $\Theta_n > T_n^+$, i.e.\ in the $n$-th piece of the process, the descendants of the fugitive reach the critical line at the latest at time $T+e^Aa^2$.
\end{itemize}
Note that $G_n\in\F_{\Theta_n}$ for every $n\ge 0$.

Recall that $p_B$ is the probability of a breakout as defined in Section~\ref{sec:BBBM_breakout}.
\begin{proposition}
\label{prop:piece}
Fix $\lambda \in \R$ and define $\widetilde{\ep} = (\pi p_Be^A)^{-1}$.
There exists a numerical constant $\numcst > 0$, such that for large $A$ and $a$, we have $\PB(G_n^c\,|\,\F_0)\Ind_{G_0} \le n\ep^{1+\numcst}$ for every $n\ge 0$ and
\begin{equation}
\label{eq:piece_fourier}
 \EB\Big[e^{i\lambda X^{[\infty]}_{\Theta_n}}\,\Big|\,\F_0\Big]\Ind_{G_0} = \exp\Big(n\widetilde{\ep}(\kappa(\lambda) + i\lambda \pi^2A + o(1))\Big)\Ind_{G_0} + O(n\ep^{1+\numcst}).
\end{equation} 
where $\kappa(\lambda)$ is the right-hand side of \eqref{eq:laplace_levy} with $c = \pi^2(\const{eq:W_expec}+\const{eq:c_log}-\log \pi)$ and where $o(1)$ and $O(\ep^\numcst)$ may depend on $\lambda$.
\end{proposition}
Proposition~\ref{prop:piece} is proved in Section~\ref{sec:piece_proof}.
Note that we will show in Proposition~\ref{prop:T} that $\Theta_1$ is approximately exponentially distributed with mean $\widetilde{\ep}a^3$. 

\subsection{More definitions}
\label{sec:BBBM_more_definitions}

The random variables $Z'_t$ and $Y'_t$ defined above are not useful when there are particles near or above the point $a$. To remedy this, we define for each $l\ge 0$ and $t\ge0$ the stopping line $\mathscr N^{(l)}_t$ containing the particles of tier $l$ that have already come back to the critical line by time $t$, as well as the descendants of those that haven't, at the moment at which they hit the critical line:
\begin{multline*}
 \mathscr N^{(l)}_t = \{(u,s)\in U\times \R_+: u\in\mathscr A_0(s),\ \tau_{l-1}(u) \le s < \tau_{l}(u),\\
\text{ and either }t<\sigma_l(u) = s\tor \sigma_l(u) \le t = s\},\ l\ge 0,
\end{multline*}
and set $\mathscr N_t = \bigcup_{l\ge 0}  \mathscr N^{(l)}_t.$ We then define for $l\ge 0$,
\[
 Z_t^{(l)} = \sum_{(u,s)\in\mathscr N^{(l)}_t} w_Z(X_u(s)),\quad Y_t^{(l)} = \sum_{(u,s)\in\mathscr N^{(l)}_t} w_Y(X_u(s)).
\]
Similarly, we set \(R^{(l)}_I = \#(\mathscr R^{(l)}_\infty\cap (U\times I))\) for every $l\ge0$ and every interval $I\subset\R_+$ and set $R^{(l)}_t = R^{(l)}_{[0,t]}$ for $t\ge0$. We then write for any symbol $S$ (e.g.~$S\in\{Z_t,Y_t,R_t,R_I\}$) and $0\le k\le l\le \infty$,
\begin{equation}
\label{eq:sum_over_tiers}
 S^{(k;l)} = \sum_{i=k}^l S^{(i)},\quad S^{(l+)} = S^{(l;\infty)},\quad S = S^{(0+)}.
\end{equation}

Later, we will also want to distinguish between the particles related to the fugitive and the others (recall from Section~\ref{sec:BBBM_tiers}, that the \emph{fugitive}, denoted by $\mathscr U$, is the particle that breaks out first). We will actually split the particles into four groups, denoted by a \emph{hat}, \emph{bar}, \emph{check} and the superscript ``fug'', respectively. The first group, the hat-particles, consists of those which share no relation with the fugitive, i.e.\ their ancestor at time 0 is not the ancestor of the fugitive. The second type of particles, denoted by a bar, are those whose most recent common ancestor with the fugitive has died between the times $\sigma_l(\mathscr U)$ and $\tau_l(\mathscr U)$ for some $l\ge 0$. The third group, denoted by a check, consists of those particles whose most recent common ancestor with the fugitive has died between the times $\tau_l(\mathscr U)$ and $\sigma_{l+1}(\mathscr U)$ for some $l\ge0$. The last group then consists of the fugitive and its 
descendants and is denoted by the superscript ``fug''.

Formally, this reads as follows. Recall the notation $u\wedge v$ for the most recent common ancestor of $u$ and $v$ (with $u\wedge v = \emptyset$ if it does not exist) and $d_u$ for the death time of the individual $u$ (see Section~\ref{sec:preliminaries_definition}).
\begin{align*}
 \mathscr{\widehat A}_0(t) &= \{u\in \mathscr A_0(t): u\wedge \mathscr U = \emptyset\},\\
 \mathscr{\widebar A}_0(t) &= \{u\in \mathscr A_0(t): d_{u\wedge \mathscr U} \in \bigcup_{l\ge 0}[\sigma_l(\mathscr U),\tau_l(\mathscr U))\},\\
 \mathscr{\widecheck A}_0(t) &= \{u\in \mathscr A_0(t): d_{u\wedge \mathscr U} \in \bigcup_{l\ge 0}[\tau_l(\mathscr U),\sigma_{l+1}(\mathscr U))\},\\
 \mathscr A^{\mathrm{fug}}_0(t) &= \{u\in \mathscr A_0(t): \mathscr U \preceq u \}.
\end{align*}
We also use this notational convention for all the other quantities, e.g.\ $\mathscr{\widehat N}_t = \{(u,s)\in\mathscr N_t:u\in\mathscr{\widehat A}_0(s)\}$, $\mathscr{\widebar R}^{(l)}_t = \{(u,s)\in\mathscr R^{(l)}_t:u\in\mathscr{\widebar A}_0(s)\}$, $\widecheck Z^{(l)}_t = \sum_{u\in \mathscr{\widecheck N}_t}w_Z(X_u(t))$, $R^{\mathrm{fug}}_t = \#\mathscr R^{\mathrm{fug}}_t$.

\paragraph{The random shift $\Delta$.}
With the previous definitions, we can now properly define the random shift $\Delta$ from the definition of the B-BBM in Section~\ref{sec:BBBM_definition_definition}. It would be tempting to define it to equal $\log(e^{-A}Z_T)$, but there is a slight problem here: In order to study the system after the time $T$, we would have to condition on the $\sigma$-field $\F_T$, otherwise the shape of the barrier, or, equivalently, the varying drift determined by $\Delta$ would be random (this would be undesirable because the results of Section~\ref{sec:interval} require the barrier function $f$ to be deterministic). The problem about conditioning on $\F_T$ is that the estimates on the distribution of particles in the B-BBM would then only be good for times $t$ with $t - T \gg a^2$. This would cause severe problems in Sections~\ref{sec:Bflat} and \ref{sec:Bsharp}.

Instead, we will condition on the particles at a time $T^- < T$ and define $\Delta$ in terms of the particles at that time. Furthermore, we do not take into account all particles, but only those which will have a non-negligible impact later on.

The actual definitions are as follows: Define $T^- = (T-e^Aa^2)\vee 0$ and set 
\begin{equation}
 \label{eq:def_check_Z_Delta}
 \widecheck Z_{\Delta} = \begin{cases}
                          \widecheck Z^{(1)}_{T^-} & \tif \tau_0(\mathscr U) \le T^-\\
                          Z^{(\mathscr U_0',\tau_0(\mathscr U))} & \tif \tau_0(\mathscr U) \in (T^-,T)\\
                          0 & \tif \tau_0(\mathscr U) = T,
                         \end{cases}
\end{equation}
where $\mathscr U_0'$ is the ancestor of $\mathscr U$ alive at time $\tau_0(\mathscr U)$. Note that $\widecheck Z_{\Delta}\ne 0$ implies $T\ne T^{(0)}$. We then define
\begin{equation}
\label{eq:Delta_def}
\Delta =  \log\Big(e^{-A}(\widehat Z_{T^-} + \widecheck Z_{\Delta} +Z^{(\mathscr U,T)})\Big).
\end{equation}

\subsection{The time of the first breakout}
\label{sec:time_breakout}

In this section, we study the law of the time of the first breakout. Specifically, we show in Proposition~\ref{prop:T} that it is approximately exponentially distributed with parameter $p_B\pi Z_0/a^3$ (we will assume throughout the section that the initial conditions are non-random). Since the law of the process until the first breakout is the same under the laws $\PB$ and $\P$, we will work with the law $\P$ in this section, as well as in Sections~\ref{sec:before_breakout_particles} and~\ref{sec:fugitive}. Define
\begin{equation}
 \label{eq:tcrit}
 \tconst{eq:tcrit}= \tfrac{1}{20 A} a^3.
\end{equation}

\begin{proposition}
\label{prop:T}
 Let $0\le t\le  \tconst{eq:tcrit}$. For $A$ and $a$ large enough, we have
\begin{equation}
\label{eq:T}
\P(T> t) = \exp\Big(-\pi p_B Z_0\frac t {a^3}\Big(1 + O(At/a^3 + p_B)\Big)+O(p_B Y_0)\Big).
\end{equation}
\end{proposition}

Before attacking the proof of Proposition~\ref{prop:T}, we will collect some results about the quantities from Section~\ref{sec:BBBM_breakout}. The proof then proceeds by a sequence of lemmas. Lemma~\ref{lem:T0} gives a estimate on $\P(T^{(0)} > t)$. This is used in Lemma~\ref{lem:T_a}, in order to obtain an estimate on $\P^a(T > t\,|\,T\ne 0)$, using a recursive argument. Finally, Proposition \ref{prop:T} is proven by combining Lemmas~\ref{lem:T0} and \ref{lem:T_a}.

Recall the definition of the breakout event $B$ with its probability $p_B$ and the associated quantities $Z$, $Y$, $W_y$ and $\sigma_{\mathrm{max}}$ from Section~\ref{sec:BBBM_breakout}. For later use of the results that follow, we allow for a varying drift given by a barrier function $f$; note that the definition of the critical line then has to be adapted as in Remark~\ref{rem:hat_f}. Define the law of BBM started at $a$ with the first particle conditioned not to break out:
\[
 \Q_f^a(\cdot) =  \P_f^a(\cdot\,|\, T > 0) = \P_f^a(\cdot\,|\, B^c) = \frac{\P_f^a(\cdot,\ B^c)}{1-p_B}.
\]
If $f\equiv 0$, then we omit it from the subscript as usual.
The following lemma, which essentially already appeared in \cite{Berestycki2010}, will be crucial to the study of the law $\Q_f^a$. Note that $\sigma_{\mathrm{max}} \le \zeta$ on the event $B^c$, by definition.
\begin{lemma}Let $f$ be a barrier function with $\|f\|$ bounded by some fixed function of $A$. 
\label{lem:ZYW}
 \begin{enumerate}
  \item On the event $\sigma_{\mathrm{max}} \le \zeta$, we have $\P_f^a$-almost surely, for large $a$,
  \[
 Z =\pi W_y\Big(1+O\Big(\frac 1 a \Big)\Big)\quad\tand\quad  Y = \frac 1 { y} W_y \Big(1+O\Big(\frac{1}{a}\Big)\Big),
\]
and, in particular, $Y \le \eta Z$ for large $a$. 
 \item There exists a random variable $W\ge 0$, whose law depends only on the reproduction law $q(k)$, such that $\P_f^a(|W_y - W| > \eta) < \eta/2$  for $\eta$ fixed and $y$ large enough. Furthermore, the variable $W$ satisfies for some constant $\const{eq:W_expec}\in\R$,
\begin{alignat}{2}
\label{eq:W_tail}
 &P(W > x) \sim \frac{1}{x},&\quad &\tas x\to\infty,\\
 \label{eq:W_expec}
 &E[W\Ind_{(W\le x)}] - \log x \to \const{eq:W_expec},&\quad &\tas x\to\infty.
\end{alignat} 
(Note: probability and expectation with respect to $W$ alone will always be denoted by $P$ and $E$, respectively).
\item For fixed $\eta$ and $y$ and large $\zeta$, we have for all $a$ and $f$, $\P_f^a(\sigma_{\mathrm{max}} > \zeta) < \eta/2$.
 \end{enumerate}
\end{lemma}
\begin{proof}
By \eqref{eq:c0_mu}, every point $(x,t)$ on the critical line with $t\le \zeta$ satisfies 
$$|x-(a-y)| \le (1-\mu)t + \|f'\|_\infty t/a^2 \le (\pi^2/2+\|f\|)\zeta/a^2.$$
By the hypothesis on $\|f\|$ and the fact that $\zeta$ depends only on $\eta$ and thus only on $A$, this is $O(1/a)$ for large $a$. This readily implies the two equalities of the first part, where for the first equality we also use the fact that for $|z|\le 1$, $w_Z(a-y + z) = e^{-y+z}(\pi y + O(z + (y-z)^3/a^2)$. The remaining inequality of the first part then follows from the fact that $y^{-1}\le \eta$ by hypothesis, see Section~\ref{sec:BBBM_parameters}. For the remaining parts, we can assume that $f\equiv 0$, since the law of $(W_y)_{y\ge0}$ and the law of $\sigma_{\mathrm{max}}$ do not depend on $f$, nor $a$, by definition of the critical line. The second part then has been proven in \cite{Berestycki2010} for the case of dyadic branching (but their arguments carry over to our setting), see Proposition~27 and the remark following it for \eqref{eq:W_tail} and the proof of Proposition~39, in particular Equation~(134), for  \eqref{eq:W_expec}. Their arguments are very ingenious but indirect; for a more direct proof, see \cite[Section~2.4]{MaillardThese}. The third part is immediate (see also \cite[Corollary~25]{Berestycki2010}).
\end{proof}

Together with \eqref{eq:eta}, Lemma~\ref{lem:ZYW} now yields for large $A$ and $a$,
\begin{align}
\nonumber
p_B &= \P_f^a(Z > \ep e^A,\,\sigma_{\mathrm{max}} \le \zeta) + \P_f^a(\sigma_{\mathrm{max}} > \zeta)\\
\nonumber
&= P(W > \pi^{-1}\ep e^A(1+o(1))) + O(\eta)\\
\label{eq:pB}
&= \Big(\pi + o(1)\Big)\frac 1 {\ep e^A} = o(1),
\end{align}
where the last equality follows from \eqref{eq:ep_lower}. Furthermore, for large $A$ and $a$,
\begin{align}
\nonumber
 \E_{\Q_f}^a[Z] &= (1-p_B)^{-1}\E_f^a[Z\Ind_{(Z\le \ep e^A)}\Ind_{(\sigma_{\mathrm{max}} \le \zeta)}]\\
 \nonumber
 &= \Big(E[\pi W\Ind_{(W\le \pi^{-1}\ep e^A(1+o(1)))}] + O(\ep e^A\P_f^a(\sigma_{\mathrm{max}} > \zeta))\Big)(1+O(p_B))\\
\label{eq:QaZ}
&= \pi(A+\log\ep+\const{eq:QaZ}+ o(1)),
\end{align}
by Lemma~\ref{lem:ZYW}, \eqref{eq:ep_upper}, \eqref{eq:ep_lower}, \eqref{eq:eta} and \eqref{eq:pB}. Here, $\const{eq:QaZ} = \const{eq:W_expec} - \log \pi$. In particular, we have for large $A$ and $a$, by \eqref{eq:ep_upper},
\begin{equation}
 \label{eq:QaZC}
\E_{\Q_f}^a[Z] \le \pi A.
\end{equation}

By \eqref{eq:W_tail}, we have $E[W^2\Ind_{(W\le x)}] \sim x$ as $x\to\infty$. Similarly as above, we then have for large $A$ and $a$,
\begin{equation}
 \label{eq:QaZsquared}
\E_{\Q_f}^a[Z^2] \le C\ep e^A.
\end{equation}

We can now start the proof of Proposition~\ref{prop:T} by stating the first lemma:
\begin{lemma}
 \label{lem:T0}
Let $0\le t\le a^3$.  Suppose that $p_B \le 1/2$. Then,
\begin{equation*}
\P(T^{(0)}> t) = \exp\Big( -\pi p_B Z_0 \frac t {a^3} \big(1+O(p_B)\big) + O(p_BY_0)\Big).
\end{equation*}
\end{lemma}
\begin{proof}
 Let $x_1,\ldots,x_n$ be the positions of the initial particles. Since the initial particles spawn independent branching Brownian motions, we have
\begin{equation*}
 \P(T^{(0)} > t) = \prod_i \P^{x_i}(T^{(0)}>t).
\end{equation*}
We have for every $x\in(0,a)$,
\begin{equation}
 \label{eq:308}
  \P^x(T^{(0)} > t) = \E^x\Big[\prod_{(u,s)\in\mathscr R^{(0)}_t} \Ind_{(B^{(u,s)})^c}\Big] = \E^x\Big[(1-p_B)^{R^{(0)}_t}\Big],
\end{equation}
since by the strong branching property, the events $B^{(u,s)}$ are independent conditioned on $\F_{\mathscr R^{(0)}_t}$. By Lemma~\ref{lem:Rt} and the assumption $t\le a^3$, we have
\begin{align}
 \label{eq:310}
&|\E^x[R^{(0)}_t] - \pi w_Z(x) t/a^3| \le C w_Y(x),\\
 \label{eq:315}
&\E^x[(R^{(0)}_t)^2] \le C (w_Z(x) t/a^3+w_Y(x)).
\end{align}
By Jensen's inequality, \eqref{eq:310} and the inequality $|\log(1-z)| \le z+z^2$ for $z\in[0,1/2]$,
\begin{equation}
 \label{eq:317}
 \E^x\Big[(1-p_B)^{R^{(0)}_t}\Big] \ge \exp\Big( -\pi p_B w_Z(x) \frac t {a^3} (1+ p_B) - Cp_Bw_Y(x)\Big).
\end{equation}
Furthermore, the inequality $(1-p)^n \le 1-np+n(n-1)p^2/2$ and Equations \eqref{eq:310} and \eqref{eq:315} give
\begin{equation}
\label{eq:325}
 \E^x\Big[(1-p_B)^{R^{(0)}_t}\Big] \le 1-\pi p_B w_Z(x) \frac t {a^3}  (1- Cp_B) + Cp_Bw_Y(x).
\end{equation}
The lemma now follows from \eqref{eq:308}, \eqref{eq:317} and \eqref{eq:325} together with the inequality $1+z \le e^z$ for $z\in\R$.
\end{proof}

\begin{lemma}
 \label{lem:T_a}
Let $0\le t\le \tconst{eq:tcrit}$. Then, for large $A$ and $a$,
\begin{equation*}
\Q^a(T > t) \ge \exp\Big(-C p_B A (\tfrac t {a^3} + \eta)\Big).
\end{equation*}
\end{lemma}
\begin{proof}
We will use the simplified notation from Section~\ref{sec:BBBM_breakout}, in particular, $\mathscr S = \mathscr S^{(\emptyset,0)}$, $Z=Z^{(\emptyset,0)}$ and $Y=Y^{(\emptyset,0)}$. By the strong branching property applied to the stopping line $\mathscr S$, we have
\begin{equation}
 \label{eq:360}
 \Q^a(T > t) = \E_\Q^a\Big[\prod_{(u,s)\in\mathscr S}\P^{X_u(s)}(T > t-s)\Big] \ge \E_\Q^a[\P^\nu(T > t)],
\end{equation}
where $\nu = \sum_{(u,s)\in\mathscr S} \delta_{X_u(s)}$. Since $T>t$ implies $T^{(0)}>t$, we have
\begin{equation}
 \label{eq:362}
\P^\nu(T > t) = \P^\nu(T > t|T^{(0)}>t) \P^\nu(T^{(0)} > t).
\end{equation}
Now note that $Y \le \eta Z$, $\Q^a$-almost surely, by the first part of Lemma~\ref{lem:ZYW}. By Lemma~\ref{lem:T0}, we have for large $A$,
 \begin{equation}
\label{eq:363}
\P^\nu(T^{(0)} > t) \ge \exp\Big(- C p_B Z\left(\tfrac t {a^3}+\eta\right)\Big).
\end{equation} 
Furthermore, with the notation from Section \ref{sec:weakly_conditioned}, with $p(s) \equiv p_B$,
\begin{align*}
 \P^\nu(T > t|T^{(0)}>t) &= \widetilde{\P}^\nu(T > t) = \widetilde \E^\nu\Big[\prod_{(u,s)\in\mathscr R^{(0)}_t}\Q^a(T>t-s)\Big] \ge \widetilde \E^\nu\Big[\Q^a(T>t)^{R^{(0)}_t}\Big].
\end{align*}
By Jensen's inequality, \eqref{eq:conditioned_Rt_constant} and Lemma~\ref{lem:Rt}, this implies
\begin{equation}
\label{eq:364}
 \P^\nu(T > t|T^{(0)}>t) \ge \Q^a(T>t)^{\widetilde{\E}^\nu[R^{(0)}_t]} \ge \Q^a(T>t)^{\pi Z (t/a^3 + O(\eta))}.
\end{equation}
Equations \eqref{eq:360}, \eqref{eq:362}, \eqref{eq:363} and \eqref{eq:364}, together with Jensen's inequality and \eqref{eq:QaZ}, now yield for large $A$ and $a$,
\begin{equation}
\label{eq:370}
 \Q^a(T > t) \ge \Q^a(T>t)^{\pi^2A (t/a^3 + O(\eta))}\times\exp\Big(- C p_B A\left(\tfrac t {a^3}+\eta\right)\Big).
\end{equation}
By the hypothesis on $t$ and \eqref{eq:eta}, the exponent of $\Q^a(T>t)$ in \eqref{eq:370} is smaller than $1/2$ for large $A$ and $a$. This yields the statement.
\end{proof}

\begin{proof}[Proof of Proposition \ref{prop:T}]
The upper bound follows from Lemma~\ref{lem:T0} and the trivial inequality $T\le T^{(0)}$. For the lower bound, we note that as in the proof of Lemma~\ref{lem:T_a}, we have by Jensen's inequality and \eqref{eq:conditioned_Rt_constant},
\begin{equation}
\label{eq:380}
 \P(T > t) = \P(T > t\,|\,T^{(0)} > t)\P(T^{(0)} > t) \ge \Q^a(T > t)^{\E[R_t^{(0)}]}\P(T^{(0)} > t).
\end{equation}
By Lemmas~\ref{lem:T_a} and \ref{lem:Rt}, we have
\begin{equation}
\label{eq:382}
 \Q^a(T > t)^{\E[R_t^{(0)}]} \ge \exp\Big(-C p_B A (\tfrac t {a^3} + \eta)(\tfrac t {a^3} Z_0 + Y_0)\Big),
\end{equation}
The lower bound in \eqref{eq:T} now follows from \eqref{eq:380}, \eqref{eq:382} and Lemma~\ref{lem:T0}, together with the hypothesis on $t$, \eqref{eq:eta} and \eqref{eq:pB}.
\end{proof}

We now state two corollaries of Proposition~\ref{prop:T}. The first gives a handy estimate for the moments of $T$, the second a coupling of $T$ with an exponentially distributed random variable.

\begin{corollary}
\label{cor:moments_T}
Define $\gamma = (\pi p_B Z_0)^{-1}$. Suppose that $Y_0\le C$. Then, for large $A$ and $a$, for every $n\in\N$ and $l\in\N\cup\{\infty\}$,
\begin{equation*}
 \E[((T^{(0;l)}/a^3) \wedge 1)^n] \le C n! (2\gamma)^n.
\end{equation*}
\end{corollary}
\begin{proof}
Follows from the equality
\[
\E[((T^{(0;l)}/a^3) \wedge 1)^n] = n\int_0^1 t^{n-1}\,\P(T^{(0;l)}>ta^3)\,\dd t,
\]
together with Lemma~\ref{lem:T0}, the trivial inequality $T^{(0;l)}\le T^{(0)}$, the hypothesis on $Y_0$ and the fact that $p_B\to0$ as $A$ and $a$ go to infinity by \eqref{eq:pB}.
\end{proof}

\begin{corollary}
\label{cor:coupling_T}
 Suppose that $e^{-A}Z_0 = 1 + O(\ep^{3/2})$ and that $Y_0 \le \eta Z_0$. Then there exists a coupling $(T,V)$, such that $V$ is a random variable which is exponentially distributed with parameter $\pi p_Be^A$, $T$ is $\sigma(V)$-measurable and $\P(|T/a^3 - V| > \ep^{3/2}) \le C\ep^2$ for large $A$ and  $a$.
\end{corollary}

\begin{proof}
Set $\gamma := \pi p_Be^A$. Define $F(t) := \P(T \ge t)$, which is a continuous function since  $T$ has no atoms except at infinity\footnote{This follows from the facts that 1) the first hitting time of a point by a Brownian motion has no atoms and 2) the number of individuals alive at a fixed time in branching Brownian motion is almost surely finite.}. We can therefore define a random variable $U$, uniformly distributed on $(0,1)$, by setting
\[
  U = F(T)\Ind_{(T<\infty)} + U' F(\infty)\Ind_{(T=\infty)},
\]
where $U'$ is a uniformly distributed random variable on $(0,1)$, independent of $T$. Now define $V = -\gamma^{-1} \log U$. Then $V$ is exponentially distributed with parameter $\gamma$ and $T = F^{-1}(e^{-\gamma V})$, where $F^{-1}$ denotes the generalised inverse of $F$. Hence, $T$ is $\sigma(V)$-measurable. On $\{T<\infty\}$, we have by Proposition \ref{prop:T}, for $a$ large enough,
\[
\begin{split}
 V &= -\gamma^{-1} (\pi p_Be^A e^{-A} Z_0 T/a^3(1+ O(AT/a^3 + p_B))) + O(p_BY_0)\\
&= T/a^3(1 + O(\ep^{3/2} + AT/a^3 + p_B)) + O(p_B e^A\eta),
\end{split}
\]
by the hypotheses on $Z_0$ and $Y_0$. Hence, by \eqref{eq:ep_lower}, \eqref{eq:eta_ep} and \eqref{eq:pB}, we have for $a$ large enough,
\begin{equation}
\label{eq:390}
 |T/a^3 - V| = O(\ep^{3/2}T/a^3 + A(T/a^3)^2) + O(\ep^{2}).
\end{equation}
But now we have by Lemma~\ref{lem:T0}, for large $A$ and $a$,
\begin{equation}
\label{eq:392}
 \P(T/a^3 > \ep^{7/8}) \le \P(T^{(0)}/a^3 > \ep^{7/8}) \le Ce^{-O(\ep^{-1/8})} \le \ep^2/2.
\end{equation}
The statement now follows from \eqref{eq:390} and \eqref{eq:392} together with \eqref{eq:ep_upper}.
\end{proof}

\subsection{The particles that do not participate in the breakout}
\label{sec:before_breakout_particles}

In this section, we fix $\vartheta \le  \tconst{eq:tcrit}$. As in the previous section, we assume that the initial condition is non-random. We define the law of BBM conditioned not to break out before $\vartheta$ from the tiers $0,\ldots,l$ by
\[
 \Phat_{(l)}(\cdot) = \P(\cdot\,|\,T^{(0;l)} > \vartheta),\quad\text{with }\Phat = \Phat_\infty.
\]
Expectation w.r.t. $\Phat_{(l)}$ is denoted by $\Ehat_{(l)}$. Under the law $\Phat_{(l)}$, the process stopped at $\mathscr R^{(0)}_\infty$ then follows the law $\Ptilde$ from Section \ref{sec:weakly_conditioned}, with
\begin{equation}
 \label{eq:def_ps}
p(t) = \P^{(a,t)}(T^{(0;l)} \le \vartheta).
\end{equation}
By Proposition~\ref{prop:T}, we have for every $x\in[0,a]$ and $t\ge0$, for large $A$ and $a$,
\begin{equation}
  \label{eq:hbar}
 \P^{(x,t)}(T^{(0;l)} > \vartheta) = \P^{x}(T^{(0;l)} > \vartheta-t) \ge \P^{x}(T > \vartheta) \ge e^{-Cp_B}.
\end{equation}
In particular, \eqref{eq:def_ps} and \eqref{eq:hbar} give,
\begin{equation}
 \label{eq:pbar_estimate}
\|p\|_\infty  \le C p_B.
\end{equation}
Moreover, \eqref{eq:hbar} implies for large $A$ and $a$,
\begin{equation}
 \label{eq:burglars}
 \forall t\le \vartheta,\quad \Q^a(T^{(0;l)} > t) \ge \Q^a(T > \vartheta) = \frac{\P^a(T>\vartheta)}{\P^a(T>0)} \ge e^{-Cp_B}.
\end{equation}

The following corollary is an immediate consequence of \eqref{eq:hbar} together with some results from Section~\ref{sec:weakly_conditioned}, namely, 
Lemma~\ref{lem:many_to_one_conditioned}, Corollary~\ref{cor:conditioned_upperbound} and Lemma~\ref{lem:conditioned_Z_lowerbound} (applied with $h(x,t) = \P^{(x,t)}(T^{(0;l)} > \vartheta)$).
\begin{corollary}
  \label{cor:hat_many_to_one}
 Let  $x\in (0,a)$, $t\in[0,a^3]$ and  $g:[0,a]\to \R_+$ be measurable. Define $S^{(0)}_t = \sum_{(u,s)\in\mathscr N^{(0)}_t} g(X_u(s))$ or $S^{(0)}_t = R^{(0)}_t$. Then, for large $A$ and $a$,
\begin{equation*}
 \Ehat_{(l)}^x[S^{(0)}_t] = (1+O(p_B)) \E^x[S^{(0)}_t] \quad\tand\quad \Ehat_{(l)}^x[(S^{(0)}_t)^2] \le (1+O(p_B)) \E^x[(S^{(0)}_t)^2].
\end{equation*}
\end{corollary}

We now first study the particles from tiers 1 and higher, the moment they come back to the critical line. We set for $k\ge0$ and $t\ge0$,
\[
 Z_{\emptyset,t}^{(k)} = \sum_{(u,s)\in\mathscr S^{(k)}_t} w_Z(X_u(s)),\quad Y_{\emptyset,t}^{(k)} = \sum_{(u,s)\in\mathscr S^{(k)}_t} w_Y(X_u(s)).
\]
For every $k\le l$ and $t\le \vartheta$, we then have by the first part of Lemma~\ref{lem:ZYW},
\begin{equation}
 \label{eq:imam}
 Y_{\emptyset,t}^{(k+1)} \le \eta Z_{\emptyset,t}^{(k+1)},\quad\text{$\Phat_{(l)}$-almost surely}.
\end{equation}

\begin{lemma}
 \label{lem:Zl_exp_upperbound}
We have for large $A$ and $a$, for all $t\le \vartheta$,
\begin{align}
 \label{eq:Z1_exp_lowerbound}
 &\Ehat_{(l)}[Z_{\emptyset,t}^{(1)}] = (\pi+O(p_B+A^{-1}\vartheta/a^3))\E_\Q^a[Z](\tfrac t {a^3}Z_0 +O(Y_0)),\tand\\
 \label{eq:Zl_exp_upperbound_2}
 &\Ehat_{(l)}[Z_{\emptyset,t}^{(k+1)}] \le C A(\tfrac t {a^3} Z_0 + Y_0)\left((\pi^2+Cp_B)A(\tfrac t {a^3} +C\eta)\right)^{k},\text{ for all $0\le k\le l$.}
\end{align}
\end{lemma}

We will need the following lemma:
\begin{lemma}
For every $l\in\N$ and $t\le \vartheta$, we have for large $A$ and $a$,
\begin{align}
 \label{eq:QhatZ}
\E_\Q^a[Z\,|\,T^{(0;l)} > t] &= (1+O(p_B+A^{-1}\vartheta/a^3)) \E_\Q^a[Z],\\
 \label{eq:QhatZsquared}
 \E_\Q^a[Z^2\,|\,T^{(0;l)} > t] &\le C\ep e^A.
\end{align}
\end{lemma}
\begin{proof}
 The upper bound in \eqref{eq:QhatZ} directly follows from \eqref{eq:burglars} and the fact that $Z$ is positive. Similarly, \eqref{eq:QhatZsquared} follows from \eqref{eq:burglars} together with \eqref{eq:QaZsquared}. As for the lower bound in \eqref{eq:QhatZ}, we have by \eqref{eq:hbar} and the first part of Lemma~\ref{lem:ZYW},
 \[
 \E_\Q^a[Z\,|\,T^{(0;l)} > t] \ge \E_\Q^a[Z\Ind_{(T^{(0;l)} > t)}] \ge \E_\Q^a[Z\,\P(T^{(0;l)} > t\,|\,\F_{\mathscr S})]\ge \E_\Q^a[Z\exp(-Cp_B(\vartheta/a^3+\eta)Z)].
 \]
 Equations~\eqref{eq:QaZsquared} and \eqref{eq:pB} now give,
\[
 \E_\Q^a[Z\exp(-Cp_B(\vartheta/a^3+\eta)Z)] \ge \E_\Q^a[Z] - Cp_B\E_\Q^a[Z^2](\vartheta/a^3+\eta) \ge \E_\Q^a[Z] - C(\vartheta/a^3+\eta).
\]
By \eqref{eq:QaZ}, \eqref{eq:ep_lower} and \eqref{eq:eta}, we have $\E_\Q^a[Z]\ge CA$ and $\eta \le p_B\E_\Q^a[Z]$ for large $A$ and $a$. The preceding equations now yield the lower bound in \eqref{eq:QhatZ}.
\end{proof}

\begin{proof}[Proof of Lemma~\ref{lem:Zl_exp_upperbound}]
Write $\G_k = \F_{\mathscr S^{(k)}_t}$ and $\H_k = \F_{\mathscr R^{(k)}_t}$, so that $\G_k \subset \H_k \subset \G_{k+1}$ for every $k$. For $0\le k\le l$,
\begin{equation*}
 \Ehat_{(l)}[Z_{\emptyset,t}^{(k+1)}\,|\,\H_{k}] = \Ehat_{(l)}\Big[\sum_{(u,s)\in\mathscr R^{(k)}_t} Z^{(u,s)}\,\Big|\,\H_{k}\Big].
\end{equation*}
Conditioned on $\H_{k}$, the random variables $Z^{(u,s)}$ in the last equation are iid under $\Phat_{(l)}$, of the same law as $Z$ under $\Q^a(\cdot\,|\,T^{(0;l-k)} > \vartheta - s)$, and independent of $\H_{k}$, by the strong branching property. With \eqref{eq:QhatZ}, this gives
\begin{equation}
\label{eq:Zl_Hl}
 \Ehat_{(l)}[Z_{\emptyset,t}^{(k+1)}\,|\,\H_{k}] = (1+O(p_B+A^{-1}\vartheta/a^3))\E_\Q^a[Z] R_t^{(k)}.
\end{equation}

Now, we have
\[
 \Ehat_{(l)}[R_t^{(k)}\,|\,\G_{k}] = \sum_{(u,s)\in\mathscr S^{(k)}_t} \E^{X_u(s)}[R_{t-s}^{(0)}\,|\,T^{(0;l-k)} > t-s],
\]
such that by Lemma~\ref{lem:Rt} and Corollary~\ref{cor:hat_many_to_one},
\begin{equation}
\label{eq:400}
\Ehat_{(l)}[R_t^{(k)}\,|\,\G_{k}] \le (1+O(p_B))\Big(\pi \frac t {a^3} Z_{\emptyset,t}^{(k)} + C Y_{\emptyset,t}^{(k)}\Big).
\end{equation}
Equations \eqref{eq:Zl_Hl} and \eqref{eq:400} now yield
\begin{equation}
 \label{eq:Zl_exp_upperbound}
 \Ehat_{(l)}[Z_{\emptyset,t}^{(k+1)}\,|\,\G_k] \le (\pi+C(p_B+A^{-1}\vartheta/a^3))\E_\Q^a[Z](\tfrac t {a^3} Z_{\emptyset,t}^{(k)} + C Y_{\emptyset,t}^{(k)}).
\end{equation}
Equation \eqref{eq:Zl_exp_upperbound_2} follows easily from this by \eqref{eq:QaZC} and \eqref{eq:imam}. Now, in the case $k=0$, we have $\G_0 = \F_0$ by definition. Denote the positions of the initial particles by $x_1,\ldots,x_n$. By Lemmas~\ref{lem:Rt} and \ref{lem:conditioned_Rt},
\[
\begin{split}
 \Ehat_{(l)}[R_t^{(0)}] = \sum_{i=1}^n \Ehat_{(l)}^{x_i}[R_t^{(0)}] &\ge \sum_{i=1}^n \E^{x_i}[R_t^{(0)}] - \|p\|_\infty  \E^{x_i}[(R_t^{(0)})^2] \\
&\ge \pi \frac t {a^3} Z_0 - C Y_0 - \|p\|_\infty  C\left(\tfrac t {a^3} Z_0 +Y_0\right),
\end{split}
\]
which, together with \eqref{eq:Zl_exp_upperbound} and \eqref{eq:pbar_estimate}, yields \eqref{eq:Z1_exp_lowerbound}.
\end{proof}

\begin{corollary}
 \label{cor:Zl_exp_upperbound}
 For every $k\ge 1$ and $t\le \vartheta$, we have for large $A$ and $a$,
 \[
  \Ehat[Z_{\emptyset,t}^{(k+)}] \le CA(\tfrac t {a^3} Z_0 + Y_0)\left(A(\tfrac t {a^3} +C\eta)\right)^{k-1}.
 \]
\end{corollary}
\begin{proof}
By the hypothesis $t\le \vartheta \le \tconst{eq:tcrit}$, \eqref{eq:eta} and \eqref{eq:pB}, we have $(\pi^2+Cp_B)A(\tfrac t {a^3} + C \eta)\le 1/2$ for large $A$ and $a$. Summing \eqref{eq:Zl_exp_upperbound} and using \eqref{eq:QaZC} yields the result.
\end{proof}

We now come to the first moment estimates for the variables $Z_t$ and $Y_t$:
\begin{lemma}
\label{lem:hat_quantities}
Suppose that $t\le \vartheta$. Then for large $A$ and $a$, we have
\begin{equation*}
 \Ehat[Z_t] = Z_0\left(1+\pi \E_\Q^a[Z] \tfrac t {a^3} + O\left(p_B + \left(A\tfrac \vartheta {a^3}\right)^2\right)\right)+O(AY_0),\quad \Ehat[Y_t] \le C \left(Y_0 + \eta A\tfrac t {a^3}Z_0\right).
\end{equation*} 
If moreover $t\ge 2a^2$, then
\begin{equation*}
\Ehat[Y_t\Ind_{(R_{[t-a^2,t]} = 0)}] \le \frac C a (Z_0 + AY_0)\quad\tand\quad \Phat(R_{[t-a^2,t]} \ne 0) \le C\eta(Z_0 + A Y_0).
\end{equation*} 
\end{lemma}
\begin{proof}
By Proposition~\ref{prop:quantities} and Corollary~\ref{cor:hat_many_to_one}, we have
\begin{align}
\label{eq:pasteque}
 \Ehat[Z_t] = \sum_{l\ge0}\Ehat\left[\Ehat[Z_t^{(l)}\mid \F_{\mathscr S_t^{(l)}}]\right] = (1+O(p_B))(Z_0 + \Ehat[Z_{\emptyset,t}^{(1)}] + \Ehat[Z_{\emptyset,t}^{(2+)}]),
\end{align}
and similarly, 
\begin{equation}
\label{eq:afatiach}
  \Ehat[Y_t] \le C(1+O(p_B))(Y_0 + \Ehat[Y_{\emptyset,t}^{(1+)}]) \le C(Y_0+\eta\Ehat[Z_{\emptyset,t}^{(1+)}]), 
\end{equation}
where the last inequality follows from \eqref{eq:imam}. The first statement now readily follows from the previous inequalities together with Corollary~\ref{cor:Zl_exp_upperbound} and \eqref{eq:Z1_exp_lowerbound}.

Now, let $t\ge 2a^2$. By \eqref{eq:pasteque} and Corollary~\ref{cor:Zl_exp_upperbound}, we have for large $A$ and $a$,
\[
\Ehat[Z_{t-a^2}] \le C(Z_0 + AY_0).
\]
Furthermore, as in \eqref{eq:afatiach}, but using \eqref{eq:Yt} for the tier 0 particles,
\[
 \Ehat[Y_{t-a^2}] \le C(a^{-1}Z_0+\eta\Ehat[Z_{\emptyset,t}^{(1+)}]) \le \eta(Z_0 + AY_0),
\]
for large $A$ and $a$. Together with Proposition~\ref{prop:quantities}, Lemma~\ref{lem:Rt} and Corollary~\ref{cor:hat_many_to_one}, this proves the other two inequalities.
\end{proof}

 If the drift of the BBM is changed according to a barrier function $f$ (see Sections~\ref{sec:BBBM_parameters} and \ref{sec:BBBM_definition_definition}), then some of the above estimates still hold under some conditions:
 \begin{lemma}
  \label{lem:hat_Z_f}
  Let $f$ be a barrier function and suppose that $\|f\|$ is bounded by a function depending on $A$ only. Furthermore, suppose that $\vartheta$ and $f$ are such that $\P^{(x,t)}_f(T > \vartheta) \ge 1-Cp_B$ for all $x\in[0,a]$ and $t\ge 0$. Then for $t\le\vartheta\le \tconst{eq:tcrit}$, for large $A$ and $a$,
  \begin{align}
  \label{eq:28}
\Ehat^x_f[Z^{(1+)}_{\emptyset,t}] \le CA\Big(\frac {t} {a^3}w_Z(x)+w_Y(x)\Big),\quad\tand\quad Y^{(1+)}_{\emptyset,t} \le \eta Z^{(1+)}_{\emptyset,t}.
\end{align}
Furthermore, with the notation of Corollary~\ref{cor:hat_many_to_one}, we have for $t\le a^3$, for large $A$ and $a$,
\begin{align}
\label{eq:ichhassedas}
 \Ehat_f^x[S^{(0)}_t] \le (1+O(p_B)) \E_f^x[S^{(0)}_t] \quad\tand\quad \Ehat_f^x[(S^{(0)}_t)^2] \le (1+O(p_B)) \E_f^x[(S^{(0)}_t)^2].
\end{align}
 If $f(s) = 0$ for all $s\le \vartheta$, we also have, for $t\le a^3$, for large $A$ and $a$,
\begin{align}
 \label{eq:scheisse}
 \Ehat_f^x[S^{(0)}_t] = (1+O(p_B)) \E_f^x[S^{(0)}_t].
\end{align}
\end{lemma}
\begin{proof}
As for Corollary~\ref{cor:hat_many_to_one}, Equations \eqref{eq:ichhassedas} and \eqref{eq:scheisse} follow from the hypothesis on $\P^{(x,t)}_f(T > \vartheta)$ together with Lemma~\ref{lem:many_to_one_conditioned}, Corollary~\ref{cor:conditioned_upperbound} and Lemma~\ref{lem:conditioned_Z_lowerbound}. It remains to prove the first inequality in \eqref{eq:28}, the second follows from the first part of Lemma~\ref{lem:ZYW}. Inspection of the proof of the corresponding result with $f\equiv 0$, namely, Corollary~\ref{cor:Zl_exp_upperbound}, shows that one only needs to perform the following changes:
\begin{itemize}
 \item Replace the uses of \eqref{eq:burglars} by the inequality $\Q^{(a,t)}_f(T>\vartheta) := \P^{(a,t)}_f(T>\vartheta|T>t) \ge e^{-Cp_B}$, which follows from the hypothesis on $\P^{(x,t)}_f(T > \vartheta)$.
 \item Replace the uses of Corollary~\ref{cor:hat_many_to_one} by \eqref{eq:28} (only the upper bound of the first-moment estimate is used from Corollary~\ref{cor:hat_many_to_one}).
 \item Use Lemma~\ref{lem:R_f} together with Lemma~\ref{lem:Rt} wherever the latter is used, together with the hypothesis on $\P^{(x,t)}_f(T > \vartheta)$.
\end{itemize}
\end{proof}
In order to verify $\P^{(x,t)}_f(T > \vartheta) \ge 1-Cp_B$ for all $x\in[0,a]$ and $t\ge0$, we cannot use Proposition~\ref{prop:T} anymore (and its proof does not adapt easily). Luckily, we will only need to look at times $\vartheta \le 2e^Aa^2$; for these times, we have the following result:
\begin{lemma}
\label{lem:rit3}
  Let $f$ be a barrier function and suppose that $\|f\|$ is bounded by a function depending on $A$ only. Then for $\vartheta \le 2e^Aa^2$, for large $A$ and $a$, we have for all $x\in[0,a]$, $t\ge0$, 
\begin{equation*}
 \P^{(x,t)}_f(T > \vartheta) \ge 1-Cp_B (w_Y(x)+w_Z(x) \vartheta/a^3) \ge 1-C'p_B.
 \end{equation*}
\end{lemma}
\begin{proof}
Let $x\in[0,a]$ and $t\ge0$. From Lemmas~\ref{lem:Rt} and~\ref{lem:R_f}, \eqref{eq:QaZC} and the inequality from the first part of Lemma~\ref{lem:ZYW}, we easily get for $\vartheta \le 2e^Aa^2$ and for large $A$ and $a$ the following estimates (we omit the details which are similar to parts of the proofs of Lemmas~\ref{lem:T0} and \ref{lem:Zl_exp_upperbound}):
 \begin{align*}
 &\P^{(x,t)}_f(T^{(0)} > \vartheta) \ge 1-Cp_B (w_Y(x)+w_Z(x) \vartheta/a^3)\\
 \tand\quad &\P^{(x,t)}_f(R^{(1)} > 0,\,T^{(0)} > \vartheta) \le C \eta A (w_Y(x)+w_Z(x) \vartheta/a^3).
\end{align*}
By \eqref{eq:pB}, \eqref{eq:eta} and \eqref{eq:ep_lower}, this gives for large $A$ and $a$,
\begin{equation*}
 \P^{(x,t)}_f(T > \vartheta) \ge \P^{(x,t)}_f(T^{(0)} > \vartheta) - \P^{(x,t)}_f(R^{(1)} > 0,\,T^{(0)} > \vartheta) \ge 1-Cp_B (w_Y(x)+w_Z(x) \vartheta/a^3).
\end{equation*}
This yields the lemma.
\end{proof}

In order to estimate second moments, we will make use of the following extension to the many-to-two lemma (Lemma~\ref{lem:many_to_two}). It is valid for time-varying drift as well, but we omit this from the notation, for simplicity. For $x,z\in(0,a)$ and $0\le t\le\vartheta$, we define $\widehat{\p}^{(l)}_t(x,z)$ to be the (expected) density of tier $l$ particles at position $z$ and time $t$ under the law $\Phat^x$, not counting the particles $u$ which haven't come back to the critical line, i.e.\ with $\tau_{l-1}(u) \le t < \sigma_l(u)$. Then set $\widehat{\p}_t = \widehat{\p}^{(0+)}_t$.

\begin{lemma}
  \label{lem:many_to_two_with_tiers}
Let $w:[0,a]\to\R_+$ be measurable, $t\le \vartheta$ and define $S_t = \sum_{(u,t)\in \mathscr N_t} w(X_u(t))$. Then
\begin{multline*}
  \Ehat^x[S_t^2] \le \Ehat^x[\sum_{(u,s)\in\mathscr N_t} w(X_u(s))^2] + \beta m_2 \int_0^t\int_0^a
   \widehat{\p}_s(x,z) \Ehat^{(z,s)}[S_t]^2\,\dd z\,\dd s\\
+ \Ehat^x\Big[\sum_{(u,s)\in\mathscr R_t} \sum_{(v_1,s_1),(v_2,s_2)\in\mathscr
  S^{(u,s)},\,v_1\ne v_2} \Ehat^{(X_{v_1}(s_1),s_1)}[S_t]\Ehat^{(X_{v_2}(s_2),s_2)}[S_t]\Big].
\end{multline*}
\end{lemma}
\begin{proof}
Recall the notation $u\wedge v$ for the most recent common ancestor of $u$ and $v$ (with $u\wedge v = \emptyset$ if it does not exist) and $d_u$ for the death time of the individual $u$ (see Section~\ref{sec:preliminaries_definition}). Define for $l\ge 0$,
\begin{align*}
\mathscr A_1^{(l+1)} &= \{((v_1,s_1),(v_2,s_2))\in \mathscr N_t^2: v_1\ne v_2 \text{ and }  \tau_l(v_0) \le d_{v_1\wedge v_2} < \sigma_{l+1}(v_0)\},\\
\mathscr A_2^{(l)} &= \{((v_1,s_1),(v_2,s_2))\in \mathscr N_t^2: v_1\ne v_2 \text{ and }  \sigma_l(v_0)\le d_{v_1\wedge v_2} < \tau_l(v_0)\}.
\end{align*}
(this is akin to the definition of the check- and bar-particles in Section~\ref{sec:BBBM_more_definitions}). Then,
\begin{equation}
 \label{eq:union}
 \mathscr N_t^2 = \{((v,s),(v,s)): (v,s)\in\mathscr N_t\} \cup \bigcup_{l\ge0} \mathscr A_1^{(l+1)}  \cup \bigcup_{l\ge0} \mathscr A_2^{(l)}.
\end{equation}

We now have
\begin{multline*}
  \Ehat^x\Big[\sum_{((v_1,s_1),(v_2,s_2))\in \mathscr A^{(l+1)}_1} w(X_{v_1}(s_1))w(X_{v_2}(s_2))\,\Big|\,\F_{\mathscr S^{(l+1)}}\Big] \\
= \sum_{(u,s)\in\mathscr R^{(l)}_t} \sum_{(v_1,s_1),(v_2,s_2)\in\mathscr S^{(u,s)},\,v_1\ne v_2} \Ehat^{(X_{v_1}(s_1),s_1)}[S_t]\Ehat^{(X_{v_2}(s_2),s_2)}[S_t].
\end{multline*}

As for the contribution of the pairs in $\mathscr A_2^{(l)}$, note that under $\Phat^x$, if particles are stopped at $a$ (and $0$), then the resulting process is the process $\Ptilde^x$ from Section~\ref{sec:weakly_conditioned}, with $p(\cdot)$ given by \eqref{eq:def_ps}. This is again a BBM with space- and time-dependent branching rate $\widehat{\beta}(y,s)$ and second factorial moment of the reproduction law $\widehat{m_2}(y,s)$. By \eqref{eq:def_tilde_quantities} and the description of this process at the beginning of Section~\ref{sec:weakly_conditioned}, we have $\widehat{\beta}(y,s)\widehat{m_2}(y,s)\le \beta m_2$ for all $y$ and $s$. Conditioning on $\mathscr R^{(l)}_t$, applying Lemma~\ref{lem:many_to_two} and then using the tower property of conditional expectation, we get
\begin{align*}
  \Ehat^x\Big[\sum_{((v_1,s_1),(v_2,s_2))\in \mathscr A^{(l)}_2} w(X_{v_1}(s_1))w(X_{v_2}(s_2))\Big] \le \beta m_2 \int_0^t\int_0^a \widehat{\p}^{(l)}_s(x,z) \Ehat^{(z,s)}[S_t]^2\,\dd z\,\dd s.
\end{align*}
Summing over $l$ in the last two displays and using \eqref{eq:union} yields the lemma.
\end{proof}
\begin{remark}
  \label{rem:many_to_two_with_tiers}
An analogous result holds for $R_t$.
\end{remark}

\begin{lemma}
  \label{lem:Zt_variance}
We have for every $t\le \vartheta$, for large $A$ and $a$,
\[
\Varhat(Z_t) \le C\ep e^A\Big(\frac{t}{a^3} Z_0 + Y_0\Big)\quad\tand\quad \Varhat(R_t) \le C\ep e^A\Big(\frac{t}{a^3} Z_0 + Y_0\Big).
\]
\end{lemma}
\begin{proof}
  We want to apply Lemma~\ref{lem:many_to_two_with_tiers} for every initial particle but first have to collect some estimates.
  By Lemma~\ref{lem:hat_quantities}, \eqref{eq:QaZC} and the hypothesis  $t\le \vartheta\le \tconst{eq:tcrit}$, we have for every $x\in(0,a)$ and $s\le t$, for large $A$ and $a$,
  \begin{equation}
    \label{eq:56}
    \Ehat^{(x,s)}[Z_t] \le C \Big(w_Z(x) + Aw_Y(x)\Big).
  \end{equation}
In particular, since $\mu > 7/8$ for large $a$, $\Ehat^{(x,s)}[Z_t] \le CAe^{-\frac 3 4 (a-x)}$ for large $A$ and $a$. Now, for $s_0\le s\le t$, let $\widehat{\p}^{(0)}_{s_0,s}(x,z)$ be defined as $\widehat{\p}^{(0)}_{s}(x,z)$, but starting from the space-time point $(x,s_0)$. As in the proof of \eqref{eq:Zt_2ndmoment}, we have by Corollary~\ref{cor:hat_many_to_one},
\begin{align}
  \nonumber
 \int_{s_0\wedge t}^t\int_0^a \widehat{\p}^{(0)}_{s_0,s}(x,z) \Ehat^{(z,s)}[Z_t]^2\,\dd z\,\dd s &\le CA^2\int_{s_0\wedge t}^t\int_0^a \p^{(0)}_{s_0,s}(x,z) e^{-\frac 3 2 (a-x)}\,\dd z\,\dd s \\
 \nonumber
 &\le CA^2 \Big(\frac{t}{a^3} w_Z(x) + w_Y(x) \Big),
\end{align}
which yields
\begin{align}
\nonumber
  \int_0^t\int_0^a \widehat{\p}_s(x,z) \Ehat^{(z,s)}[Z_t]^2\,\dd z\,\dd s &\le C \Ehat^x\Big[\sum_{(u,s_0)\in\mathscr S_t} \int_{s_0\wedge t}^t\int_0^a \widehat{\p}^{(0)}_{s_0,s}(X_u(s_0),z) \Ehat^{(z,s)}[Z_t]^2\,\dd z\,\dd s\Big]\\
  \nonumber
 &\le CA^2 \Big(\frac{t}{a^3} \Big(w_Z(x) + \Ehat^x[Z_{\emptyset,t}^{(1+)}]\Big) + w_Y(x) + \Ehat^x[Y_{\emptyset,t}^{(1+)}]\Big)\\
  \nonumber
 &\le CA^2 \Big(\frac{t}{a^3} w_Z(x) +  w_Y(x)\Big), 
\end{align}
by Corollary~\ref{cor:Zl_exp_upperbound}, \eqref{eq:imam}, \eqref{eq:eta} and the hypothesis $t\le \vartheta\le \tconst{eq:tcrit}$. Furthermore, 
\begin{align*}
\Ehat^x\Big[&\sum_{(u,s)\in\mathscr R_t} \sum_{(v_1,s_1),(v_2,s_2)\in\mathscr S^{(u,s)},\,v_1\ne v_2} \Ehat^{(X_{v_1}(s_1),s_1)}[Z_t]\Ehat^{(X_{v_2}(s_2),s_2)}[Z_t]\Big]\hspace{-30cm} \\
&\le C \Ehat^x\Big[\sum_{(u,s)\in\mathscr R_t} (Z^{(u,s)}+AY^{(u,s)})^2\Big] && \text{by \eqref{eq:56}}\\
&\le C \ep e^A \Ehat^x[R_t] && \text{by \eqref{eq:imam},  \eqref{eq:eta} and \eqref{eq:QhatZsquared}}\\
&\le C\ep e^A \Big(\frac{t}{a^3} w_Z(x) +  w_Y(x)\Big)&&\text{by Lemma~\ref{lem:Rt} and Corollary~\ref{cor:hat_many_to_one}}
\end{align*}
We can now apply Lemma~\ref{lem:many_to_two_with_tiers} together with the two previous displays,  Lemma~\ref{lem:hat_quantities} and the inequality $w_Z^2 \le Cw_Y$, to get
\begin{equation*}
  \Varhat^x(Z_t) \le \Ehat^x[Z_t^2] \le C (A^2+\ep e^A)\Big(\frac{t}{a^3} w_Z(x) +  w_Y(x)\Big).
\end{equation*}
Summing over the initial particles and using \eqref{eq:ep_lower} yields the inequality for $\Varhat(Z_t)$. The proof of the second inequality is similar, relying on Lemma~\ref{lem:Rt} and Remark~\ref{rem:many_to_two_with_tiers}.
\end{proof}

We finish the section with a corollary which will be useful in the next section.
\begin{corollary}
\label{cor:hat_quantities}
 Suppose that $t\le \vartheta$ and $x\in(0,a)$. Then for large $A$ and $a$, we have $\Ehat^x[Z_t] \le CAe^{-(a-x)/2}$, $\Ehat^x[Z_t^2] \le C\ep e^Ae^{-(a-x)/2}$ and $\Ehat^x[Y_t] \le Ce^{-(a-x)/2}$.
\end{corollary}
\begin{proof}
  Immediate from Lemmas~\ref{lem:hat_quantities} and \ref{lem:Zt_variance} and \eqref{eq:c0_mu}.
\end{proof}

\subsection{The fugitive and its family}
\label{sec:fugitive}

We now describe the BBM conditioned to break out at a given time. Recall that $\mathscr U$ denotes the fugitive. For simplicity, write $\tau_l$ and $\sigma_l$ for $\tau_l(\mathscr U)$ and $\sigma_l(\mathscr U)$, respectively, $l\in\N$. 
By the strong branching property, we have the following decomposition:

\begin{lemma}
  \label{lem:decomposition_Tl}
Let $k\in\N$, $l\in \N\cup\{\infty\}$ with $l\ge k$ and $t\ge 0$. Conditioned on $\F_0$, $T^{(0;l)}=T^{(k)}=t$, $\tau_0$ and $\mathscr U_0$, the BBM admits the following recursive decomposition:
\begin{enumerate}[nolistsep]
\item The initial particles $u\ne \mathscr U_0$ spawn independent BBM conditioned on $T^{(0;l)}>t$,
\item independently, the particle $\mathscr U_0$ spawns BBM conditioned on the event\footnote{This event has zero probability  but the conditioned law can be defined in obvious ways, for example by a limiting procedure. Note that this law is described in Lemma~\ref{lem:many_to_one_fugitive}.} that 
\begin{itemize}
 \item there exists a particle (call it $\mathscr U_0'$) hitting the point $a$ for the first time at time $\tau_0$ and
 \item no particle which is a descendant of $\mathscr U_0$ but not of $\mathscr U_0'$ breaks out before time $t$ from tiers $0,\ldots,l$.
\end{itemize}
Then, at time $\tau_0$:
\begin{enumerate}
\item If $k=0$, the particle $\mathscr U_0'$ spawns BBM conditioned on the event $B$ of a breakout.
\item If $k>0$, it spawns BBM starting at the space-time point $(a,\tau_0)$ conditioned on $T^{(0;l)}=T^{(k)}=t$. In particular, if we write $\mathscr S' = \mathscr S^{(\mathscr U_0',\tau_0)}$, then conditioned on $\F_{\mathscr S'}$\footnote{This $\sigma$-algebra is defined in Section~\ref{sec:stopping_lines}.}, the particles in $\mathscr S'$ spawn BBM starting from the collection of space-time points $\mathscr S'$, conditioned on $T^{(0;l-1)}=T^{(k-1)} = t$.
\end{enumerate}
\end{enumerate}
\end{lemma}
Note that in the case 2(b) above, the subtree spawned by $(\mathscr U_0',\tau_0)$ follows the law $\Q^a$ \emph{conditioned} on $T^{(0;l)}=T^{(k)} = t$, hence the law of $\mathscr S^{(\mathscr U_0',\tau_0)}$ is not the same as under $\Q^a$. In particular,  $\E[Z^{(\mathscr U_0',\tau_0)}] \not\le CA$. Indeed, conditioning on one of its descendants breaking out at a later time corresponds to a kind of size-bias on the number of particles. However, it is still true that $Z^{(\mathscr U_0',\tau_0)} < \ep e^A$, by the definition of a breakout.

In order to study the particles related to the fugitive, the following definition will be helpful. Let $L$ be the tier from which the breakout arises, i.e.\ $T^{(L)} = T$. For $l\le L$, let $\mathscr U_l$ and $\mathscr U'_l$ be the ancestors of the fugitive alive at $\sigma_l$ and $\tau_l$, respectively. We define a random proper line $\mathscr L_{\mathscr U}$ to contain the descendants and the cousins of the fugitive, at the time they are born or, if they are born between times $\tau_l(\mathscr U)$ and $\sigma_{l+1}(\mathscr U)$ for some $l\ge0$, at the time they hit the critical barrier. Formally,
\begin{align*}
 \mathscr L_{\mathscr U} &= \mathscr L^{\mathrm{fug}}_{\mathscr U}\wedge \widecheck{\mathscr L}_{\mathscr U}\wedge \widebar{\mathscr L}_{\mathscr U},\quad\text{where}\\
  \mathscr L^{\mathrm{fug}}_{\mathscr U} &= \mathscr S^{(\mathscr U,T)}\\
  \widecheck{\mathscr L}_{\mathscr U} &= \bigwedge_{l=0}^{L-1} \Big(\mathscr S^{(\mathscr U'_l,\tau_l)}\backslash\{(\mathscr U_{l+1},\sigma_l)\}\Big)\\
  \widebar{\mathscr L}_{\mathscr U} &= \bigcup_{u\prec \mathscr U,\,d_u\in \bigcup_{l=0}^L[\sigma_l,\tau_l)} \{(uk,d_u): k\in\{1,\ldots,k_u\},\,uk\npreceq \mathscr U\}.
\end{align*}
Here, we recall that $k_u$ and $d_u$ denote the number of children, resp., the time of death of the individual $u$. Note that the $\wedge$ symbols could have been replaced by unions in the above formulas. We further write
\[
 \F_{\mathscr U} = \F_{\mathscr L_{\mathscr U}},
\]
as defined in Section~\ref{sec:stopping_lines}. The following statement is a direct corollary of Lemma~\ref{lem:decomposition_Tl}.
\begin{corollary}
 \label{cor:L_U}
 With the notation introduced above, for every $j\le l$, conditioned on the event $T^{(0;l)}=T^{(k)} = t$ and on the $\sigma$-algebra $\mathscr F_{\mathscr U}$, the (space-time) tier-$j$ particles on the line $\mathscr L_{\mathscr U}$ spawn independent BBM conditioned not to break out from tiers $0,\ldots,l-j$ before time $t$.
\end{corollary}

The following $\F_{\mathscr U}$-measurable functional will be of use in the study of first and second moments of the bar-quantities conditioned on $\F_{\mathscr U}$ (Corollary~\ref{cor:hat_quantities} tells us why):
\begin{equation}
  \label{eq:E_U}
\mathscr E_{\mathscr U} = \sum_{(u,s)\in\widebar{\mathscr L}_{\mathscr U}} e^{-(a-X_u(s))/2} = \sum_{l=0}^L\  \sum_{u\prec \mathscr U,\,d_u\in[\sigma_l,\tau_l)} (k_u-1)e^{-(a-X_u(d_u-))/2}.
\end{equation}

\begin{lemma}
  \label{lem:E_U}
For large $A$ and $a$, we have for every $l\in\N$,
\[
\E[\mathscr E_{\mathscr U}\,|\,T=T^{(l)} \le \tconst{eq:tcrit}] \le C(l+1).
\]
\end{lemma}
\begin{proof}
By Lemma~\ref{lem:decomposition_Tl} and Corollary~\ref{cor:many_to_one_fugitive} (which can be applied because of \eqref{eq:pbar_estimate}), we have for every $l\in\N$,
\[
\E[\mathscr E_{\mathscr U}\,|\,T=T^{(l)} \le \tconst{eq:tcrit}]\le C  \sum_{i=0}^l \E\Big[W_i\Big[\int_0^{\tau_i-\sigma_i} e^{-(a-X_t)/2}\,\dd t\Big]\,\Big|\,T=T^{(l)}\le \tconst{eq:tcrit}\Big],
\]
where $W_i$ is the law of a Brownian bridge of length $\tau_i-\sigma_i$ from $X_{\mathscr U}(\sigma_i)$ to $X_{\mathscr U}(\tau_i)$. The statement now follows readily from Lemma~\ref{lem:k_integral}.
\end{proof}

We can now bound the probability that the fugitive stems from tier 1 or 2.
\begin{lemma}
\label{lem:T1_T2}
Suppose that $C_1 e^A \le Z_0 \le C_2e^A$ and $Y_0 \le \eta Z_0$. Then for large $A$ and $a$, we have for $l\in\{1,2\}$,
 \[
  \P(T^{(l+)} < T^{(0;l-1)}\le \tconst{eq:tcrit}) \le C(\ep A)^l,
 \]
\end{lemma}
\begin{proof}
Fix $t\le \tconst{eq:tcrit}$ and $l\ge 1$ and write for this proof $\Phat_{(j)}(\cdot) = \P(\cdot\,|\,T^{(0;j)} > t)$ for all $j\le l-1$. Let $\nu = \sum_{i=1}^n \delta_{x_i}$, where $x_1,\ldots,x_n$ are the positions of the initial particles (we assume w.l.o.g.\ that $\nu$ is non-random). Conditioned on $T^{(0;l-1)}=t$, let $p_i$ be the probability that the first breakout from tiers $0,\ldots,l-1$ is caused by a descendant of the $i$-th initial particle. Then,
\begin{equation}
  \label{eq:1}
 \P^\nu(T^{(l+)} < t\,|\,T^{(0;l-1)} = t) \le \sum_{i=1}^n p_i \left(\Phat_{(l-1)}^{\nu-\delta_{x_i}}(T^{(l+)} < t)
+\P^{x_i}(T^{(l+)} < t\,|\,T^{(0;l-1)} = t)\right).
\end{equation}

Define $c_t = 10A( t /{a^3} + C\eta)$, where $C$ is the constant from \eqref{eq:Zl_exp_upperbound_2}. Then $c_t$ is less than $2/3$ for large $A$ and $a$ by the hypothesis on $t$ and \eqref{eq:eta}. By Lemma~\ref{lem:Zl_exp_upperbound}, we have  for every $j\le l-1$ and every  $(x,s)\in (0,a)\times[0,t]$,
\begin{equation}
 \label{eq:deezer}
\Ehat_{(l-j-1)}^{(x,s)}[Z_{\emptyset,t}^{(l-j)}] \le CA(\tfrac t{a^3}w_Z(x)+w_Y(x))c_t^{l-j-1}.
\end{equation}
We now have for $l\ge 1$,
\begin{align*}
 \Phat_{(l-1)}^{\nu-\delta_{x_i}}(T^{(l+)} < t) &\le Cp_B\Ehat_{(l-1)}^{\nu-\delta_{x_i}}[\tfrac t {a^3} Z_{\emptyset,t}^{(l)} + Y_{\emptyset,t}^{(l)}] && \text{by Proposition~\ref{prop:T}}\\
 &\le Cp_BA^{-1}c_t\Ehat_{(l-1)}^{\nu-\delta_{x_i}}[Z_{\emptyset,t}^{(l)}] && \text{by \eqref{eq:imam}}\\
 &\le Cp_BZ_0A^{-1} c_t^{l+1} && \text{by \eqref{eq:deezer} and the hyp.\ on $Y_0$}.\\
 &\le C(A\ep)^{-1} c_t^{l+1} && \text{by \eqref{eq:pB} and the hyp.\ on $Z_0$.}
\end{align*}
With \eqref{eq:pB} and the hypothesis on $Z_0$, this gives
\begin{equation}
\label{eq:7777}
 \Phat_{(l-1)}^{\nu-\delta_{x_i}}(T^{(l+)} < t) \le C(A\ep)^{-1} c_t^{l+1}.
\end{equation}

In order to bound the second term on the RHS of \eqref{eq:1}, we note that as above, we have for every $x\in(0,a)$ and every $k=0,\ldots,l-1$, by Proposition~\ref{prop:T} and \eqref{eq:imam},
\begin{equation}
 \label{eq:olives}
 \P^x(T^{(l+)} < t\,|\,T^{(k)} = T^{(0;l-1)} = t) \le Cp_BA^{-1}c_t\E^x[\widecheck Z^{(l)}_{\emptyset,t} + \widebar Z^{(l)}_{\emptyset,t}\,|\,T^{(k)} = T^{(0;l-1)} = t].
\end{equation}
By \eqref{eq:deezer}, Lemma~\ref{lem:E_U} and the inequality $w_Z(x) + w_Y(x) \le Ce^{-(a-x)/2}$, we have for every $k=0,\ldots,l-1$,
\begin{equation}
  \label{eq:14}
  \E^x[\widebar Z^{(l)}_{\emptyset,t}\,|\,T^{(k)} = T^{(0;l-1)} = t] \le C(k+1)A c_t^{l-k-1}.
\end{equation}
Moreover, if $\mathscr U_j'$ denotes the ancestor of $\mathscr U$ at time $\tau_j(\mathscr U)$, then by \eqref{eq:deezer},
\begin{align}
\nonumber
\E^x[\widecheck Z^{(l)}_{\emptyset,t}\,|\,T^{(k)} = T^{(0;l-1)} = t] &\le \E^x\Big[\sum_{j=0}^{k-1} CA(\tfrac t{a^3}Z^{(\mathscr U_j',\tau_j(\mathscr U))}+Y^{(\mathscr U_j',\tau_j(\mathscr U))}) c_t^{l-j-2}\Big]\\
\label{eq:13}
&\le C \ep e^A c_t^{l-k},
\end{align}
where the last equation follows from \eqref{eq:imam}, the fact that $Z^{(\mathscr U_j',\tau_j(\mathscr U))}\le \ep e^A$ by the definition of a breakout, and the fact that $c_t \le 2/3$.
Equations \eqref{eq:olives}, \eqref{eq:14} and \eqref{eq:13} and the law of total expectation now give
\begin{equation}
  \label{eq:115}
 \P^x(T^{(l+)} < t\,|\,T^{(0;l-1)} = t) \le C p_B (A^{-1}\ep e^A c_t^2 + lc_t) \le C (A^{-1}c_t^2 + l(\ep e^A)^{-1} c_t),
\end{equation}
by \eqref{eq:pB}.

Summing up the above results, we have by \eqref{eq:1}, \eqref{eq:7777} and \eqref{eq:115},
\begin{equation*}
 \P^\nu(T^{(l+)} < t\,|\, T^{(0;l-1)} = t) \le C ((A\ep)^{-1}c_t^{l+1} + A^{-1}c_t^2 + l(\ep e^A)^{-1} c_t),
\end{equation*}
so that, with $\widehat T = T^{(0;l-1)}\wedge \tconst{eq:tcrit}$,
\begin{align}
\nonumber
\P(T^{(l+)} < T^{(0;l-1)}\le \tconst{eq:tcrit}) &= \int_0^{\tconst{eq:tcrit}} \P^\nu(T^{(l+)} < t\,|\, T^{(0;l-1)} = t) \P^\nu(T^{(0;l-1)} \in \dd t)\\
\label{eq:techno}
&\le C \E^\nu\Big[(A\ep)^{-1}c_{\widehat T}^{l+1} + A^{-1}c_{\widehat T}^2 + l(\ep e^A)^{-1} c_{\widehat T}\Big],
\end{align}
By the hypotheses on $Z_0$ and $Y_0$, Corollary~\ref{cor:moments_T} and \eqref{eq:eta_ep} now give for $j=1,2,3$,
\[
 \E^\nu[c_{\widehat T}^j] \le C(A\ep)^j.
\]
The lemma now follows easily from this equation together with \eqref{eq:techno} and \eqref{eq:ep_lower}.
\end{proof}

\begin{remark}
  One may wonder whether one can simply calculate $\P(T^{(l)}\in\dd t\,|\,T^{(0;l-1)}> t)$ for every $l\ge 1$ and $t\le\tconst{eq:tcrit}$, using only the tools from Section~\ref{sec:before_breakout_particles}. This would require fine estimates on the density of the point process formed by the particles from tier $l-1$ hitting $a$ just before $t$. These estimates can be most easily obtained if one stops descendants of the particles hitting $a$ at a (large) fixed time $\zeta$ instead of the line from the first part of Lemma~\ref{lem:ZYW}, with which the results in this paper would hold as well. However, in order not make our results dependent on a particular form of the critical line, we stick to Lemma~\ref{lem:T1_T2}, which is enough for our purposes.
\end{remark}

\subsection{Proof of Proposition \ref{prop:piece}}
\label{sec:piece_proof}

Fix $\lambda \in \R$. In this section, the symbols $o(1)$ and $O(\cdot)$ may depend on~$\lambda$ (we could make the dependence precise, but we won't need it). Also, in this section, we always assume that $A$ and $a$ are large and will sometimes omit mentioning it explicitly.

Recall that  $\widetilde{\ep} = (\pi p_Be^A)^{-1}$. In particular, by \eqref{eq:pB} and \eqref{eq:ep_upper}, we have for large $A$ and $a$,
\begin{equation}
\label{eq:ep_tilde_ep}
 \widetilde{\ep} = (\ep/\pi^2)(1+o(1)) = O(\ep) = o(1).
\end{equation}

In order to prove Proposition~\ref{prop:piece}, we will prove the following statements: There exists a numerical constant $\numcst > 0$, such that for large $A$ and $a$, we have
\begin{align}
\label{eq:P_G1}
&\PB(G_1^c\,|\,\F_0)\Ind_{G_0} = O(\ep^{1+\numcst})\\
\label{eq:fourier_G1}
&\EB\Big[e^{i\lambda X^{[\infty]}_{\Theta_1}}\Ind_{G_1}\,\Big|\,\F_0\Big] = \exp\Big(\widetilde{\ep}(\kappa(\lambda) + i\lambda \pi^2A + o(1)) \Big)\Ind_{G_0} +  O(\ep^{1+\numcst}),
\end{align}
with $\kappa(\lambda)$ as in the statement of Proposition~\ref{prop:piece}.
Let us show how this implies the proposition. Since the process starts afresh at the stopping time $\Theta_n$, it is easily proven from \eqref{eq:P_G1} by induction that
\begin{equation*}
\PB(G_n^c\,|\,\F_0)\Ind_{G_0} = O(n\ep^{1+\numcst}),\quad\forall n\ge1. 
\end{equation*}
This yields the first statement of Proposition~\ref{prop:piece}. In order to show \eqref{eq:piece_fourier}, it is then enough to show that 
\begin{equation}
\label{eq:fourier_Gn}
\EB\Big[e^{i\lambda X^{[\infty]}_{\Theta_n}}\Ind_{G_n}\,\Big|\,\F_0\Big] = \exp\Big(n\widetilde{\ep}(\kappa(\lambda) + i\lambda \pi^2A + o(1))\Big)\Ind_{G_0} + O(n\ep^{1+\numcst}).
\end{equation}
We prove this again by induction: for every $n$, we have by \eqref{eq:fourier_G1},
\begin{align*}
 \EB\Big[e^{i\lambda X^{[\infty]}_{\Theta_{n+1}}}\Ind_{G_{n+1}}\,\Big|\,\F_n\Big] &= \EB\Big[e^{i\lambda (X^{[\infty]}_{\Theta_{n+1}}-X^{[\infty]}_{\Theta_{n}})}\Ind_{G_{n+1}}\,\Big|\,\F_n\Big]e^{i\lambda X^{[\infty]}_{\Theta_{n}}}\\
 &= \exp\Big(\widetilde{\ep}(\kappa(\lambda) + i\lambda \pi^2A + o(1))\Big)e^{i\lambda X^{[\infty]}_{\Theta_{n}}}\Ind_{G_n} + O(\ep^{1+\numcst}).
\end{align*}
Assume now that $\lambda \ne 0$ (for $\lambda=0$, \eqref{eq:fourier_G1} follows from \eqref{eq:P_G1}). Then $\Re(\kappa(\lambda) + o(1))< 0$ for large $A$ and $a$. In particular, $|\exp(\widetilde{\ep}(\kappa(\lambda) + i\lambda \pi^2A + o(1)))| \le 1$ for large $A$ and $a$. 
From this, one easily concludes by induction that \eqref{eq:fourier_Gn} holds for every $n$.

Now let us turn to the proof of \eqref{eq:P_G1} and \eqref{eq:fourier_G1}. We first introduce some notation. Define the random proper line
\[
 \mathscr L_\Delta = \widehat{\mathscr N}_{T^-} \wedge \mathscr S^{(\mathscr U,T)} \wedge \widecheck{\mathscr N}^{(1)}_{T^-}\wedge(\widecheck{\mathscr S}^{(1)}_T \backslash \widecheck{\mathscr S}^{(1)}_{T^-}).
\]
By definition, we then have ($\widecheck Z_{\Delta}$ was defined in \eqref{eq:def_check_Z_Delta})
\[
 Z_\Delta := \sum_{(u,s)\in\mathscr L_\Delta} w_Z(X_u(s)) = \widehat Z_{T^-}+Z^{(\mathscr U,T)}+\widecheck Z_{\Delta},
\]
so that  $\Delta = \log(e^{-A}Z_\Delta)$ by \eqref{eq:Delta_def}. Define $Y_\Delta$ and $\widecheck Y_{\Delta}$ analogously to $Z_\Delta$ and $\widecheck Z_{\Delta}$. 

Recall the definition of $\F_{\mathscr U}$ from Section~\ref{sec:fugitive}. We set
\[
 \F_\Delta := \F_{\mathscr U} \vee \widehat \F_{T^-}\vee \widecheck \F_{T^-}.
\]
We then have $\F_{\mathscr L_\Delta} \subset \F_\Delta$, in particular, the random variables $Z_\Delta$, $Y_\Delta$ and $\Delta$ are measurable with respect to $\F_\Delta$. Furthermore, it follows from Corollary~\ref{cor:L_U} and the strong branching property (applied to BBM conditioned not to break out) that \emph{conditioned on $\F_\Delta$, the (space-time) particles on the line $\mathscr L_\Delta$ spawn independent BBM conditioned not to break out before time $T$}.

In what follows, we will work on several ``good sets'', all of which we define here for easy reference, since they will be reused in the following sections. For this reason, some of them (for example $\widehat G$) are actually more restrictive than what would be necessary for the proof of Proposition~\ref{prop:piece}.

\begin{itemize}[nolistsep]
\item $G_{\mathscr U} =\{2e^Aa^2\le T\le \sqrt \ep a^3,\,T=T^{(0;1)},\,\mathscr E_{\mathscr U} \le e^{A/3}\}$ ($\mathscr E_{\mathscr U}$ was defined in \eqref{eq:E_U}).
\item $G_{\mathrm{fug}} = \{\sigma_{\mathrm{max}}^{(\mathscr U,T)}\le \zeta,\ Z^{(\mathscr U,T)} \le e^A/\ep,\ Y^{(\mathscr U,T)}\le \eta Z^{(\mathscr U,T)}\}$.
\item $\widehat G = \{|e^{-A}\widehat Z_{T^-} - 1| \le \ep^{1/4},\ \widehat Y_{T^-} \le a^{-1}e^{3A/2}\}$.
\item $\widecheck G = \{\widecheck Z_{\Delta} \le \ep^{1/4}e^A,\ \widecheck Y_{\Delta} \le e^{-A/2}\}$.
\item $G_\Delta = G_{\mathscr U}\cap G_{\mathrm{fug}}\cap \widehat G\cap \widecheck G \in \F_\Delta$.
\item $G_{\mathrm{nbab}}$ (``nbab'' stands for ``no breakout after breakout'') is the event that no bar-particle breaks out between $T$ and $\Theta_1+e^Aa^2$ and that no descendant of the particles in $\mathscr L_\Delta$ hits $a$ between $T^-$ and $\Theta_1+e^Aa^2$. We also define $\PB_{\mathrm{nbab}} = \PB(\cdot\,|\,G_{\mathrm{nbab}})$.
\end{itemize}
Note that for large $A$ and $a$, by \eqref{eq:eta} and \eqref{eq:ep_lower},
\begin{equation}
 \label{eq:ZYDelta}
 \text{on $G_\Delta$,}\quad Z_\Delta \le 2e^A/\ep\quad\tand\quad Y_\Delta \le 2e^{-A/2}.
\end{equation}

Finally, we recall that from  time $T^+ = T+\sigma_{\mathrm{max}}^{(\mathscr U,T)}$ on, we have the drift $-\mu_t = -\mu-f_\Delta((t-T^+)/a^2)/a^2$. On $G_{\mathrm{fug}}$, we have $T^+ \le T+\zeta$ and $\Theta_1 = T+e^Aa^2$. Furthermore, for large $A$, $-1\le \Delta \le C\log(1/\ep)$ on $G_\Delta$. By the hypotheses on the family $(f_x)_{x\in\R}$ from Section~\ref{sec:BBBM_definition_definition} and \eqref{eq:ep_lower} it follows that for large $A$ and $a$, for some function $g$,
\begin{equation}
 \label{eq:Delta_f}
\ton G_\Delta: \text{$f_\Delta$ is a barrier function,}\quad \|f_\Delta\| \le g(A)\quad\tand\quad \Delta - f_\Delta((\Theta_1-T^+)/a^2)= O(e^{-A/2}),
\end{equation}
where $\|\cdot\|$ is defined in \eqref{eq:def_f_norm}.

\paragraph{The probability of $G_\Delta$.}

We say that hypothesis (HB$_0$) is verified if $\nu_0$ is deterministic and such that $G_0$ holds.
\begin{lemma} 
\label{lem:GDelta}
Suppose (HB$_0$). For large $A$ and $a$, we have $\PB(G_\Delta) \ge 1 - C\ep^{5/4}$.
\end{lemma}
\begin{proof}
Since under $\PB$, the drift does not change before $T^+$ and the event $G_\Delta$ only depends on the process up to this time, we have $\PB(G_\Delta) = \P(G_\Delta)$. It is therefore enough to show that $\P(G_\Delta) \ge 1-C\ep^{5/4}$.

First note that $\sqrt\ep a^3 \le \tconst{eq:tcrit}$ for large $A$ and $a$, by \eqref{eq:ep_upper}. With Proposition~\ref{prop:T} and Lemma~\ref{lem:T1_T2},
\begin{equation*}
  \P(2e^Aa^2 \le T \le \sqrt{\ep}a^3,\,T=T^{(0;1)}) \ge 1 - CA^2\ep^2.
\end{equation*}
With Lemma~\ref{lem:E_U} and Markov's inequality, this yields
\begin{equation}
  \label{eq:32}
  \P(G_{\mathscr U}) \ge 1 - CA^2\ep^2 - Ce^{-A/3} \ge 1 - CA^2\ep^2,
\end{equation}
where the last inequality follows from \eqref{eq:ep_lower}. Furthermore, by Lemma~\ref{lem:ZYW}, with the notation used there and in Section~\ref{sec:BBBM_breakout},
\begin{align*}
 \P(G_{\mathrm{fug}}^c) &= \P^a(Z > e^A/\ep,\,\sigma_{\mathrm{max}} \le \zeta\,|\,B) + \P^a(\sigma_{\mathrm{max}}> \zeta\,|\,B)\\
 &\le p_B^{-1}(P(\pi (W+\eta) > (1+O(1/a))e^A/\ep) + \eta)\\
 &\le C p_B^{-1}(\ep e^{-A} + \eta),
\end{align*}
where the last inequality follows from \eqref{eq:W_tail} and \eqref{eq:eta}. With \eqref{eq:pB}, \eqref{eq:ep_lower} and \eqref{eq:eta}, this now gives
\begin{equation}
 \label{eq:Gfug}
\P(G_{\mathrm{fug}}) \ge 1- C\ep^2.
\end{equation}

In order to estimate the probability of $\widehat G$, we will calculate first and second moments of the quantities in the definition of $\widehat G$. These estimates will be needed again later on, when we will calculate the Fourier transform of the random variable $\Delta$. Recall that $\E_\Q^a[Z]\le \pi A$ for large $A$ and~$a$ by \eqref{eq:QaZC} and that $T^- = (T-e^Aa^2)\vee 0$. Also, by hypothesis, we have $\widehat Z_0 = (1+O(\ep^{3/2}))e^A$ and $\widehat Y_0 \le \eta$. By Lemma~\ref{lem:hat_quantities} (applied with $\vartheta = T$), we then have for large $A$ and $a$,
\begin{align}
\nonumber
 \text{on $G_{\mathscr U}$,}\quad\E[e^{-A}\widehat Z_{T^-}\,|\,\F_{\mathscr U}] &= (1+O(\ep^{3/2}))\Big(1+\pi \E_\Q^a[Z] \tfrac {T^-} {a^3} + O\Big(p_B + \big(A\tfrac {T} {a^3}\big)^2\Big)\Big)+O(A\eta)\\
  \label{eq:710}
&= 1+\pi \E_\Q^a[Z] \tfrac T {a^3} + O\Big(\ep^{3/2} + \big(A\tfrac T {a^3}\big)^2\Big),
\end{align}
where the last inequality follows from the fact that $T\le \sqrt{\ep}a^3$ on $G_{\mathscr U}$, together with \eqref{eq:pB}, \eqref{eq:ep_lower} and \eqref{eq:eta}. As for the second moment, Lemma~\ref{lem:Zt_variance} and the hypothesis give
\[
 \Var(e^{-A}\widehat Z_{T^-}\,|\,\F_{\mathscr U})\Ind_{G_{\mathscr U}} \le C(\ep\tfrac T {a^3} + \ep e^{-A}\eta)\Ind_{G_{\mathscr U}}.
\]
Together with \eqref{eq:710}, \eqref{eq:eta_ep} and the inequality $(x+y)^2 \le 2(x^2+y^2)$, this yields
\begin{equation}
\label{eq:712}
\E[(e^{-A}\widehat Z_{T^-}-1)^2\,|\,\F_{\mathscr U}]\Ind_{G_{\mathscr U}} 
\le C\left(\ep \tfrac T {a^3}+(A\tfrac T {a^3})^2+\ep^3\right)\Ind_{G_{\mathscr U}},
\end{equation}

In order to integrate away the conditioning on $\F_{\mathscr U}$ in the above estimates, we need to estimate the first and second moment of $T\Ind_{G_{\mathscr U}}$. First note that $\P(G_{\mathscr U}^c) = O(\ep^{3/2})$ by \eqref{eq:32} and \eqref{eq:ep_upper}. This gives
\begin{align}
\nonumber
\E[T\Ind_{G_\mathscr U}] &= \E[T \wedge \tconst{eq:tcrit}] - \E[(T \wedge \tconst{eq:tcrit})\Ind_{G^c_\mathscr U}]\\
\nonumber
&= \int_0^{\tconst{eq:tcrit}} \P(T > t)\,\dd t + O(A^{-1}a^3\P(G_{\mathscr U}^c))\\
\nonumber
 &= (\widetilde{\ep} + O(A^{-1}\ep^{3/2}))a^3,
\end{align}
where the last equality is an easy consequence of Proposition~\ref{prop:T} and the hypothesis, together with the usual inequalities \eqref{eq:ep_upper}, \eqref{eq:ep_lower}, \eqref{eq:eta_ep} and \eqref{eq:pB}. Furthermore, Corollary~\ref{cor:moments_T} and the hypothesis give
\begin{equation*}
 \E[\big(\tfrac T {a^3}\big)^2\Ind_{G_\mathscr U}] \le  \E[\big(\tfrac T {a^3}\wedge 1\big)^2] \le C\ep^2.
\end{equation*}
With \eqref{eq:710}, \eqref{eq:QaZ} and \eqref{eq:ep_upper}, this now gives 
\begin{align}
 \nonumber 
  \E[e^{-A}\widehat Z_{T^-}\,|\,G_{\mathscr U}] &= (1+O(\ep^{3/2}))\Big(1+\pi \E_\Q^a[Z] \E[\tfrac T {a^3}\Ind_{G_\mathscr U}] + O\Big(\ep^{3/2} + A^2\E\big[\big(\tfrac T {a^3}\big)^2\Ind_{G_\mathscr U}\big]\Big)\Big)\\
   \label{eq:714}
  &= 1+\pi^2\widetilde{\ep} (A+\log \ep+\const{eq:QaZ}+o(1)) + O(\ep^{3/2}).
\end{align}
Similarly, with \eqref{eq:712} instead of \eqref{eq:710}, we get,
\begin{equation}
 \label{eq:716}
 \E[(e^{-A}\widehat Z_{T^-}-1)^2\,|\,G_{\mathscr U}] = O(A^2\ep^2).
\end{equation}
Furthermore, by the second part of Lemma~\ref{lem:hat_quantities} and Markov's inequality, together with \eqref{eq:eta} and the hypothesis, we have
\begin{align}
  \nonumber
  \P(\widehat Y_{T^-}>e^{3A/2}/a,\,G_{\mathscr U}) &\le \P(\widehat Y_{T^-}>e^{3A/2}/a,\,\widehat R_{[T^--a^2,T^-]}=0, \,G_{\mathscr U}) + \P(\widehat R_{[T^--a^2,T^-]}\ne 0, \,G_{\mathscr U})\\
  \label{eq:117}
  &\le C e^{-A/2}.
\end{align}

Turning now to the check-quantities, note that with the notation introduced in Section~\ref{sec:before_breakout_particles}, $\widecheck Z_{\Delta} = \widecheck Z^{(1)}_{T^-}$ if $\tau_0(\mathscr U) \le T^-$ and $\widecheck Z_{\Delta}=\widecheck Z^{(1)}_{\emptyset,T}$ otherwise, and similarly for $\widecheck Y_{\Delta}$. Recall that by definition of a breakout, $\widecheck Z^{(1)}_{\emptyset,T} \le \ep e^A\Ind_{(T=T^{(1+)})}$ and $\widecheck Y^{(1)}_{\emptyset,T} \le \eta \widecheck Z^{(1)}_{\emptyset,T}$.
By Lemmas~\ref{lem:hat_quantities} and~\ref{lem:Zt_variance} and \eqref{eq:eta} we then have (conditioning on $\mathscr S^{(1)}_T$ inside the expectations)
\begin{align}
\label{eq:121}
 \E[e^{-A}\widecheck Z_{\Delta}\,|\,\F_{\mathscr U}]\Ind_{G_{\mathscr U}} &\le C \E[e^{-A}(\widecheck Z^{(1)}_{\emptyset,T} + A \widecheck Y^{(1)}_{\emptyset,T})\,|\,\F_{\mathscr U}]\Ind_{G_{\mathscr U}} \le C \ep\Ind_{(T=T^{(1)})}\\
 \nonumber
 \E[(e^{-A}\widecheck Z_{\Delta})^2\,|\,\F_{\mathscr U}]\Ind_{G_{\mathscr U}} &\le C \E[e^{-2A}\left(\ep e^A(\widecheck Z^{(1)}_{\emptyset,T} + \widecheck Y^{(1)}_{\emptyset,T}) + (\widecheck Z^{(1)}_{\emptyset,T} + A \widecheck Y^{(1)}_{\emptyset,T})^2\right)\,|\,\F_{\mathscr U}]\Ind_{G_{\mathscr U}}\\
  \label{eq:119}
  &\le C \ep^2\Ind_{(T=T^{(1)})}\\
\label{eq:122}
\E[\widecheck Y_{\Delta}\,|\,\F_{\mathscr U}]\Ind_{G_{\mathscr U}} &\le C\E[\widecheck Y^{(1)}_{\emptyset,T} + \eta \widecheck Z^{(1)}_{\emptyset,T}\,|\,\F_{\mathscr U}]\Ind_{G_{\mathscr U}} \le Ce^{A}\eta \le Ce^{-A}.
\end{align}
Equations \eqref{eq:117} and \eqref{eq:122} and Markov's inequality applied to \eqref{eq:716} and \eqref{eq:119} together with \eqref{eq:ep_upper} now give for large $A$ and $a$,
\begin{equation}
  \label{eq:118}
  \P((\widehat G^c\cup\widecheck G^c)\cap G_{\mathscr U}) \le \ep^{5/4}.
\end{equation}
The lemma now follows from \eqref{eq:32}, \eqref{eq:Gfug} and \eqref{eq:118}.
\end{proof}

\paragraph{The probability of $G_{\mathrm{nbab}}$.}

\begin{lemma}
  \label{lem:Gnbab}
Suppose (HB$_0$). Then $\PB((G_{\mathrm{nbab}})^c\,|\,\F_\Delta)\Ind_{G_\Delta} \le C\ep^2$ for large $A$ and $a$.
\end{lemma}
\begin{proof}
Define $R_1 = \widehat R_{[T^-,\Theta_1+e^Aa^2]} + R^{\mathrm{fug}}_{[T,\Theta_1+e^Aa^2]} +\widecheck R^{(1)}_{[T^-,\Theta_1+e^Aa^2]}$. The second condition in the definition of $G_{\mathrm{nbab}}$ then is equivalent to $R_1 = 0$. As mentioned at the beginning of the section, conditioned on $\F_\Delta$, the (space-time) particles on the line $\mathscr L_\Delta$ spawn independent BBM conditioned not to break out before time $T$. Corollary~\ref{cor:conditioned_upperbound} and \eqref{eq:hbar} together with  Lemmas~\ref{lem:Rt} and \ref{lem:R_f} and Markov's inequality then imply
\begin{align*}
  \PB(R_1 > 0\,|\,\F_\Delta)\Ind_{G_\Delta}\le C(Z_\Delta \cdot o_a(1) + Y_\Delta) \le C\ep^2,
\end{align*}
where the last inequality follows \eqref{eq:ZYDelta} and \eqref{eq:ep_lower}.

It remains to prove the statement about the bar-particles. First, by  Corollary~\ref{cor:hat_quantities} and the definition of $G_\Delta$,
\begin{align}
\label{eq:die_antwoord}
 \EB[\widebar Z_T\,|\,\F_\Delta]\Ind_{G_\Delta} \le CA\mathscr E_{\mathscr U}\Ind_{G_\Delta} \le CAe^{A/3}\quad\tand\quad\EB[\widebar Y_T\,|\,\F_\Delta]\Ind_{G_\Delta} \le C \mathscr E_{\mathscr U}\Ind_{G_\Delta} \le Ce^{A/3}.
\end{align}
We then apply Lemma~\ref{lem:rit3} to the bar-particles at time $T$, the hypotheses being verified on $G_\Delta$ by \eqref{eq:Delta_f} and since we apply it with $\vartheta = \Theta_1+e^Aa^2 - T = 2e^Aa^2$ on $G_\Delta$. With the tower property of conditional expectation, this gives for large $A$ and $a$,
\begin{multline*}
 \PB(\text{a bar-particle breaks out between $T$ and $\Theta_1+e^Aa^2$}\,|\,\F_\Delta)\Ind_{G_\Delta} \\
 \le Cp_B\EB[\widebar Z_T\cdot o_a(1)+\widebar Y_T] \le Cp_Be^{A/3} \le C\ep^2,
\end{multline*}
where the last inequality follows from \eqref{eq:pB} and \eqref{eq:ep_lower}. This finishes the proof.
\end{proof}

\paragraph{The particles at time $\Theta_1$. The probability of $G_1$.}

\begin{lemma}
\label{lem:G1}
Suppose (HB$_0$). Then $\PB(G_1) \ge 1-C\ep^{5/4}$ for large $A$ and $a$.
\end{lemma}

\begin{proof}
We first bound $\widebar Z_{\Theta_1}$ and  $\widebar Y_{\Theta_1}$. As in the proof of Lemma~\ref{lem:Gnbab}, we note that the hypotheses of Lemma~\ref{lem:rit3} are verified on $G_\Delta$ by \eqref{eq:Delta_f}, so that we can apply the results of Lemma~\ref{lem:hat_Z_f}. Equation~\eqref{eq:28} together with Corollary~\ref{cor:conditioned_upperbound} and Proposition~\ref{prop:quantities} then gives for large $A$ and $a$,
\begin{align}
 \EB_{\mathrm{nbab}}[\widebar Z_{\Theta_1}\,|\,\F_\Delta]\Ind_{G_\Delta} &\le C \EB[\widebar Z_T + A \widebar Y_T]\,|\,\F_\Delta]\Ind_{G_\Delta} \le C A e^{A/3}\\
 \nonumber
 \EB_{\mathrm{nbab}}[\widebar Y_{\Theta_1}\Ind_{(\widebar R_{[\Theta_1-2a^2,\Theta_1]} = 0)}\,|\,\F_\Delta]\Ind_{G_\Delta} &\le C\EB[\widebar Z_{\Theta_1-2a^2}/a] \le  C \EB[\widebar Z_T/a + A \widebar Y_T/a\,|\,\F_\Delta]\Ind_{G_\Delta}\\
 &\le CAe^{A/3}/a.
\end{align}
where the last inequalities in each numbered equation follow from \eqref{eq:die_antwoord}. We can similarly bound the first moment of $\widebar R_{[\Theta_1-2a^2,\Theta_1]}$. If we start afresh the definition of the tiers at time $T$, we have by \eqref{eq:ichhassedas} and \eqref{eq:28}, together with Lemmas~\ref{lem:R_f} and \ref{lem:Rt},
\begin{multline*}
 \EB_{\mathrm{nbab}}[\widebar R_{[\Theta_1-2a^2,\Theta_1]}\,|\,\F_\Delta]\Ind_{G_\Delta} = \EB_{\mathrm{nbab}}[\widebar R^{(0)}_{[\Theta_1-2a^2,\Theta_1]} + \widebar R^{(1+)}_{[\Theta_1-2a^2,\Theta_1]}\,|\,\F_\Delta]\Ind_{G_\Delta}\\
 \le C \EB[\widebar Z_T/a + \eta A(\widebar Y_T + \widebar Z_T e^A/a)\,|\,\F_\Delta]\Ind_{G_\Delta} \le C \eta A e^{A/3}.
\end{multline*}
With Markov's inequality we then get, for large $A$ and $a$,
\begin{align}
 \nonumber
 \PB_{\mathrm{nbab}}&(\widebar R_{[\Theta_1-2a^2,\Theta_1]} > 0 \tor e^{-A}\widebar Z_{\Theta_1} > \ep^{3/2}/2\tor \widebar Y_{\Theta_1}> \eta/2\,|\,\F_\Delta)\Ind_{G_\Delta} \\
 \label{eq:fuck}
 &\le  C (\eta A e^{A/3} + Ae^{-2A/3}\ep^{-3/2} + Ae^{A/3}/(\eta a)) 
 \le C\ep^2,
\end{align}
where the last inequality follows from \eqref{eq:eta} and \eqref{eq:ep_lower}. Note that $\widebar R_{[\Theta_1-2a^2,\Theta_1]} = 0$ implies $\supp \nu_{\Theta_1} \subset (0,a)$ on $G_{\mathrm{nbab}}$ for large $A$ and $a$.

As for the remaining particles, denote by $Z^\Diamond_{\Theta_1}$ the contribution to $Z_{\Theta_1}$ of the descendants of the particles in $\mathscr L_\Delta$ which have not hit $a$ after $\mathscr L_\Delta$. On  $G_{\mathrm{nbab}}\cap G_\Delta$, we then have $Z_{\Theta_1} = Z^\Diamond_{\Theta_1} + \widebar Z_{\Theta_1}$. Define analogously $Y^\Diamond_{\Theta_1}$.  Proposition~\ref{prop:quantities} together with \eqref{eq:scheisse}, \eqref{eq:ichhassedas} and \eqref{eq:Delta_f} then gives for large $A$ and $a$,
 \begin{align*}
   \EB[e^{-A}Z^\Diamond_{\Theta_1}\,|\,\F_\Delta]\Ind_{G_\Delta} &= e^{-\Delta - A}Z_\Delta(1+O(e^{-A/2}+p_B))\Ind_{G_\Delta} = 1+O(\ep^3)\\
\VarB(e^{-A} Z^\Diamond_{\Theta_1}\,|\,\F_\Delta)\Ind_{G_\Delta} &\le C e^{- 2A} \left(Z_\Delta  \cdot o_a(1) + Y_\Delta\right)\Ind_{G_\Delta} \le C e^{-2A}\\
  \EB[Y^\Diamond_{\Theta_1}\,|\,\F_\Delta]\Ind_{G_\Delta} &\le C e^{-\Delta} Z_\Delta\Ind_{G_\Delta} \cdot o_a(1) \le o_a(1).
 \end{align*}
where the last inequality in the first line follows from \eqref{eq:pB} and \eqref{eq:ep_lower} and in the second and third lines from \eqref{eq:ZYDelta}. By Chebychev's inequality and \eqref{eq:ep_lower}, the previous inequalities then give,
\begin{equation}
  \label{eq:123}
  \PB(\{|e^{-A}Z^\Diamond_{\Theta_1} - 1| \ge \ep^{3/2}/2\} \cap G_{\mathrm{nbab}}\,|\,\F_\Delta)\Ind_{G_\Delta} \le C\ep^2.
\end{equation}
Moreover, by Markov's inequality, for large $A$ and $a$,
\begin{equation}
 \label{eq:123a}
 \PB(\{Y^\Diamond_{\Theta_1} > \eta/2\} \cap G_{\mathrm{nbab}}\,|\,\F_\Delta)\Ind_{G_\Delta} \le o_a(1) \le \ep^2.
\end{equation}
The lemma now follows from \eqref{eq:fuck}, \eqref{eq:123} and \eqref{eq:123a} together with Lemmas~\ref{lem:GDelta} and \ref{lem:Gnbab}.
\end{proof}

\begin{remark}
  \label{rem:piece}
Lemma~\ref{lem:G1} obviously still holds if one replaces $G_1$ by $G_1\cap\{|e^{-A}Z_{\Theta_1}-1| < \tfrac 1 2 \ep^{3/2}\}$. This will be needed in Sections~\ref{sec:Bflat} and \ref{sec:Bsharp}.
\end{remark}

\paragraph{The Fourier transform of the barrier process.}
In this paragraph we will prove \eqref{eq:fourier_G1}. We will assume throughout that hypothesis (HB$_0$) is verified. Fix $\lambda\in\R$. Throughout the proof, the symbols $O(\cdot)$, $o(1)$ and $C$ may depend on $\lambda$ (but are uniform in $\rho$ appearing below). Define $\Delta_{\mathrm{drift}} = e^{-A} (\widehat Z_{T^-} + \widecheck Z_{\Delta}) - 1$ and $\Delta_{\mathrm{jump}} = e^{-A} Z^{(\mathscr U,T)}$, so that
\[
 \Delta = \log(1+\Delta_{\mathrm{drift}} + \Delta_{\mathrm{jump}}) =  \log(1+\Delta_{\mathrm{drift}}) + \log(1+\frac{\Delta_{\mathrm{jump}}}{1+\Delta_{\mathrm{drift}}}).
\]
Recall that $X^{[\infty]}_{\Theta_1} = f_\Delta((\Theta_1-T^+)/a^2)$ on $G_\Delta$.
Equation~\eqref{eq:Delta_f} together with Lemmas~\ref{lem:GDelta} and~\ref{lem:G1} and \eqref{eq:ep_lower} then yield,
\begin{equation}
\label{eq:782a}
\begin{split}
 \EB[e^{i\lambda X^{[\infty]}_{\Theta_1}}\Ind_{G_1}] &= \E[e^{i\lambda \Delta}\Ind_{G_{\mathscr U}}] + O(\ep^{5/4})\\
&= \E[e^{i\lambda \log(1+\Delta_{\mathrm{drift}})}e^{i\lambda\log(1+\frac{\Delta_{\mathrm{jump}}}{1+\Delta_{\mathrm{drift}}})}\Ind_{G_{\mathscr U}}] + O(\ep^{5/4}).
\end{split}
\end{equation}
We now claim that for every $|\rho|<1/2$, we have for large $A$ and $a$,
\begin{gather}
 \label{eq:789}
\E[e^{i\lambda\log(1+\frac{\Delta_{\mathrm{jump}}}{1+ \rho})}] = \exp(\Psi + O(\ep |\log\ep|\rho)),\\
\nonumber
\text{where}\quad \Psi := \widetilde{\ep}  \Big(\pi^2\int_0^\infty [e^{i\lambda x}-1-i\lambda x\Ind_{(x\le 1)}]\,\Lambda(\dd x) + i\lambda \pi^2 (-\log\ep  + \const{eq:c_log} + o(1)) \Big),
\end{gather}
with $\Lambda(\dd x)$ as in the statement of Theorem~\ref{th:1} and $\const{eq:c_log}$ as defined in \eqref{eq:c_log}. Let us show how \eqref{eq:789} implies \eqref{eq:fourier_G1}. We use throughout without mention that $\P(G_{\mathscr U}) = 1-O(A^2\ep^2)$ by \eqref{eq:32}. First, by \eqref{eq:121} and Lemma~\ref{lem:T1_T2}, we have
\begin{equation*}
  \E[e^{-A}\widecheck Z_{\Delta}\,|\,G_{\mathscr U}]\le C\ep \P(T=T^{(1)}\,|\,G_{\mathscr U}) \le CA\ep^2,
\end{equation*}
so that with \eqref{eq:714} and \eqref{eq:ep_upper}, we get,
\begin{equation}
 \label{eq:Delta_drift}
\E\left[\Delta_{\mathrm{drift}}\,|\,G_{\mathscr U}\right] = \pi^2\widetilde{\ep}(A+\log\ep+\const{eq:QaZ}+o(1)) + O(\ep^{3/2}).
\end{equation}
Furthermore, \eqref{eq:716} and \eqref{eq:119} and the inequality $(x+y)^2 \le 2(x^2+y^2)$ yield
\begin{equation}
 \label{eq:Delta_drift_2}
\E\left[(\Delta_{\mathrm{drift}})^2\,|\,G_{\mathscr U}\right] = O(A^2\ep^2) = O(\ep^{3/2}),
\end{equation}
where the last inequality follows from \eqref{eq:ep_upper}. In particular, Markov's inequality applied to \eqref{eq:Delta_drift_2} yields,
\begin{equation}
 \label{eq:Delta_drift_bound}
\P(|\Delta_{\mathrm{drift}}| \ge 1/2\,|\, G_{\mathscr U}) = O(\ep^{3/2}).
\end{equation}
Equations \eqref{eq:782a}, \eqref{eq:789} and \eqref{eq:Delta_drift_bound} then give (note that $\Delta_{\mathrm{jump}}$ is independent of $\Delta_{\mathrm{drift}}$ and $G_{\mathscr U}$ and that $\Psi$ is deterministic)
\begin{align*}
 \E[e^{i\lambda X^{[\infty]}_{\Theta_1}}\Ind_{G_1}] &= \E[e^{i\lambda \log(1+\Delta_{\mathrm{drift}})}\Ind_{G_{\mathscr U}}\Ind_{(|\Delta_{\mathrm{drift}}| < 1/2)}e^{\Psi + O(\ep |\log\ep|\Delta_{\mathrm{drift}})}] + O(\ep^{5/4})\\
 &= e^{\Psi}\E[(1+(i\lambda + O(\ep |\log\ep|)) \Delta_{\mathrm{drift}} + O(\Delta_{\mathrm{drift}}^2))\Ind_{G_{\mathscr U}}\Ind_{(|\Delta_{\mathrm{drift}}| < 1/2)}] + O(\ep^{5/4})\\
 &= e^{\Psi}\E[(1+i\lambda) \Delta_{\mathrm{drift}}\Ind_{G_{\mathscr U}}] + O(\ep^{5/4}),
\end{align*}
where the last equality follows from \eqref{eq:Delta_drift_2} and \eqref{eq:Delta_drift_bound} together with standard applications of the Cauchy--Schwarz inequality and the bound $\Psi = O(1)$ valid for large $A$ and $a$. With \eqref{eq:Delta_drift} and \eqref{eq:ep_tilde_ep}, this gives
\begin{align*}
  \E[e^{i\lambda X^{[\infty]}_{\Theta_1}}\Ind_{G_1}] &= e^\Psi (1+i\lambda\pi^2\widetilde{\ep}(A+\log\ep+\const{eq:QaZ}+o(1))) + O(\ep^{5/4})\\
  &= \exp\Big\{\widetilde{\ep}\Big(\pi^2\int_0^\infty [e^{i\lambda x}-1-i\lambda x\Ind_{(x\le 1)}]\,\Lambda(\dd x) + i\lambda\pi^2(A+\const{eq:QaZ}+\const{eq:c_log}+o(1))   \Big)\Big\} + O(\ep^{5/4}).
\end{align*}
But this implies \eqref{eq:fourier_G1} with $c = \pi^2(\const{eq:QaZ}+\const{eq:c_log}) = \pi^2(\const{eq:W_expec} + \const{eq:c_log} - \log \pi)$.

It remains to show \eqref{eq:789} for any $|\rho|<1/2$. By the definition of the breakout event $B$ in \eqref{eq:def_B}, we have
\begin{equation*}
\E[e^{i\lambda \log(1+\frac{\Delta_{\mathrm{jump}}}{1+ \rho})}] = \E^a[e^{i\lambda \log(1+\frac{e^{-A}Z}{1+ \rho})}\,|\,B] = p_B^{-1}\E^a[e^{i\lambda \log(1+\frac{e^{-A}Z}{1+ \rho})}\Ind_{(Z > \ep e^A,\,\sigma_{\text{max}} > \zeta)}].
\end{equation*}
Recall that $\P^a(\sigma_{\text{max}} > \zeta) = O(\eta) = O(e^{-2A})$ by the third part of Lemma~\ref{lem:ZYW} and \eqref{eq:eta}. Together with \eqref{eq:pB}, this yields,
\begin{equation}
 \label{eq:PaZpB}
 \begin{split}
\P^a(Z>\ep e^A) &= \P^a(Z > \ep e^A,\,\sigma_{\text{max}} > \zeta) - O\left(\P^a(\sigma_{\text{max}} > \zeta)\right)\\
&= p_B + O(\eta) = p_B(1+O(e^{-A})).
\end{split}
\end{equation}
Furthermore, setting
\(
 g(x) = \exp(i\lambda \log(1+ x)),
\)
the above two equations yield
\begin{equation}
\label{eq:785}
\begin{split}
\E[e^{i\lambda \log(1+\frac{\Delta_{\mathrm{jump}}}{1+ \rho})}] &= p_B^{-1}\E^a[e^{i\lambda \log(1+\frac{e^{-A}Z}{1+ \rho})}\Ind_{(Z>\ep e^A)}] + O(p_B^{-1}\P^a(\sigma_{\text{max}} > \zeta)),\\
&=\E^a[g(\tfrac{e^{-A}Z}{1+ \rho})\,|\,Z>\ep e^A] + O(e^{-A}).
\end{split}
\end{equation}
Define $h(x) =  g(x) - 1 - i\lambda x \Ind_{(x\le 1)}$ and $h_\rho(x) = h(\tfrac x {1+\rho})$ for $x\ge 0$, so that
\begin{multline}
\label{eq:786}
\E^a[g(\tfrac{e^{-A}Z}{1+ \rho})\,|\,Z>\ep e^A] =  1 + \frac{i\lambda}{1+ \rho}\E^a[e^{-A}Z\Ind_{(Z\le e^A)}\,|\,Z>\ep e^A] + \E^a[h_\rho(e^{-A}Z)\,|\,Z>\ep e^A].
\end{multline}
Recall that $\widetilde{\ep} = (\pi p_B e^A)^{-1}$. With \eqref{eq:PaZpB}, this gives for large $A$ and $a$,
\begin{align*}
\E^a[e^{-A}Z\Ind_{(Z\le e^A)}\,|\,Z>\ep e^A] &= (1+O(e^{-A}))p_B^{-1}\E^a[e^{-A}Z\Ind_{(\ep e^A<Z\le  e^A)}] \\
 &= \pi \widetilde{\ep}\E^a[Z\Ind_{(\ep e^A<Z\le  e^A)}](1+O(e^{-A})).
\end{align*}
By the first two parts of Lemma~\ref{lem:ZYW} and \eqref{eq:eta}, we have for large $A$ and $a$,
\begin{align*}
 \E^a[Z\Ind_{(\ep e^A<Z\le  e^A)}] &= E[\pi W\Ind_{(\pi^{-1}\ep e^A+O(\eta) < W \le \pi^{-1}e^A+O(\eta))}] + O(e^A\eta)\\
 &= \pi\left(\log\left(\pi^{-1}e^A+O(\eta)\right) - \log \left(\pi^{-1}\ep e^A+O(\eta)\right)\right) + o(1)\\
 &= -\pi\log \ep + o(1).
\end{align*}
The two previous equations yield together,
\begin{equation}
\label{eq:786_0}
\E^a[e^{-A}Z\Ind_{(Z\le e^A)}\,|\,Z>\ep e^A] = \pi^2\widetilde{\ep} (-\log \ep + o(1)).
\end{equation}
As for the last expectation in \eqref{eq:786} first note that $|h_\rho(x)|=O(1\wedge x^2)$ for every $x\ge0$ and $|\rho|\le 1/2$ (recall that we allow the symbol $O(\cdot)$ to depend on $\lambda$ in this proof), so that 
\begin{align*}
 \E^a[h_\rho(e^{-A}Z)\Ind_{(Z>\ep e^A)}] &= \E^a[h_\rho(e^{-A}Z)] + O(\E^a[(e^{-A}Z)^2\Ind_{(Z\le\ep e^A)}]) \\
 &= \E^a[h_\rho(e^{-A}Z)] + O(\ep e^{-A}),
\end{align*}
where the last equality follows for the same reasons as \eqref{eq:QaZsquared}.
Furthermore, since $h_\rho$ has a bounded derivative, uniformly in $|\rho|\le 1/2$, except at $x=1+\rho$, the first two parts of Lemma~\ref{lem:ZYW} yield for large $A$ and $a$,
\begin{align*}
 \E^a[h_\rho(e^{-A}Z)] &= E[h_\rho(e^{-A}\pi W)] + O(\eta + P(|\pi W-(1+\rho)e^A| \le C\eta))\\
 &= E[h_\rho(e^{-A}\pi W)] + O(\eta) + e^{-A}o(1).
\end{align*}
The previous equations and \eqref{eq:eta} now yield
\begin{equation}
 \label{eq:786_1}
  \E^a[h_\rho(e^{-A}Z)\Ind_{(Z>\ep e^A)}] = E[h_\rho(e^{-A}\pi W)] + e^{-A}o(1).
\end{equation}

We wish to express $E[h_\rho(e^{-A}\pi W)]$ for large $A$.  Denote by $h'(x)$ the left derivative of $h$, which satisfies $|h'(x)| \le C(x\wedge x^{-1})$ for $x\ge 0$. Integration by parts gives for every $\alpha>0$,
\begin{align*}
 E[h(\alpha W)] = \int_0^\infty h'(x) P(\alpha W>x)\,\dd x + (h(1+)-h(1))P(\alpha W > 1).
\end{align*}
By the second part of Lemma~\ref{lem:ZYW}, we have $P(W>x)\sim x^{-1}$ as $x\to\infty$, in particular, 
\[
h'(x)\alpha^{-1} P(\alpha W>x) \le h'(x)\alpha^{-1}C(\alpha^{-1}x)^{-1}\le C(1\wedge x^{-2}),\quad\text{ for all $\alpha>0$ and $x\ge0$.} 
\]
Setting $\alpha = \pi e^{-A}/(1+\rho)$ and using dominated convergence then yields the following limit, uniformly in $|\rho|\le 1/2$,
\begin{multline*}
 \lim_{A\to\infty}(1+\rho)\pi^{-1} e^A E[h_\rho(e^{-A}\pi W)]
 = \lim_{\alpha\to 0} \alpha^{-1} E[h(\alpha W)]\\
 = \int_0^\infty h'(x) \frac 1 x\,\dd x + (h(1+)-h(1)) = \int_0^\infty h(x)\frac 1 {x^2}\,\dd x,
\end{multline*}
where the last equality follows again by integration by parts.
Together with \eqref{eq:786_1}, \eqref{eq:PaZpB} and the definition of $\widetilde{\ep}$, this gives
\begin{equation}
 \label{eq:786_2}
 \E^a[h_\rho(e^{-A}Z)\,|\,Z>\ep e^A]
 = \pi^2\widetilde{\ep}\Big(\int_0^\infty h(x)\frac 1 {x^2}\,\dd x +O(\rho)+o(1)\Big).
\end{equation}
Collecting \eqref{eq:785}, \eqref{eq:786}, \eqref{eq:786_0} and \eqref{eq:786_2} and using the fact that $e^{-A} = \widetilde{\ep}\cdot o(1)$ by \eqref{eq:ep_tilde_ep} and \eqref{eq:ep_lower}, we now have
\begin{align}
 \nonumber
\E[e^{i\lambda \log(1+\frac{\Delta_{\mathrm{jump}}}{1+ \rho})}]
 &=  1 + \widetilde{\ep}\left(-\frac{i\lambda}{1+ \rho} \pi^2 \log \ep + \pi^2 \int_0^\infty h(x)\frac 1 {x^2}\,\dd x + O(\rho)+o(1)\right) \\
\nonumber
&=\exp\left(\widetilde{\ep}\left(-i\lambda \pi^2 \log \ep + \pi^2 \int_0^\infty h(x)\frac 1 {x^2}\,\dd x + o(1)\right) + O(\ep|\log\ep|\rho)\right)
\end{align}
Set
\begin{equation}
 \label{eq:c_log}
  \const{eq:c_log} = \int_0^\infty [\log(1+x) \Ind_{(\log(1+x)\le 1)} - x \Ind_{(x\le 1)}]x^{-2}\,\dd x,
\end{equation}
which is well defined since $\log(1+x) = x + O(x^2)$ for $|x|\le 1/2$. We have,
\begin{align*}
 \int_0^\infty h(x) \frac 1 {x^2}\,\dd x &=  \int_0^\infty [e^{i\lambda \log(1+x)} - 1 - i\lambda \log(1+x) \Ind_{(\log(1+x)\le 1)}]\frac 1 {x^2}\,\dd x + i\lambda \const{eq:c_log} \\ 
 &= \int_0^\infty [e^{i\lambda x} - 1 - i\lambda x \Ind_{(x\le 1)}]\,\Lambda(\dd x) + i\lambda \const{eq:c_log},
\end{align*}
where $\Lambda$ denotes the push-forward/image of the measure $x^{-2}\,\dd x$ by the map $x\mapsto \log(1+x)$, as in the statement of Theorem~\ref{th:1}. The preceding equations now finally yield  \eqref{eq:789}, which was the last missing piece in the proof of \eqref{eq:fourier_G1}. Together with \eqref{eq:P_G1} (which holds by Lemma~\ref{lem:G1}), this finally yields Proposition~\ref{prop:piece}.

\subsection{Proof of Theorems \ref{th:barrier} and \ref{th:barrier2}}
\label{sec:BBBM_proofs}

We start with two lemmas: the first, Lemma~\ref{lem:Hperp}, shows that hypothesis (H$_\perp$) implies the event $G_0$ with high probability, the second, Lemma~\ref{lem:Xt_second}, shows Skorokhod convergence to the L\'evy process $(L_t)_{t\ge0}$ from Theorem~\ref{th:barrier} for an auxiliary process $(X_t'')_{t\ge0}$. This process is essentially a time-change of the process $(X_t')_{t\ge0}$; Lemma~\ref{lem:Xt_second} thus technically bridges the gap between Proposition~\ref{prop:piece} and Theorem~\ref{th:barrier2}. The step from Lemma~\ref{lem:Xt_second} to Theorem~\ref{th:barrier2} is done through a coupling of the random times $(\Theta_n)_{n\ge0}$ with a Poisson process of intensity $\widetilde\ep^{-1}$. Theorem~\ref{th:barrier} is then easily derived from Theorem~\ref{th:barrier2}.

\begin{lemma}
 \label{lem:Hperp} (H$_\perp$) implies $\PB(G_0) \ge 1-o_a(1)$.
\end{lemma}
\begin{proof}
Recall that under  (H$_\perp$), there are $\lfloor 2\pi  e^Aa^{-3}e^{\mu a} \rfloor$ particles distributed independently according to the probability density proportional to $\sin(\pi x/a)e^{-\mu x}\Ind_{(0,a)}(x)$.  An elementary calculation yields that
\begin{equation}
  \label{eq:49}
  \EB[Z_0'] = e^A(1+o_a(1)),\quad\VarB(Z_0') \le Ce^A/a^3,\quad \EB[Y_0'] \le Ce^A/a.
\end{equation}
This immediately yields the statement, by Chebychev's and Markov's inequalities.  
\end{proof}

\begin{lemma}
 \label{lem:Xt_second}
As $A$ and $a$ go to infinity, supposing $\PB(G_0)\to1$, the process $(X_t'')_{t\ge 0}$, defined by $X_t'' = X^{[\infty]}_{\Theta_{\lfloor t\widetilde{\ep}^{-1}\rfloor}} - \pi^2A t$, converges in law (with respect to Skorokhod's topology) to the L\'evy process $(L_t)_{t\ge 0}$ defined in Theorem~\ref{th:barrier}.
\end{lemma}
\begin{proof}
Note that the process $(X_t'')_{t\ge0}$ is adapted to the filtration $(\F_t'')_{t\ge0} := (\F_{\Theta_{\lfloor t\widetilde{\ep}^{-1}\rfloor}})_{t\ge0}$. In order to show convergence of the finite-dimensional distributions, it is then enough to show (see Proposition~3.1 in \cite{Kurtz1975} or Lemma~8.1 in \cite{Ethier1986}, p.\ 225), that for every $\lambda\in\R$ and $t,s\ge 0$,
\begin{equation}
\begin{split}
\label{eq:910}
 \EB\Big[\Big|\EB[e^{i\lambda (X_{t+s}''-X_t'')}\,|\,\F_t''] - e^{s\kappa(\lambda)}\Big|\Big] \to 0,
\end{split}
\end{equation} 
as $A$ and $a$ go to infinity, where $\kappa(\lambda)$ denotes the right-hand side of \eqref{eq:laplace_levy}. Fix $\lambda\in\R$ and $t,s\ge 0$ and define $m := \lfloor t\widetilde{\ep}^{-1}\rfloor$ and $n := \lfloor (t+s)\widetilde{\ep}^{-1}\rfloor$. Then, $(n-m)\widetilde{\ep} = s+A^{-1}o(1)$, by \eqref{eq:ep_upper} and \eqref{eq:ep_tilde_ep}. By Proposition \ref{prop:piece}, we have
\begin{equation}
\label{eq:912}
 \begin{split}
  \EB[e^{i\lambda (X_{t+s}''-X_t'')}\,|\,\F_{\Theta_m}]\Ind_{G_m} &= e^{-i\lambda \pi^2 As} \EB[e^{i\lambda (X^{[\infty]}_{\Theta_n} - X^{[\infty]}_{\Theta_m})}\,|\,\F_{\Theta_m}]\Ind_{G_m}\\
&= \exp\Big(s\big(\kappa(\lambda) + o(1) + O(\ep^\numcst)\big)\Big)\Ind_{G_m}.
 \end{split}
\end{equation}
In total, we get for $A$ and $a$ large enough,
\[
 \EB\Big[\Big|\EB[e^{i\lambda (X_{t+s}''-X_t'')}\,|\,\F_t''] - e^{s\kappa(\lambda)}\Big|\Big] \le e^{s \kappa(\lambda)}\EB[|e^{s(o(1) + O(\ep^\numcst))} - 1|] + \PB(G_m^c),
\]
where we used the fact that $\Re \kappa(\lambda) \le 0$ and therefore $|e^{s\kappa(\lambda)}|\le 1$.
By Proposition \ref{prop:piece}, this goes to $0$ as $A$ and $a$ go to infinity, which proves \eqref{eq:910}.

In order to show tightness in Skorokhod's topology, we use Aldous' famous criterion \cite{Aldous1978} (see also \cite{Billingsley1999}, Theorem 16.10): If for every $M > 0$, every family of $(\F_t'')$-stopping times $\tau = \tau(A,a)$ taking only finitely many values, all in $[0,M]$, and every $h = h(A,a)\ge 0$ with $h(A,a) \to 0$ as $A$ and $a$ go to infinity, we have
\begin{equation}
\label{eq:920}
 X_{\tau+h}'' - X_{\tau}'' \to 0,\quad\text{ in probability as $A$ and $a$ go to infinity},
\end{equation}
then tightness follows for the processes $(X_t'')_{t\ge0}$ (note that the second point in the criterion, namely tightness of $X_t''$ for every fixed $t$, follows from the convergence in finite-dimensional distributions proved above). Now let $\tau$ be such a stopping time and let $V_\tau$ be the (finite) set of values it takes. We first note that since $G_n \supset G_{n+1}$ for every $n\in \N$, we have for every $t\in V_\tau$ and every $A$ and $a$ large enough,
\begin{equation}
\label{eq:922}
\PB(G^c_{\lfloor t \widetilde{\ep}^{-1}\rfloor}) \le \PB(G^c_{\lfloor M \widetilde{\ep}^{-1}\rfloor}) = O(M \ep^\numcst).
\end{equation}
by Proposition \ref{prop:piece} and \eqref{eq:ep_tilde_ep}. Moreover, we have for every $\lambda > 0$,
\begin{align*}
 \EB[e^{i\lambda(X_{\tau+h}'' - X_{\tau}'')}] &= \sum_{t\in V_\tau}\EB\Big[e^{i\lambda(X_{t+h}'' - X_t'')}\Ind_{(\tau = t)}\Big]\\
&= \sum_{t\in V_\tau}\EB\Big[\EB[e^{i\lambda(X_{t+h}'' - X_t'')}\,|\,\F_t'']\Ind_{(\tau = t)}\Ind_{G_{\lfloor t\widetilde{\ep}^{-1}\rfloor}}\Big] + O(M \ep^\numcst) && \text{by \eqref{eq:922}}\\
&= e^{h (\kappa(\lambda)+o(1) + O(\ep^\numcst))}(1-O(M\ep^\numcst)) + O(M \ep^\numcst), && \text{by \eqref{eq:912}},
\end{align*}
which converges to $1$ as $A$ and $a$ go to infinity. This implies \eqref{eq:920} and therefore proves tightness in Skorokhod's topology, since $M$ was arbitrary. Together with the convergence in finite-dimensional distributions proved above, the lemma follows.
\end{proof}

\paragraph{A coupling with a Poisson process.} 
Let $(V_n)_{n\ge 1}$ be a sequence of independent exponentially distributed random variables with mean $\widetilde{\ep}$. In order to prove convergence of the processes $(X_t)_{t\ge0}$ and $(X_t')_{t\ge0}$ (the latter was defined in the statement of Theorem~\ref{th:barrier2}), we are going to couple the BBM with the sequence $(V_n)_{n\ge1}$ in the following way: Suppose we have constructed the BBM until time $\Theta_{n-1}$. Now, on the event $G_{n-1}$, by Corollary~\ref{cor:coupling_T}, the strong Markov property of BBM and the transfer theorem (a theorem which allows to ``transfer'' random variables to another probability space, see \cite{Kallenberg1997}, Theorem 5.10), we can construct the BBM up to time $\Theta_n$ such that $\PB(|(T_n-\Theta_{n-1})/a^3 - V_n| > \ep^{3/2}) = O(\ep^2)$ (recall that $T_n$ denotes the time of the first breakout after $\Theta_{n-1}$). On the event $G_{n-1}^c$, we simply let the BBM evolve independently of $(V_j)_{j\ge n}$. Setting $G_n' = G_n \cap \{\forall j\le n: |(T_j-\Theta_{j-1})/a^3 - V_j| \le \ep^{3/2}\}$, there exists by Corollary~\ref{cor:coupling_T} and Proposition~\ref{prop:piece} a 
numerical constant $\numcst > 0$, such 
that for large $A$ and $a$,
\begin{equation}
 \label{eq:prob_Gnprime}
\PB(G_n') \ge \PB(G_0) - nO(\ep^{1+\numcst})
\end{equation}
Furthermore, we have $\Theta_n = T_n + e^Aa^{2}$ on $G_n'$, whence for large $A$ and $a$,
\begin{equation*}
 \ton G_n': |(\Theta_n-\Theta_{n-1})/a^3 - V_n| \le 2\ep^{3/2}.
\end{equation*} 
This construction now permits us to prove Theorem~\ref{th:barrier2}:

\begin{proof}[Proof of Theorem~\ref{th:barrier2}]
The main idea of the proof is to compare the process $(X_t')_{t\ge0}$ with the process $(X_t'')_{t\ge0}$ and deduce convergence of the former from the convergence of the latter. For this, we first recall some basic facts about Skorokhod convergence.

Let $d$ denote the Skorokhod metric on the space of real-valued cadlag functions $D([0,\infty))$ (see \cite{Ethier1986}, Section 3.5). Let $\Phi$ be the space of strictly increasing, continuous, maps of $[0,\infty)$ onto itself. Let $x,x_1,x_2,\ldots$ be elements of $D([0,\infty))$. Then (\cite{Ethier1986}, Proposition 3.5.3), $d(x_n,x) \to 0$ as $n\to\infty$ if and only if for every $M > 0$ there exists a sequence $(\varphi_n)$ in $\Phi$, such that
\begin{equation}
 \label{eq:940}
\sup_{t\in [0,M]}|\varphi_n(t) - t| \to 0,\quad\tand\quad\sup_{t\in [0,M]} |x_n(\varphi_n(t)) - x(t)| \to 0.
\end{equation}
If $(x_n')_{n\in \N}$ is another sequence of functions in $D([0,\infty))$, with $d(x_n',x)\to 0$, then by the triangle inequality and the fact that $\Phi$ is stable under the operations of inverse and convolution, we have $d(x_n,x)\to 0$ if and only if there exists a sequence $(\varphi_n)$ in $\Phi$, such that the first inequality in \eqref{eq:940} holds and
\begin{equation*}
\sup_{t\in [0,M]} |x_n(\varphi_n(t)) - x_n'(t)| \to 0.
\end{equation*}
For every $A$ and $a$, we now define the (random) map $\varphi_{A,a} \in \Phi$ by
\[
 \varphi_{A,a}(\widetilde{\ep}(n+r)) = ((1-r) \Theta_n + r \Theta_{n+1})a^{-3}, \text{ with } n\in\N,\ r\in [0,1].
\]
Let $M > 0$ and define $n_M = \lceil M\widetilde{\ep}^{-1} \rceil$. Let $(X''_t)_{t\ge0}$ be the process from  Lemma~\ref{lem:Xt_second}. Then,
\begin{align*}
\sup_{t\in [0,M]} |\varphi_{A,a}(t) - t| &\le  \max_{n\in\{0,\ldots,n_M\}} \left|a^{-3}\Theta_n - \widetilde{\ep} n\right|,\\
\sup_{t\in [0,M]}|X''_t - X'_{\varphi_{A,a}(t)}| &\le  \max_{n\in\{0,\ldots,n_M\}} \pi^2 A \left|a^{-3}\Theta_n - \widetilde{\ep} n\right|.
\end{align*}
By Doob's $L^2$ inequality and the fact that $\widetilde{\ep} = \EB[V_1]$, we get
\[
 \PB\Big(\max_{n\in\{0,\ldots,n_M\}} \left|\sum_{i=1}^n V_i - n\widetilde{\ep}\right| > \ep^{1/3}\Big) \le C \ep^{-2/3} n_M \VarB(V_i) \le CM\ep^{1/3},
\]
by \eqref{eq:ep_tilde_ep}.
Furthermore, on the event $G_{n_M}'$ defined above, we have for every $n\le n_M$,
\(
|\Theta_n - \sum_{i=1}^n V_i| \le Cn_M \ep^{3/2} \le CM\ep^{1/2}.
\)
By Lemma~\ref{lem:Hperp} and \eqref{eq:prob_Gnprime}, $\PB(G_{n_M}')\to1$ as $A$ and $a$ go to infinity.
In total,
\begin{equation}
\label{eq:948}
 \forall M > 0:\sup_{t\in [0,M]} |\varphi_{A,a}(t) - t| \vee |X''_t - X'_{\varphi_{A,a}(t)}| \to 0,\quad\text{in probability},
\end{equation}
as $A$ and $a$ go to infinity, which is equivalent to 
\begin{equation}
\label{eq:949}
 \sum_{M=1}^\infty 2^{-M}\Big[1\wedge\Big(\sup_{t\in [0,M]} |\varphi_{A,a}(t) - t| \vee |X''_t - X'_{\varphi_{A,a}(t)}|\Big)\Big] \to 0,\quad\text{in probability}.
\end{equation}
Now, suppose that $A$ and $a$ go to infinity along a sequence $(A_n,a_n)_{n\in\N}$ and denote by $X'_{A_n,a_n}$, $X''_{A_n,a_n}$ and $\varphi_{A_n,a_n}$ the processes corresponding to these parameters. By Lemma~\ref{lem:Xt_second} and Skorokhod's representation theorem (\cite{Billingsley1999}, Theorem 6.7), there exists a probability space on which the sequence $(X''_{A_n,a_n})$ converges almost surely as $n\to\infty$ to the limiting L\'evy process $L = (L_t)_{t\ge 0}$ stated in the theorem. Applying again the representation theorem as well as the transfer theorem, we can transfer the processes $X'_{A_n,a_n}$ and $\varphi_{A_n,a_n}$ to this probability space in such a way that the convergence in \eqref{eq:949} holds almost surely, which implies that the convergence in \eqref{eq:948} holds almost surely as well. By the remarks at the beginning of the proof, it follows that on this new probability space,
\[
 d(X'_{A_n,a_n},L) \le d(X'_{A_n,a_n},X''_{A_n,a_n}) + d(X''_{A_n,a_n},L)\to 0,
\]
almost surely, as $n\to\infty$. This proves the theorem.
\end{proof}

\begin{proof}[Proof of Theorem~\ref{th:barrier}]
Let $(t_i^{A,a})_{i=1}^k$ be as in the statement of the theorem (with $n$ replaced by $k$) and write $t_i=t_i^{A,a}$. By Theorem~\ref{th:barrier2}, it suffices to show that
\begin{equation}
\label{eq:950}
 \PB\Big(\forall i: X^{[\infty]}_{t_ia^3} = J_{t_ia^3}\Big) \to 1.
\end{equation}
Let $n := \lceil 2(t_k+2)/\widetilde{\ep}\rceil$, so that $n = O(\ep^{-1})$, by \eqref{eq:ep_tilde_ep}, where we allow the $O(\cdot)$ symbol to depend on $t_k$. By Chebychev's inequality, we then have
\begin{equation}
 \label{eq:952}
\PB(\sum_{i=1}^n V_i \le t_k+2) \le \PB\Big(\sum_{i=1}^n (V_i - \widetilde{\ep}) \le -\frac n 2 \widetilde{\ep}\Big) = O(n\VarB(V_i)) = O(\ep).
\end{equation}
Furthermore, define the intervals $I_i = t_i + [-2n\ep^{3/2}-e^A/a,2n\ep^{3/2}]$, $i=1,\ldots,k$ and denote by $\mathscr P$ the point process on the real line with points at the positions $V_1,V_1+V_2,V_1+V_2+V_3,\ldots$ Then $\mathscr P$ is a Poisson process with intensity $\widetilde{\ep}^{-1} = O(\ep^{-1})$ and thus, for large $A$ and $a$,
\begin{equation}
 \label{eq:954}
\PB\Big(\mathscr P \cap \bigcup_{i=1}^k I_i \ne \emptyset\Big) \le Ck\ep^{1/2}.
\end{equation}
We now have
\begin{align*}
 \PB\Big(\forall i: X^{[\infty]}_{t_ia^3} = J_{t_ia^3}\Big) &\ge \PB\Big(\nexists (i,j): t_ia^3\in[\Theta_j-T_{j-1},\Theta_j]\Big) && \text{by definition}\\
&\ge \PB\Big(G_{n}',\ \sum_{i=1}^{n}V_i > t_k+2,\ \mathscr P \cap \bigcup_{i=1}^k I_i = \emptyset\Big) && \text{by definition of $G_n'$}\\
&\ge \PB(G_0) - O(\ep^\numcst) && \text{by \eqref{eq:prob_Gnprime}, \eqref{eq:952}, \eqref{eq:954}}.
\end{align*}
Letting $A$ and $a$ go to infinity and using again Lemma~\ref{lem:Hperp}, yields \eqref{eq:950} and thus proves the theorem.
\end{proof}

\section{The \texorpdfstring{\Bfl}{Bb}-BBM}
\label{sec:Bflat}
In this section, we prove the parts of Theorem~\ref{th:Bflat_Bsharp} and Proposition~\ref{prop:Bflat_Bsharp_med} concerning the \Bfl-BBM, which was defined in Section~\ref{sec:Bflat_definition}. This section relies very much on Section~\ref{sec:BBBM} and we will use all of the notation introduced there. 

\subsection{More definitions and results}
\label{sec:Bflat_moredefs}
Recall from the beginning of Section~\ref{sec:Bflat_Bsharp_definition} that we fix $\delta\in(0,1/100)$ and that the phrase ``for large $A$ and $a$'' may now depend on $\delta$. Depending on $\delta$, we fix $K\ge 1$ such that $E_K \le \delta/10$, where $E_K$ is defined in \eqref{eq:def_E}.  We will use the symbols $C_\delta$ and $C_{\delta,\alpha}$, which have the same meaning as $C$ (see Section~\ref{sec:notation}), except that they may depend on $\delta$ or $\delta$ and $\alpha$ as well, $\alpha$ being defined later.

For a Borel set $S\subset\R_+$, we define the stopping line $\mathscr L^{\mathrm{red}}_{\Box,S}$ by $(u,t)\in \mathscr L^{\mathrm{red}}_{\Box,S}$ if and only if the particle $u$ gets coloured red at time $t$ and has been white up to time $t$, with $t\in S$. We then set $Z^{\mathrm{red}}_{\Box,S}$ and  $Y^{\mathrm{red}}_{\Box,S}$ by summing respectively $w_Z$ and $w_Y$ over the particles of this stopping line. Furthermore, we define $\mathscr N^{\mathrm{red}}_t$ and $\mathscr N^{\mathrm{white}}_t$ to be the subsets of $\mathscr N_t$ formed by the red and white particles, respectively, and define $Z^{\mathrm{red}}_t$, $Y^{\mathrm{red}}_t$, $Z^{\mathrm{white}}_t$ and $Y^{\mathrm{white}}_t$ accordingly. Recall that we kill all red particles at every time $\Theta_n$, $n\ge0$, so that $\mathscr N^{\mathrm{red}}_{\Theta_n} = \emptyset$ and $\mathscr N_{\Theta_n} = \mathscr N^{\mathrm{white}}_{\Theta_n}$ for every $n\ge0$.

Note that the law of the process until time $\Theta_1$ is the same as under $\PB$, which allows us to use the results from Section~\ref{sec:BBBM}.

Recall that $\nu^\flat_t$ denotes the empirical measure of the white particles at time $t$ and abuse notation by setting $\nu^\flat_n =\nu^\flat_{\Theta_n}$. We set $G^{\flat}_{-1}=\Omega$ and for each $n\in\N$, we define the event $G^\flat_n$ to be the intersection of $G^\flat_{n-1}$ with the following events:
\begin{itemize}[nolistsep]
 \item $\supp\nu^\flat_n\subset (0,a)$,
 \item $\mathscr N^{\mathrm{white}}_{\Theta_n} \subset U\times \{\Theta_n\}$ and $\Theta_n > T_n^+$ (for $n>0$),
 \item $|e^{-A}Z^{\mathrm{white}}_{\Theta_n} -1| \le \ep^{3/2}$ and $Y^{\mathrm{white}}_{\Theta_n} \le \eta$.
 \item $\Pfl\Big(Z^{\mathrm{red}}_{\Box,[\Theta_n,\Theta_n+Ka^2]} + Y^{\mathrm{red}}_{\Box,[\Theta_n,\Theta_n+Ka^2]} \le a^{-1/2}\,\Big|\,\F_{\Theta_n}\Big) \ge 1-\ep^2$.
\end{itemize}
The last event is of course uniquely defined up to a set of probability zero. Note that $G^\flat_n\in\F_{\Theta_n}$ for each $n\in\N$.

We now state the main results from this section, which will imply the \Bfl-BBM parts of Theorem~\ref{th:Bflat_Bsharp} and Proposition~\ref{prop:Bflat_Bsharp_med}. They are proved in Section~\ref{sec:Bflat_main_results}. Recall the definition of (H$_\perp$) from Section~\ref{sec:BBBM_results}.
\begin{lemma}
\label{lem:HBflat}
(H$_\perp$) implies that $\Pfl(G^\flat_0)\to 1$ as $A$ and $a$ go to infinity.
\end{lemma}

\begin{proposition}
  \label{prop:Bflat}
Proposition~\ref{prop:piece} still holds with $G_n$, $\PB$, $\EB$ replaced by $G_n^\flat$, $\Pfl$, $\Efl$.
\end{proposition}

\subsection{The total number of particles: upper bounds}
\label{sec:Bflat_num_particles}
In this section, we will establish some fine estimates for the number of particles  of the process, which will be used later to bound the number of creations of red particles. For this, we define $N_t(r)$ to be the number of particles to the right of $r$ at time $t$, i.e.\ $N_t(r) = \sum_{u\in\mathscr A_0(t)}\Ind_{(X_u(t) \ge r)}$. Unfortunately, this quantity is not very convenient to work with, because of the ``in between'' particles, namely, particles $u\in\mathscr A_0(t)$ with $\tau_l(u) \le t < \sigma_{l+1}(u)$ for some $l\ge 0$ (see Section~\ref{sec:BBBM_tiers}). We therefore also define for $l\ge0$,
\[
  N^{(l)}_t(r) = \sum_{(u,s)\in\mathscr N^{(l)}_t}\Ind_{(X_u(s) \ge r)},
\]
and define  $N^{(l+)}_t(r)$ etc.\ by \eqref{eq:sum_over_tiers} as before. Note that contrary to $Z_t$, $Y_t$ and $R_t$, we may have $N_t(r) \ne N^{(0+)}_t(r)$!
We also apply the superscripts ``white'' and ``red'' to all of these quantities with the obvious meanings.

Write for short $N = N^\flat = \lfloor 2\pi  e^{A+\delta}a^{-3}e^{\mu a}\rfloor$. As in Section~\ref{sec:piece_proof}, we say that hypothesis (HB$_0$) is verified if $\nu_0$ is deterministic and such that $G_0$ holds.
The main lemma in this section is the following:
\begin{lemma}
\label{lem:Bflat_N} Suppose (HB$_0$). Let $t\in [Ka^2,\tconst{eq:tcrit}\wedge \sqrt\ep a^3]$.  Then, for $0\le r\le 9a/10$ and every $\alpha > 0$ there exists $C_{\delta,\alpha}$, such that for large $A$ and $a$,
\[
\P(N_t(r) \ge N\,|\,T > t) \le C_{\delta,\alpha} A^2\ep\Big(\frac t {a^3}+\eta\Big)e^{- (2 - \alpha)  r}
\]
Furthermore, conditioned on $\F_\Delta$, for $t\in[T,\Theta_1+e^Aa^2]$, for large $A$ and $a$,
\[
\Pfl_{\mathrm{nbab}}(N_t(r) \ge N\,|\,\F_\Delta)\Ind_{G_\Delta}\le C_{\delta,\alpha} A^2\ep^2e^{- (2 - \alpha)  r}.
\]
\end{lemma}
The proof of this lemma goes by a first-second moment argument, making use of the results from Section~\ref{sec:interval_number} for BBM in an interval. It is more convenient to calculate the moments of $N^{(0+)}_t(r)$ instead of $N_t(r)$, which is why we will work with this quantity for the next lemmas. 

The following lemma about BBM conditioned not to break out will be used many times in the proof. Recall the definition of a barrier function $f$ and the norm $\|f\|$ from Section~\ref{sec:BBBM_definition_definition}. Also, see Remark~\ref{rem:hat_f} about the definition of the tier structure in the case of varying drift.

\begin{lemma}
  \label{lem:N_with_tiers}
Let $f$ be a barrier function, $0\le t\le t_0\le \tconst{eq:tcrit}$ and suppose that $\|f\|$ is bounded by a function depending on $A$ only and that either $f\equiv 0$ or $t_0 \le 2e^Aa^2$. For $x\in(0,a)$, define $\Phat_f^x = \P_f^x(\cdot\,|\,T>t_0)$.   Then, for all $r\le (9/10)a$ and for large $A$ and $a$,
\begin{align}
  \label{eq:66}
&\Ehat_f^x[N^{(1+)}_{t}(r)] \le C Ae^{-A}N(1+r^2)e^{-\mu r}\Big(\frac {t} {a^3}w_Z(x)+w_Y(x)\Big),\\
\label{eq:66a}
&\Ehat_f^x[N^{(0+)}_{t}(r)] \le C Ae^{-A}N(1+ r^2)e^{-\mu r}e^{-3(a-x)/4},\\
  \label{eq:67}
&\Ehat_f^x[(N^{(0+)}_{t}(r))^2] \le C (1+r^4)e^{- 2  r} \ep e^{-A} N^2 \Big(\frac {t} {a^3}w_Z(x)+w_Y(x)\Big).
\end{align}
\end{lemma}
\begin{proof}
First note that by the hypotheses together with Proposition~\ref{prop:T} and Lemma~\ref{lem:rit3}, we can apply the results of Lemma~\ref{lem:hat_Z_f}. By \eqref{eq:ichhassedas}, Lemma~\ref{lem:mu_t_density},  Corollary~\ref{cor:N_expec_upper_bound}, the definition of $N$ and the hypothesis on $r$, we have for every $x\in(0,a)$ and $s\in[0,t]$,
\begin{equation}
 \label{eq:metro}
 \Ehat_f^{(x,s)}[N^{(0)}_t(r)] \le \begin{cases}
                                    C e^{-A} N (1+r^2) e^{-\mu r} (w_Z(x) + w_Y(x)) & \tif x\ge 19a/20\\
                                    C e^{-A} N e^{-\mu r} a^3 w_Y(x)                & \tif x< 19a/20.
                                   \end{cases}
\end{equation}
Equation \eqref{eq:66} then follows from \eqref{eq:metro} and \eqref{eq:28}, noting that $a-y-f(\zeta/a^2) \ge 19a/20$ for large $a$, by the hypothesis on $f$. Equation~\eqref{eq:66a} easily follows from \eqref{eq:66} and \eqref{eq:metro}.

For the proof of \eqref{eq:67}, we omit $f$ from the notation, for simplicity. By Lemma~\ref{lem:many_to_two_with_tiers}, we have for all $x\in(0,a)$,
  \begin{multline}
    \label{eq:27}
    \Ehat^x[(N^{(0+)}_{t}(r))^2] \le \Ehat^x[N^{(0+)}_{t}(r)] + C\int_0^t\int_0^a
    \widehat{\p}^{(0+)}_s(x,z) \Ehat^{(z,s)}[N^{(0+)}_{t}(r)]^2\,\dd z\,\dd s\\
+ \Ehat^x\Big[\sum_{(u,s)\in\mathscr R_{t}} \Big(\sum_{(v,t')\in\mathscr
  S^{(u,s)}} \Ehat^{(X_v(t'),t')}[N^{(0+)}_{t}(r)]\Big)^2\Big].
  \end{multline}
We first bound the second summand in \eqref{eq:27}. By \eqref{eq:ptilde_estimate_f},
\begin{multline}
  \label{eq:30}
\int_0^t\int_0^a \widehat{\p}^{(0+)}_s(x,z)\Ehat^{(z,s)}[N^{(0+)}_{t}(r)]^2\,\dd z\,\dd s\\
\le C \Ehat^x\Big[\sum_{(v,t')\in\mathscr S_t} \int_{t'\wedge t}^t\int_0^a \p^{(0)}_{s-t'}(X_v(t'),z)\Ehat^{(z,s)}[N^{(0+)}_{t}(r)]^2\,\dd z\,\dd s\Big]
\end{multline}
As in the proof of Lemma~\ref{lem:N_2ndmoment}, but using \eqref{eq:66a} instead of Corollary~\ref{cor:N_expec_upper_bound}, we have for every $t'\le t$,
\begin{equation}
  \label{eq:29}
  \int_{t'}^t\int_0^a \p^{(0)}_{s-t'}(x,z)  \Ehat^{(z,s)}[N^{(0+)}_{t}(r)]^2\,\dd z\,\dd s \le CA^2e^{-2A}N^2(1+r^4)e^{-2\mu r}\Big(\frac t {a^3} w_Z(x)+w_Y(x)\Big).
\end{equation}
Equations \eqref{eq:30} and \eqref{eq:29} together with \eqref{eq:28} and the hypothesis on $t$ now give
\begin{equation}
\label{eq:29b}
\int_0^t\int_0^a \widehat{\p}^{(0+)}_s(x,z)\Ehat^{(z,s)}[N^{(0+)}_{t}(r)]^2\,\dd z\,\dd s \le 
CA^2e^{-2A}N^2(1+r^4)e^{-2\mu r}\Big(\frac t {a^3} w_Z(x)+A w_Y(x)\Big).
\end{equation}
As for the last summand in \eqref{eq:27}, let $(u,s)\in\mathscr R_{t}$. Note again that for every $(v,t')\in\mathscr S^{(u,s)}$, $X_v(t') \ge 19a/20$ for large $a$. By \eqref{eq:66}, \eqref{eq:metro} and the hypothesis on $t$, we then have
\[
\sum_{(v,t')\in\mathscr S^{(u,s)}} \Ehat^{(X_v(t'),t')}[N^{(0+)}_{t}(r)] 
\le C e^{-A} N (1+r^2) e^{-\mu r} (Z^{(u,s)} + AY^{(u,s)}).
\]
Furthermore, we have $AY^{(u,s)} \le A\eta Z^{(u,s)} \le Z^{(u,s)}$ for large $A$ and $a$ by the first part of Lemma~\ref{lem:ZYW} and \eqref{eq:eta}. This gives,
\begin{align}
\nonumber \Ehat^x\Big[&\sum_{(u,s)\in\mathscr R_{t}} \Big(\sum_{(v,t')\in\mathscr
  S^{(u,s)}} \Ehat^{(X_v(t'),t')}[N^{(0+)}_{t}(r)]\Big)^2\Big]\\
\label{eq:31} &\le C (1+r^4)e^{-2\mu r} e^{-2A}N^2 \Ehat^x\Big[\sum_{(u,s)\in\mathscr R_{t}} (Z^{(u,s)})^2\Big].
\end{align}
By now familiar arguments (namely, first \eqref{eq:QaZsquared}, then \eqref{eq:28}, \eqref{eq:ichhassedas}, Lemmas~\ref{lem:R_f} and~\ref{lem:Rt} and the hypothesis on $t$) give 
\begin{equation}
 \label{eq:31a}
 \Ehat^x\Big[\sum_{(u,s)\in\mathscr R_{t}} (Z^{(u,s)})^2\Big] \le C \ep e^A \Ehat^x[R_t] \le C\ep e^A\Big(\frac {t} {a^3} w_Z(x) + w_Y(x)\Big).
\end{equation}
Equations \eqref{eq:29b}, \eqref{eq:31} and \eqref{eq:31a}, together with \eqref{eq:ep_lower} and \eqref{eq:mu}, now bound the second and third summands in \eqref{eq:27}. The first summand is easily bounded by \eqref{eq:66} and \eqref{eq:metro}. This yields \eqref{eq:67} and finishes the proof.
\end{proof}

\begin{lemma}
  \label{lem:N_bar_with_tiers}
Suppose (HB$_0$). For large $A$ and $a$, we have for every $0\le r\le 9a/10$ and $t\in[T,\Theta_1+e^Aa^2]$, 
  \begin{equation}
    \label{eq:65}
  \Pfl_{\mathrm{nbab}}(\overline N^{(0+)}_t(r)> \delta N/4\,|\,\F_\Delta)\Ind_{G_\Delta}\le C_\delta \ep e^{-2A/3}(1+r^4)e^{-2 r}.
  \end{equation}
\end{lemma}
\begin{proof}
Recall the line $\widebar{\mathscr L}_{\mathscr U}$ from Section~\ref{sec:fugitive}. By independence of its descendants, we have with $f_{\Delta}^+ = f_\Delta(\cdot - T^+/a^2)$,
\begin{equation*}
\Varfl_{\mathrm{nbab}}(\overline N^{(0+)}_t(r)\,|\,\F_\Delta)\Ind_{G_\Delta}\le \sum_{(u,s)\in \widebar{\mathscr L}_{\mathscr U}} \E^{(X_u(s),s)}_{f_\Delta^+}[N^{(0+)}_t(r)^2\,|\,T>\Theta_1+e^Aa^2],
\end{equation*}
and Lemma \ref{lem:N_with_tiers} now implies
\begin{equation}
  \label{eq:68}
 \Varfl_{\mathrm{nbab}}(\overline N^{(0+)}_t(r)\,|\,\F_\Delta)\Ind_{G_\Delta}\le C (1+r^4)e^{- 2  r} \ep e^{-A} N^2 \mathscr E_{\mathscr U}\Ind_{G_\Delta} \le C (1+r^4)e^{- 2  r}\ep e^{-2A/3}N^2,
\end{equation}
where the last inequality follows from the fact that $\mathscr E_{\mathscr U}\le e^{A/3}$ on $G_{\mathscr U}\subset G_\Delta$. 
Furthermore, by Lemma~\ref{lem:N_with_tiers},
\begin{equation}
  \label{eq:70}
  \Efl_{\mathrm{nbab}}[\overline N^{(0+)}_t(r)\,|\,\F_\Delta]\Ind_{G_\Delta} \le  C Ae^{-A}N(1+r^2)e^{-\mu r} \mathscr E_{\mathscr U}\Ind_{G_\Delta} \le CAe^{-2A/3}N.
\end{equation}
Equation \eqref{eq:65} now follows from \eqref{eq:68} and \eqref{eq:70} together with the conditional Chebychev inequality.
\end{proof}

\begin{proof}[Proof of Lemma \ref{lem:Bflat_N}]
For simplicity, we will only prove the lemma with $N_t(r)$ replaced by $N^{(0+)}_t(r)$. This is certainly enough if $q(0) = 0$, since then $N_t(r) \le N^{(0+)}_t(r)$ almost surely. If $q(0) \ne 0$, one can bound the second moment of the number of in-between particles at time $t$ by a constant $C_{A,\zeta}$ depending on $A$ and $\zeta$ only (all results needed for this have been introduced, most notably the second statement of Lemma~\ref{lem:Zt_variance}). This yields the bound $\P(N_t(r)-N^{(0+)}_t(r) > (\delta/10) N) \le 100 \delta^{-2} C_{A,\zeta}/N^2$, which can be used with the bounds in the proof to yield the lemma. We omit the (technical) details.

Assume  $Ka^2 \le t \le  \tconst{eq:tcrit}\wedge \sqrt\ep a^3$. By Lemma~\ref{lem:N_expec_large_t} and Corollary~\ref{cor:hat_many_to_one}, we have for large $A$ and $a$,
\begin{align}
\E[N^{(0)}_{t}(r)\,|\,T>t] 
\le 2\pi a^{-3} e^{\mu a} Z_0 \big(1+E_K+\big(\tfrac{1+r}{a}\big)^2\big) (1+\mu r)e^{-\mu r}
\label{eq:22}
\le (1-3\delta/4)N,
\end{align}
where the last inequality follows from hypothesis (HB$_0$) and the definitions of $K$ and $N$ from  Sections~\ref{sec:Bflat_moredefs} and \ref{sec:Bflat_definition}, respectively.
Moreover, we have by Lemma~\ref{lem:N_with_tiers},
 \begin{equation}
   \label{eq:23}
\E[N^{(1+)}_{t}(r)\,|\,T>t] \le C(1+r^2) e^{-\mu r}N A e^{-A}(\frac t {a^3}Z_0 + Y_0) \le C\sqrt{\ep}AN,
 \end{equation}
by hypothesis (HB$_0$), the hypothesis on $t$ and \eqref{eq:eta}. Equations \eqref{eq:22}, \eqref{eq:23} and \eqref{eq:ep_upper} now give for large $A$ and $a$, 
\begin{equation}
  \label{eq:33}
   \E[ N^{(0+)}_t(r)\,|\,T>t] \le (1-\delta/2) N.
\end{equation}
Furthermore, by Lemma~\ref{lem:N_with_tiers} and the hypotheses on $Z_0$ and $Y_0$, we have
\begin{equation}
  \label{eq:75}
\Var( N^{(0+)}_t(r)\,|\,T>t) \le C \ep N^2(1+r^4)e^{-2r} (t/a^3 +\eta).
\end{equation}
Chebychev's inequality, \eqref{eq:33} and \eqref{eq:75} yield the first equation of the lemma.

Conditioned on $\F_\Delta$, let $t\in[T,\Theta_1+e^Aa^2]$. Define $N^{(0),\mathrm{bulk}}_t(r)$ to be the number of hat- and check-particles to the right of $r$ at time $t$ that have not hit $a$ between $T^-$ and $t$ and likewise $N^{(0),\mathrm{fug}}_t(r)$ the number of fug-particles with the same properties. Set $M_t = e^{- X^{[1]}_{t}} = (1+\Delta_t)^{-1}$, where $\Delta_t =  \thbar((t-T^+)/a^2)(e^\Delta-1)$. Note that we have on $G_\Delta$:  $|e^\Delta-1 - e^{-A}Z^{(\mathscr U,T)}| \le 2\ep^{1/4}$.

Define $\alpha_r = (1+\mu r)e^{-\mu r}$. By Lemmas~\ref{lem:mu_t_density} and~\ref{lem:N_expec_large_t}, together with \eqref{eq:pB} and \eqref{eq:ichhassedas} from Lemma~\ref{lem:hat_Z_f} (which we can apply by Lemma~\ref{lem:rit3}), we have for large $A$ and $a$,
\begin{equation}
  \label{eq:48}
  \begin{split}
  \Efl[N^{(0),\mathrm{bulk}}_t(r)\,|\,\F_\Delta]\Ind_{G_\Delta} \le \alpha_re^{o(1)}NM_t(1+E_K)e^{-A}(\widehat Z_{T^-}+\widecheck Z_{\Delta})\Ind_{G_\Delta}
\le \alpha_rNM_t(1-3\delta/4),
  \end{split}
\end{equation}
by the definition of $\widehat G$ and $\widecheck G$.
Furthermore, we have by Corollary~\ref{cor:N_expec_thbar} and Lemma~\ref{lem:mu_t_density}, for large $a$,
\begin{align}
\nonumber
  \Efl[N^{(0),\mathrm{fug}}_t(r)\,|\,\F_\Delta] \Ind_{G_\Delta} &\le \alpha_re^{o_a(1)} N M_t \left(\thbar\left((t-T)/a^2\right) + O\left((y+\Delta+r)^2\eta/a\right)\right) e^{-A}Z^{(\mathscr U,T)}\Ind_{G_\Delta}\\
    \label{eq:51}
&\le \alpha_r N M_t \left(\Delta_t+ O(\ep^{1/4}+a^{-1}r^2)\right)\Ind_{G_\Delta}.
\end{align}
Equations \eqref{eq:48} and \eqref{eq:51} now give,
\begin{equation}
  \label{eq:55}
\Efl[N^{(0),\mathrm{bulk}}_t(r)+N^{(0),\mathrm{fug}}_t(r)\,|\,\F_\Delta]\Ind_{G_\Delta} \le \alpha_r  N \left(1-\delta/2 + O(a^{-1}r^2)\right).
\end{equation}
Similarly, one has by Lemma~\ref{lem:N_2ndmoment} and \eqref{eq:ichhassedas},
\begin{equation}
  \label{eq:40}
\Varfl(N^{(0),\mathrm{bulk}}_t(r)+N^{(0),\mathrm{fug}}_t(r)\,|\,\F_\Delta) \Ind_{G_\Delta}\le Ce^{-A} (1+r^4) e^{-2\mu r}N^2.
\end{equation}
Equations \eqref{eq:55} and \eqref{eq:40} and the conditional Chebychev inequality now yield for large $A$ and $a$,
\[
\Pfl(N^{(0),\mathrm{bulk}}_t(r)+N^{(0),\mathrm{fug}}_t(r) \ge(1-\delta/4)N\,|\,\F_\Delta)\Ind_{G_\Delta} \le C_\delta (1+r^4)e^{-A}e^{-2 \mu r}.
\]
This, together with Lemma \ref{lem:N_bar_with_tiers} and Lemma~\ref{lem:Gnbab} and the fact that the hat-, fug- and check-particles do not hit $a$ on $G_{\mathrm{nbab}}$ finishes the proof of the lemma.
\end{proof}

We finish this section with a result which extends Lemma~\ref{lem:Bflat_N} to $t\le Ka^2$ under (H$_\perp$).

\begin{lemma}
  \label{lem:Bflat_start}
Suppose (H$_\perp$) and let $\alpha > 0$. Then for large $A$ and $a$, we have for every $t\le Ka^2$ and $0\le r\le 9a/10$,
\[
\Pfl(N_t(r) \ge N,\,R^{(0)}_{Ka^2} = 0) \le C_{\alpha,\delta} e^{-(2 -\alpha) r}/a,\quad\tand\quad \Pfl(R^{(0)}_{Ka^2} = 0) \ge 1-C_\delta e^A/a.
\]
\end{lemma}
\begin{proof}
On the event $\{R^{(0)}_{Ka^2} = 0\}$, we have $N_t(r) = N^{(0)}_t(r)$ and $T>Ka^2$, so it is enough to prove the statement with $N_t(r)$ replaced by $N^{(0)}_t(r)$ and $\PB$ replaced by $\P$. Lemmas~\ref{lem:Rt} together with  \eqref{eq:49} from the proof of Lemma~\ref{lem:Hperp} yields
  \begin{equation*}
\P(R^{(0)}_{Ka^2} \ge 1) \le \E[R^{(0)}_{Ka^2}] \le C \Big(\frac K a \E[Z_0] + \E[Y_0]\Big) \le C \frac {Ke^A} a,
  \end{equation*}
which gives the second statement. As for the first one, note that $\lfloor 2\pi  e^Aa^{-3}e^{\mu a} \rfloor$ equals $\lfloor e^{-\delta} N\rfloor$ or $\lceil e^{-\delta} N\rceil$. Suppose the former for simplicity. Now, since the density $\phi(x)$ is stationary w.r.t. BBM with absorption at $0$ and $a$, we have
\begin{equation}
  \label{eq:69}
  \E[N_t^{(0)}(r)] =  \E[N_0^{(0)}(r)] = \lfloor e^{-\delta}N\rfloor \int_r^a \phi(x)\,\dd x \le e^{-\delta}N (1+\mu r) e^{-\mu r}.
\end{equation}
 Furthermore, by the independence of the initial particles and the tower property of conditional expectation,
\begin{equation*}
  \Var(N_t^{(0)}(r)) \le N E_X[\E^X[(N_t^{(0)}(r))^2]],
\end{equation*}
where $X$ is a random variable distributed according to the density $\phi$ and the outer expectation is with respect to $X$.  
By Lemma~\ref{lem:N_2ndmoment}, we have for every $x\in [0,a]$,
\[
  \E^x[(N_t^{(0)}(r))^2] \le Ce^{-2A}N^2(1+r^4)e^{-2\mu r}((t/a^3)w_Z(x)+w_Y(x)),
\]
and a simple calculation then yields for $t\le Ka^2$,
\begin{equation}
  \label{eq:84}
  \Var(N_t^{(0)}(r)) \le N \int_0^a \E^x[(N_t^{(0)}(r))^2] \phi(x)\,\dd x \le C(K/a)e^{-A} N^2 (1+r^4)e^{-2\mu r}.
\end{equation}
The lemma now follows from \eqref{eq:69} and \eqref{eq:84}, together with Chebychev's inequality.
\end{proof}

\def\Nwtr{N^{\mathrm{wtr}}}

\subsection{Bounds on the number of red particles}
For a   Borel $S\subset\R_+$ and $r\in[0,a]$, denote by $\Nwtr(S,r)$ (``wtr'' stands for ``white to red'') the number of white particles turning red to the right of $r$ at some time $t\in S$. More precisely, set 
\[
 \Nwtr(S,r) = \# \{(u,t)\in\mathscr L^{\mathrm{red}}_{\Box,S}:X_u(t)\ge r\}.
\]
Furthermore, we denote by $N^{\mathrm{white}}_{t}(r)$ the number of white particles with positions $\ge r$ at time $t$ (including the ``in between particles'').

\begin{lemma}
\label{lem:Nwtr}
 For a Borel set $S\subset\R_+$, write $|S|$ for its Lebesgue measure. Then for every $\alpha > 0$, for sufficiently small $\delta$, there exists $C_{\delta,\alpha}$, such that for large $A$ and $a$, the following holds:
  \begin{enumerate}
   \item  Suppose (HB$_0$). For every $r\in [0,9a/10]$, $t\in[Ka^2,\tconst{eq:tcrit}\wedge\sqrt\ep a^3]$ and every interval $I\subset[Ka^2,t]$,
  \[
   \E[N^{\mathrm{wtr}}(I,r)\,|\,T>t] \le C_{\delta,\alpha}\ep\Big(\frac t {a^3} +\eta\Big)e^{-(2-\alpha) r}N (|I|+1).
  \]
  \item  Suppose (HB$_0$). Conditioned on $\F_\Delta$, for every $r\in [0,9a/10]$ and every interval $I\subset [T,\Theta_1+e^Aa^2]$,
  \[
   \Efl_{\mathrm{nbab}}[N^{\mathrm{wtr}}(I,r)\,|\,\F_\Delta]\Ind_{G_\Delta} \le C_{\delta,\alpha}\ep e^{-2A/3} e^{-(2-\alpha) r}N (|I|+1)
  \]
  \item Suppose (H$_\perp$). For every $r\in [0,9a/10]$,
  \[
   \Efl[N^{\mathrm{wtr}}([0,Ka^2],r)\Ind_{(R_{Ka^2} = 0)}] \le C_{\delta,\alpha}a^{-1}e^{-(2-\alpha) r}N.
  \]
 
  \end{enumerate}

\end{lemma}
\begin{proof}
Define $\widetilde N^{\mathrm{wtr}}(I,r)$ like $N^{\mathrm{wtr}}(I,r)$ but not counting those particles which become red because of an in-between particle branching. We will only prove the lemma with $N^{\mathrm{wtr}}(I,r)$ replaced by $\widetilde N^{\mathrm{wtr}}(I,r)$, in which case one can actually take $|I|$ instead of $|I|+1$ in the inequalities from the statement of the lemma. Taking account of the remaining particles, i.e.\ of the difference $\widetilde N^{\mathrm{wtr}}(I,r) - N^{\mathrm{wtr}}(I,r)$ goes through straightforward but technical and uninteresting arguments.

We first prove the first statement. Fix $r\in[0,a]$ and $t\in [Ka^2,\tconst{eq:tcrit}\wedge\sqrt\ep a^3]$. Write $\Ehat[\cdot] =  \E[\cdot\,|\,T>t]$ and define the measure $\widetilde{\mathfrak{m}}_r$ on $[Ka^2,t]$ by 
\[
 \widetilde{\mathfrak{m}}_r(S) = \Ehat[\widetilde N^{\mathrm{wtr}}(S,r)],\quad\text{ for all Borel $S\subset[Ka^2,t]$.}
\]
We want to show that the measure $\widetilde{\mathfrak{m}}_r$ is dominated by a constant multiple of Lebesgue measure (with the constant given in the statement of the lemma). For this, it is enough to bound the limit of  $\widetilde{\mathfrak{m}}_r([s,s+h])/h$ as $h\to 0$, for every $s\in [Ka^2,t]$. Now, one easily sees that in this limit, only those events contribute where $N$ white particles lie to the right of $r$ at time $s$. More precisely, fix $s\in [Ka^2,t]$, let $\widetilde{\mathcal N}^{\mathrm{white}}_s(r)$ be the stopping line formed by the non-in-between white particles to the right of $r$ at time $s$. Then, since a white particle branching into $k\ge0$ descendants creates $(k-1)_+$ new white particles, we have,
\begin{equation}
\label{eq:234798}
\lim_{h\to0} \widetilde{\mathfrak{m}}_r([s,s+h])/h = \lim_{h\to0} \Ehat\Big[\Ind_{(N^{\mathrm{white}}_s(r) = N)}\sum_{(u,s)\in\widetilde{\mathcal N}^{\mathrm{white}}_s(r)} \Ehat^{(X_u(s),s)}[(N_{s+h}(r) - 1)_+]\Big]/h. 
\end{equation}
Since for every $x\in(0,a)$ the law  $\Phat^{(x,s)}$ is dominated in the sense of measures by $2\P^{(x,s)}$ for large $A$ and $a$, we have $\Ehat^{(x,s)}[(N_{s+h}(r)-1)_+] \le Ch$ for small $h$. Together with \eqref{eq:234798} this gives
\[
 \lim_{h\to0}  \widetilde{\mathfrak{m}}_r([s,s+h])/h \le CN\Phat(N^{\mathrm{white}}_s(r) = N) \le CN\Phat(N_s(r) \ge N) \le C_{\delta,\alpha}\ep\Big(\frac t {a^3} +\eta\Big)e^{-(2-\alpha) r}N,
\]
by the first part of Lemma~\ref{lem:Bflat_N}. This proves the first statement. The second and third statements are proven similarly, drawing on the second part of Lemma~\ref{lem:Bflat_N} and on Lemma~\ref{lem:Bflat_start}, respectively.
\end{proof}

For an interval $I\subset \R_+$, let $G^{\mathrm{wtr}}_I$ be the event that $\Nwtr(I,2a/3)$ equals zero. 
\begin{corollary}
  \label{cor:Ga2}
Suppose (HB$_0$). Then, for large $A$ and $a$,
\[
\P((G^{\mathrm{wtr}}_{[Ka^2,T\wedge \sqrt\ep a^3]})^c) \le C_\delta o_a(1)\quad\tand\quad \Pfl_{\mathrm{nbab}}\left((G^{\mathrm{wtr}}_{[T,\Theta_1+e^Aa^2]})^c\,|\,\F_\Delta\right)\Ind_{G_\Delta}\le C_\delta o_a(1).
\]
\end{corollary}
\begin{proof}
The second inequality follows directly from the second part of Lemma~\ref{lem:Nwtr} and Markov's inequality. For the first inequality, we have for every $r\le 9a/10$, by the first part of Lemma~\ref{lem:Nwtr},
\begin{align}
\nonumber
 \E[\Nwtr([Ka^2,T\wedge \sqrt\ep a^3],r)] &= \int_{Ka^2}^{\sqrt\ep a^3} \E[\Nwtr(\dd s,r)\Ind_{(T>s)}] \\
 \label{eq:horn}
 &\le C_{\delta}\ep e^{-1.9 r}N \int_{Ka^2}^{\sqrt\ep a^3} \Big(\frac s {a^3}+\eta\Big) \P(T>s)\,\dd s.
 \end{align}
In particular, $\E[\Nwtr([Ka^2,T\wedge \sqrt\ep a^3],2a/3)] \le C_\delta o_a(1)$, since the last integral in \eqref{eq:horn} is trivially bounded by $Ca^3$. The first inequality follows from this together with Markov's inequality.
\end{proof}

\begin{lemma} Suppose (HB$_0$). Define the intervals $I = [Ka^2,T\wedge \sqrt\ep a^3]$ and $J = [T,\Theta_1+e^Aa^2]$. For large $A$ and $a$,
\label{lem:ZY_red}
\begin{align*}
\E[(Z^{\mathrm{red}}_{\Box,I} + Y^{\mathrm{red}}_{\Box,I})\Ind_{G^{\mathrm{wtr}}_{I}}] \le C_{\delta} \ep^3 e^A,\quad\tand\quad
\Efl_{\mathrm{nbab}}[(Z^{\mathrm{red}}_{\Box,J} + Y^{\mathrm{red}}_{\Box,J})\Ind_{G^{\mathrm{wtr}}_{J}}\,|\,\F_\Delta]\Ind_{G_\Delta} \le C_{\delta} \ep e^{4A/3}/a. 
\end{align*}
Now suppose (H$_\perp$). Then $\E[(Z^{\mathrm{red}}_{\Box,[0,Ka^2]} + Y^{\mathrm{red}}_{\Box,[0,Ka^2]})\Ind_{(R_{Ka^2}=0)}] \le C_{\delta}e^A/a^2.$
\end{lemma}
\begin{proof}
We only prove the first inequality, the others follow similarly, using the second and third parts of Lemma~\ref{lem:Nwtr} instead of the first. By integration by parts, we have on the event $G^{\mathrm{wtr}}_{I}$,
\[
 Z^{\mathrm{red}}_{\Box,I} = \int_0^{2a/3} w_Z'(r) \Nwtr(I,r)\,\dd r\quad\tand\quad Y^{\mathrm{red}}_{\Box,I} = w_Y(0)\Nwtr(I,0) + \int_0^{2a/3} w_Y'(r) \Nwtr(I,r)\,\dd r.
\]
Now note that by \eqref{eq:horn}, we have for every $r\le 2a/3$, by Corollary~\ref{cor:moments_T} and \eqref{eq:eta_ep},
\[
 \E[\Nwtr(I,r)] \le C_{\delta}\ep e^{-1.9 r}N a^3\E\Big[\Big(\frac T{a^3}\wedge 1\Big)^2 + \eta \Big(\frac T{a^3}\wedge 1\Big) \Big] \le C_{\delta}\ep^3 e^{-1.9 r}N a^3.
\]
The inequality then follows easily from the previous inequalities, since $w_Z'(x)\le C(1+x)e^{x-a}$ and $w_Y'(x) \le Cw_Y(x) \le 2Ce^{x-a}$ for $x\in[0,a]$ and large $a$.
\end{proof}

\subsection{The probability of \texorpdfstring{$G^{\flat}_1$}{Gb1}}

We say that hypothesis (HB$^\flat_0$) is verified if $\nu^\flat_0$ is deterministic and such that $G^\flat_0$ holds.

\begin{lemma} Suppose (HB$^\flat_0$). Let $G_{\mathrm{fred}}$ be the event that the fugitive does not get coloured red. Then, for large $A$ and $a$,
  \label{lem:fred}
\[
\Pfl(G_{\mathrm{fred}}^c \cap (G_{\mathscr U}\cap G^{\mathrm{wtr}}_{[Ka^2,T]}) ) \le C_\delta \ep^2.
\]
\end{lemma}
\begin{proof}
Set $G^{\mathrm{wtr}} = G^{\mathrm{wtr}}_{[Ka^2,T]}$. 
Let $T_{\mathrm{red}}$ be the first breakout of a red particle. Then
\begin{equation}
  \label{eq:92}
  \Pfl(G_{\mathrm{fred}}^c \cap (G_{\mathscr U}\cap G^{\mathrm{wtr}})) = \P(T = T_{\mathrm{red}},\,G_{\mathscr U}\cap G^{\mathrm{wtr}}) \le \P(T_{\mathrm{red}} \le \sqrt\ep a^3,\,G^{\mathrm{wtr}}).
\end{equation} 
Define $T' = T\wedge \sqrt\ep a^3$. Then, since $\mathscr L^{\mathrm{red}}_{\Box,[0,T]}$ is a stopping line, we have by Proposition~\ref{prop:T} and \eqref{eq:pB},
\begin{equation}
\label{eq:93}
\P(T_{\mathrm{red}} \le \sqrt\ep a^3\,|\,\F_{\mathscr L^{\mathrm{red}}_{\Box,[0,T]}})\Ind_{G^{\mathrm{wtr}}} \le C \ep^{-1/2}e^{-A} (Z^{\mathrm{red}}_{\Box,[0,T']}+Y^{\mathrm{red}}_{\Box,[0,T']}).
\end{equation}
Recall that $\P(Z^{\mathrm{red}}_{\Box,[0,Ka^2]} + Y^{\mathrm{red}}_{\Box,[0,Ka^2]} \le a^{-1/2}) \ge 1-\ep^2$ by hypothesis.
By the tower property of conditional expectation and \eqref{eq:93}, this gives
\begin{align}
  \P(T_{\mathrm{red}} \le \sqrt\ep a^3,\,G^{\mathrm{wtr}}) 
  \label{eq:9470}
\le C\E[\ep^{-1/2}e^{-A} (Z^{\mathrm{red}}_{\Box,[Ka^2,T']}+Y^{\mathrm{red}}_{\Box,[Ka^2,T']})\Ind_{G^{\mathrm{wtr}}}] + \ep^2 + o_a(1).
\end{align}
The lemma now follows from \eqref{eq:92} and \eqref{eq:9470} together with Lemma~\ref{lem:ZY_red}.
\end{proof}

\begin{lemma}
\label{lem:G_Bflat}
Suppose (HB$^\flat_0$). There exists a numerical constant $\numcst > 0$, such that for large $A$ and $a$,
\[
\Pfl(G^{\flat}_1)\ge 1- \ep^{1+\numcst}.
\]
\end{lemma}
\begin{proof}
Recall the event $G_1$ from Section~\ref{sec:BBBM}. Define $G_1' = G_1 \cap \{|e^{-A}Z'_{\Theta_1}-1| \le \ep^{3/2}/2\}$. Define the random variable 
\[
X = \Pfl(Z^{\mathrm{red}}_{\Box,[\Theta_1,\Theta_1+Ka^2]} + Y^{\mathrm{red}}_{\Box,[\Theta_1,\Theta_1+Ka^2]} \le a^{-1/2}\,|\,\F_{\Theta_1}).
\]
By Proposition~\ref{prop:piece} and Remark~\ref{rem:piece}, we have $\Pfl(G_1') \ge 1-O(\ep^{1+\numcst})$, for some numerical constant $\numcst>0$. It therefore suffices to show that for large $A$ and $a$,
\begin{equation}
\label{eq:3947}
\Pfl(e^{-A}Z^{\mathrm{red}}_{\Theta_1} > \ep^{3/2}/2) + \Pfl(X < 1-\ep^2) \le \ep^{1+\numcst},
\end{equation}
By Markov's inequality and Lemma~\ref{lem:ZY_red}, together with Corollary~\ref{cor:Ga2}, we have
\begin{equation*}
  \Efl_{\mathrm{nbab}}[1-X\,|\,\F_\Delta]\Ind_{G_\Delta} \le C_\delta o_a(1).
\end{equation*}
A second application of Markov's inequality yields
\begin{equation}
  \label{eq:52}
  \Pfl(X<1-\ep^2,\ G_\Delta\cap G_{\mathrm{nbab}}) \le C_\delta o_a(1).
\end{equation}
This bounds the second summand in \eqref{eq:3947}. The first summand will be bounded by a first moment estimate, restricted to a set of high probability. Define 
\[
G_{\mathrm{red}} = \{Ka^2\le T\le\sqrt\ep a^3,\,Z^{\mathrm{red}}_{\Box,[0,Ka^2]} + Y^{\mathrm{red}}_{\Box,[0,Ka^2]} \le a^{-1/2}\}\cap G^{\mathrm{wtr}}_{I}\cap G_{\mathrm{fred}},
\] where $I = [Ka^2,T\wedge \sqrt\ep a^3]$. Conditioned on $T$ and $\F_{\mathscr L^{\mathrm{red}}_{\Box,[0,T]}}$ and on the event $G_{\mathrm{red}}$, the particles from the stopping line $\mathscr L^{\mathrm{red}}_{\Box,[0,T]}$ then all spawn BBM conditioned not to break out before $T$ (because neither the fugitive nor any in-between particles are on the stopping line). By  Lemmas~\ref{lem:hat_quantities} and~\ref{lem:ZY_red} and \eqref{eq:ep_upper}, we then have
\begin{align}
\label{eq:beaufort}
  \E[Z^{\mathrm{red}}_{T}\Ind_{G_{\mathrm{red}}}] &\le C\Big(\E[(Z^{\mathrm{red}}_{\Box,I} + AY^{\mathrm{red}}_{\Box,I})\Ind_{G^{\mathrm{wtr}}_{I}}]+Aa^{-1/2}\Big) \le C_\delta A\ep^3e^A\\
\label{eq:beaufort2}
  \E[Y^{\mathrm{red}}_{T}\Ind_{G_{\mathrm{red}}}] &\le C\Big(\E[(Z^{\mathrm{red}}_{\Box,I} + Y^{\mathrm{red}}_{\Box,I})\Ind_{G^{\mathrm{wtr}}_{I}}]+Aa^{-1/2}\Big) \le C_\delta \ep^3e^A.
\end{align}

Now define an event $\widetilde G_{\mathrm{nbab}}$ similarly to $G_{\mathrm{nbab}}$, but with the difference that red hat-, check- and fug-particles are now only required not to break out before time $\Theta_1+e^Aa^2$. Trivially, $\widetilde G_{\mathrm{nbab}}\supset G_{\mathrm{nbab}}$. Inspecting the proof of Lemma~\ref{lem:Bflat_N}, one sees that its second part is still valid if $\Pfl_{\mathrm{nbab}}$ is replaced by $\Pfltilde_{\mathrm{nbab}} = \Pfl(\cdot\,|\,\widetilde G_{\mathrm{nbab}})$ and $N_t(r)$ is replaced by $N^{\mathrm{white}}_t(r)$. As a consequence, the second part of Lemma~\ref{lem:Nwtr} and the second inequality in Lemma~\ref{lem:ZY_red} are still valid if $\Efl_{\mathrm{nbab}}$ is replaced by $\Efltilde_{\mathrm{nbab}}$. This permits to bound $Z^{\mathrm{red}}_{\Theta_1}$ in expectation under $\Pfltilde_{\mathrm{nbab}}$: Define a random line
\[
 \mathscr L' = \mathscr N_T^{\mathrm{red}} \wedge \mathscr L^{\mathrm{red}}_{\Box,J},
\]
where  $J = [T,\Theta_1+e^Aa^2]$. Then conditioned on $\F_{\mathscr L'}$ and on the event\footnote{When $a$ is large, this event prevents that $\mathscr L'$ contains in-between particles.} $G^{\mathrm{wtr}}_J$, under the law  $\Pfltilde_{\mathrm{nbab}}$ the descendants of the particles in $\mathscr L'$ follow independent BBM (with drift according to the barrier function) conditioned not to break out before $\Theta_1+e^Aa^2$. By Lemmas~\ref{lem:hat_Z_f} and~\ref{lem:rit3} and Proposition~\ref{prop:quantities}, we then have
\begin{align}
  \label{eq:95}
\Efltilde_{\mathrm{nbab}}[Z^{\mathrm{red}}_{\Theta_1}\,|\,\F_{\mathscr L'}]\Ind_{G^{\mathrm{wtr}}_J} \le C \big(Z^{\mathrm{red}}_T + Z^{\mathrm{red}}_{\Box,J}+ A(Y^{\mathrm{red}}_T + Y^{\mathrm{red}}_{\Box,J})\big).
\end{align}
Note that $G_\Delta\cap G^{\mathrm{red}} \cap G^{\mathrm{wtr}}_J \in \F_{\mathscr L'}$. Equations \eqref{eq:95}, \eqref{eq:beaufort} and \eqref{eq:beaufort2} together with the second inequality of Lemma~\ref{lem:ZY_red} then give,
\begin{equation*}
  \Efl[Z^{\mathrm{red}}_{\Theta_1}\Ind_{G_\Delta\cap \widetilde G_{\mathrm{nbab}} \cap G_{\mathrm{red}}\cap G^{\mathrm{wtr}}_J }] \le C_\delta A\ep^3 e^A,
\end{equation*}
so that by Markov's inequality and the inclusion $\widetilde G_{\mathrm{nbab}} \supset G_{\mathrm{nbab}}$,
\begin{equation}
 \label{eq:96}
 \Pfl(e^{-A}Z^{\mathrm{red}}_{\Theta_1} > \ep^{3/2}/2,\,G_\Delta\cap G_{\mathrm{nbab}} \cap G_{\mathrm{red}}\cap G^{\mathrm{wtr}}_J ) \le C_\delta A\ep^{3/2}.
\end{equation}
Finally, by the hypothesis, Corollary~\ref{cor:Ga2} and Lemmas~\ref{lem:GDelta}, \ref{lem:Gnbab} and \ref{lem:fred}, we have
\begin{equation}
\label{eq:979}
\Pfl(G_\Delta\cap G_{\mathrm{nbab}}\cap G_{\mathrm{red}}\cap G^{\mathrm{wtr}}_J ) \ge 1-\ep^{1+\numcst},
\end{equation}
for some numerical constant $\numcst > 0$.
The lemma now follows from \eqref{eq:52}, \eqref{eq:96} and \eqref{eq:979}, together with \eqref{eq:ep_upper}.
\end{proof}

\subsection{Proofs of the main results}
\label{sec:Bflat_main_results}

\begin{proof}[Proof of Lemma~\ref{lem:HBflat}]
Note that all particles are white at time $0$, so that the first three events in the definition of $G^\flat_0$ are contained in $G_0$. By Lemma~\ref{lem:Hperp}, it then remains to estimate the probability of the last event. Define the random variable $X = \Pfl(Z^{\mathrm{red}}_{\Box,[0,Ka^2]} + Y^{\mathrm{red}}_{\Box,[0,Ka^2]} \le a^{-1/2}\,|\,\F_0)$. Then,
\[
 \Efl[1-X] \le \Pfl(Z^{\mathrm{red}}_{\Box,[0,Ka^2]} + Y^{\mathrm{red}}_{\Box,[0,Ka^2]} \ge a^{-1/2},\,R_{Ka^2}=0) + \Pfl(R_{Ka^2}>0) \le C_\delta o_a(1),
\]
where the first summand is bounded by Lemma~\ref{lem:ZY_red} and Markov's inequality, and the second by Lemma~\ref{lem:Bflat_start}. Applying Markov's inequality once again yields $\Pfl(X\le 1-\ep^2)  \le C_\delta o_a(1)$, which finishes the proof.
\end{proof}

\begin{proof}[Proof of Proposition~\ref{prop:Bflat} and the \Bfl-BBM part of Theorem~\ref{th:Bflat_Bsharp}]
Let $n\ge 0$. Conditioned on $\F_{\Theta_n}$,  the barrier process until $\Theta_{n+1}$ is by definition the same in B$^\flat$-BBM and B-BBM. Furthermore, $G^\flat_n\subset G_n$ and $G^\flat_n\in\F_{\Theta_n}$. Proposition~\ref{prop:Bflat} then follows by induction from Lemma~\ref{lem:G_Bflat}, as in the beginning of the proof of Proposition~\ref{prop:piece} in Section~\ref{sec:piece_proof}.

As for the \Bfl-BBM part of Theorem~\ref{th:Bflat_Bsharp}, inspection of the proofs of Theorems~\ref{th:barrier} and \ref{th:barrier2} shows that they only rely on Proposition~\ref{prop:piece} and on the existence of the coupling with a Poisson process constructed in Section~\ref{sec:BBBM_proofs}. But this construction only relied on the law of $T_{n+1}$ conditioned on $\F_{\Theta_n}$ and $G_n$ and thus readily transfers to the \Bfl-BBM.
\end{proof}

For the proof of the first part of Proposition~\ref{prop:Bflat_Bsharp_med}, we need the following lemma. Recall that the superscript $^{(0)}$ denotes restriction to the particles from tier 0.
\begin{lemma}
\label{lem:Nred}
Suppose (HB$^\flat_0$). Let $r\in[0,9a/10]$ and $(K+1)a^2\le t\le \tconst{eq:tcrit}\wedge \sqrt\ep a^3$. Then, for all $\gamma > 0$, for large $A$ and~$a$,
  \begin{equation*}
  \P(N^{{(0)},\mathrm{red}}_{t}(r)>\gamma N,\,T>t) \le C_\delta \gamma^{-1}(1+ \mu r)e^{-\mu r}\ep \Big(\frac t {a^3} +\eta\Big)+ C\ep^2.
  \end{equation*}
\end{lemma}
\begin{proof}
Define the event 
\(G = G^{\mathrm{wtr}}_{[0,t]}\cap \{Z^{\mathrm{red}}_{\Box,[0,Ka^2]} + Y^{\mathrm{red}}_{\Box,[0,Ka^2]} \le a^{-1/2}\}.\)
Then $\P(G^c,\,T>t) \le C\ep^2$ for large $A$ and $a$, by Corollary~\ref{cor:Ga2} and the hypothesis. By Markov's inequality, it therefore suffices to show that
\begin{equation}
\label{eq:soso}
 \Ehat[N^{{(0)},\mathrm{red}}_t(r)\Ind_{G}] \le  C_{\delta}(1+\mu r)e^{-\mu r}\ep\Big(\frac t {a^3} +\eta\Big)N,
\end{equation}
where $\Ehat = \E[\cdot\,|\,T>t]$. Define a measure $\mathfrak{m}$ on $[0,a]\times [0,t]$ by
\[
\mathfrak{m}(S',S) = \Ehat[\#\{(u,s)\in\mathscr L^{(0),\mathrm{red}}_{\Box,[0,t]}: X_u(s)\in S',\,s\in S\}], 
\]
for all Borel $S'\subset[0,a],\,S\subset[0,t]$. Further set for every $x\in[0,a]$, $\mathfrak{m}_x(S) = \mathfrak{m}([x,a],S)$. By Lemma~\ref{lem:Nwtr}, we then have for $x\le 9a/10$ and every interval $I\subset[Ka^2,t]$
\begin{equation}
 \label{eq:m_r}
 \mathfrak{m}_x(I) \le C_\delta\ep\Big(\frac t {a^3} +\eta\Big)e^{-(3/2) x}N (|I|+1).
\end{equation}
Let $N^{{(0)},\mathrm{red}}_t(x,S)$ denote the number of red tier-0 particles to the right of $x$ at time $t$ which have turned red during some time $s\in S$.
If $\Ehat^{(x,s)}[N^{(0)}_{t}(r)]\le f(x,s)$ for some positive measurable function $f$, then by definition,
\begin{align}
 \label{eq:tacky}
   \Ehat[N^{{(0)},\mathrm{red}}_t(r,S)\Ind_{G}] \le \int_{S\times [0,2a/3]} f(x,s)\,\mathfrak{m}(\dd x,\dd s),
\end{align}
 for Borel $S\subset[0,t]$.
If furthermore $f(x,s)$ is jointly differentiable in $x$ and continuous in $s$, then by \eqref{eq:tacky} and integration by parts,
  \begin{align}
  \label{eq:tacky2}
  \Ehat[N^{{(0)},\mathrm{red}}_t(r,S)\Ind_{G}] 
\le \int_0^{2a/3} \int_B \frac \dd {\dd x}f(x,s)\,\mathfrak{m}_x(\dd s)\,\dd x + \int_B f(0,s)\,\mathfrak{m}_0(\dd s).
  \end{align}
Define $I_1 = [0,t-a^2]$ and $I_2 = [t-a^2,t]$. By Lemma~\ref{lem:N_expec_large_t}, \eqref{eq:250} and Corollary~\ref{cor:hat_many_to_one}, we have for $x\in[0,2a/3]$,
\begin{equation}
\label{eq:wordcrimes}
  \Ehat^{(x,s)}[N^{(0)}_{t}(r)] \le 
  \begin{cases} 
  C(1+\mu r)e^{-\mu r}N e^{-A}w_Z(x), &\tif s\in I_1\\
  Ce^{\mu x}\int_r^\infty e^{-\mu z} p^a_{t-s}(x,y),& \tif s\in I_2.
  \end{cases}
\end{equation}
By \eqref{eq:tacky2} and \eqref{eq:wordcrimes} and the inequality $w_Z'(x)\le C(1+x)e^{x-a}$, we have for large $A$ and $a$,
\begin{align}
\nonumber
\Ehat[N^{{(0)},\mathrm{red}}_t(r,I_1)\Ind_{G}]  &\le Ce^{-A}N(1+\mu r)e^{-\mu r} \Big(\int_0^{2a/3}\mathfrak{m}_x([Ka^2,t-a^2])(1+x)e^{x-a}\,\dd x + \Ehat[Z^{\mathrm{red}}_{\Box,[0,Ka^2]}\Ind_G]\Big)\\
\label{eq:prefinally}
&\le C_\delta (1+\mu r)e^{-\mu r}N \big(\ep((t/a^{3})+\eta)(t/a^{3})+o_a(1)\big),
\end{align}
the last inequality following from \eqref{eq:m_r} and the definition of $G$. 

We now further subdivide the interval $I_2$ into $I_2' = [t-a^2,t-1]$ and $I_2'' = [t-1,t]$. Writing $p_s^\infty(x,y) = (2/(\pi s))^{1/2} \exp(-(x^2+y^2)/2s)\sinh(xy/s)$ for the transition density of Brownian motion killed at $0$, we have trivially,
\begin{equation}
 \label{eq:wordcrimes2}
 \Ehat^{(x,s)}[N^{(0)}_{t}(r)] \le \begin{cases}  Ce^{\mu x}\int_r^\infty e^{-\mu y} p^\infty_{t-s}(x,y),& \tif s\in I_2'\\
 Ce^{\mu x} e^{-\mu y},& \tif s\in I_2''.
 \end{cases}
\end{equation}
By \eqref{eq:tacky2}, \eqref{eq:wordcrimes2} and Fubini's theorem, we now have
\[
 \Ehat[N^{{(0)},\mathrm{red}}_t(r,I_2')\Ind_{G}] \le C \int_r^\infty e^{-\mu y}  \int_{I_2'\times [0,2a/3]}    e^{\mu x}p^\infty_{t-s}(x,y)\,\mathfrak{m}(\dd x,\dd s)\,\dd y.\\
\]
In order to use our estimate on $\mathfrak{m}_x$, we regularise the integrand. For this, write $\overline p^\infty_s(x,y) = \max_{z\in[\lfloor x \rfloor, \lfloor x \rfloor+1]} p^\infty_s(z,y)$. By elementary calculations, using notably the unimodality of $p^\infty_s(x,y)$ in~$x$, we have $\overline p^\infty_s(x,y) \le C(p^\infty_s(x+1,y)+p^\infty_s(x-1,y) + s^{-3/2}\Ind_{(\max(x,y)<2)})$ for every $s\ge 1$ and all $x,y\ge0$, where we set $p^\infty_s(x,y) = 0$ for $x<0$. From the above inequality, we then get,
\begin{align}
\nonumber
 \Ehat[N^{{(0)},\mathrm{red}}_t(r,I_2')\Ind_{G}]&\le C \int_r^\infty e^{-\mu y} \sum_{x\in\N\cap[0,2a/3]} \int_{I_2'}  e^{\mu x} \overline p^\infty_{t-s}(x,y) (\mathfrak{m}_x(\dd s) - \mathfrak{m}_{x+1}(\dd s))\,\dd y\\
 \nonumber
 &\le C \int_r^\infty e^{-\mu y}\sum_{x\in\N\cap[0,2a/3+1]} e^{\mu x} \int_{I_2'}[p^\infty_{t-s}(x,y)+(t-s)^{-3/2}\Ind_{(\max(x,y)<2)}]\,\mathfrak{m}_{x-1}(\dd s)\,\dd y.
\end{align}
It is well-known that $\int_0^\infty p^\infty_s(x,y)\,\dd s = 2(x\wedge y) \le 2x$ (take for example $a\to\infty$ in \eqref{eq:p_potential}). One can also check that $\int_1^\infty \max_{\lfloor s\rfloor \le s'\le \lfloor s\rfloor+1} p^\infty_{s'}(x,y)\,\dd s \le Cx$. Together with the last display and \eqref{eq:m_r}, this finally gives
\begin{equation}
\label{eq:finally}
 \Ehat[N^{{(0)},\mathrm{red}}_t(r,I_2')\Ind_{G}] \le C_\delta e^{-\mu r} \ep\Big(\frac t {a^3} +\eta\Big)N.
\end{equation}

By \eqref{eq:m_r}, \eqref{eq:tacky2} and \eqref{eq:wordcrimes2}, one easily sees that \eqref{eq:finally} also  holds with $I_2'$ replaced by $I_2''$. Together with \eqref{eq:prefinally}, this yields \eqref{eq:soso} and therefore finishes the proof.
\end{proof}

\begin{proof}[Proof of the first part of Proposition~\ref{prop:Bflat_Bsharp_med}]
Define $\rho = x_{\alpha e^{2\delta}}$, so that $(1+\mu\rho)e^{-\mu\rho} = \alpha e^{2\delta}(1+o_a(1))$. 
For simplicity, we will only show that for every $t>0$, $\Pfl(\med^N_\alpha(\nu^{\flat}_{ta^3}) \ge \rho) \to 1$ as $A$ and $a$ go to infinity. The general case is a straightforward extension (and the case $t= 0$ is an easy calculation). 
Fix $t>0$ and set $n_0 := \lceil \widetilde{\ep}^{-1}(t+1)\rceil$. 
Define $\widetilde n$ to be the (random) number, such that $ta^3 \in [\Theta_{\widetilde n},\Theta_{\widetilde n+1})$. Define the good set
\[
 G = G^\flat_{n_0} \cap \{\widetilde n < n_0\} \cap \{ta^3 \in [\Theta_{\widetilde n}+e^Aa^2,T_{\widetilde n+1})\}.
\]
Then $\Pfl(G)\to 1$ as $A$ and $a$ go to infinity, since by Proposition~\ref{prop:Bflat}, we have $\Pfl(G^\flat_{n_0}) \ge 1-O((t+1)\ep^{\numcst})$ for large $A$ and $a$, and using the coupling  of the process $(\Theta_n/a^3)_{n\ge 0}$ with a Poisson process  of intensity $\widetilde{\ep}^{-1}$ from  Section~\ref{sec:BBBM_proofs}, it is easy to show that the other events in the definition of $G$ happen with high probability as well (see e.g.\ the proof of Theorem~\ref{th:barrier}). We now have
\begin{multline*}
 \Pfl(\med^N_\alpha(\nu^{\flat}_{ta^3}) \ge \rho,\,G) = \sum_{n=0}^{n_0-1} \Pfl(\med^N_\alpha(\nu^{\flat}_{ta^3}) \ge \rho,\,G,\,\widetilde n = n)\\
 \le \sum_{n=0}^{n_0-1} \Efl\Big[\Pfl(\med^N_\alpha(\nu^{\flat}_{ta^3}) \ge \rho,\, T_{n+1} > ta^3\,|\,\F_{\Theta_n})\Ind_{G^\flat_n \cap \{ta^3 - \Theta_n \in [e^Aa^2,\sqrt\ep a^3]\}}\Big].
\end{multline*}
In order to finish the proof, by \eqref{eq:ep_tilde_ep} it remains to show that for every initial configuration $\nu$ such that $G^\flat_0$ is satisfied, and for every $s\in[e^Aa^2,\sqrt\ep a^3]$, we have for large $A$ and $a$, for some $b>0$,
\begin{equation}
 \label{eq:3049}
 \P^\nu(N^{\mathrm{white}}_{s}(\rho) < \alpha N,\,T> s) = O(\ep^{1+\numcst}),
\end{equation}
where we recall from Section~\ref{sec:Bflat_num_particles} that $N^{\mathrm{white}}_{s}(\rho)$ denotes the number of white particles to the right of $\rho$ at time $s$. Write $\Phat^\nu = \P^\nu(\cdot\,|\,T>s)$. We have for large $A$ and $a$,
\begin{align*}
  &\Ehat^\nu[N^{(0)}_s(\rho)] \ge (1-O(p_B + E_{e^A}) - o_a(1))e^{2\delta} \alpha e^{-\delta} N \ge (1+\delta/2) \alpha N,\\
  &\Varhat^\nu(N^{(0)}_s(\rho)) \le C_\alpha e^{-2A} N^2,
\end{align*}
where the  first line follows from Lemma~\ref{lem:N_expec_large_t} and Corollary~\ref{cor:hat_many_to_one} together with \eqref{eq:pB} and \eqref{eq:def_E} and the second from Lemma~\ref{lem:N_2ndmoment} and Corollary \ref{cor:hat_many_to_one}. It follows that
\begin{align*}
  \P^\nu(N^{{(0)},\mathrm{white}}_{s}(\rho) < \alpha N,\,T>s) &\le \Phat^\nu(N^{(0)}_{s}(\rho) < (1+\delta/4)\alpha N) +  \P^\nu(N^{{(0)},\mathrm{red}}_{s}(\rho) \ge (\delta/4)\alpha N,\,T>s)\\
&\le C_{\delta,\alpha} \left(e^{-2A} + \ep(s/a^3+\eta) + \ep^2\right).
\end{align*}
by Chebychev's inequality applied to the first term and Lemma~\ref{lem:Nred} to the second (with $\gamma = (\delta/4)\alpha$). Since $N^{{(0)},\mathrm{white}}_{s}(\rho) \le N^{\mathrm{white}}_{s}(\rho)$, this implies \eqref{eq:3049} (using \eqref{eq:ep_lower} and \eqref{eq:eta_ep}). This finishes the proof.
\end{proof}

\section{The \texorpdfstring{\Bsh}{B\#}-BBM}
\label{sec:Bsharp}

In this section, we prove the parts of Theorem~\ref{th:Bflat_Bsharp} and Proposition~\ref{prop:Bflat_Bsharp_med} concerning the \Bsh-BBM, which was defined in Section~\ref{sec:Bsharp_definition}. As for the previous section, this section relies very much on Section~\ref{sec:BBBM} and we will use all of the notation introduced there. 

\subsection{More definitions}
\label{sec:Bsharp_more_defs}
As in the previous section, recall that we fix $\delta\in(0,1/100)$ and that the phrase ``for large $A$ and $a$'' may now depend on $\delta$. Define $K$ to be the smallest number, such that $K \ge 1$ and $E_K \le \delta/10$. Write for short $N = N^\sharp = \lfloor 2\pi  e^{A-\delta}a^{-3}e^{\mu a} \rfloor $. The symbols $C_\delta$ and $C_{\delta,\alpha}$ have the same meaning as in the last section.

For bookkeeping, we add a shade of grey to the white particles which have hit 0 at least once (and call them hence the grey particles). We then add the superscripts ``nw'', ``gr'', ``blue'' or ``tot'' to the quantities referring respectively to the non-white, grey, blue or all the particles. Quantities without this superscript refer to the white particles. This convention enables us often to use interchangeably $\Psh$, $\P$ and $\PB$ (or $\Esh$, $\E$ and $\EB$) in formulae concerning the white particles; we will do so without further mention.

We will also use the notation $N_t(r)$ and its variants from Section~\ref{sec:Bflat_num_particles} and set $N_t = N_t(0)$, $N^{(0)}_t = N^{(0)}_t(0)$ etc. In particular, $N^{\mathrm{tot}}_t$ denotes the total number of particles to the right of the origin at time $t$. We also define $B_n$ and $B^{\mathrm{tot}}_n$ to be the number of white, respectively, white and grey particles touching the left barrier during the time interval $I_n$ with less than $N$ particles to their right (i.e.\ those which are coloured blue). 

Recall that $\nu^\sharp_t$ denotes the empirical measure of all particles at time $t$ and abuse notation by setting $\nu^\sharp_n =\nu^\sharp_{\Theta_n}$. We set $G^{\sharp}_{-1}=\Omega$ and for each $n\in\N$, we define the event $G^\sharp_n$ to be the intersection of $G^\sharp_{n-1}$ with the following events:
\begin{itemize}[nolistsep]
 \item $\supp\nu^\sharp_n\subset (0,a)$,
 \item $\mathscr N^{\mathrm{tot}}_{\Theta_n} \subset U\times \{\Theta_n\}$ and $\Theta_n > T_n^+$ (for $n>0$),
 \item $|e^{-A}Z^{\mathrm{tot}}_{\Theta_n} -1| \le \ep^{3/2}$ and $Y^{\mathrm{tot}}_{\Theta_n} \le \eta$,
 \item $N^{\mathrm{tot}}_{\Theta_n} \ge N$ and for all $j\ge 0$ with $\Theta_{n-1}+t_j<\Theta_n$: $N^{\mathrm{tot}}_{\Theta_{n-1}+t_j} \ge N$,
 \item $\Psh\Big(B^{\mathrm{tot}}_{[\Theta_n,\Theta_n+t_1]} \le e^{-A} e^{\mu a}/a\,\Big|\,\F_{\Theta_n}\Big) \ge 1-\ep^2$.
\end{itemize}
The last event is of course uniquely defined up to a set of probability zero. Note that $G^\sharp_n\in\F_{\Theta_n}$ for each $n\in\N$.
 
We now state  the main results from this section, which will imply the \Bsh-BBM parts of Theorem~\ref{th:Bflat_Bsharp} and Proposition~\ref{prop:Bflat_Bsharp_med}. They are proved in Section~\ref{sec:Bsharp_main_results}. Recall the definition of (H$_\perp$) from Section~\ref{sec:BBBM_results}.
\begin{lemma}
\label{lem:HBsharp}
(H$_\perp$) implies that $\Psh(G^\sharp_0)\to 1$ as $A$ and $a$ go to infinity.
\end{lemma}

\begin{proposition}
  \label{prop:Bsharp}
Proposition~\ref{prop:piece} still holds with $G_n$, $\PB$ and $\EB$ replaced by $G_n^\sharp$, $\Psh$ and $\Esh$.
\end{proposition}

\subsection{The number of white particles: lower bounds}
\label{sec:Bsharp_num_particles}

In this section, we bound the probability that the number of white particles is less than $N$ at a given time $t$. 

\begin{lemma}
\label{lem:Nt_prob_less_N}
Let $\Phat = \P(\cdot\,|\,T>t')$ with $t'\le \tconst{eq:tcrit}$ and let $Ka^2\le t\le C_\delta a^2$. We have for $A$ and $a$ large enough,
\[
 \Phat(N^{(0)}_{t} < N\,|\,\F_0)\Ind_{(Z_0 \ge (1-\delta/2)e^A)} \le C_\delta e^{-A} (a^{-1} + e^{-A}Y_0).
\]
\end{lemma}
\begin{proof}
By Lemma~\ref{lem:N_expec_large_t}, Corollary~\ref{cor:hat_many_to_one}, \eqref{eq:pB} and the definitions of $N$ and $K$, we have for  $A$ and $a$ large enough,
\begin{equation*}
\Ehat[N^{(0)}_{t}\,|\,\F_0] \ge (1-Cp_B)\E[N^{(0)}_{t}\,|\,\F_0] \ge (1+(3/4)\delta) N e^{-A}Z_0.
\end{equation*}
By the conditional Chebychev inequality, we then have for $A$ and $a$ large enough, since $\delta\le 1/100$,
\begin{equation*}
\Phat(N^{(0)}_{t} < N\,|\,\F_0)\Ind_{(Z_0\ge (1-\delta/2)e^A)} \le C \frac{\Varhat(N^{(0)}_{t}\,|\,\F_0)}{(e^{-A} N\delta Z_0)^2}.
\end{equation*}
By Lemma~\ref{lem:N_2ndmoment} and Corollary~\ref{cor:hat_many_to_one}, $\Varhat(N^{(0)}_{t}\,|\,\F_0) \le C (e^{-A}N)^2((t/a^3)Z_0+Y_0)$. The lemma now follows from the previous inequalities and the hypothesis on $t$.
\end{proof}
\begin{corollary}
\label{cor:Nt_prob_less_N}
Let $\Phat = \P(\cdot\,|\,T>t')$ with $t'\le \tconst{eq:tcrit}$ and let $t\le t'$ and $s\in[a^2,t]$ such that $t-s \in[Ka^2,C_\delta a^2]$. 
Then, for $A$ and $a$ large enough,
\[
 \Phat(N^{(0)}_t < N,\,Z^{(0)}_s \ge (1-\delta/2)e^A\,|\,\F_0) \le C_\delta e^{-A}(1+e^{-A}Z_0)/a.
\]
\end{corollary}
\begin{proof}
 First condition on $\F_s$ and apply Lemma~\ref{lem:Nt_prob_less_N}. Then condition on $\F_0$ and apply \eqref{eq:Yt} and Corollary~\ref{cor:hat_many_to_one}.
\end{proof}

The following lemma is the analogue of Lemma~\ref{lem:Nt_prob_less_N} for the system after the breakout. For an individual $u$, we define $\widehat N^{\neg u}_t$ and $N^{\mathrm{fug},\neg u}_t$ to be the number of white hat-, respectively, fug-particles at time $t$ which are not descendants of $u$.

\begin{lemma}
  \label{lem:Nt_prob_less_N_moving_wall}
For large $A$ and $a$, we have for each $t\in[T,\Theta_1]$ and $(u,t_u)\in \widehat{\mathscr N}_{T^-}\cup \mathscr S^{(\mathscr U,T)}$,
\[
 \Psh_{\mathrm{nbab}}(\widehat N^{\neg u}_t + N^{\mathrm{fug},\neg u}_t<N\,|\,\F_\Delta)\Ind_{G_\Delta} \le C_{\delta} \eta e^{-A}/\ep.
\]
\end{lemma}
\begin{proof}
As in the proof of Lemma~\ref{lem:Bflat_N}, set $M_t = e^{- X^{[1]}_{t}} = (1+\Delta_t)^{-1}$, where $\Delta_t = \thbar((t-T^+)/a^2)(e^\Delta-1)$. Recall from that proof that on $G_\Delta$: $|e^\Delta-1 - e^{-A}Z^{(\mathscr U,T)}| \le 2\ep^{1/4}$. For a particle $(u,t_u)\in \widehat{\mathscr N}_{T^-}\cup \mathscr S^{(\mathscr U,T)}$, we define $\widehat N^{(0),\neg u}_t$ and $N^{(0),\mathrm{fug},\neg u}_t$ to be the number of hat-, respectively, fug-particles which are not descendants of $u$ and which have not hit $a$ after the time $T^-$, respectively, $T$.

As noted in the proof of Lemma~\ref{lem:G1}, we can apply the results of Lemma~\ref{lem:hat_Z_f} on $G_\Delta$, by Lemma~\ref{lem:rit3}. 
Then, by Lemma~\ref{lem:mu_t_density}, Lemma~\ref{lem:N_expec_large_t} and \eqref{eq:scheisse} for the first, and by Lemma~\ref{lem:mu_t_density}, Corollary~\ref{cor:N_expec_thbar} and the first part of Lemma~\ref{lem:ZYW} for the second inequality, we have for large $A$ and $a$,
\begin{align*}
 \Esh[\widehat N^{(0),\neg u}_t\,|\,\F_\Delta]\Ind_{G_\Delta} & \ge (1+\delta-o(1)) N e^{-A}M_t(\widehat Z_{T^-}-\|w_Z\|_\infty)\Ind_{G_\Delta}\ge N M_t(1+3\delta/4)\Ind_{G_\Delta},\tand\\
  \Esh[N^{(0),\mathrm{fug},\neg u}_t\,|\,\F_\Delta]\Ind_{{G_\Delta}} &\ge (1+\delta-o(1)) Ne^{-A}M_t (Z^{(\mathscr U,T)}-\|w_Z\|_\infty)(\thbar((t-T)/a^2) + o_a(1))\Ind_{{G_\Delta}}\\
 &\ge NM_t((1+\delta/2)\Delta_t -\delta/4)\Ind_{{G_\Delta}}.
\end{align*}
In total, this gives for large $A$ and $a$,
\begin{equation}
  \label{eq:7}
  \Esh[\widehat N^{(0),\neg u}_t + N^{(0),\mathrm{fug},\neg u}_t\,|\,\F_\Delta]\Ind_{{G_\Delta}} \ge N(1+\delta/2)\Ind_{{G_\Delta}}.
\end{equation}
Moreover, by Lemma~\ref{lem:N_2ndmoment}, \eqref{eq:ichhassedas} and independence, we have for large $A$ and $a$,
\begin{multline}
  \label{eq:3}
  \Varsh(\widehat N^{(0),\neg u}_t+N^{(0),\mathrm{fug},\neg u}_t\,|\,\F_\Delta)\Ind_{G_\Delta} \\
\le C(Ne^{-A})^2 ((\widehat Z_{T^-}+Z^{(\mathscr U,T)})(t-T^-)/a^3 +(\widehat Y^{(0)}_{T^-}+Y^{(\mathscr U,T)}))\Ind_{G_\Delta}\le CN^2 \eta e^{-A}/\ep,
\end{multline}
by the definition of $G_\Delta$. Now note that under $\Psh_{\mathrm{nbab}}$, $\widehat N^{\neg u}_t=\widehat N^{(0),\neg u}_t$ and $N^{\mathrm{fug},\neg u}_t=N^{(0),\mathrm{fug},\neg u}_t$ almost surely. Equations \eqref{eq:7} and \eqref{eq:3}, together with Lemma~\ref{lem:Gnbab} and the conditional Chebychev inequality now yield the lemma.
\end{proof}

\subsection{Bounds on the number of blue particles}

The bulk of this section will be to bound the number of blue particles. For this, first moment estimates will turn out to be enough. We start with some lemmas which bound the number of particles hitting the origin. Define for every interval $I\subset \R_+$ and for every $n\ge0$ the random variables
\begin{equation*}
L_I = \sum_{(u,t)\in\mathscr L_{H_0}}\Ind_{(t\in I)} \quad\tand\quad L_n = L_{I_n}.
\end{equation*}

\begin{lemma}
\label{lem:Bsharp_left_border}
Let $f$ be a barrier function, $0\le t\le \tconst{eq:tcrit}$ and suppose that $\|f\|$ is bounded by a function depending on $A$ only and that either $f\equiv 0$ or $t\le 2e^Aa^2$. Furthermore, let $I = [t_l,t_r] \subset [0,t]$ and $x\in[0,a]$. Then, for large $A$ and $a$, we have
\[
\E^{x}_f[L_I\,|\,T>t] \le \begin{cases}
C e^{\mu a}a^{-3}(t_r-t_l+a^2)(Aw_Y(x) + w_Z(x)), &\tif t_l \ge a^2 \tor x\ge a/2\\
C e^{\mu a} (w_Y(x) +a^{-3}(t_r-t_l+a^2)w_Z(x)), &\text{ for any $t_l$ and $x$.}
\end{cases}
\]
\end{lemma}
\begin{proof}
As in Lemma~\ref{lem:N_with_tiers}, the hypotheses allow us to use the results of Lemma~\ref{lem:hat_Z_f} without further mention.
Write $\Ehat[\cdot] = \E[\cdot\,|\,T>t]$. Define for every interval $I$, $L^{(0)}_I = \sum_{(u,s)\in\mathscr L_{H_0}}\Ind_{(H_a(X_u) > s\in I)}$. By \eqref{eq:ichhassedas} and as in the proof of Lemma~\ref{lem:mu_t_density}, we have $\Ehat^{x}_f[L^{(0)}_I] \le C e^{O(\|f\|/a)}\E[L^{(0)}_I ]$ for every $x\in[0,a]$ and for large $A$ and $a$. The hypothesis on $f$, Lemma~\ref{lem:many_to_one_simple}, Girsanov's theorem and \eqref{eq:Ia} then yield
\begin{equation}
\label{eq:left_border}
\Ehat^{x}_f[L^{(0)}_I] \le C W^x_{-\mu}[e^{\frac{1}{2}H_0}\Ind_{(H_a > H_0\in I)}] = C e^{\mu x} W^{a-x}[e^{\frac{\pi^2}{2a^2}H_a}\Ind_{(H_0 > H_a\in I)}] = C e^{\mu x} I^a(a-x,I).
\end{equation}
By  \eqref{eq:IJ_scaling} and  Lemma~\ref{lem:I_estimates} we have 
\(
 I^a(a-x,I) \le C\left(a^{-2}(t_r-t_l+a^2)\sin(\pi x/a)+\Ind_{(t_l< a^2\tand x\le a/2)}\right).
\)
Together with \eqref{eq:left_border}, this gives for large $A$ and $a$,
\begin{align}
\nonumber
\Ehat^{x}_f[L_I] &= \Ehat^{x}_f[L^{(0)}_I ] + \Ehat^{x}_f\Big[\sum_{(u,s)\in  \mathscr S^{(1+)}_{t_r}} \Ehat^{(X_u(s),s)}_f[L^{(0)}_I] \Big]\\
  \label{eq:99}
&\le Ce^{\mu a}\left(\frac {t_r-t_l+a^2} {a^3} w_Z(x) + w_Y(x)\Ind_{(t_l< a^2\tand x\le a/2)} + \frac {t_r-t_l+a^2} {a^3} \Ehat^{x}_f[Z^{(1+)}_{\emptyset}]\right).
\end{align}
The lemma now follows from \eqref{eq:99} and \eqref{eq:28}, together with the hypothesis $t_r \le \tconst{eq:tcrit} \le Ca^3/A$. 
\end{proof}

The following lemma is crucial. It will permit us to estimate the number of particles turning blue upon hitting the origin. 

\begin{lemma}
  \label{lem:blue_particles}
Suppose that $\nu_0$ is deterministic with $Z_0 \ge (1-\delta/4)e^A$.
Let $t\le\tconst{eq:tcrit}$ and $I=[t_l,t_r] \subset [Ka^2,t]$ with $t_r\le C_\delta a^2$. Write $\Phat=\P(\cdot\,|\,T>t).$ Then, for large $A$ and $a$,
\begin{equation*}
  \Ehat\Big[\sum_{(u,s)\in\mathscr L_{H_0}} \Ind_{(s\in I,\,N^{(0)}_s < N)}\Big] \le C_\delta e^{-A}(a^{-1} + e^{-A}Y_0)\Ehat[L_I].
\end{equation*}
\end{lemma}
\begin{proof}
For an individual $u\in U$, denote by $u_0$ its ancestor at time 0. If $u\in\mathscr A(0)$ and $s\ge 0$, define $N^{(0),\neg u}_s = \sum_{(v,s)\in \mathscr N^{(0)}_s}\Ind_{(v\nsucceq u)}$. By the trivial inequality $N^{(0)}_s \ge N^{(0),\neg u}_s$ for every $u$ and by the independence of the initial particles,
  \begin{align*}
  \Ehat\Big[\sum_{(u,s)\in\mathscr L_{H_0}} \Ind_{(s\in I,\,N^{(0)}_s < N)}\Big] 
  &\le \Ehat\Big[\sum_{(u,s)\in\mathscr L_{H_0}} \Ind_{(s\in I)} \Phat(N^{(0),\neg u_0}_s < N)\Big]\\
  &\le \Big(\sup_{u\in\mathscr A(0),\,s\in I}  \Phat(N^{(0),\neg u}_s < N)\Big) \Ehat[L_I].
 \end{align*}
The lemma now follows from Lemma~\ref{lem:Nt_prob_less_N}, since for every $u\in\mathscr A(0)$, we have $Z_0 - w_Z(X_u(0)) \ge Z_0 - C \ge (1-\delta/2)e^A$ for large $A$, by hypothesis.
\end{proof}

Finally, the following lemma is needed to take care of the particles which survive upon hitting the origin (i.e.\ the blue particles). We introduce the random variables $Z^{\mathrm{free}}_t$, $Y^{\mathrm{free}}_t$, $R^{\mathrm{free}}_t$ and $N^{\mathrm{free}}_t(r)$, which are defined as $Z_t$, $Y_t$, $R_t$ and $N_t(r)$ but taking into account \emph{all} particles, including those who have hit the origin. 
\begin{lemma}
\label{lem:BBM_0}
Let $t \ge 0$ and $f$ be a barrier function. Then,
\begin{equation*}
\E^0_f[Z^{\mathrm{free}}_t] \le Ca e^{\frac{\pi^2}{2a^2}t - \mu a},\quad
\E^0_f[Y^{\mathrm{free}}_t] \le Ce^{\frac{\pi^2}{2a^2}t - \mu a}\quad\tand\quad
\E^0_f[R^{\mathrm{free}}_t] \le Ce^{\frac{\pi^2}{2a^2}t - \mu a}.
\end{equation*}
\end{lemma}
\begin{proof}
Follows easily from Lemma~\ref{lem:many_to_one_simple} and Girsanov's theorem, together with the inequality $Z_t\le a Y_t$ for the first inequality.
\end{proof}

\paragraph{Creation of blue particles: before the breakout.}

Recall the definitions from Section~\ref{sec:Bsharp_more_defs}, in particular of the random variables $B_n$ and $B^{\mathrm{tot}}_n$. The goal in this section is to bound the first moment of these random variables under the (sub-probability) measure $\P_{[n]} := \P(\cdot,\ T>t_{n+1})$, $n\in\N$. Since $\P_{[n]}$ is a subprobability measure, we introduce the notation $\E_{[n]}[\cdot\,|\,\F] = \E[\,\cdot\,\Ind_{(T>t_{n+1})}\,|\,\F]$ for a sigma-field $\F$. Furthermore, we define the good events
\begin{equation*}
G_{Z,t} = \left\{\sup_{0\le s\le t} |Z^{(0)}_s - e^A| < (\delta/4)e^A \right\}\quad\text{ and }\quad G_{Z,n} = G_{Z,t_n}.
\end{equation*}
\begin{lemma}
\label{lem:B_hat}
For every $n\ge 1$ with $t_{n+1}\le \tconst{eq:tcrit}$, we have for large $A$ and $a$,
 \[
\E_{[n]}[B_n\Ind_{G_{Z,n}\cap\{Y_0\le 1\}}] \le C_\delta \frac{e^{\mu a}}{a^2} \le C_\delta e^{-A} a N.
 \]
\end{lemma}
\begin{proof}
For simplicity, assume that $\P(G_{Z,0}\cap\{Y_0\le 1\}) = 1$, the general case is a simple adaptation.
Let $n\ge 1$ and set  $t'=t_n-(K+1)a^2 = t_{n-1}+2a^2$.
In this proof, if a quantity has a superscript~$^{(0)}$, it refers to the particles which have not hit $a$ before time $t'$ (i.e. to the descendants of $\mathscr N_{t'}^{(0)}$), and if it has a superscript~$^{(1+)}$ it refers to the remaining particles (i.e. to the descendants of $\mathscr R_{t'}^{(0)}$).
By Lemmas~\ref{lem:Bsharp_left_border} and \ref{lem:blue_particles} and the inequality $Z_{t'}^{(0)}\le aY_{t'}^{(0)}$,
\begin{equation*}
  \begin{split}
 \E_{[n]}[B^{(0)}_n\Ind_{G_{Z,n}}\,|\,\F_{t'}] &\le  C_\delta e^{ - A} \Big(\frac 1 a + e^{-A}Y_{t'}^{(0)}\Big)\Ind_{G_{Z,t'}} \E_{[n]}[L_n^{(0)}\,|\,\F_{t'}]\\
&\le C_\delta e^{\mu a - A} \Big(\frac 1 a Y_{t'}^{(0)} + e^{-A}(Y_{t'}^{(0)})^2\Big)\Ind_{G_{Z,t'}},
  \end{split}
\end{equation*}
Proposition~\ref{prop:quantities} now gives, since $Z_0\le 2e^A$ on $G_{Z,0}$ and $t' = t_{n-1}+2a^2 \ge 2a^2$,
\begin{align}
  \nonumber
  \E_{[n]}[B^{(0)}_n\Ind_{G_{Z,n}}] &\le C_\delta e^{\mu a - A}  \Big(\frac 1 a \E[Y_{t'}^{(0)}] + e^{-A}\E\big[\Var(Y_{t'}^{(0)}\,|\,\F_{a^2}) + (\E[Y_{t'}^{(0)}\,|\,\F_{a^2}])^2\Ind_{G_{Z,a^2}}\big]\Big)\\
  \label{eq:132}
  &\le C_\delta e^{\mu a} \Big(a^{-2} + e^{-2A}\big(\E[a^{-1}Y_{a^2}^{(0)}] + a^{-2}e^{2A}\big)\Big) \le C_\delta e^{\mu a} /a^2.
\end{align}
As for the remaining particles, by the strong branching property, with $\Phat_{[n]} = \P(\cdot\,|\,T>t_{n+1})$,
\begin{equation*}
\begin{split}
   \E_{[n]}[B^{(1+)}_n\Ind_{G_{Z,n}}\,|\,\F_{\mathscr L_{H_a\wedge t'}}] &= \E_{[n]}\Big[\sum_{(u,s)\in\mathscr R^{(0)}_{t'}} \sum_{(v,r)\in\mathscr L_{H_0},\,v\succeq u}\Ind_{(r\in I_n,\,N^{(0)}_r< N)}\Ind_{G_{Z,n}}\,\Big|\,\F_{\mathscr L_{H_a\wedge t'}}\Big]\\
   &\le \Big[\sup_{r\in I_n}\Phat_{[n]}(N^{(0)}_r< N,\,G_{Z,n}\,|\,\F_{\mathscr L_{H_a\wedge t'}})\Big] \sum_{(u,s)\in\mathscr R^{(0)}_{t'}} \E_{[n]}^{(a,s)}[L_{I_n}]\\
&\le C_\delta e^{-A}(e^{\mu a}/a^2) R^{(0)}_{t'} \Q^a[Z+AY],
\end{split}
\end{equation*}
the last inequality following from Lemma~\ref{lem:Bsharp_left_border} together with Corollary~\ref{cor:Nt_prob_less_N} and the fact that $Z_{t'} \le 2e^A$ on $G_{Z,t'}$. By the first part of Lemma~\ref{lem:ZYW}, \eqref{eq:QaZC} and \eqref{eq:eta}, we have $\Q^a[Z+AY]\le CA$. This gives
\begin{equation}
  \label{eq:134}
  \E_{[n]}[B^{(1+)}_n\Ind_{G_{Z,n}}] \le C_\delta (e^{\mu a}/a^2) e^{-A} A \E[R^{(0)}_{t'}] \le C_\delta e^{\mu a}/a^2,
\end{equation}
the last inequality following from Lemma~\ref{lem:Rt} and the hypotheses on $Z_0$ and $Y_0$. The lemma now follows from \eqref{eq:132} and \eqref{eq:134}.
\end{proof}

Up to now, we have only considered the \emph{white} particles turning blue. In order to estimate $B_n^{\mathrm{tot}}$, we will use an inductive argument. For $n\ge 0$ and $0\le k \le n-2$, define $B_{k,n}^{\mathrm{tot}}$ to be the number of (white or grey) particles that turn blue at a time $t\in I_n$, have an ancestor that turned blue at a time $t'\in I_k$ and have none that has hit 0 between $t_{k+2}$ and $t_n$.

Let $G_{\mathrm{b},n}$ be the event that no blue particle hits $a$ before $t_n$ and $G_{\mathrm{g},n}$ the event that there exists no grey particle that breaks out before $\tconst{eq:tcrit}$ and has been grey all the time between $t_n$ and the time of its breakout. Then set 
\[
G_n^{\mathrm{tot}} = G_{Z,n}\cap G_{\mathrm{b},n} \cap G_{\mathrm{g},n}\cap 
\{B_0 \le e^Ae^{\mu a}/a^2,\,Y_0\le \eta\}.
\]

\begin{lemma}
 \label{lem:B_kn}
For every $n\ge 2$ with $t_{n+1}\le \tconst{eq:tcrit}$ and $k\le n-2$, we have for large $A$ and $a$,
\[
 \E_{[n]}[B_{k,n}^{\mathrm{tot}}\Ind_{G_n^{\mathrm{tot}}}] \le C_{\delta} \E_{[k]}[B^{\mathrm{tot}}_k\Ind_{G_k^{\mathrm{tot}}}] \frac{e^{-A}}{a}.
\]
\end{lemma}
\begin{proof}
Let $\mathscr B_k$ be the stopping line consisting of the particles that turn blue during $I_k$, at the moment at which they turn blue (hence, $B^{\mathrm{tot}}_k = \#\mathscr B_k$). Note that $\mathscr B_k\wedge \mathscr N^{(0)}_{t_{n-1}} = \mathscr B_k \cup \mathscr N^{(0)}_{t_{n-1}}$, so that the descendants of $\mathscr B_k$ and of $\mathscr N^{(0)}_{t_{n-1}}$ are independent conditioned on $\F_{\mathscr B_k\wedge \mathscr N^{(0)}_{t_{n-1}}}$, by the strong branching property. Note also that $G_k^{\mathrm{tot}}\in\F_{\mathscr B_k\wedge \mathscr N^{(0)}_{t_{n-1}}}$.
It follows that with $\Phat_{[n]} = \P(\cdot\,|\,T > t_{n+1})$,
\begin{multline}
 \label{eq:103}
 \E_{[n]}[B_{k,n}^{\mathrm{tot}}\Ind_{G_n^{\mathrm{tot}}}\,|\,\F_{\mathscr B_k\wedge \mathscr N^{(0)}_{t_{n-1}}}] 
\le \sup_{t\in I_n}\Big(\Phat_{[n]}(N^{(0)}_t < N,\,G_{Z,n}\,|\,\F_{\mathscr N^{(0)}_{t_{n-1}}}) \Ind_{G_k^{\mathrm{tot}}\cap G_{Z,n-1}}\Ind_{(T>t_{n-1})}\Big)\\
\times \sum_{(u,s)\in\mathscr B_k}\E_{[n]}^{(0,s)}\Big[\sum_{(v,t)\in\mathscr L_{H_0^{k+2}}}\Ind_{(t\in I_n,\,H_a(X_v) > t_{k+2},\,X_v(t_{k+2})\in (0,a))}\Big],
\end{multline}
where here $H_0^{k+2}(X) = \inf\{t\ge t_{k+2}: X_t = 0\}$. By Lemma~\ref{lem:Bsharp_left_border} and the inequality $w_Z(x) \le aw_Y(x)$, each summand in the sum on the right-hand side of the above inequality is bounded by $C e^{\mu a} \E^0[Y^{\mathrm{free}}_{t_{k+2}-s}]$ (with the notation of Lemma~\ref{lem:BBM_0}), which by Lemma~\ref{lem:BBM_0} is bounded by $Ce^{CK}$ for $s\in I_k$. Furthermore, by Corollary~\ref{cor:Nt_prob_less_N} and the fact that $Z^{(0)}_{t_{n-1}} \le 2e^A$ on $G_{Z,n}$, the supremum in \eqref{eq:103} is bounded by $C_{\delta} e^{-A}/a$.
The lemma follows by taking expectation on both sides of \eqref{eq:103} and using the fact that $t_{k+1} \le t_{n-1}$ by the hypothesis on $k$.
\end{proof}
\begin{lemma}
\label{lem:B_n}
For all $n\ge 1$ with $t_{n+1}\le \tconst{eq:tcrit}$, we have $\E_{[n]}[B^{\mathrm{tot}}_n\Ind_{G_n^{\mathrm{tot}}}] \le C_\delta e^{\mu a}/a^2$ for large $A$ and $a$.
\end{lemma}
\begin{proof}
By Lemma~\ref{lem:B_hat}, the statement is true for $n=1$, because $B^{\mathrm{tot}}_1 = B_1$ by definition. Now we have for every $n$, by Lemmas~\ref{lem:B_hat} and \ref{lem:B_kn} and the fact that $B^{\mathrm{tot}}_0 = B_0 \le e^{A}e^{\mu a}/a^2$ on $G_n^{\mathrm{tot}}$, 
\begin{equation}
 \label{eq:840}
\E_{[n]}[B^{\mathrm{tot}}_n\Ind_{G_n^{\mathrm{tot}}}] = \sum_{k=0}^{n-2} \E_{[n]}[B_{k,n}^{\mathrm{tot}}\Ind_{G_n^{\mathrm{tot}}}] + \E_{[n]}[B_n\Ind_{G_n^{\mathrm{tot}}}] \le C_\delta\Big(\frac{e^{-A}}{a}\sum_{k=1}^{n-2} \E_{[k]}[B^{\mathrm{tot}}_k\Ind_{G_k^{\mathrm{tot}}}] + \frac{e^{\mu a}}{a^2}\Big).
\end{equation}
The lemma now follows easily by induction over $n$, since $n\le a$ by hypothesis.
\end{proof}

Define the random variable $n_T:= \lfloor T/t_1\rfloor$ and note that $t_{n_T-1}>T^-$ for large $A$. Define the events $G_{n_T}^{\mathrm{tot}} = \bigcup_n (G_n^{\mathrm{tot}}\cap \{n_T=n\})$ and $G_{\mathrm{nw}} = \{Z^{\mathrm{gr}}_{t_{n_T}}\le e^{A/2},\,Y^{\mathrm{gr}}_{t_{n_T}} \le e^{A/2}/a,\,B^{\mathrm{gr}}_{n_T-1}\le \eta e^{\mu a}/a\}$. The next two lemmas show that these events happen with high probability.

\begin{lemma}
\label{lem:G_n0_tot}
 Suppose $\P(G^\sharp_0) = 1$. For large $A$ and $a$, we have $\P(G_{n_T}^{\mathrm{tot}}) \ge 1-C_\delta \ep^2$.
\end{lemma}
\begin{proof}
Let $n\ge 1$. By definition, the event $G_{n-1}^{\mathrm{tot}}\backslash G_{\mathrm{b},n}$ implies that a particle which turned blue during $I_{n-2}$ or $I_{n-1}$ has a descendant which hits $a$ before $t_n$. Markov's inequality and Lemmas~\ref{lem:BBM_0} and~\ref{lem:B_n} then imply,
\begin{equation}
  \label{eq:90}
  \P_{[n]}(G_{n-1}^{\mathrm{tot}}\backslash G_{\mathrm{b},n}) \le \E_{[n]}[(B^{\mathrm{tot}}_{n-2}+B^{\mathrm{tot}}_{n-1})\Ind_{G_{n-1}^{\mathrm{tot}}}] \times \sup_{t\le t_2} \E^0[R^{\mathrm{free}}_t]\le C_\delta a^{-2},
\end{equation}
Furthermore, the event $(G_{n-1}^{\mathrm{tot}}\cap G_{\mathrm{b},n}) \backslash G_{\mathrm{g},n}$ implies that there exists a grey particle which breaks out for the first time before $\tconst{eq:tcrit}$, has not hit the origin between time $t_n$ and the time of its breakout, and which has an ancestor that turned blue during $I_{n-2}$. Proposition~\ref{prop:T} and Lemmas~\ref{lem:BBM_0} and \ref{lem:B_n}, then give,
\begin{multline}
  \label{eq:106}
  \P_{[n]}((G_{n-1}^{\mathrm{tot}}\cap G_{\mathrm{b},n}) \backslash G_{\mathrm{g},n}) \le \E_{[n]}[B^{\mathrm{tot}}_{n-2}\Ind_{G_{n-1}^{\mathrm{tot}}}]\times C\sup_{t\le t_2}p_B\E^0[ \tfrac{\tconst{eq:tcrit}}{a^3}Z^{\mathrm{free}}_t+Y^{\mathrm{free}}_t]\\
  \le C_\delta a^{-2} p_B (a+1) \le C_\delta e^{-2A/3} a^{-1},
\end{multline}
where the last inequality follows from \eqref{eq:pB} and \eqref{eq:ep_lower}. 
Equations \eqref{eq:90} and \eqref{eq:106} now give for every $n\ge 1$, for large $A$ and $a$,
\[
 \P_{[n]}(G_{Z,n} \cap G_{n-1}^{\mathrm{tot}} \backslash G_n^{\mathrm{tot}}) \le C_\delta e^{-2A/3}a^{-1}.
\]
Define $G_{Z,n_T} = \bigcup_n (G_{Z,n}\cap \{n_T=n\})$ and set  $n_0 = \lfloor a^3/t_1\rfloor = O(a)$. Then,
\begin{multline*}
 \P(G_{Z,n_T}\cap G_0^{\mathrm{tot}} \cap \{T\le  a^3\} \backslash G_{n_T}^{\mathrm{tot}}) = \sum_{n=1}^{n_0} \P(\{n_T = n\} \cap G_{Z,n}\cap G_0^{\mathrm{tot}} \backslash G_n^{\mathrm{tot}})\\
 = \sum_{n=1}^{n_0}\sum_{k=1}^{n} \P(\{n_T=n\}\cap G_{Z,n}\cap G_{k-1}^{\mathrm{tot}}\backslash G_k^{\mathrm{tot}})
 \le \sum_{k=1}^{n_0} \P(\{n_T \ge k\} \cap G_{Z,k}\cap G_{k-1}^{\mathrm{tot}}\backslash G_k^{\mathrm{tot}}),
\end{multline*}
where we used the fact that $(G_{Z,n})_{n\ge0}$ and $(G_n^{\mathrm{tot}})_{n\ge0}$ are decreasing sequences.
Now note that $n_T \ge n$ implies $T > t_n$. The last two equations then give
\begin{equation*}
 \P(G_{Z,n_T}\cap G_0^{\mathrm{tot}} \cap \{T\le  a^3\} \backslash G_{n_T}^{\mathrm{tot}}) \le C_\delta e^{-2A/3}.
\end{equation*}
By the hypothesis and Proposition~\ref{prop:T}, we now have for large $A$ and $a$,
\[
 \P(G_0^{\mathrm{tot}} \cap \{T\le  a^3\}) \ge 1-2\ep^2.
\]
It therefore remains to bound the probability of $G_{Z,n_T}^c\cap\{T\le a^3\}$. Since $(G_{Z,n}^c)_{n\ge0}$ is a decreasing sequence of events, we have
\(
 \P(G_{Z,n_T}^c\cap\{T\le a^3\}) \le \P(G_{Z,n_0}^c),
\)
We first note that under $G_0^\sharp$, $|Z_0-e^A| \le \delta e^A/8$ for large $A$ and $a$. Since $(Z^{(0)}_t)_{t\ge0}$ is a martingale under $\P$ (see Section~\ref{sec:interval}), we then have by Doob's $L^2$-inequality,
\[
 \P(G_{Z,n_0}^c\,|\,\F_0) \le \P(\sup_{t\in[0,t_{n_0}]} |Z^{(0)}_t - Z_0| > \delta e^A/8\,|\,\F_0) \le C \frac{\Var(Z^{(0)}_{t_{n_0}}\,|\,\F_0)}{(\delta e^A)^2} \le C_\delta e^{-A},
\]
where the last inequality follows from Proposition~\ref{prop:quantities}.
Together with the previous inequalities and \eqref{eq:ep_lower}, the statement follows.
\end{proof}

\begin{lemma}
  \label{lem:G_ex} Suppose $\P(G^\sharp_0) = 1$. For large $A$ and $a$, $\P(G_{n_T}^{\mathrm{tot}}\cap G_{\mathrm{nw}}) \ge 1- C_\delta \ep^2.$
\end{lemma}
\begin{proof}
By Lemma~\ref{lem:G_n0_tot} and Proposition~\ref{prop:T}, it remains to show that
$\P((G_{n_T}^{\mathrm{tot}}\cap \{T\le  a^3\})\backslash G_{\mathrm{nw}}) \le C_\delta \ep^2.$
The event $\{T\le  a^3\}$ is $\F_{\mathscr L_{H_0}}$-measurable, since by definition, $T$ is the time of the first breakout of a particle which has not hit the origin yet. Lemmas~\ref{lem:hat_quantities}, \ref{lem:BBM_0} and \ref{lem:B_n} then give
\begin{equation}
  \label{eq:2}
  \begin{split}
  \E[Z^{\mathrm{gr}}_{t_{n_T}} \Ind_{G_{n_T}^{\mathrm{tot}}\cap \{T\le  a^3\}}]&\le C\sup_{t\le t_2}\E^0[Z^{\mathrm{free}}_t + AY^{\mathrm{free}}_t]  \E\Big[\sum_{n=0}^{n_T-2}\E_{[n]}[B^{\mathrm{tot}}_n\Ind_{G_n^{\mathrm{tot}}}\,|\,\F_{\mathscr L_{H_0}}]\Ind_{(T\le  a^3)}\Big]\\
&\le C_\delta a e^{\mu a} \sum_{n=0}^{\lfloor a^3/t_1\rfloor} \E_{[n]}[B^{\mathrm{tot}}_n\Ind_{G_n^{\mathrm{tot}}}]\le C_\delta.
  \end{split}
\end{equation}
Similarly, we have
\begin{align}
  \label{eq:36}
\E[Y^{\mathrm{gr}}_{t_{n_T}}\Ind_{(R^{\mathrm{gr}}_{I_{n_T-1}} = 0)}\Ind_{G_{n_T}^{\mathrm{tot}}\cap \{T\le  a^3\}}] \le C_\delta/a\tand \P(R^{\mathrm{gr}}_{I_{n_T-1}} > 0,\ G_{n_T}^{\mathrm{tot}},\ T\le  a^3) \le C_\delta \eta.
\end{align}
Finally, we have as in \eqref{eq:2}, by Lemmas~\ref{lem:B_kn} and \ref{lem:B_n},
\begin{align}
\nonumber
   \E[B^{\mathrm{gr}}_{n_T-1} \Ind_{G_{n_T-1}^{\mathrm{tot}}\cap \{T\le  a^3\}}] &= \E\Big[\sum_{k=0}^{n_T-2}\E[B_{k,{n_T-1}}^{\mathrm{tot}}\Ind_{G_{n_T-1}^{\mathrm{tot}}}\,|\,\F_{\mathscr L_{H_0}}]\Ind_{(T\le  a^3)}\Big] \\
    \label{eq:38}
&\le C_\delta (e^{-A}/a) \sum_{k=0}^{\lfloor  a^3/t_1\rfloor} \E_{[k]}[B^{\mathrm{tot}}_k\Ind_{G_k^{\mathrm{tot}}}] \le C_\delta e^{-A+\mu a}/a^2.
\end{align}
The lemma now follows from \eqref{eq:2}, \eqref{eq:36} and \eqref{eq:38}, together with Markov's inequality and \eqref{eq:ep_lower}.
\end{proof}

\paragraph{Creation of blue particles: after the breakout.}

We study now the system after the breakout. Recall that $n_T = \lfloor T/t_1\rfloor$ by definition and define $n_\Theta := \lfloor \Theta_1/t_1 +
 1\rfloor$. 
We define $\F^\sharp_\Delta = \F_\Delta\cap \F_{\mathscr N^{\mathrm{gr}}_{t_{n_T}}\wedge \mathscr B^{\mathrm{gr}}_{n_T-1}}$, where, analogous to the proof of Lemma~\ref{lem:B_kn}, $\mathscr B^{\mathrm{gr}}_{n_T-1}$ denotes the stopping line of the grey particles turning blue during the interval $I_{n_T-1}$. We then define the event $G^\sharp_\Delta = G_\Delta\cap G_{\mathrm{nw}}\in\F^\sharp_\Delta$. We denote by the superscript ``gr$\preceq$" the quantities relative to particles descending from those that were grey before or at time $t_{n_T}$. Then let $G^\sharp_{\mathrm{nbab}}$ be the intersection of $G_{\mathrm{nbab}}$ with the event that none of these particles hits $a$ between $t_{n_T}$ and $\Theta_1+e^Aa^2$ before hitting 0. Define then the (sub-probability) measure $\P^\sharp_{\mathrm{nbab}} = \P(\cdot,\,G^\sharp_{\mathrm{nbab}})$.

\begin{lemma}
\label{lem:Gsharp_Delta_nbab}
For large $A$ and $a$, $\Psh(G^\sharp_\Delta)\ge 1-C\ep^{5/4}$ and $\Psh((G^\sharp_{\mathrm{nbab}})^c\,|\,\F^\sharp_\Delta)\Ind_{G^\sharp_\Delta} \le C\ep^2$.
\end{lemma}
\begin{proof}
The first inequality follows from Lemmas~\ref{lem:GDelta} and \ref{lem:G_ex}. For the second statement, we have by Lemmas~\ref{lem:Rt} and \ref{lem:R_f} and the definition of $G_{\mathrm{nw}}$,
\begin{equation*}
  \Psh(R^{\mathrm{gr}\preceq}_{[t_{n_T},\Theta_1+e^Aa^2]} > 0\,|\,\F^\sharp_\Delta)\Ind_{G^\sharp_\Delta} \le C((e^A/a)Z^{\mathrm{gr}}_{t_{n_T}}+Y^{\mathrm{gr}}_{t_{n_T}})\Ind_{G^\sharp_\Delta} \le Ce^{3A/2}/a.
\end{equation*}
Together with Lemma~\ref{lem:Gnbab}, this implies the lemma.
\end{proof}

\begin{lemma}
 For large $A$ and $a$ and $n\ge n_T-1$, 
\(
\E^\sharp_{\mathrm{nbab}}[B_n+B_n^{\mathrm{gr}\preceq}\,|\,\F^\sharp_\Delta]\Ind_{G^\sharp_\Delta} \le C_{\delta}e^{-5A/3}{e^{\mu a}}/{a}.
\)
\end{lemma}
\begin{proof}
Conditioned on $\F^\sharp_\Delta$, let $n\ge n_T-1$. As in Lemma~\ref{lem:blue_particles}, we have for $t\in I_n$, with $f_{\Delta}^+ = f_\Delta(\cdot - T^+/a^2)$,
\begin{equation*}
  \begin{split}
  \E^\sharp_{\mathrm{nbab}}[\widehat B_n + B^{\mathrm{fug}}_n + \widecheck B_n\,|\,\F^\sharp_\Delta]\Ind_{G^\sharp_\Delta} &\le C\sum_{(u,t)\in \mathscr L_\Delta} \Ehat_{f_{\Delta}^+}^{(X_u(t),t)}[L_n]\sup_{s\in I_n}\P^\sharp_{\mathrm{nbab}}(\widehat N^{\neg u}_s + N^{\mathrm{fug},\neg u}_s < N\,|\,\F_\Delta)\Ind_{G^\sharp_\Delta}\\
&\le C_\delta \eta (\ep e^{A})^{-1} \sum_{(u,t)\in \mathscr L_\Delta}\Ehat_{f_{\Delta}^+}^{(X_u(t),t)}[L_n]\Ind_{G^\sharp_\Delta}
  \end{split}
\end{equation*}
the last inequality following from Lemma~\ref{lem:Nt_prob_less_N_moving_wall}. By the first inequality in Lemma~\ref{lem:Bsharp_left_border} and the definition of $G_\Delta$, this gives for large $A$ and $a$,
\begin{equation*}
  \E^\sharp_{\mathrm{nbab}}[\widehat B_n  + B^{\mathrm{fug}}_n+ \widecheck B_n\,|\,\F^\sharp_\Delta]\Ind_{G^\sharp_\Delta} \le C_\delta e^{\mu a}\eta (\ep e^Aa)^{-1}(Z_\Delta + A Y_\Delta) \Ind_{G^\sharp_\Delta}
\le C_\delta e^{\mu a}\eta\ep^{-2}/a.
\end{equation*}
The reasoning is similar for the bar-particles, using the first inequality of Lemma~\ref{lem:Bsharp_left_border} for those $(x,s)\in\mathscr L_{\mathscr U}$ with $x\ge a/2$ and the second inequality together with the bound  $w_Y(x) \le a^{-1}e^{-(a-x)/2}$ valid for $x\le a/2$ and large $a$, to get
\begin{equation*}
 \E^\sharp_{\mathrm{nbab}}[\widebar B_n\,|\,\F^\sharp_\Delta]\Ind_{G^\sharp_\Delta} \le C_\delta e^{\mu a}\eta (\ep e^{A}a)^{-1}(A \mathscr E_{\mathscr U})\Ind_{G^\sharp_\Delta}
 \le C_\delta e^{\mu a} \eta /a,
\end{equation*}
by \eqref{eq:ep_upper} and the definition of $G_{\mathscr U}\supset G_\Delta$. 
As for $B^{\mathrm{gr}\preceq}_n$, if $n=n_T-1$ then $B^{\mathrm{gr}\preceq}_n \le \eta e^{\mu a}/a$ on $G_{\mathrm{nw}} \supset G^\sharp_\Delta$. Furthermore, if $n\ge n_T$, then by the second inequality of Lemma~\ref{lem:Bsharp_left_border},
\begin{equation*}
  \E^\sharp_{\mathrm{nbab}}[B^{\mathrm{gr}\preceq}_n\,|\,\F^\sharp_\Delta]\Ind_{G^\sharp_\Delta} \le C_\delta e^{\mu a}\eta (\ep e^{A})^{-1}(e^A Z^{\mathrm{gr}}_{t_{n_T}}/a + Y^{\mathrm{gr}}_{t_{n_T}})\Ind_{G^\sharp_\Delta} \le C_\delta e^{\mu a} \eta \ep^{-1} e^{-A/2}/a,
\end{equation*}
by the definition of $G_{\mathrm{nw}}$.
The lemma now follows from the above,
together with \eqref{eq:eta} and \eqref{eq:ep_lower}.
\end{proof}

For $n\ge n_T-1$, define $\widetilde G^{\mathrm{tot}}_n$ to be the event that no descendant of a particle which has been coloured blue between $t_{n_T-1}$ and  $t_n$ hits $a$ before $\Theta_1+e^Aa^2$.
\begin{lemma}
\label{lem:B_n_moving_wall}
 We have for every $n\ge n_T-1$ and for large $A$ and $a$,
\[
\E^\sharp_{\mathrm{nbab}}[B^{\mathrm{tot}}_n\Ind_{\widetilde G^{\mathrm{tot}}_n}\,|\,\F^\sharp_\Delta]\Ind_{G^\sharp_\Delta} \le C_{\delta}e^{-5A/3}\frac{ e^{\mu a}}{a}. 
\]
\end{lemma}
\begin{proof}
  The proof is similar to the proof of Lemma~\ref{lem:B_n}: As in the proof of Lemma~\ref{lem:B_kn}, with Lemma~\ref{lem:Nt_prob_less_N_moving_wall} instead of Corollary~\ref{cor:Nt_prob_less_N}, one first shows that for every $n_T-1 \le k\le n-2$, one has
  \begin{equation*}
    \E^\sharp_{\mathrm{nbab}}[B_{k,n}^{\mathrm{tot}}\Ind_{\widetilde G^{\mathrm{tot}}_n}\,|\,\F_{\mathscr B_k}]\Ind_{G^\sharp_\Delta} \le C_\delta \eta \ep^{-1} e^{-A} \E^\sharp_{\mathrm{nbab}}[B^{\mathrm{tot}}_k\Ind_{\widetilde G^{\mathrm{tot}}_k}\,|\,\F^\sharp_\Delta]\Ind_{G^\sharp_\Delta}.
  \end{equation*}
By a recurrence similar to \eqref{eq:840}, using $B_n+B_n^{\mathrm{gr}\preceq}$ instead of $B_n$, this yields the lemma.
\end{proof}

\subsection{The probability of \texorpdfstring{$G^\sharp_1$}{G\#1}}

\begin{lemma}
  \label{lem:G_n_prime}
 For large $A$ and $a$, $\P^\sharp_{\mathrm{nbab}}((\widetilde G^{\mathrm{tot}}_{n_\Theta})^c\,|\,\F^\sharp_\Delta)\Ind_{G^\sharp_\Delta} \le C_\delta/a.$
\end{lemma}
\begin{proof}
  Similarly to the proof of \eqref{eq:90} and \eqref{eq:106} in the proof of Lemma~\ref{lem:G_n0_tot}, but using Lemmas~\ref{lem:Rt} and \ref{lem:R_f} instead of Proposition~\ref{prop:T}, we have for large $A$ and $a$,
\[
\P^\sharp_{\mathrm{nbab}} (\widetilde G^{\mathrm{tot}}_{n-1}\backslash \widetilde G^{\mathrm{tot}}_n\,|\,\F^\sharp_\Delta)\Ind_{G^\sharp_\Delta} \le C_\delta \E^\sharp_{\mathrm{nbab}}[B^{\mathrm{tot}}_{n-2}\Ind_{\widetilde G^{\mathrm{tot}}_{n-1}}\,|\,\F^\sharp_\Delta]\Ind_{G^\sharp_\Delta} e^{A-\mu a}/a \le C_\delta a^{-2},
\]
where the last inequality follows from Lemma~\ref{lem:B_n_moving_wall}. The lemma now follows by induction over $n$.
\end{proof}

\begin{lemma}
\label{lem:G_sharp}
Suppose $\Psh(G^\sharp_0) = 1$. Then $\Psh(G^\sharp_1) \ge 1-C\ep^{5/4}$  for large $A$ and $a$.
\end{lemma}
\begin{proof}
By Corollary~\ref{cor:Nt_prob_less_N} and the fact that $T\le a^3$ on $G_{\mathscr U}$, we have by a union bound,
\begin{equation}
  \label{eq:113}
\Psh(\exists n \le n_T:N^{\mathrm{tot}}_{t_n} < N,\ G_{Z,n_T},\ G_{\mathscr U}) \le \sum_{n=1}^{\lfloor a \rfloor} \P_{[n]}(N^{(0)}_{t_n} < N,\,G_{Z,n}) \le C_{\delta} e^{-A}.
\end{equation}
Furthermore, by Lemma~\ref{lem:Nt_prob_less_N_moving_wall}, we have
\begin{equation}
  \label{eq:9}
\Psh(\exists n_T < n < n_\Theta:N^{\mathrm{tot}}_{t_n} < N\tor N^{\mathrm{tot}}_{\Theta_1} < N,\ G_\Delta) \le C_{\delta} \eta/\ep \le C_\delta e^{-A},
\end{equation}
by \eqref{eq:eta} and \eqref{eq:ep_lower}.
Now, by Proposition~\ref{prop:quantities} and Lemmas~\ref{lem:BBM_0} and \ref{lem:Gsharp_Delta_nbab},
\begin{equation}
  \label{eq:4}
  \begin{split}
 \E^\sharp_{\mathrm{nbab}}[Z^{\mathrm{nw}}_{\Theta_1}\Ind_{\widetilde G^{\mathrm{tot}}_{n_\Theta}}\,|\,\F^\sharp_\Delta]\Ind_{G^\sharp_\Delta} &\le C\Big(Z^{\mathrm{gr}}_{t_{n_T}}+ C_\delta a e^{\mu a}\sum_{n=n_T-1}^{n_\Theta-1}  \E^\sharp_{\mathrm{nbab}}[B^{\mathrm{tot}}_n\Ind_{\widetilde G^{\mathrm{tot}}_n}\,|\,\F^\sharp_\Delta]\Big)\Ind_{G^\sharp_\Delta}\\
&\le Ce^{A/2} + C_\delta  \le C_\delta e^{A/2},
  \end{split}
\end{equation}
by Lemma~\ref{lem:B_n_moving_wall}. Similarly, we get
\begin{equation}
  \label{eq:5}
  \E^\sharp_{\mathrm{nbab}}[Y^{\mathrm{nw}}_{\Theta_1}\Ind_{\widetilde G^{\mathrm{tot}}_{n_\Theta}}\,|\,\F^\sharp_\Delta]\Ind_{G^\sharp_\Delta}  \le Ce^{A/2}/a.
\end{equation}
Moreover, we have by Lemma~\ref{lem:B_n_moving_wall},
\begin{equation*}
\E^\sharp_{\mathrm{nbab}}[(B^{\mathrm{tot}}_{n_\Theta-1}+B^{\mathrm{tot}}_{n_\Theta}) \Ind_{\widetilde G^{\mathrm{tot}}_{n_\Theta}}\,|\,\F^\sharp_\Delta]\Ind_{G^\sharp_\Delta}  \le C_\delta e^{\mu a -5A/3}/a.
\end{equation*}
Setting $X = \Psh\Big(B^{\mathrm{tot}}_{[\Theta_1,\Theta_1+t_1]} \le e^{-A} e^{\mu a}/a\,\Big|\,\F_{\Theta_1}\Big)$, we then have $\Esh[(1-X) \Ind_{G^\sharp_{\mathrm{nbab}}\cap \widetilde G^{\mathrm{tot}}_{n_\Theta} \cap G^\sharp_\Delta}] \le C_\delta e^{-2A/3}$ by Markov's inequality.
Applying the Markov inequality once more to $X$ as in the proof of Lemma~\ref{lem:HBflat} yields 
\begin{equation}
  \label{eq:8}
\Psh(X \le 1-\ep^2,\,G^\sharp_{\mathrm{nbab}}\cap \widetilde G^{\mathrm{tot}}_{n_\Theta} \cap G^\sharp_\Delta) \le C_\delta\ep^{-2}e^{-2A/3} \le C_\delta \ep^2,
\end{equation}
by \eqref{eq:ep_lower}.
The lemma now follows from \eqref{eq:113}, \eqref{eq:9} and \eqref{eq:8}, Lemmas~\ref{lem:Gsharp_Delta_nbab} and \ref{lem:G_n_prime} and Markov's inequality applied to \eqref{eq:4} and \eqref{eq:5}.
\end{proof}

\subsection{Proofs of the main results}
\label{sec:Bsharp_main_results}

\begin{proof}[Proof of Lemma~\ref{lem:HBsharp}]
Set $N_0 = \lfloor 2\pi  e^Aa^{-3}e^{\mu a} \rfloor$, then $|N_0 - \lceil e^\delta N\rceil| \le 1$.
By Lemma~\ref{lem:Hperp}, it remains to estimate the probability of the last event in the definition of $G^\sharp_0$. Denote the initial particles by $1,\dots,N_0$ and let $N^{\neg i}_{t_1}$ be the number of particles at time $t$ which do not descend from $i\in\{1,\ldots,N_0\}$. Again as in the proof of Lemma~\ref{lem:Bflat_start}, we have by \eqref{eq:49}, \eqref{eq:69} and \eqref{eq:84}, for every $i\le N_0$, by independence,
\begin{equation*}
\P(N^{\neg i}_{t_1} < N,\,R_{t_1}=0\,|\,X_i(0)) \le C_\delta e^{-A}/a\quad\tand\quad\P(R_{t_1}> 0) \le C_\delta e^A/a.
\end{equation*}
Using Lemma~\ref{lem:Bsharp_left_border}, we now have,
\begin{equation*}
  \begin{split}
  \E[B_0\Ind_{(R_{t_1}=0)}] &\le \E\Big[ \sum_{i=1}^{N_0} \P(N^{\neg i}_{t_1} < N,\,R_{t_1}=0\,|\,X_i(0)) \E^{X_i(0)}[L_0\Ind_{(R_{t_1}=0)}]\Big]\\
& \le C_\delta (e^{-A}/a) e^{\mu a} \E[Y_0] \le C_\delta e^{\mu a}/a^2,
  \end{split}
\end{equation*}
by \eqref{eq:49}. Setting $X = \P\Big(B_{0} \le e^{-A} e^{\mu a}/a\,\Big|\,\F_0\Big)$, we then have $\E[1-X] \le C_\delta e^{A}/a$ by Markov's inequality. Applying Markov's inequality once more to $X$ as in the proof of Lemma~\ref{lem:HBflat} yields the lemma.
\end{proof}

\begin{proof}[Proof of Proposition~\ref{prop:Bsharp} and the \Bsh-BBM part of Theorem~\ref{th:Bflat_Bsharp}]
Exactly the same reasoning as for the \Bfl-BBM, but using Lemma~\ref{lem:G_sharp} instead of Lemma~\ref{lem:G_Bflat}.
\end{proof}

\begin{proof}[Proof of the second part of Proposition~\ref{prop:Bflat_Bsharp_med}]
Define $\rho = x_{\alpha e^{-2\delta}}$, so that $(1+\mu\rho)e^{-\mu\rho} = \alpha e^{-2\delta}(1+o_a(1))$. Recall that $N^{\mathrm{tot}}_t(\rho)$ denotes the number of particles to the right of $\rho$ at time $t$. As in the proof of the first part of Proposition~\ref{prop:Bflat_Bsharp_med} in Section~\ref{sec:Bflat_main_results}, it is enough to show that
\begin{equation*}
\P^\nu(N^{\mathrm{tot}}_t(\rho) > \alpha N,\,T>t) \le C_\delta \ep^{5/4},
\end{equation*}
for large $A$ and $a$, where $\nu$ is such that $G^\sharp_0$ is satisfied and $t\in[e^Aa^2,\sqrt\ep a^3]$. Fix such $\nu$ and $t$. Write $\P_{[t]} = \P(\cdot,\,T>t)$. Let $n$ be the largest integer such that $t_{n+1} < t$; note that $n\ge1$ for large $A$. 
By Lemmas~\ref{lem:hat_quantities} and \ref{lem:Zt_variance} together with Chebychev's and Markov's inequalities and Equations \eqref{eq:ep_lower} and \eqref{eq:ep_upper}, we have for large $A$ and $a$,
\begin{equation}
  \label{eq:16}
\P_{[t]}(|e^{-A}Z_{t_n}-1|\le \ep^{1/8},\,Y_{t_n}\le e^{4A/3}/a) \ge 1-C_\delta  \ep^{5/4}.
\end{equation}
And as in the proof of Lemma~\ref{lem:G_ex}, we have $\E_{[t]}[Z^{\mathrm{gr}}_{t_n}\Ind_{G_n^{\mathrm{tot}}\cap G_{\mathscr U}}]\le C_\delta$, $\E_{[t]}[Y^{\mathrm{gr}}_{t_n}\Ind_{G_n^{\mathrm{tot}}\cap G_{\mathscr U}\cap\{R^{\mathrm{gr}}_{I_{n-1}} = 0\}} ]\le C_\delta/a$ and $\P_{[t]}(R^{\mathrm{gr}}_{I_{n-1}} > 0,\,G_n^{\mathrm{tot}},\,G_{\mathscr U}) \le C_\delta \eta$, so that by Markov's inequality and Lemmas~\ref{lem:G_n0_tot} and \ref{lem:GDelta}, together with \eqref{eq:16}, we get
\begin{equation}
  \label{eq:89}
\P_{[t]}(|e^{-A}Z^{\mathrm{wg}}_{t_n}-1|\le 2\ep^{1/8},\,Y^{\mathrm{wg}}_{t_n}\le 2e^{4A/3}/a) \ge 1-C_\delta \ep^{5/4},
\end{equation}
Lemmas~\ref{lem:Rt}, \ref{lem:Bsharp_left_border} and \ref{lem:BBM_0} and Corollary~\ref{cor:hat_many_to_one} then show that 
\begin{equation}
 \label{eq:7208}
\P_{[t]}(R^{\mathrm{tot}}_{[t_n,t]} > 0,\,|e^{-A}Z^{\mathrm{wg}}_{t_n}-1|\le 2\ep^{1/8},\,Y^{\mathrm{wg}}_{t_n}\le 2e^{4A/3}/a) \le o_a(1). 
\end{equation}

Let $\widetilde N^{\mathrm{wg}}_t(\rho)$ be the particles to the right of $\rho$ at time $t$ which descend from the white and grey particles at time $t_n$ and which have hit neither $0$ nor $a$ between $t_n$ and $t$. By Lemmas~\ref{lem:N_expec_large_t} and~\ref{lem:N_2ndmoment} and Corollary~\ref{cor:hat_many_to_one}, we have for large $A$ and $a$, with $\Phat = \P(\cdot\,|\,T>t)$,
\begin{align}
  \Ehat[\widetilde N^{\mathrm{wg}}_t(\rho)\,|\,\F_{t_n}] &\le (1-\delta/2) \alpha N e^{-A} Z^{\mathrm{wg}}_{t_n},\\
  \Varhat(\widetilde N^{\mathrm{wg}}_t(\rho)\,|\,\F_{t_n}) &\le C_{\delta,\alpha} N^2 e^{-2A} (a^{-1}Z^{\mathrm{wg}}_{t_n}+Y^{\mathrm{wg}}_{t_n}).
\end{align}
Chebychev's inequality and \eqref{eq:89} then give for large $A$ and $a$,
\begin{equation}
  \label{eq:94}
  \Phat(\widetilde N^{\mathrm{wg}}_t(\rho) > (1-\delta/4) \alpha N) \le C_\delta \ep^{5/4}.
\end{equation}
Furthermore, denote by $\widetilde N^{\mathrm{blue}}_t(\rho)$ the number of particles to the right of $\rho$ at time $t$ which have turned blue after $t_{n}+a^2$ and which have not hit $a$ between $t_n$ and $t$. Then as in Lemma~\ref{lem:blue_particles}, we have by Lemma~\ref{lem:Nt_prob_less_N} and \eqref{eq:left_border},	
\begin{align}
\nonumber
&\Ehat[\widetilde N^{\mathrm{blue}}_t(\rho)\,|\,\F_{t_n}]\Ind_{G_{Z,n}}\\
\nonumber
& \le \sup_{(u,t_n)\in\mathscr N^{(0)}_{t_n},\,r\in[t_n+a^2,t]} \Phat(N^{(0),\neg u}_r < N\,|\,\F_{t_n})\Ind_{G_{Z,n}}\sum_{(u,s)\in\mathscr N^{\mathrm{wg}}_{t_n}}\int_{t_{n}+a^2}^t \Ehat^{(X_u(s),s)}[L_{[\tau,\tau+\dd \tau]}(u,s)]\E^0[N^{\mathrm{free}}_{{t-\tau}}(\rho)]\\
\label{eq:73}
& \le C_\delta e^{-A}(a^{-1} + e^{-A}Y^{\mathrm{wg}}_{t_n})\sum_{(u,s)\in\mathscr N^{\mathrm{wg}}_{t_n}}e^{\mu X_u(s)}\int_{a^2}^{t-t_n} I^a(a-X_u(s),\dd \tau)\E^0[N^{\mathrm{free}}_{t-t_n-\tau}(\rho)].
\end{align}
Set $D = t-t_n$. Note that $(K+3)a^2\le D \le 2(K+3)a^2$. By Lemma~\ref{lem:many_to_one_simple} and Girsanov's theorem, we have for every $\tau\ge 0$,
\begin{equation}
\label{eq:53}
  \E^0[N^{\mathrm{free}}_{\tau}({\rho})] = e^{\beta m\tau}W^0_{-\mu}(X_\tau > \rho) = e^{\frac{\pi^2}{2a^2}\tau}W^0[e^{-\mu X_\tau}\Ind_{(X_\tau > \rho)}] = e^{\frac{\pi^2}{2a^2}\tau}\int_\rho^\infty e^{-\mu z}g_\tau(z)\,\dd z,
\end{equation}
where $g_\tau(x)=(2\pi\tau)^{-1/2}e^{-x^2/(2\tau)}$ is the Gaussian density with variance $\tau$. If $\tau\ge a^2$, then $\sup_z g_\tau(z) \le C/a$, so that for every $x\in[0,a]$ and $z\ge\rho$,
\begin{equation}
\label{eq:45}
  \int_{a^2}^{D-a^2}I^a(a-x,\dd\tau) g_{D-\tau}(z) \le Ca^{-1} I^a(a-x,[a^2,D-a^2]) \le CKa^{-1}\sin(\pi x/a),
\end{equation}
by Lemma~\ref{lem:I_estimates}. Moreover, by Lemma~\ref{lem:I_estimates}, we have $I^a(a-x,\dd \tau) \le Ca^{-2} \sin(\pi x/a)\,\dd\tau$ for every $\tau\ge a^2$, so that
\begin{equation}
  \label{eq:72}
  \int_{D-a^2}^D I^a(a-x,\dd\tau)g_{D-\tau}(z) \le Ca^{-2}\sin(\pi x/a)\int_0^{a^2} \tau^{-1/2}\,\dd\tau \le Ca^{-1}\sin(\pi x/a).
\end{equation}
Equations \eqref{eq:73}, \eqref{eq:53}, \eqref{eq:45} and \eqref{eq:72} together with Fubini's theorem now yield
\begin{equation}
  \label{eq:74}
  \begin{split}
 \Ehat[\widetilde N^{\mathrm{blue}}_t(\rho)\,|\,\F_{t_n}]\Ind_{G_{Z,n}} \le C_\delta e^{-A}(a^{-1} + e^{-A}Y^{\mathrm{wg}}_{t_n}) e^{\mu(a-\rho)}a^{-2}Z^{\mathrm{wg}}_{t_n}
 \le C_\delta \alpha e^{-2A}N(1+ae^{-A}Y^{\mathrm{wg}}_{t_n}) Z^{\mathrm{wg}}_{t_n}.
  \end{split}
\end{equation}
Furthermore, if $N^{\mathrm{rest}}_t(\rho)$ denotes the  particles to the right of $\rho$ at time $t$ descending from those turning blue between $t_{n-1}$ and $t_n+a^2$, then by \eqref{eq:53} and the supremum bound on $g_\tau(z)$,
\begin{equation}
  \label{eq:91}
  \Ehat[N^{\mathrm{rest}}_t(\rho)G_n^{\mathrm{tot}}] \le C_\delta a^{-1}e^{-\mu \rho} \Ehat[(B_n^{\mathrm{tot}} +B_{n-1}^{\mathrm{tot}})\Ind_{G_n^{\mathrm{tot}}}] \le C_\delta \alpha e^{-A}N,
\end{equation}
by Lemma~\ref{lem:B_n}. Markov's inequality applied to \eqref{eq:91} and \eqref{eq:74} gives, together with Lemma~\ref{lem:G_n0_tot} and \eqref{eq:ep_lower},
\begin{equation}
  \label{eq:125}
  \Phat(\widetilde N_t^{\mathrm{blue}}(\rho)+N_t^{\mathrm{rest}}(\rho) > (\delta/10) \alpha N,\,Z^{\mathrm{wg}}_{t_n}\le 2e^A,\,Y^{\mathrm{wg}}_{t_n}\le 2e^{4A/3}/a) \le C_\delta \ep^{2}.
\end{equation}
The statement now follows from  \eqref{eq:89}, \eqref{eq:7208}, \eqref{eq:94} and \eqref{eq:125}.
\end{proof}

\begin{proof}[Proof of Lemma~\ref{lem:Csharp}]
Follows simply from the fact that B$^\sharp$-BBM and C$^\sharp$-BBM coincide until $\Theta_n$ on the set $G^\sharp_n$.
\end{proof}

\appendix

\section{Preliminaries on branching Markov processes}
\label{sec:preliminaries}

In this section we settle the notation and recall some basic properties of branching Markov processes in a Polish space $\mathscr E$. We use the setup of Neveu's marked trees\footnote{Other definitions of branching Markov processes have been used in the literature, e.g.\ as a Markov process in the space $\bigcup_n \mathscr E^n/_{S_n}$, where $S_n$ is the group of coordinate permutations \cite{Ikeda1968} or as a measure-valued Markov process \cite{Dawson1991}. However, these definitions do not include the genealogy and are therefore less suited for our purposes.} but omit some technical details. These can be found in the original works by Neveu and Chauvin \cite{Neveu1986,Chauvin1988,Chauvin1991}. 

\subsection{Definition and notation}
\label{sec:preliminaries_definition}

To define a branching Markov process, we need four ingredients:
\begin{itemize}[nolistsep]
 \item $\mathscr E$, a Polish space (we will only use $\mathscr E = \R$ or $=\R^2$),
 \item $X = (X_t)_{t\ge 0}$, a (conservative) strong Markov process on $\mathscr E$ with cadlag paths; law of and expectation w.r.t.\ this process started at $x$ are denoted by $P^x$ and $E^x$,
 \item $R: \mathscr E\to\R_+$, a measurable function (the ``killing rate''), and
 \item $((q(x,k))_{k\in \N})_{x\in\mathscr E}$, a family of probability measures (``reproduction laws'') on $\N$, measurable with respect to $x$.
\end{itemize}
The branching Markov process starting at $x$ is informally defined as follows: A particle moves according to the process $X$, starting at $x$, and gets killed at the position-dependent rate $R$. When it gets killed at the point $y$, say, it gets replaced by a random number of particles according to the reproduction law $q(y,\cdot)$. These particles then independently repeat this process, starting at $y$. 

In order to maintain a record of the genealogy, we associate to each particle, or \emph{individual}, a label $u\in U = \{\emptyset\} \cup \bigcup_{n\ge 1} (\N^*)^n$, the \emph{universe}. Here, the ancestor is denoted by $\emptyset$ and the $i$-th child of the individual $u$ is denoted by $ui$, the concatenation of $u$ and $i$. The universe is endowed with the ordering relations $\preceq$ and $\prec$ defined by 
\[
 u \preceq v \iff \exists w\in U: v = uw\quad\tand\quad u\prec v \iff u\preceq v \tand u \ne v.
\]
The most recent common ancestor of two individuals $u$ and $v$, i.e.\ the maximal individual smaller than $u$ and $v$ (in the sense of $\preceq$), is then denoted by $u\wedge v$.

For $t\ge0$, we denote by $\A(t)\subset U$ the set of individuals \emph{alive} at time $t$, i.e.\ those which are born before or at time $t$ and which die after time $t$. The death time of an individual $u$ is denoted by $d_u$ and the number of its children by $k_u$. The \emph{position} of an individual $u$ at time $t$ is denoted by $X_u(t)$; if $u$ is not alive at time $t$, it is the position of its ancestor alive at that time (or undefined if such an ancestor does not exist).

The individuals that are alive at some time, i.e. the set $\A_\infty = \bigcup_t \A(t)$, forms a tree, and together with $(X_u(t))_{u\in \A_\infty,\,t\ge0}$ this forms an object called a \emph{marked tree}. The law of the branching Markov process defined above is then formally a law on the space of marked trees which we denote by $\P^x$. Expectation w.r.t.\ this law will be denoted by $\E^x$. Note that by looking at the space-time process $(X_t,t)_{t\ge 0}$, we can (and will) extend this definition to the time-inhomogeneous case.

The above definitions can be extended to arbitrary initial configurations of particles distributed according to a finite counting measure $\nu$ on $\mathscr E$. In this case, there is no individual $\emptyset$; the initial particles are denoted by $1,\ldots,n$. The law of this process is denoted by $\P^\nu$, expectation w.r.t.\ this law by $\E^\nu$.

\subsection{Stopping lines}
\label{sec:stopping_lines}

The analogue to stopping times for branching Markov processes are \emph{(optional) stopping lines}, which have been defined in various degrees of generality \cite{Jagers1989,Chauvin1991,Biggins2004}; often only in the discrete-time case, but the definitions carry over. The most general definition is contained in \cite{Jagers1989}, which the reader should consult to find justifications of the definitions below.

For $(v,s),(u,t)\in U\times [0,\infty)$, write $(v,s)\preceq(u,t)$ if  $v\preceq u$ and $s\le t$, and $(v,s)\prec(u,t)$ if $(v,s)\preceq(u,t)$ and $(v,s)\ne (u,t)$. Now define a \emph{line} $\ell$ to be a subset of $U\times [0,\infty)$, such that $(u,t)\in\ell$ implies $(v,s)\notin \ell$ for all $(v,s) \preceq (u,t)$. We extend the meaning of the ordering relation $\preceq$ to lines as follows: For a pair $(u,t)\in U\times [0,\infty)$ and a line $\ell$, write $\ell \preceq (u,t)$ if there exists $(v,s)\in \ell$, such that $(v,s)\preceq(u,t)$. For a subset $A\subset U\times [0,\infty)$, write $\ell \preceq A$ if $\ell \preceq (u,t)$ for all $(u,t)\in A$. The relation $\prec$ is extended similarly. If $\ell_1$ and $\ell_2$ are two lines, we define the line $\ell_1 \wedge \ell_2$ to be the maximal line (with respect to $\preceq$), which is smaller than both lines; it is easy to see that this line exists.

We say that a line $\ell$ is \emph{proper} if $u\in\A(t)$ for all $(u,t)\in \ell$.

For a line $\ell$, we define the $\sigma$-algebra 
\[
\F_\ell = \sigma\Big(\{u\in\A(t)\}, X_u(t); (u,t)\in U\times [0,\infty): \ell \not \prec (u,t)\Big).
\]
Informally, $\F_\ell$ contains the information about everything except the descendants of the line $\ell$.
 A \emph{stopping line} $\mathscr L$ is then a random proper line, such that for every line $\ell$, $\{\mathscr L\preceq \ell\} \in \F_\ell$. Informally, this means that $\mathscr L$ does not depend on the individuals ``behind'' it. A stopping line $\mathscr L$ (in fact any random proper line), defines a $\sigma$-algebra $\F_{\mathscr L}$ by
\[
 \F_{\mathscr L} = \sigma\Big(\{u\in\A(t), X_u(t) \in A, \mathscr L \not \prec (u,t)\}; (u,t)\in U\times [0,\infty), A\subset \mathscr E \text{ measurable}\Big).
\]

The \emph{strong branching property} \cite[Theorem 4.14]{Jagers1989} states that for every stopping line $\mathscr L$, conditioned on $\F_\mathscr L$, the subtrees rooted at the pairs $(u,t)\in\mathscr L$ are independent with respective distributions $\P^{(X_u(t),t)}$.

For $t\ge0$, the line $\mathscr L_t = \A(t)\times\{t\}$ is a stopping line. More generally, if $T$ is a stopping time for the Markov process $X$, viewed as a functional of the path, then
\[
 \mathscr L_T = \{(u,t)\in U\times [0,\infty): u\in \A(t)\tand t = T(X_u)\}
\]
is a stopping line as well.

\subsection{Moment formulae}
\label{sec:spine}

In this section, we recall first- and second moment formulae (called \emph{many-to-one} and \emph{many-to-two} lemmas) for additive functionals of branching Markov processes, which will be the main tools for the calculations in this article. For fixed time, these are classical (see e.g.\ \cite[Theorems~4.1 and~4.15]{Ikeda1968} or \cite{Sawyer1976}), see also  \cite{Kyprianou2004,Hardy2006,Harris2011} for recent proofs using the spine decomposition from \cite{LPP1995}. With stopping lines like the ones we consider below, they have been proved in the discrete-time setting \cite[Lemma~14.1]{Biggins2004} and their arguments can be used to adapt the proofs in \cite{Hardy2006} and \cite{Harris2011} to yield the results stated here. Again, they can be immediately extended to the time-inhomogeneous case by considering the space-time process.

In both lemmas below,  $f:\mathscr E\to \R_+$ is a measurable function and $H$ the hitting time functional of a closed set $F\subset \mathscr E$ which satisfies $P^x(H<\infty) = 1$ for every $x\in\mathscr E$. Define $m(x) = \sum_{k\ge 0} (k-1) q(x,k)$ and $m_2(x) = \sum_{k\ge 0} k(k-1) q(x,k)$ for every  $x\in\mathscr E$.

\begin{lemma}[Many-to-one]
\label{lem:many_to_one_simple}
For all $x\in\mathscr E$,
\[
 \E^x\Big[\sum_{(u,t)\in \mathscr L_H} f(X_u(t))\Big] = E^x\Big[e^{\int_0^H R(X_t) m(X_t)\,\dd t} f(X_H)\Big].
\]
\end{lemma}
Thus, calculating the expectation of an additive functional reduces to calculating the expectation of a functional w.r.t.\ the Markov process $X$. Note that if $R$ and $m$ are constants, then $e^{Rmt}$ is the expected number of particles at time $t$.

Now suppose that the Markov process $X$ admits a transition density w.r.t.\ some measure $\dd x$ on $\mathscr E$. Define the \emph{density of the branching Markov process before $\mathscr L_H$ (w.r.t.\ $\dd x$)} by
\begin{equation}
\label{eq:def_density}
 p_H(x,y,t)\,\dd y = \E^x\Big[\sum_{u\in \A(t)} \Ind_{(X_u(t)\in \dd y,\ t < H(X_u))}\Big].
\end{equation} 
\begin{lemma}[Many-to-two]
\label{lem:many_to_two}
For all $x\in\mathscr E$,
\begin{multline*}
 \E^x\Big[\Big(\sum_{(u,t)\in \mathscr L_H} f(X_u(t))\Big)^2\Big] = \E^x\Big[\sum_{(u,t)\in \mathscr L_H} (f(X_u(t)))^2\Big]\\
 + \int_0^\infty \int_{\mathscr E} p_H(x,y,s) R(y)m_2(y) \Big(\E^y\Big[\sum_{(u,t)\in \mathscr L_H} f(X_u(t))\Big]\Big)^2\,\dd y \,\dd s.
\end{multline*}
\end{lemma}

\section{Brownian motion in an interval}
\label{sec:bm}

In this section, we recall some explicit formulae concerning real-valued Brownian motion killed upon exiting or conditioned not to exit an interval. We also prove Lemma~\ref{lem:k_integral}.

\subsection{Brownian motion killed upon exiting an interval}
\label{sec:killed_bm}

We define for $x\in\R$ and $t>0$ the following function of Jacobi theta-type:
\begin{eqnarray}
\label{eq:theta}
\theta(x,t) =& \frac 1 2 + \sum_{n=1}^\infty e^{- \tfrac{\pi^2}{2} n^2 t} \cos(\pi n x)\\
\label{eq:theta_gaussian}
=& \sum_{n\in\Z} \frac{1}{\sqrt{2\pi t}} \exp\Big(-\frac{(x-2n)^2}{2t}\Big),
\end{eqnarray}
the two representations being related by the Poisson summation formula (see \cite{Bellman}, §9). We denote the derivative of $\theta$ with respect to $x$ by  $\theta'$. Note that \eqref{eq:theta_gaussian} implies that 
\begin{equation}
 \label{eq:theta_Cinfty}
\theta \in C^\infty\Big(\big(\R\times[0,\infty)\big)\backslash \{(2n,0):n\in\Z\}\Big).
\end{equation}
Furthermore, the function $\theta$ has the following basic properties:
\begin{itemize}[nolistsep]
  \item It is $2$-periodic in $x$, i.e. $\theta(x,t) = \theta(x+2,t)$ for every $x\in\R$, $t\ge 0$.
  \item For every $t>0$, $\theta(\cdot,t)$ and $\theta(1+\cdot,t)$ are even functions, in particular, $\theta'(0,t) = \theta'(1,t) = 0$ for all $t>0$.
  \item For every $t>0$, the function $\theta(x,t)$ is strictly decreasing (resp., increasing) in $x$ on every interval $[2n,2n+1]$ (resp.,  $[2n-1,2n]$), $n\in\Z$. This is easily seen from the infinite product representation of theta functions, see e.g.\ \cite{Bellman}, §32.
  \item It solves the heat equation $\partial_t \theta = \tfrac 1 2 \partial_x^2 \theta$ on $\R\times(0,\infty)$ with initial condition $\theta(\cdot,0) = \sum_{n\in\Z} \delta_{2n}$.
\end{itemize}
We then define $\thbar(t)$ by
\begin{equation}
\label{eq:thbar}
 \thbar(t) = \frac{2}{\pi^2} e^{\frac{\pi^2}{2}t}\frac{\dd}{\dd t} \theta(1,t) = \frac{1}{\pi^2} e^{\frac{\pi^2}{2}t} \theta''(1,t),
\end{equation}
which is a smooth function on $\R_+$. By \eqref{eq:theta} and \eqref{eq:theta_gaussian}, one can show that $\thbar(t)$ is strictly increasing\footnote{More precisely, by elementary computations, \eqref{eq:theta} gives $t_0\in\R_+$, such that $\thbar$ is strictly increasing on $(t_0,\infty)$ and \eqref{eq:theta_gaussian} gives $t_1 > t_0$, such that $\thbar$ is strictly increasing on $[0,t_1)$.} with $\thbar(0) = 0$ and $\thbar(+\infty) = 1$.

Various quantities of Brownian motion killed upon exiting an interval can be expressed using $\theta(x,t)$. For $x\in\R$, let $W^x$ be the law of Brownian motion started at $x$, let $(X_t)_{t\ge0}$ be the canonical process and let $H_y = \inf\{t\ge0: X_t = y\}$. For $a>0$, let $p^a_t(x,y)$ be the transition density of Brownian motion started at $x$ and killed upon leaving the interval $(0,a)$, i.e.
\begin{equation}
\label{eq:def_p}
p^a_t(x,y)\,\dd y = W^x(X_t \in \dd y,\
H_0\wedge H_a > t), \quad x,y\in [0,a].
\end{equation}
Then (see \cite{ItoMcKean}, Problem 1.7.8 or \cite{BorodinSalminen}, formula 1.1.15.8),
\begin{equation}
\label{eq:p_theta}
p^a_t(x,y) = a^{-1}\left(\theta\Big(\frac{x-y}{a},\frac t {a^2}\Big) - \theta\Big(\frac{x+y}{a},\frac t {a^2}\Big)\right).
\end{equation}
One can easily deduce from the above-mentioned properties of $\theta(x,t)$, or otherwise, the following basic properties of this quantity:
\begin{itemize}[nolistsep]
	\item $p_t^a(x,y) > 0$ for all $x,y\in(0,a)$ and $t>0$.
	\item $p_t^a(0,y) = p_t^a(a,y) = p_t^a(x,0) = p_t^a(x,a) = 0$ for all $x,y\in[0,a]$ and $t>0$.
	\item For all $t>0$, the following quantities are strictly positive and continuous functions of $x$, $y$ and $t$: 
	\[
	\frac{\partial}{\partial x}p_t^a(x,y)\Big|_{x=0}, -\frac{\partial}{\partial x}p_t^a(x,y)\Big|_{x=a}, \frac{\partial}{\partial y}p_t^a(x,y)\Big|_{y=0}, -\frac{\partial}{\partial y}p_t^a(x,y)\Big|_{y=a}
	\]
\end{itemize}
These properties imply that for every $\ep>0$, the function $(x,y,t)\mapsto p_t^1(x,y) / (\sin(\pi x)\sin(\pi y))$ is continuous and strictly positive on the compact set $[0,1]^2\times[\ep,\ep^{-1}]$. By scaling, this readily implies the following: There exists a (universal) family $(C_t)_{t>0}$ of positive constants, continuous in $t$, such that for every $t>0$,
\begin{equation}
\label{eq:p_sin_coarse}
\frac{2C_{t/a^2}}{a} \sin\Big(\frac{\pi x}{a}\Big) \sin\Big(\frac{\pi y}{a}\Big) \le p_t^a(x,y) \le \frac{2C_{t/a^2}^{-1}}{a} \sin\Big(\frac{\pi x}{a}\Big) \sin\Big(\frac{\pi y}{a}\Big),\quad \forall x,y\in[0,a].
\end{equation}
For large $t$, this can be made much more precise: by \eqref{eq:theta}, we have
\begin{equation}
\label{eq:p_sin}
p^a_t(x,y) = \frac{2}{a} \sum_{n=1}^\infty e^{-\frac{\pi^2}{2a^2} n^2 t} \sin(\pi n \tfrac x a) \sin(\pi n \tfrac y a).
\end{equation}
Define 
\begin{equation}
\label{eq:def_E}
 E_t = \pi^2 \sum_{n=2}^{\infty} n^2e^{-\pi^2(n^2-1) t/2}.
\end{equation}
By \eqref{eq:p_sin} and the inequality $|\sin nx| \le n \sin x$, $x\in [0,\pi]$, one now sees that
\begin{equation}
 \label{eq:p_estimate}
 \left|e^{\frac{\pi^2}{2a^2} t}p^a_t(x,y) - \frac 2 a \sin\Big(\frac{\pi x}{a}\Big) \sin\Big(\frac{\pi y}{a}\Big)\right| \le E_{t/a^2}\frac 2 a \sin\Big(\frac{\pi x}{a}\Big) \sin\Big(\frac{\pi y}{a}\Big).
\end{equation}
In particular, the optimal constant $C_t$ in \eqref{eq:p_sin_coarse} converges to $1$ as $t\to\infty$.

For small $t$, we often resort to the integral over $t$ of $p_t^a(x,y)$, i.e.\ the Green function of the killed Brownian motion. It is given by (see e.g.\ \cite{Kallenberg1997}, Lemma 20.10, p379)
\begin{equation}
\label{eq:p_potential}
 \int_0^\infty p^a_t(x,y)\,\dd t = 2 a^{-1} (x\wedge y)(a-x\vee y).
\end{equation}

Set $H=H_0\wedge H_a$ and define $r^a_t(x) = W^x(H \in \dd t,\ X_H = a)/\dd t$. Then (see \cite{BorodinSalminen}, formula 1.3.0.6),
\begin{equation*}
r^a_t(x) = -\frac 1 2 \frac{\partial}{\partial y} p_t^a(x,y)\Big|_{y=a} = \frac{1}{a^2} \theta'\left(\frac{x}{a}-1,\frac t {a^2}\right).
\end{equation*}

The following two integrals are going to appear several times throughout the article, which is why we give some useful estimates here. For a measurable subset $S\subset \R$, define
\begin{align}
 \label{eq:Ia} 
  I^a(x,S) &= W^x\Big[e^{\frac{\pi^2}{2a^2}H_a} \Ind_{(H_0 > H_a,\,H_a\in S)}\Big] = \int_{S\cap (0,\infty)} e^{\frac{\pi^2}{2a^2} s} r_s^a(x)\dd s,\quad\tand\\
 \label{eq:Ja}
  J^a(x,y,S) &= \int_{S\cap (0,\infty)} e^{\frac{\pi^2}{2a^2} s} p_s^a(x,y)\dd s,
\end{align}
which satisfy the scaling relations 
\begin{equation}
 \label{eq:IJ_scaling}
 I^a(x,S) = I\Big(\frac x a,\frac S {a^2}\Big),\quad J^a(x,y,S) = a J\Big(\frac x a,\frac y a,\frac S {a^2}\Big),
\end{equation}
with $I = I^1$ and $J=J^1$. The following lemma provides estimates on $I(x,S)$ and $J(x,y,S)$. It is easily proven from the above equations (see \cite[Lemma~2.2.1]{MaillardThese}).
\begin{lemma}
 \label{lem:I_estimates}
There exists a universal constant $C$, such that for every $x\in[0,1]$ and every measurable $S\subset \R_+$, we have
\[
\begin{split}
 |I(x,S) - \pi \lambda(S)\sin(\pi x)| &\le C\Big(x \wedge E_{\inf S} (1\wedge \lambda(S))\sin(\pi x)\Big),\quad\tand\\
|J(x,y,S) - 2 \lambda(S)\sin(\pi x)\sin(\pi y)| &\le C\Big([(x\wedge y)(1-(x\vee y))] \wedge E_{\inf S} \sin(\pi x) \sin(\pi y)\Big),
\end{split}
\]
where $\lambda(S)$ denotes the Lebesgue measure of $S$ and $E_{\inf S}$ is defined in \eqref{eq:def_E}.
\end{lemma}

\subsection{The Brownian taboo process}
\label{sec:taboo}
 The \emph{Brownian taboo process} on the interval $(0,a)$ is the diffusion with infinitesimal generator
$\frac 1 2 \left(\frac \dd {\dd x}\right)^2 + \frac \pi a \cot \frac {\pi x} a \frac \dd {\dd x}$.
The name of this process was coined by F.\ Knight \cite{Knight1969} who showed that it can be interpreted as Brownian motion conditioned to stay inside the interval $(0,a)$. It is the Doob transform of Brownian motion killed at $0$ and $a$, with respect to the space-time harmonic function $h(x,t) = \sin(\pi x/a) \exp(\pi^2t/(2a^2))$. As a consequence, its transition density is given for $x,y\in[0,a]$ by
\begin{equation*}
 p^{\text{taboo},a}_t(x,y) = \frac {\sin (\pi y/a)} {\sin (\pi x/a)}e^{\frac {\pi^2}{2a^2} t}\, p^a_t(x,y).
\end{equation*}
In particular, by \eqref{eq:p_sin_coarse} and \eqref{eq:p_estimate}, there exists 
a (universal) family $(C_t)_{t>0}$ of positive constants, continuous in $t$
 and such that $C_t\to 1$ as $t\to\infty$, such that
\begin{equation}
\label{eq:ptaboo_estimate}
 \frac {2C_{t/a^2}} a \sin^2(\pi y/a) \le p^{\text{taboo},a}_t(x,y) \le \frac {2C_{t/a^2}^{-1}} a \sin^2(\pi y/a).
\end{equation}
For $x\in [0,a]$ we denote the law of the Brownian taboo process on $(0,a)$ started from $x$ by $W^x_{\text{taboo},a}$. Often we will drop the $a$ if its value is clear from the context, similarly for $p^{\text{taboo},a}_t(x,y)$. We also denote by $W^{x,t,y}_{\text{taboo}}$ the law of the taboo bridge from $x$ to $y$ of length $t$. Note that the taboo process is self-dual in the sense that for a measurable functional $F$ and $t>0$, we have
\(
 W^{x,t,y}_{\text{taboo}}[F((X_s;0\le s\le t))] = W^{y,t,x}_{\text{taboo}}[F((X_{t-s};0\le s\le t))].
\)

The following lemma is used a few times in the article, e.g.\ for bounding from  below the density of particles in the B-BBM conditioned not the break out from Section~\ref{sec:before_breakout_particles} (in the proof of Lemma~\ref{lem:conditioned_Z_lowerbound}), and in Section~\ref{sec:fugitive} for bounding the contribution of the fugitive's descendants in B-BBM.

\begin{lemma}
\label{lem:k_integral}
Let $c>0$. There exists a constant $C$, depending only on $c$, such that we have for every $x,y\in[0,a]$, for every $t\ge0$,
\begin{equation}
 \label{eq:taboo_integral_error_1}
 W^{x}_{\mathrm{taboo}}\Big[\int_0^t e^{-c X_s}\,\dd s\Big] \le C \Big(t/a^3 + (1\wedge x^{-1})\Big),
\end{equation}
and for $t\ge a^2$, 
\begin{equation}
 \label{eq:taboo_integral_error_2}
 W^{x,t,y}_{\mathrm{taboo}}\Big[\int_0^t e^{-c X_s}\,\dd s\Big] \le C \Big(t/a^3 +(1\wedge x^{-1}) + (1\wedge y^{-1})\Big).
\end{equation}
 If $t\le a^2$, we still have,
\begin{equation}
 \label{eq:taboo_integral_error_3}
 W^{x,t,y}_{\mathrm{taboo}}\Big[\int_0^t e^{-c X_s}\,\dd s\Big] \le C.
\end{equation}
\end{lemma}

\begin{proof}
Let $c>0$. We first show that \eqref{eq:taboo_integral_error_1} implies \eqref{eq:taboo_integral_error_2}. By the self-duality of the taboo process, we have for every $x,y\in[0,a]$,
\[
 W^{x,t,y}_{\mathrm{taboo}}\Big[\int_0^t e^{-cX_s}\,\dd s\Big] = W^{x,t,y}_{\mathrm{taboo}}\Big[\int_0^{t/2} e^{-cX_s}\,\dd s\Big] + W^{y,t,x}_{\mathrm{taboo}}\Big[\int_0^{t/2} e^{-cX_s}\,\dd s\Big].
\]
Conditioning on $\sigma(X_s; 0\le s\le t/2)$ and using \eqref{eq:ptaboo_estimate}, there exists a universal constant $C$, such that for $t\ge a^2$, for every $x,y\in[0,a]$,
\begin{align*}
 W^{x,t,y}_{\mathrm{taboo}}\Big[\int_0^{t/2} e^{-cX_s}\,\dd s\Big] &= W^x_{\mathrm{taboo}}\Big[\frac{p^{\mathrm{taboo}}_{t/2}(X_{t/2},y)}{p^{\mathrm{taboo}}_t(x,y)}\int_0^{t/2} e^{-cX_s}\,\dd s\Big] \\
 &\le C\, W^x_{\mathrm{taboo}}\Big[\int_0^{t/2} e^{-cX_s}\,\dd s\Big].
\end{align*}
Equation \eqref{eq:taboo_integral_error_1} therefore implies \eqref{eq:taboo_integral_error_2}.

Heuristically, one can estimate the left side of \eqref{eq:taboo_integral_error_1} for large $a$ in the following way: Since $e^{-cx}$ is decreasing very fast, only the times at which $X_s$ is of order $1$ contribute to the integral. When started from the stationary distribution, the process takes a time of order $a^3$ to reach a point at distance $O(1)$ from the origin \cite{Lambert2000} and it stays there for a time of order $1$, hence the integral is of order $t/a^3$. When started from the point $x$, an additional error is added, which is of order $1$, when $x$ is at distance of order $1$ away from $0$. Adding both terms gives the bound appearing in the statement of the lemma.

The actual calculations can be performed in the following way: Let $Y$ be a random variable with values in $(0,a)$ distributed according to $\widetilde{m}(\dd x) := (2/a) \sin^2(\pi x/a)\,\dd x$, which is the stationary probability measure of the taboo process. Let $H_Y = \inf\{t>0: X_s = Y\}$. We then have
\[
\begin{split}
 W^x_{\mathrm{taboo}}\Big[\int_0^t e^{-cX_s}\,\dd s\Big] &= W^x_{\mathrm{taboo}}\Big[\int_0^{H_Y} e^{-cX_s}\,\dd s+\int_{H_Y}^t e^{-cX_s}\,\dd s\Big]\\
&\le W^x_{\mathrm{taboo}}\Big[\int_0^{H_Y} e^{-cX_s}\,\dd s\Big] + W^{\widetilde m}_{\mathrm{taboo}}\Big[\int_0^t e^{-cX_s}\,\dd s\Big].
\end{split}
\]
Both of these quantities can now be expressed through the scale function and speed measure of the taboo process and after some calculations (see \cite[Lemma~2.2.2]{MaillardThese} for details), one obtains \eqref{eq:taboo_integral_error_1}.

Now let $t\le a^2$. In order to prove \eqref{eq:taboo_integral_error_3}, a different method is needed. We may assume that $x,y \le a/2$, otherwise we decompose the path at the first and/or last time it hits $a/2$ and bound the parts above $a/2$ trivially by $a^2e^{-ca/2} = O(c^{-2})$. The transition density of the taboo bridge can be written for every $s\in(0,t)$ by,
\[
 W^{x,t,y}_{\mathrm{taboo}}(X_s\in \dd z) = \frac{p_s^a(x,z)p_{t-s}^a(z,y)}{p_t^a(x,y)}\,\dd z.
\]
If we denote by $p_t^\infty(x,y) = (2/(\pi t))^{1/2} \exp(-(x^2+y^2)/2t)\sinh(xy/t)$ the transition density of Brownian motion killed at $0$, then we have the trivial inequality $p_t^a(x,y) \le p_t^\infty(x,y)$ and furthermore $p_t^a(x,y) \ge C^{-1} p_t^\infty(x,y)$ for $x,y\le a/2$ and $t\le a^2$ by Brownian scaling, where $C>0$ is a universal constant. It follows that
\[
 W^{x,t,y}_{\mathrm{taboo}}\Big[\int_0^t e^{-cX_s}\,\dd s\Big] \le C R^{x,t,y}\Big[\int_0^t e^{-cX_s}\,\dd s\Big],
\]
where $R^{x,t,y}$ denotes the law of the Bessel bridge of dimension 3, the Doob transform of Brownian motion killed at $0$ and started at $x$ with respect to the space-time harmonic function $h_y(z,s) = p_{t-s}^\infty(z,y)/p_t^\infty(x,y)$. It is Brownian motion with additional time-dependent drift
\[
 \frac{\dd}{\dd z}(\log h_y(z,s)) = -\frac{z}{t-s}+\frac{\dd}{\dd z}\log \sinh\frac{zy}{t-s} = -\frac{z}{t-s} + \frac{y}{t-s} \coth \frac{zy}{t-s}.
\]
Elementary calculations show that this is an increasing function in $y$ and standard comparison theorems for diffusions (see e.g.\ \cite{Revuz1999}, Theorem IX.3.7) now yield that for $y_1\le y_2$, we have
\(
 R^{x,t,y_2}[e^{-cX_s}] \le R^{x,t,y_1}[e^{-cX_s}],
\)
since $e^{-cx}$ is decreasing in $x$. This is true in particular for $y_1 = 0$. Using the self-duality of the Bessel bridge, we can repeat the same reasoning with $x$. We thus have altogether
\[
 W^{x,t,y}_{\mathrm{taboo}}\Big[\int_0^t e^{-cX_s}\,\dd s\Big] \le C R^{0,t,0}\Big[\int_0^t e^{-cX_s}\,\dd s\Big] = 2C R^{0,t,0}\Big[\int_0^{t/2} e^{-cX_s}\,\dd s\Big],
\]
for any $x,y\le a/2$. This last expectation can be bounded by a constant: Let $r_s(x,y) = y p^\infty_s(x,y)/x$ be the transition density of the Bessel-3 process, extended to $x=0$ by continuity. Then for every $s\in(0,t]$,
\[
  \lim_{y\to 0} \frac{r_{s}(x,y)}{r_t(0,y)} = \Big(\frac t s\Big)^{3/2} e^{-\frac{x^2}{2s}} \le  \Big(\frac t s\Big)^{3/2},
\]
and so, with $R^0$ the law of the Bessel-3 process started at $0$,
\[
 R^{0,t,0}\Big[\int_0^{t/2} e^{-cX_s}\,\dd s\Big] = R^0\Big[\Big(\int_0^{t/2} e^{-cX_s}\,\dd s\Big)\lim_{y\to 0} \frac{r_{t/2}(x,y)}{r_t(0,y)}\Big] \le 2^{3/2}  R^0\Big[\int_0^\infty e^{-cX_s}\Big].
\]
The Green function of the Bessel-3 process is $G^R(x,y) = 2y(x\wedge y)/x$ (let $a\to\infty$ in \eqref{eq:p_potential}), and so,
\[
 R^0\Big[\int_0^\infty e^{-cX_s}\Big] = \int_0^\infty G^R(0,x) e^{-cx}\,dx = \int_0^\infty 2xe^{-cx}\,dx = 2c^{-2}.
\]
Together with the previous inequalities, this yields \eqref{eq:taboo_integral_error_3}.
\end{proof}

\section{BBM in an interval}
\label{sec:interval}

In this section we study branching Brownian motion killed upon exiting an interval. We remark that the results of Sections~\ref{sec:ZY} and \ref{sec:right_border} as well as Lemma~\ref{lem:N_expec_large_t} are in essence already contained in \cite{Berestycki2010} (with $f\equiv 0$), but we reprove them here for completeness and with streamlined proofs. The remaining results are, to our knowledge, new.

The results obtained in this section are independent of the main text of the article and so is the notation that we introduce, although there are many similarities. However, there are also some conflicts that the reader should be aware of, notably the definitions of $Z_t$, $Y_t$ and $R_t$ (which would correspond to $Z^{(0)}_t$, $Y^{(0)}_t$ and $R^{(0)}_t$ in the main text). These conflicts have been voluntarily accepted in order to simplify the notation in this section.

\subsection{Notation}
\label{sec:interval_notation}
Let $q(k)$ be a law on $\N$ and $L$ a random variable with law $q(k)$. Define $m = \E[L-1]$, $m_2 = \E[L(L-1)]$ and $\beta = 1/(2m)$ and suppose that $m>0$ and $m_2 < \infty$. Furthermore, we let $a \ge \pi$ and set $\mu = \sqrt{1 - \frac{\pi^2}{a^2}}$. We then denote by $\P^x$ the law of the branching Markov process (in the sense of Section~\ref{sec:preliminaries_definition}) where, starting with a single particle at the point $x\in\R$, particles move according to Brownian motion with variance $1$ and drift $-\mu$ and branch at rate $\beta$ into $k$ particles according to the reproduction law $q(k)$. Expectation with respect to $\P^x$ is denoted by $\E^x$. On the space of continuous functions from $\R_+$ to $\R$, we define $H_0$ and $H_a$ to be the hitting time functionals of $0$ and $a$. We further set $H = H_0\wedge H_a$. By Lemma~\ref{lem:many_to_one_simple}, the density of the branching Brownian motion before $\mathscr L_H$, as defined in \eqref{eq:def_density}, is given for $t>0$ and $x,y \in (0,a)$ by
\begin{equation}
 \label{eq:density_bbm}
\p_t(x,y) = e^{\mu(x-y) + \frac{\pi^2}{2 a^2} t} p_t^a(x,y),
\end{equation}
where $p_t^a$ was defined in \eqref{eq:def_p}.

Now let $f$ be a barrier function as defined in Section~\ref{sec:BBBM_definition_definition}. Define
\begin{equation*}
\mu_t = \mu + \frac{1}{a^2} f'(t/a^2).
\end{equation*}
We denote by $\P^x_f$ the law of the branching Brownian motion described above, but with (time-dependent) drift $-\mu_t$. Expectation with respect to $\P^x_f$ is denoted by $\E^x_f$ and the density of the process before $\mathscr L_H$ is denoted by $\p^f_t(x,y)$. We also denote by $W^x_{-\mu_t}$ and $W^x_{-\mu}$ the laws of Brownian motion with (time-dependent) drift $-\mu_t$ and (constant) drift $-\mu$, respectively.

We also extend the definitions of $\P^x$, $\E^x$, $\P^x_f$ and $\E^x_f$ to arbitrary initial configurations of particles distributed according to a finite counting measure $\nu$ on $(0,a)$, as mentioned at the end of Section~\ref{sec:preliminaries_definition}, as well as to a single particle starting from a space-time point $(x,t)$.

\subsection{The processes \texorpdfstring{$Z_t$ and $Y_t$}{Z\_t and Y\_t}}
\label{sec:ZY}

Recall from Section \ref{sec:preliminaries} that the set of particles alive at time $t$ is denoted by $\mathscr A(t)$. We define
\[
 \mathscr A_{0,a}(t) = \{u\in \mathscr A(t): H(X_u) > t\},
\]
where $H$ was defined in the previous subsection. Now set 
\begin{equation*}
w_Z(x) = ae^{\mu (x-a)}\sin(\pi x/a)\Ind_{(x\in[0,a])}\quad\tand\quad w_Y(x) = e^{\mu(x-a)},
\end{equation*}
and define the processes $(Z_t)_{t\ge0}$ and $(Y_t)_{t\ge0}$ by
\[
 Z_t = \sum_{u\in \mathscr A_{0,a}(t)} w_Z(X_u(t))\quad\tand\quad Y_t = \sum_{u\in \mathscr A_{0,a}(t)} w_Y(X_u(t)).
\]
The usefulness of the process $(Z_t)_{t\ge0}$ comes from the fact that it is a martingale under $\P^x$, as observed in~\cite{Berestycki2010}. The proof of this fact is standard and relies on the branching property, the many-to-one lemma (Lemma~\ref{lem:many_to_one_simple}) and the fact that $e^{t/2} w_Z(X_t)$ is a martingale for $(X_t)_{t\ge0}$ a Brownian motion with drift $-\mu$ killed at $0$ and $a$, which is easily seen by It\={o}'s formula, for example.

The following lemma relates the density of BBM with variable drift to BBM with fixed drift. 
\begin{lemma}
 \label{lem:mu_t_density} For all $x,y\in[0,a]$ and $t\ge0$, we have
\(
 \p^f_t(x,y) = \p_t(x,y) e^{- f(t/a^2) + O(\|f\|/a)}.
\)
\end{lemma}
\begin{proof}
 By Lemma~\ref{lem:many_to_one_simple} and Girsanov's theorem, we have
\begin{equation}
\label{eq:005}
\begin{split}
 \p^f_t(x,y)\,\dd y &= e^{\beta mt}W_{-\mu_t}^x\left(X_t\in \dd y,\, H > t\right)\\
&= \exp\left(t/2 - \int_0^t\frac{\mu_s^2-\mu^2}{2} \dd s\right)W_{-\mu}^x\left[\exp\Big(-\int_0^t [\mu_s-\mu]\,\dd X_s\Big),\,X_t\in\dd y, H > t\right].
\end{split}
\end{equation}
By integration by parts, under $W^x_{-\mu}$, we have $\int_0^t[\mu_s-\mu]\,\dd X_s = (\mu_t - \mu) X_t - \int_0^tX_s\,\dd \mu_s.$
Since $X_s \in (0,a)$ for all $s<H$, this gives on the event $\{H>t\}$,
\begin{align}
\label{eq:0101}
 \Big|\int_0^t[\mu_s-\mu]\,\dd X_s\Big| &\le a |\mu_t - \mu| + a \int_0^t |\dd \mu_s| 
 \le \frac 1 a \Big(\|f'\|_\infty + \int_0^\infty|\dd f'(s)|\Big)
\end{align}
Furthermore, we have for all $s\ge 0$,
\[
 \frac{\mu_s^2}{2} = \frac{\mu^2}{2} + \frac{\mu}{a^2}f'(s/a^2)+\frac{f'(s/a^2)^2}{2a^4},
\]
so that
\begin{equation}
 \label{eq:020}
 \Big|\int_0^t \frac{\mu_s^2-\mu^2}{2} \dd s - \mu f(t/a^2)\Big| \le \int_0^\infty \frac{f'(s/a^2)^2}{2a^4}\dd s = \frac 1 {2a^2}  \int_0^\infty f'(s)^2\,\dd s.
\end{equation}
Equations \eqref{eq:005}, \eqref{eq:0101} and \eqref{eq:020} and another application of Lemma~\ref{lem:many_to_one_simple} now give
\[
  \p^f_t(x,y) = \p_t(x,y) e^{-\mu f(t/a^2) + O(\|f\|/a)},
\]
and the lemma now follows from \eqref{eq:c0_mu}.
\end{proof}

\begin{proposition}
 \label{prop:quantities}
Under any initial configuration of particles at positions $x_1,\ldots,x_n$, we have for every $t\ge 0$
\begin{equation}
 \label{eq:Zt_expectation}
\E_f[Z_t] = Z_0e^{- f(t/a^2) + O(\|f\|/a)},
\end{equation}
and if in addition $\mu \ge 1/2$, then
\begin{equation}
 \label{eq:Zt_variance} 
\Var_f(Z_t) \le \sum_{i=1}^n \E^{x_i}_f[Z_t^2] \le C e^{- f(t/a^2)+O(\|f\|/a)} \Big(\frac{t}{a^3} Z_0 + Y_0 \Big).
\end{equation}
Furthermore, we have for every $t\ge 0$ (without hypothesis on $\mu$),
\begin{equation}
 \label{eq:Yt_weak}
\E_f[Y_t] \le Ce^{- f(t/a^2)+O(\|f\|/a)}Y_0.
\end{equation}
and for $t\ge a^2$,
\begin{equation}
  \label{eq:Yt}
\E_f[Y_t] \le Ce^{- f(t/a^2)+O(\|f\|/a)} \frac{Z_0}{a}.
\end{equation}
Moreover, for every $a^2\le t\le a^3$, we have
\begin{equation}
  \label{eq:Yt_variance}
\Var_f(Y_t) \le \sum_{i=1}^n \E^{x_i}_f[Y_t^2] \le Ce^{- f(t/a^2)+O(\|f\|/a)} \frac{Y_0}a.
\end{equation}
\end{proposition}

\begin{proof} 
Equation \eqref{eq:Zt_expectation} follows from Lemma~\ref{lem:mu_t_density} and the fact that $Z_t$ is a martingale under $\P^x$. In order to show \eqref{eq:Yt_weak} and \eqref{eq:Yt}, it suffices by Lemma~\ref{lem:mu_t_density} to consider the case without variable drift. We first suppose that $t\ge a^2$. By \eqref{eq:density_bbm} and \eqref{eq:p_estimate}, we get
\begin{align*}
 \E^x[Y_t] &\le e^{\mu (x-a)}\int_0^a e^{\frac{\pi^2}{2a^2}t} p^a_t(x,y)\,\dd y
\le C e^{\mu (x-a)} \sin(\pi x/a)\int_0^a \frac 1 a \sin(\pi y/a)\,\dd y\\
&\le C e^{\mu (x-a)} \sin(\pi x/a).
\end{align*}
Summing over $x$ yields \eqref{eq:Yt} as well as \eqref{eq:Yt_weak} in the case $t\ge a^2$. Now, if $t< a^2$, by Lemma~\ref{lem:many_to_one_simple} and Girsanov's theorem, we have
\[
 \E^x[Y_t] = e^{\beta mt} W^x_{-\mu}\Big[e^{\mu (X_t-a)},\ H > t\Big] = e^{\pi^2t/(2a^2)} W^x(H > t) e^{\mu (x-a)} \le Ce^{\mu(x-a)}.
\]
Summing over $x$ yields \eqref{eq:Yt_weak}.

In order to prove \eqref{eq:Zt_variance}, we have by Lemma~\ref{lem:many_to_two},
\begin{equation*}
 \E^x_f[Z_t^2] = \E_f^x\Big[\sum_{u\in \mathscr A_{0,a}(t)}w_Z(X_u(t))^2\Big] + \beta m_2\int_0^a \int_0^t \p_s^f(x,y) (\E_f^{(y,s)}[Z_t])^2\,\dd s\,\dd y.
\end{equation*}
By Lemma~\ref{lem:mu_t_density}, \eqref{eq:Zt_expectation} and the fact that $f(t)-f(s)\ge -1$ for all $s<t$, this yields
\begin{equation}
 \label{eq:050}
 \E_f^x[Z_t^2] \le C e^{- f(t/a^2)+O(\|f\|/a)} \left(\E^x\Big[\sum_{u\in \mathscr A_{0,a}(t)}w_Z(X_u(t))^2\Big] + \int_0^a \int_0^t \p_s(x,y) w_Z(y)^2\,\dd s\,\dd y\right).
\end{equation}
Now we have
\(
 w_Z(x)^2 = (a\sin(\pi x/a) e^{-\mu(a-x)})^2 \le \pi^2 (a-x)^2e^{-2\mu(a-x)} \le C w_Y(x)
\)  for $x\in (0,a)$,
because $\mu \ge 1/2$ by hypothesis and so $(a-x)^2e^{-\mu(a-x)} \le C$ for $x\in(0,a)$. This yields
\begin{equation}
 \label{eq:S1}
 S_1 := \E^x\Big[\sum_{u\in \mathscr A_{0,a}(t)}w_Z(X_u(t))^2\Big] \le C \E^x[Y_t] \le C w_Y(x),
\end{equation}
by \eqref{eq:Yt_weak}. Now, by \eqref{eq:density_bbm} and \eqref{eq:Ja}, we have
\[
 S_2 := \int_0^a \int_0^t \p_s(x,y) w_Z(y)^2\,\dd s\,\dd y = ae^{\mu (x-a)}\int_0^aae^{\mu (y-a)}\sin^2(\pi y/a) J^a(x,y,[0,t])\,\dd y.
\]
Lemma~\ref{lem:I_estimates} and \eqref{eq:IJ_scaling} now give after the change of variables $y\mapsto a-y$, 
\begin{equation}
\label{eq:S2}
\begin{split}
 S_2 &\le C ae^{\mu (x-a)}\int_0^a e^{-\mu y}\sin^2(\pi y/a)\Big(t\sin(\pi x/a)\sin(\pi y/a) + ay\Big)\,\dd y\\
&\le C ae^{\mu (x-a)} \Big(\sin(\pi x/a)\frac t{a^3} + \frac 1 a\Big)\int_0^\infty e^{-\mu y} y^3\,\dd y,
\end{split}
\end{equation}
the last line following from the inequality $\sin x\le x$ for $x\ge0$. Using again the fact that $\mu \ge 1/2$, Equations \eqref{eq:050}, \eqref{eq:S1} and \eqref{eq:S2} now imply
\begin{equation}
 \label{eq:Zt_2ndmoment}
\E^x_f[Z_t^2] \le C e^{- f(t/a^2)+O(\|f\|/a)} \Big(\frac{t}{a^3} w_Z(x) + w_Y(x) \Big).
\end{equation}
Writing now the positions of the initial particles as $x_1,\ldots,x_n$, then by independence,
\begin{equation}
\label{eq:41}
 \Var_f(Z_t) = \sum_i \Var^{x_i}_f(Z_t) \le \sum_i \E^{x_i}_f[Z_t^2].
\end{equation}
Equations \eqref{eq:Zt_2ndmoment} and \eqref{eq:41} now prove \eqref{eq:Zt_variance}. Equation \eqref{eq:Yt_variance} is proven similarly.
\end{proof}

\subsection{The number of particles}
\label{sec:interval_number}
In this subsection, we establish precise first and second moment estimates for the number of particles which have not been absorbed until time $t$.  For $r\in[0,a]$ and $t\ge0$, we denote by $N_t(r)$ the number of not absorbed particles (i.e.\ individuals $u\in\mathscr A_{0,a}(t)$) located in the interval $[r,a]$ at time $t$. Recall the definition $E_t = \pi^2 \sum_{n=2}^{\infty} n^2e^{-\pi^2(n^2-1) t/2}$ from \eqref{eq:def_E}. The first result is useful when $t\gg a^2$:
\begin{lemma}
 \label{lem:N_expec_large_t}
 Suppose $\mu \ge 1/2$. Then for every $x,r\in(0,a)$ and $t > 0$, we have
\begin{equation}
 \label{eq:N_expec_large_t}
\E^x[N_t(r)] = \frac{2\pi (1+\mu r)e^{\mu (a-r)}}{ a^3} w_Z(x) \Big(1+\mathrm{err}\Big),
\end{equation}
with $|\mathrm{err}|\le E_{t/a^2}+O\Big((1+E_{t/a^2})\big(\tfrac{1+r}{a}\big)^2\Big)$.
\end{lemma}
\begin{proof}
By \eqref{eq:density_bbm} and \eqref{eq:p_estimate}, we have for every $x,r\in(0,a)$ and $t>0$,
\begin{equation*}
\E^x[N_t(r)] = (1+\mathrm{err}') w_Z(x) \frac {2 e^{\mu a}} {a^2} \int_r^a e^{-\mu y} \sin(\pi y/a)\,\dd y,
\end{equation*}
with $|\mathrm{err}'|\le E_{t/a^2}$. Now expand $\sin(\pi y/a)$ around $y=0$ and use $\int_r^\infty xe^{-\mu x}\,dx = \mu^{-2}(1+\mu r)e^{-\mu r}$ and $\int_r^\infty x^2e^{-\mu x} = O(\mu^{-3} (1+(\mu r)^2)e^{-\mu r})$. Together with  \eqref{eq:c0_mu} and the hypothesis on $\mu$, this yields \eqref{eq:N_expec_large_t}.
\end{proof}

The following proposition has a different formula for $\E^x[N_t(r)]$ which is useful for every $t\ge 0$. We do not use it directly 
in this article, but the two corollaries that follow it are invoked very often.
\begin{proposition}
\label{prop:N_expec}
Suppose $\mu \ge 1/2$. Let $t\ge0$, $x,r\in [0,a]$ and suppose that $x\ge (r+a/{20})\Ind_{(t\le a^2)}$. Then
\begin{equation*}
 \E^x[N_{t}(r)] = (1+\mu r)e^{\mu (x-r)}\left[-\frac{2e^{\pi^2 t/(2a^2)}}{a^2} \theta'\Big(\frac x a,\frac t {a^2}\Big) + O\left((1+r^2)\frac{\sin(\pi x/a) \vee \Ind_{(t\le a^2)}}{a^4}\right)\right].
\end{equation*}
\end{proposition}

\begin{proof}
 Let $t\ge0$ and $x,r\in[0,1]$. By \eqref{eq:density_bbm} and Brownian scaling, we have
\begin{equation}
 \label{eq:250}
 \E^{ax}[N_{ta^2}(ar)] = e^{\mu ax + \pi^2 t/2} \int_{r}^1 e^{-\mu a z} p_t^1(x,z)\,\dd z.
\end{equation}
Because of the exponential in the integral, the value of the integral is determined by the value of the integrand near the origin. This suggests an expansion of $p_t^1(x,z)$ around $z=0$. By \eqref{eq:p_theta}, we have $(\partial/\partial z) p_t^1(x,z)|_{z=0} = -2\theta'(x,t)$ and $(\partial/\partial z)^2 p_t^1(x,z)|_{z=0} = 0$. Taylor's formula then guarantees the existence of $\xi\in [0,z]$, such that
\begin{equation}
 \label{eq:252}
 p_t^1(x,z) = - 2z \theta'(x,t) + \frac {z^3} {6!} \frac{\partial^3}{\partial z^3} p_t^1(x,z)\big|_{z=\xi}.
\end{equation}
Now, if $t\ge 1$, we have by \eqref{eq:p_sin}, for every $z\in[0,1]$ and $\xi\in[0,z]$,
\begin{equation*}
 \Big|\frac{\partial^3}{\partial z^3} p_t^1(x,z)\big|_{z=\xi}\Big| = \pi^3\Big|\sum_{n=1}^\infty e^{-\pi^2 n^2 t/2} n^3 \sin(\pi n x)\cos(\pi n \xi)\Big| \le C \sin(\pi x) e^{-\pi^2 t/2},
\end{equation*}
by the inequality $|\sin nx| \le n \sin x$, $x\in[0,\pi]$. Similarly, by \eqref{eq:theta}, $|\theta'(x,t)| \le C\sin(\pi x) e^{-\pi^2 t/2}$ for all $x\in[0,1]$ and $t\ge 1$. With \eqref{eq:252}, this yields for $t\ge 1$,
\begin{equation}
 \label{eq:256}
 \int_{r}^1 e^{-\mu a z} p_t^1(x,z)\,\dd z = - \frac{2\theta'(x,t)}{\mu^2 a^2} (1+\mu ar) e^{-\mu ar} + O\Big((1+(ar)^3) e^{-\mu ar}\frac{\sin(\pi x)}{a^4} e^{-\pi^2 t/2}\Big).
\end{equation}
Now suppose that $t\le1$ and $x\ge r+1/{20}$. By \eqref{eq:p_theta} and \eqref{eq:theta_Cinfty}, we have for every $z\in[0,x-1/40]$, 
\begin{equation*}
 \Big|\frac{\partial^3}{\partial z^3} p_t^1(x,z)\Big| = |\theta'''(x+z,t)+\theta'''(x-z,t)| \le C
\end{equation*}
Furthermore, $|\theta'(z,t)| \le C$ for all $z\in[1/20,1]$ by \eqref{eq:theta_Cinfty}. With \eqref{eq:252},  this yields for $t\ge 1$, with $x' = x-1/40 \ge r+1/40$,
\begin{equation}
\label{eq:260}
 \int_{r}^{x'} e^{-\mu a z} p_t^1(x,z)\,\dd z = - \frac{2\theta'(x,t)}{\mu^2 a^2} (1+\mu ar)e^{-\mu ar} + O\Big(\frac{(1+(ar)^3)e^{-\mu ar}}{a^4}\Big).
\end{equation}
Furthermore, since $\int_0^1 p_t^1(x,z)\,\dd z \le 1$ for all $x\in[0,1]$ and $t\ge 0$, we have
\begin{equation}
 \label{eq:262}
\int_{x'}^1 e^{-\mu a z} p_t^1(x,z)\,\dd z \le e^{-\mu ax'} \int_{x'}^1 p_t^1(x,z)\,\dd z \le e^{-\mu ax'} \le e^{-\mu ar - \mu a/40}.
\end{equation}
Equations \eqref{eq:250}, \eqref{eq:256}, \eqref{eq:260}, \eqref{eq:262} and the hypothesis on $\mu$ now yield the proposition.
\end{proof}

\begin{corollary}
  \label{cor:N_expec_upper_bound}Suppose $\mu \ge 1/2$. Let $t\ge0$, $x,r\in
  [0,a]$ and suppose that $x\ge (r+a/{20})\Ind_{(t\le a^2)}$. Then,
  \begin{equation*}
\E^x[N_{t}(r)] \le C (1+r)\frac{e^{\mu (a-r)}}{a^3}\Big(w_Z(x)+\frac{1+r^2} a w_Y(x)\Ind_{(t\le a^2)}\Big).
  \end{equation*}
Furthermore, we have for all $x,r\in[0,a]$ and $t\ge 0$, $\E^x[N_{t}(r)] \le C e^{\mu(x-r)}.$
\end{corollary}
\begin{proof}
The second statement follows from the inequality $N_t(r) \le e^{\mu (a-r)} Y_t$ together with \eqref{eq:Yt_weak}. As for the first statement, it is enough by Lemma~\ref{lem:N_expec_large_t} to show it for $t\le a^2$. We then note that by \eqref{eq:theta_Cinfty}, we have $\theta''(x,t) \le C$ for $(x,t)\in [1/{20},1]\times[0,1]$. The mean value theorem and the fact that $\theta'(1,t)=0$ then give $\theta'(x,t) \le C(1-x)$. Furthermore, $1-x \le C\sin(\pi x)$ for $x\ge 1/20$. This, together with Proposition~\ref{prop:N_expec}, yields the statement.
\end{proof}

\begin{corollary}
 \label{cor:N_expec_thbar}
Suppose $\mu\ge 1/2$. Let $t\ge 0$ and $r\in[0,a]$. For every $x\in[r+a/20,a]$, we have,
\begin{equation*}
\E^x[N_{t}(r)] = \frac{2\pi(1+\mu r)e^{\mu (a-r)}}{ a^3}\Big[w_Z(x)\thbar\left(\frac t {a^2}\right)+O\left(\frac{1+r^2 + (a-x)^2}{a}w_Y(x)\right)\Big].
\end{equation*}
\end{corollary}
\begin{proof}
Let $x\in[0,1]$. By Taylor's formula and \eqref{eq:thbar}, there exists $\xi\in[x,1]$, such that for all $t\ge 0$,
\begin{equation*}
-e^{\frac{\pi^2}{2}t}\theta'(x,t) = \pi^2(1-x)\thbar(t) - \tfrac 1 2(1-x)^2e^{\frac{\pi^2}{2}t}\theta'''(\xi,t) = \pi\sin(\pi x)\thbar(t) + O((1-x)^2),
\end{equation*}
where the last equality follows from the fact that  $\pi(1-x) = \sin(\pi x) + O((1-x)^2)$ and that  $|e^{\frac{\pi^2}{2}t}\theta'''(x',t)| \le C$ for all $x'\ge 1/20$ and $t\ge0$, this follows from \eqref{eq:theta_Cinfty} for $t\in[0,1]$ and from \eqref{eq:theta} for $t\ge 1$. Together with Proposition~\ref{prop:N_expec}, this finishes the proof.
\end{proof}

We now turn to second moment estimates.

\begin{lemma}
\label{lem:N_2ndmoment}
Suppose $\mu\ge 1/2$ and $r\le 9a/10$. For every $t\ge 0$ and $x\in[0,a]$, we have,
\begin{equation*}
 \E^x_f[N_{t}(r)^2] \le Ce^{O(\|f\|/a)} \left(\frac{e^{\mu a}}{a^3}\right)^2 (1+r^4)e^{-2 \mu r}\Big(w_Z(x)  \frac t {a^3} + w_Y(x)\Big)
\end{equation*}
\end{lemma}
\begin{proof}
As in the proof of \eqref{eq:Zt_variance}, we have by Lemmas~\ref{lem:many_to_two} and \ref{lem:mu_t_density},
\begin{equation}
\label{eq:272}
 \E^x_f[N_{t}(r)^2] \le e^{- f(t/a^2)+O(\|f\|/a)} \Big(\E^x[N_{t}(r)] + \beta m_2 \int_0^a \int_0^t \p_s(x,z) (\E^z[N_{t-s}(r)])^2\,\dd s\,\dd z\Big).
\end{equation}
The first summand is easily handled by the second part of Corollary~\ref{cor:N_expec_upper_bound}:
\begin{equation}
 \label{eq:273}
 \E^x[N_{t}(r)] \le Ce^{\mu(x-r)} = Cw_Y(x) e^{\mu(a-r)} \le Cw_Y(x) \left(\frac{e^{\mu a}}{a^3}\right)^2 e^{-2 \mu r},
\end{equation}
where the last inequality follows from the hypotheses on $r$ and $\mu$. As for the second summand, we split the integral over $z$ into two pieces at $z=19a/20$ and apply the second and first parts of Corollary~\ref{cor:N_expec_upper_bound}, respectively. This yields for every $s\in[0,t]$,
\begin{multline*}
  \int_0^a \p_s(x,z) (\E^z[N_{t-s}(r)])^2\,\dd z \le C \left(\frac{e^{\mu a}}{a^3}\right)^2  e^{-2\mu r} \\
  \times \Big[ a^6 \int_0^{19a/20} \p_s(x,z) w_Y(z)^2\,\dd z + \int_{19a/20}^a \p_s(x,z) (1+r^2)\Big(w_Z(z)^2 + (1+r^2)w_Y(z)^2\Big)\,\dd z\Big]
\end{multline*}
Changing variables by $z\to a-z$ and applying \eqref{eq:density_bbm} and the inequality $\sin x \le x$ for $x\ge 0$, we get,
 \begin{multline}
\label{eq:278}
  \int_0^a \p_s(x,z) (\E^z[N_{t-s}(r)])^2\,\dd z \le C \left(\frac{e^{\mu a}}{a^3}\right)^2 e^{\mu(x-a-2r)+\pi^2s/(2a^2)}\\
\times \Big[a^6  \int_{a/{20}}^a p_s^a(x,a-z) e^{-\mu z} \,\dd z + (1+r^4) \int_0^{a/20} p_s^a(x,a-z) e^{-\mu z}(1+z^2) \,\dd z\Big].
 \end{multline}
Integrating over $s$ from $0$ to $t$ gives
\begin{multline*}
\int_0^a \int_0^t \p_s(x,z) (\E^z[N_{t-s}(r)])^2\,\dd s\,\dd z \le C \left(\frac{e^{\mu a}}{a^3}\right)^2 e^{\mu(x-a-2r)}\\
\times \Big[a^6  \int_{a/{20}}^a J^a(x,a-z,[0,t]) e^{-\mu z} \,\dd z + (1+r^4) \int_0^{a/20} J^a(x,a-z,[0,t]) e^{-\mu z}(1+z^2) \,\dd z\Big],
\end{multline*}
with $J^a$ from \eqref{eq:Ja}. By \eqref{eq:IJ_scaling} and Lemma~\ref{lem:I_estimates}, we have $J^a(x,a-z,[0,t]) \le Cz(t\sin(\pi x/a)/a^2 + 1)$. Together with the hypothesis on $\mu$, this gives,
\begin{equation}
\label{eq:280}
\int_0^a \int_0^t \p_s(x,z) (\E^z[N_{t-s}(r)])^2\,\dd s\,\dd z \le C\left(\frac{e^{\mu a}}{a^3}\right)^2 (1+r^4)e^{-2\mu r} \Big( w_Z(x) \frac t {a^3}+w_Y(x) \Big).
\end{equation}
The lemma then follows from \eqref{eq:272}, \eqref{eq:273} and \eqref{eq:280}, together with the fact that $f(t)\ge -1$ for all $t\ge0$ by definition.
\end{proof}

\subsection{The particles hitting the right border}
\label{sec:right_border}

For a measurable subset $S\subset \R$, define $R_S$ to be the number of particles killed at the right border during the (time) interval $S$, i.e.
\[
 R_S = \#\{(u,t):u\in \mathscr A(t)\tand H_0(X_u)>H_a(X_u) = t\in S\}.
\]
Furthermore, define $R_t = R_{[0,t]}$. The first lemma estimates first and second moments of these variables, in the case $f\equiv 0$.
\begin{lemma}
\label{lem:Rt}
 For any initial configuration $\nu$ and any $0\le s\le t$, we have
\begin{equation}
 \label{eq:Rt_nu_expec}
 |\E[R_{[s,t]}] - \frac{\pi (t-s)}{a^3} Z_0| \le C \Big(Y_0 \wedge E_{s/a^2} (1\wedge (t-s)/a^3) Z_0\Big),
\end{equation}
where $E_s$ is defined in \eqref{eq:def_E}. Furthermore, if $\mu \ge 1/2$ and $0\le t\le a^3$, then for each $x\in (0,a)$,
\begin{equation*}
 \E^x[R_t^2]\le C \Big(\frac{t}{a^3} w_Z(x) + w_Y(x)\Big),
\end{equation*}
\end{lemma}

\begin{proof}
Recall that $H_0$ and $H_a$ denote the hitting time functionals of $0$ and $a$ and $H = H_0\wedge H_a$. Note that $W^x_{-\mu}(H<\infty) = 1$ for all $x\in [0,a]$. By Lemma~\ref{lem:many_to_one_simple}, we then have
\begin{equation}
\label{eq:RS_x_expec}
\E^x[R_S] = \E^x\Big[\sum_{(u,t)\in \mathscr L_H} \Ind_{(X_u(t) = a,\,t\in S)}\Big] = W^x_{-\mu}\Big[e^{\beta mH_a} \Ind_{(H_0 > H_a\in S)}\Big]=e^{\mu(x-a)} I^a(x,S),
\end{equation}
where the last equality follows from Girsanov's transform and \eqref{eq:Ia}.
Equation \eqref{eq:Rt_nu_expec} now follows from Lemma~\ref{lem:I_estimates}. The second-moment estimate is obtained from Lemma~\ref{lem:many_to_two} and \eqref{eq:RS_x_expec} by a calculation similar to the one in the proof of Proposition~\ref{prop:quantities} (see \cite[Lemma~5.8]{MaillardThese} for details).
\end{proof}

The next lemma treats the case of general $f$. 
\begin{lemma}
 \label{lem:R_f}
For every $x\in (0,a)$ and every measurable $S\subset \R$, we have
\[
 \E^x_f[R_S] \le e^{1+ O(\|f\|/a)}\E^x[R_S].
\]
\end{lemma}
\begin{proof}
 As in the proof of Lemma~\ref{lem:Rt}, we have
\[
 \E^x_f[R_S] = W^x_{-\mu_t}\Big[e^{\beta mH_a} \Ind_{(H_0 > H_a\in S)}\Big].
\]
Similarly to the proof of Lemma~\ref{lem:mu_t_density}, we then have
\[
 \E^x_f[R_S] \le e^{O(\|f\|/a)}W^x_\mu\Big[e^{-\mu f(H_a/a^2)} e^{\beta mH_a} \Ind_{(H_0 > H_a\in S)}\Big] \le e^{1+O(\|f\|/a)}\E^x[R_S],
\]
where in the last inequality we used the fact that $\mu \le 1$ and $f(t) \ge -1$ for all $t\ge0$. This proves the lemma.
\end{proof}

\subsection{Penalizing the particles hitting the right border}
\label{sec:weakly_conditioned}
In this section, let $(V_u)_{u\in U}$ be iid random variables, uniformly distributed on $(0,1)$, independent of the branching Brownian motion. Furthermore, let $p:\R_+\to [0,1]$ be measurable and such that $p(t) = 0$ for large enough $t$. Recall that $H = H_0 \wedge H_a$. We define the event
\[
 E = \{\nexists (u,t)\in \mathscr L_H: X_u(t) = a \tand V_u \le p(t)\}.
\]
Define $\Ptilde_f^x = \P_f^x(\cdot|E)$. In the main text, we will set $p(t)$ to be for example the probability of a particle hitting $a$ at time $t$ to break out before some fixed time $t_0$ (see Section~\ref{sec:before_breakout_particles}).  Then $\Ptilde_f^x$ is the law of the tier 0 particles of BBM conditioned not to break out before time $t_0$.

In order to describe $\Ptilde_f^x$, set
\begin{align}
  \label{eq:def_tilde_quantities}
 h(x,t) = \P_f^{(x,t)}(E),\quad Q(x,t) = \sum_{k=0}^\infty q(k)h(x,t)^{k-1},\quad \widetilde q(x,t,k) = q(k) h(x,t)^{k-1} / Q(x,t)
\end{align}
Under $\Ptilde_f^x$, the BBM stopped at $\mathscr L_H$ is then again a branching Markov process where (see \cite[Section~3.4]{MaillardThese} and the errata in the arXiv version)
\begin{itemize}[nolistsep]
 \item particles move according to the law obtained from Brownian motion with drift $-\mu_t$ (stopped at $0$ and $a$) through a change of measure by the martingale 
\[
 h(X_{t\wedge H},t\wedge H) \exp\left(-\int_0^{t\wedge H} \beta (1-Q(X_s,s))\,\dd s\right).
\]
 \item a particle located at the point $x\in(0,a)$ at time $t$ branches at rate $\beta Q(x,t)\Ind_{x\in (0,a)}$, branching into $k$ offspring with probability $\widetilde q(x,t,k)$.
\end{itemize}

Together with  Lemma~\ref{lem:many_to_one_simple}, this immediately gives the following useful many-to-one lemma for the conditioned process stopped at the stopping line $\mathscr L_{H \wedge t}$: Define the function
\begin{equation}
 \label{eq:def_e}
 e(x,t) = \beta \sum_{k\ge 0} k (1-h(x,t)^{k-1}) q(k) \le \beta m_2(1-h(x,t)).
\end{equation}
\begin{lemma}
\label{lem:many_to_one_conditioned}
 For any measurable function $g:[0,a]\to \R_+$, $x\in (0,a)$ and $t\ge 0$, we have
\begin{equation}
\label{eq:many_to_one_conditioned}
 \Etilde_f^x\Big[\sum_{u\in \mathscr L_{H \wedge t}} g(X_u(t))\Big] = W_{-\mu_t}^x\Big[g(X_{H\wedge t})\frac{h(X_{H\wedge t},H\wedge t)}{h(x,0)} e^{(H\wedge t)/2 -\int_0^{H\wedge t} e(X_s,s) \,\dd s}\Big].
\end{equation}
In particular, if we denote by $\widetilde{\p}^f_t(x,y)$ the density of the $\Ptilde_f^x$-BBM before $\mathscr L_H$, then for $f\equiv 0$,
\begin{equation*}
\widetilde{\p}_t(x,y) =  \frac{h(y,t)}{h(x,0)} \p_t(x,y) W^{x,t,y}_{\mathrm{taboo}}\Big[e^{-\int_0^t e(X_s,s) \,\dd s}\Big],
\end{equation*}
and for general $f$,
\begin{equation}
  \label{eq:ptilde_estimate_f}
  \widetilde{\p}^f_t(x,y) \le \frac{h(y,t)}{h(x,0)} \p^f_t(x,y).
\end{equation}
\end{lemma}

This lemma immediately gives an upper bound for the quantities we are interested in:
\begin{corollary}
\label{cor:conditioned_upperbound}
 Let  $x\in (0,a)$, $t\ge 0$ and  $g:[0,a]\to \R_+$ be measurable with $g(0) = g(a) = 0$. Define $S_t = \sum_{u\in \mathscr L_{H \wedge t}} g(X_u(t))$. Then,
 \begin{align}
  \label{eq:conditioned_upperbound_1}
  \Etilde_f^x[S_t] &\le (h(x,0))^{-1} \E_f^x[S_t],\\
  \label{eq:conditioned_upperbound_2}
  \Etilde_f^x[S_t^2] &\le (h(x,0))^{-1} \E_f^x[S_t^2].
 \end{align}
\end{corollary}
\begin{proof}
Equation \eqref{eq:conditioned_upperbound_1} immediately follows from \eqref{eq:ptilde_estimate_f}. In order to prove the second-moment estimates, we note that by Lemma~\ref{lem:many_to_two} and the description of the conditioned process,
\[
 \begin{split}
  \Etilde_f^x[S^2_t] &= \Etilde_f^x\Big[\sum_{u\in \mathscr L_{H \wedge t}} g(X_u(t))^2\Big] + \int_0^t\int_0^a\widetilde{\p}^f_s(x,y) \widetilde{m_2}(y,s)\beta Q(y,s)\left(\Etilde_f^{(y,s)}[S_t]\right)^2\,\dd y\,\dd s,
 \end{split}
\]
where $\widetilde{m_2}(x,t) = \sum_{k\ge 0} k(k-1) \widetilde{q}(x,t,k).$ By \eqref{eq:def_tilde_quantities}, we have $\widetilde{m_2}(x,t)Q(x,t)\le h(x,t) m_2$ for all $x\in(0,a)$ and $t\ge0$. Using this inequality with \eqref{eq:ptilde_estimate_f} and \eqref{eq:conditioned_upperbound_1}, we obtain \eqref{eq:conditioned_upperbound_2} after another application of Lemma~\ref{lem:many_to_two}.
\end{proof}

The following lemma gives a good lower bound on the first-moment estimates in a special case.
\begin{lemma}
\label{lem:conditioned_Z_lowerbound}
Suppose $\mu \ge 1/2$, $t\le a^3$ and suppose that there exists $s_0\in[0,a^3]$, such that $p(s) = 0$ for $s\ge s_0$ and $f(s) = 0$ for $s\le s_0$. Let $S_t$ be as in Corollary \ref{cor:conditioned_upperbound}. We have
\begin{equation*}
 \Etilde_f^x[S_t] \ge \E_f^x[S_t](1-C\|p\|_\infty ).
\end{equation*}
\end{lemma}

This follows from the following estimate on $h(x,0)$.

\begin{lemma} Suppose $p(s) = 0$ for all $s\ge a^3$. Then for all $x\in(0,a)$, we have
 \label{lem:h_estimate}
\[
1-h(a-x,0) \le C \|p\|_\infty  e^{O(\|f\|/a)} (x + 1) e^{-\mu x}.
\]
\end{lemma}
\begin{proof}
By Markov's inequality, we have
\[
  1-h(x,0) \le \E_f^x(\#\{(u,s)\in \mathscr L_{H\wedge a^3}: X_u(s)=a,\ V_u \le \|p\|_\infty\}) \le \|p\|_\infty  \E_f^x(R_{a^3}).
\]
The lemma now follows from Lemmas~\ref{lem:Rt} and \ref{lem:R_f} and the inequality $\sin x\le x$, $x\in[0,\pi]$.
\end{proof}

\begin{proof}[Proof of Lemma~\ref{lem:conditioned_Z_lowerbound}]
For all $s\ge s_0$ and $x\in[0,a]$, we have by hypothesis $h(x,s) = 1$ and therefore $e(x,s) = 0$. With \eqref{eq:many_to_one_conditioned} and Lemma~\ref{lem:many_to_one_simple}, this gives
\begin{equation}
 \label{eq:210}
 \Etilde_f^x[S_t] \ge \E_f^x[S_t]\inf_{y\in(0,a)}\Big(h(y,t\wedge s_0)W_{\mathrm{taboo}}^{x,t\wedge s_0,y}\Big[e^{-\int_0^{t\wedge s_0} e(X_s,s) \,\dd s}\Big]\Big).
\end{equation}
By Lemma~\ref{lem:h_estimate} (with $f\equiv 0$) and the hypothesis on $\mu$, we have for every $y\in (0,a)$ and $s\le s_0$,
\begin{equation}
 \label{eq:218}
h(a-y,s) \ge h(a-y,0) \ge 1-C\|p\|_\infty e^{-y/3}.
\end{equation}
Together with \eqref{eq:def_e}, this gives, for every $y\in(0,a)$,
\begin{align}
\nonumber
 W_{\mathrm{taboo}}^{x,t\wedge s_0,y}\Big[e^{-\int_0^{t\wedge s_0} e(X_s,s) \,\dd s}\Big] 
 &\ge 1 - W_{\mathrm{taboo}}^{x,t\wedge s_0,y}\Big[\int_0^{t\wedge s_0} e(X_s,s) \,\dd s\Big] \\
  \label{eq:219}
 &\ge 1-W_{\mathrm{taboo}}^{x,t\wedge s_0,y}\Big[\int_0^{t\wedge s_0} Ce^{-(a-X_s)/3} \,\dd s\Big] 
 \ge 1-C\|p\|_\infty,
\end{align}
the last inequality following from Lemma~\ref{lem:k_integral}. The lemma now follows from \eqref{eq:210}, \eqref{eq:218} and \eqref{eq:219}.
\end{proof}

Finally, we study the law of $R_t$ under the new probability.
\begin{lemma}
 \label{lem:conditioned_Rt}
We have for every $x\in [0,a]$,
\begin{equation}
\label{eq:conditioned_Rt}
 \E_f^x[R_t] - \|p\|_\infty \E_f^x[R_t^2] \le \Etilde_f^x[R_t] \le (h(x,0))^{-1}\E_f^x[R_t],
\end{equation}
and if there is $p\in[0,1]$, such that $p(s) \equiv p$ for $s\le t$, then we even have
\begin{equation}
\label{eq:conditioned_Rt_constant}
 \Etilde_f^x[R_t] \le \E_f[R_t].
\end{equation}
\end{lemma}

\begin{proof}
Define the stopping line $\mathscr R_t = \{(u,s)\in \mathscr L_{H\wedge t}: X_u(s) = a\}$.
By the definition of the law $\Ptilde$,
\begin{equation}
 \label{eq:230}
 \Etilde_f^x[R_t] = \frac{\E_f^x\Big[R_t \prod_{(u,s)\in\mathscr R_t} (1-p(s))\Big]}{\E_f^x\Big[\prod_{(u,s)\in\mathscr R_t} (1-p(s))\Big]}.
\end{equation}
The denominator is $h(x,0)$ by \eqref{eq:def_tilde_quantities}, which yields the right-hand side of \eqref{eq:conditioned_Rt}. The left-hand side follows by noticing that
\[
 \E_f^x\Big[R_t \prod_{(u,s)\in\mathscr R_t} (1-p(s))\Big] \ge \E_f^x[R_t(1-\|p\|_\infty )^{R_t}] \ge \E_f^x[R_t] - \|p\|_\infty \E_f^x[R_t^2].
\]
For \eqref{eq:conditioned_Rt_constant}, we note that if $p(s) \equiv p$ for $s\le t$, then by \eqref{eq:230},
\[
 \Etilde_f^x[R_t] = \frac{\E_f^x[R_t(1-p)^{R_t}]}{\E_f^x[(1-p)^{R_t}]}.
\]
Since $(1-p)^k$ is decreasing in $k$, this yields \eqref{eq:conditioned_Rt_constant}.
\end{proof}

\subsection{Conditioning the penalised process to hit the right border at a given time}

In this section, we consider the law $\Ptilde^x$ from the previous section (with $f\equiv 0$) conditioned on the event that a particle hits the right border at a given time $t>0$, i.e.\ the conditional law $\Ptilde^x(\cdot\,|\,\exists \mathscr U \in U : (\mathscr U,t) \in \mathscr R_\infty)$, where $\mathscr R_\infty = \{(u,s)\in\mathscr L_H: X_u(s) = a\}$. For this, it is enough to consider expectation with respect to test functionals $Y$ of the form 
\begin{equation}
 \label{eq:test_Y}
 Y = \sum_{(u,s)\in \mathscr L_H} Y_u \Ind_{(u=\mathscr U)},
\end{equation}
where $Y_u$ is a non-negative $\F_{\mathscr L_H}$-measurable random variable for every $u\in U$ and $\mathscr U$ is the individual hitting the right border at time $t$.

For this, we define a branching Markov process with a selected genealogical line, called the \emph{spine}\footnote{Such processes are classical, see e.g.\ \cite{Chauvin1988} for an early example.}. Recall the definitions of $Q(x,t)$ and $\widetilde q(x,t,k)$ from \eqref{eq:def_tilde_quantities} and define $\widetilde m_1(x,t,k) = \sum_k k \widetilde q(x,t,k)$.
\begin{itemize}[nolistsep]
\item Initially, a single spine particle moves according to standard Brownian motion starting at $x\in\R$ and  absorbed at 0 and $a$.
\item As long as this spine particle has not been absorbed yet, it branches at rate $\widetilde m_1(y,s)R(y)$ when at position $y$ at time $s$, throwing offspring according to the size-biased distribution of $\widetilde q(y,s,\cdot)$ defined by $\widetilde q^*(y,s,k) = k\widetilde q(y,s,k)/\widetilde m_1(y,s)$.
\item Amongst those offspring, the next individual on the spine is chosen uniformly. This individual repeats the behaviour of its parent (started at the point $y$).
\item The other offspring initiate independent branching Markov processes according to the law $\Ptilde^y$, independently of the spine.
\end{itemize}
The law of and expectation w.r.t.\ this process are denoted by $\Ptilde^{*,x}$ and $\Etilde^{*,x}$, respectively. We further denote the individual on the spine alive at time $s$ by $\xi_s$ and the trajectory of the spine by $X_\xi$. We then have the following result, which resembles \cite[Theorem~1]{Chauvin1991a}. Recall the definition of $e(x,s)$ from \eqref{eq:def_e}.

\begin{lemma}
 \label{lem:many_to_one_fugitive}
Let $t \ge 0$, $x\in(0,a)$ and $Y$ be a functional as in \eqref{eq:test_Y}. Then,
\begin{equation*}
  \Etilde^x\Big[Y\,\Big|\,\exists \mathscr U \in U : (\mathscr U,t) \in \mathscr R_\infty\Big] = \frac{\Etilde^{*,x}\Big[Y_{\xi_t}e^{-\int_0^t e(X_\xi(s),s)\,\dd s} \,\Big|\,H_0(X_\xi) > H_a(X_\xi)=t\Big]}{\Etilde^{*,x}\Big[e^{-\int_0^t e(X_\xi(s),s)\,\dd s} \,\Big|\,H_0(X_\xi) > H_a(X_\xi)=t\Big]}.
\end{equation*}
\end{lemma}
\begin{proof}
Since almost surely, no two particles hit the point $a$ at the same time, we have
\begin{equation}
 \label{eq:600}
 \Etilde^x\Big[Y\,\Big|\,\exists \mathscr U\in U : (\mathscr U,t) \in \mathscr R_\infty\Big]
  = \frac{\Etilde^x\Big[\sum_{(u,s)\in \mathscr R_\infty}\Ind_{(H_a(X_u) \in \dd t)} Y_u\Big]}{\Etilde^x\Big[\sum_{(u,s)\in \mathscr R_\infty}\Ind_{(H_a(X_u) \in \dd t)}\Big]}.
\end{equation}
Lemma~\ref{lem:many_to_one_simple} is too restricted to deal with such an expectation. However, the ``full'' many-to-one lemma (see \cite{Kyprianou2004,Hardy2006,Harris2011} for proofs with fixed time and \cite[Lemma~14.1]{Biggins2004} for a proof with stopping lines in the discrete-time setting, which can be adapted to continuous time), yields
\begin{equation}
  \label{eq:47}
 \Etilde^x\Big[\sum_{(u,s)\in \mathscr R_\infty}\Ind_{(H_a(X_u) \in \dd t)} Y_u\Big] = \widetilde \Etilde^{*,x}\Big[Y_{\xi_t}e^{\int_0^t \widetilde m(X_\xi(s),s)\beta Q(X_\xi(s),s)\,\dd s} \Ind_{(H_0(X_\xi) > H_a(X_\xi)\in\dd t)}\Big],
\end{equation}
where $\widetilde \Etilde^{*,x}$ is defined as $\Etilde^{*,x}$ but with the motion of the spine being the same as the motion of the other particles. According to the description of $\Ptilde^x$ in Section \ref{sec:weakly_conditioned}, the law of this motion is obtained from Brownian motion with drift $-\mu$ (stopped at $0$ and $a$) through a change of measure by the martingale
\(
\left(h(X_{s\wedge H},s\wedge H) \exp(-\int_0^{s\wedge H} \beta (1-Q(X_r,r))\,\dd r)\right)_{s\ge 0},
\)
where $h$ and $Q$ are defined in \eqref{eq:def_tilde_quantities}. With Girsanov's transform, \eqref{eq:47} then yields,
  \begin{equation}
    \label{eq:57}
  \Etilde^x\Big[\sum_{(u,s)\in \mathscr R_\infty}\Ind_{(H_a(X_u) \in \dd t)} Y_u\Big] = e^{-\mu(a-x)+\frac{\pi^2}{2a^2}t}\frac{h(a,t)}{h(x,0)} \Etilde^{*,x}\Big[Y_{\xi_t}e^{-\int_0^t e(X_\xi(s),s)\,\dd s} \Ind_{(H_0(X_\xi) > H_a(X_\xi)\in\dd t)}\Big].
  \end{equation}
Plugging  \eqref{eq:57} into \eqref{eq:600}, with $Y_u\equiv 1$ in the denominator, yields the lemma.
\end{proof}

\begin{corollary}
\label{cor:many_to_one_fugitive}
In addition to the assumptions in Lemma~\ref{lem:many_to_one_fugitive}, suppose $\mu\ge 1/2$, $t\in [0,a^3]$ and $p(s)= 0$ for $s\ge a^3$. Then,
\begin{equation*}
 \Etilde^x\Big[Y\,\Big|\,\exists \mathscr U\in U : (\mathscr U,t) \in \mathscr R_\infty\Big]
 \le e^{C\|p\|_\infty} \Etilde^{*,x}\Big[Y_{\xi_t}\,\Big|\, H_0(X_\xi) > H_a(X_\xi)\in\dd t\Big].
\end{equation*}
\end{corollary}
\begin{proof}
By \eqref{eq:def_e}, Lemma~\ref{lem:h_estimate} and the hypothesis on $\mu$, we have $e(x,s) \le C\|p\|_\infty e^{-(a-x)/3}$ for all $x\in(0,a)$ and $s\ge 0$. By Lemma~\ref{lem:k_integral}, we then have
\(
 W^{x,t,a}_{\mathrm{taboo}}[\int_0^t e(X_s,s)\,\dd s] \le C\|p\|_\infty.
\)
Jensen's inequality now yields 
\begin{equation}
  \label{eq:610}
 \Etilde^{*,x}\Big[e^{-\int_0^t e(X_\xi(s),s)\,\dd s} \,\Big|\,H_0(X_\xi) > H_a(X_\xi)=t\Big] = W^{x,t,a}_{\mathrm{taboo}}[e^{-\int_0^t e(X_s,s)\,\dd s}] \ge e^{-C\|p\|_\infty}.
\end{equation}
The statement now follows from \eqref{eq:610} and Lemma~\ref{lem:many_to_one_fugitive}.
\end{proof}

\section*{Acknowledgments}
\addcontentsline{toc}{section}{Acknowledgements}
I wish to thank my PhD advisor Zhan Shi for his continuous support and encouragements and Julien Berestycki for numerous discussions and his interest in this work. I further thank two anonymous referees who read and checked the manuscript in detail, gave very valuable suggestions regarding the organisation of the paper and spotted several typographical errors.

{
\small
\bibliography{n-bbm1}
}
\end{document}